\documentclass[reqno, 10pt]{amsart}

\usepackage[usenames,dvipsnames]{color}

\usepackage{amsthm,amsfonts,amssymb,amsmath,amsxtra}
\usepackage[all]{xy}
\SelectTips{cm}{}
\usepackage{xr-hyper}
\usepackage[colorlinks=citecolor=Black, linkcolor=Red,
   urlcolor=Blue]{hyperref}
\usepackage{verbatim}

\usepackage[margin=1.45in]{geometry}
\usepackage{mathrsfs}

\RequirePackage{xspace}
\RequirePackage{etoolbox}
\RequirePackage{varwidth}
\RequirePackage{enumitem}
\RequirePackage{tensor}
\RequirePackage{mathtools}
\RequirePackage{longtable}
\RequirePackage{multirow}

\usepackage{scalerel}

\newcommand\reallywidehat[1]{\arraycolsep=0pt\relax%
\begin{array}{c}
\stretchto{
  \scaleto{
    \scalerel*[\widthof{\ensuremath{#1}}]{\kern-.5pt\bigwedge\kern-.5pt}
    {\rule[-\textheight/2]{1ex}{\textheight}} 
  }{\textheight} %
}{0.8ex}\\           
#1\\                 
\rule{-1ex}{0ex}
\end{array}
}

\newcommand{\bdot}{\boldsymbol{.}}
  \newcommand{\bdtimes}{\buildrel{\bdot}\over\times}

\newcommand{\sB}{\ensuremath{\mathscr{B}}\xspace}

\newcommand{\sE}{\ensuremath{\mathscr{E}}\xspace}

\newcommand{\sG}{\ensuremath{\mathscr{G}}\xspace}
\newcommand{\sH}{\ensuremath{\mathscr{H}}\xspace}

\newcommand{\sM}{\ensuremath{\mathscr{M}}\xspace}
\newcommand{\sN}{\ensuremath{\mathscr{N}}\xspace}

\newcommand{\sP}{\ensuremath{\mathscr{P}}\xspace}
\newcommand{\sQ}{\ensuremath{\mathscr{Q}}\xspace}
\newcommand{\sR}{\ensuremath{\mathscr{R}}\xspace}
\newcommand{\sS}{\ensuremath{\mathscr{S}}\xspace}
\newcommand{\sT}{\ensuremath{\mathscr{T}}\xspace}

\newcommand{\sV}{\ensuremath{\mathscr{V}}\xspace}
\newcommand{\sW}{\ensuremath{\mathscr{W}}\xspace}
\newcommand{\sX}{\ensuremath{\mathscr{X}}\xspace}
\newcommand{\sY}{\ensuremath{\mathscr{Y}}\xspace}
\newcommand{\sZ}{\ensuremath{\mathscr{Z}}\xspace}

\newcommand{\eK}{{\sf K}}
\newcommand{\eE}{{\sf E}}

\newcommand{\eH}{{\sf H}}

\newcommand{\eG}{{\sf G}}

\newcommand{\fkm}{\ensuremath{\mathfrak{m}}\xspace}

\newcommand{\fkM}{\ensuremath{\mathfrak{M}}\xspace}

\newcommand{\fkX}{\ensuremath{\mathfrak{X}}\xspace}

\newcommand{\fkZ}{\ensuremath{\mathfrak{Z}}\xspace}

\newcommand{\mar}[1]{\marginpar{\tiny #1}}

\newcommand{\et}{{\text{\rm \'et}}}

\newcommand{\diam}{{\Diamond}}

\newcommand{\sdiam}{{\blacklozenge}}

 \newcommand{\bDiam}{\blacklozenge}

\renewcommand{\inf}{{\rm inf}}

\newcommand{\BA}{\ensuremath{\mathbb {A}}\xspace}
\newcommand{\BB}{\ensuremath{\mathbb {B}}\xspace}
\newcommand{\BC}{\ensuremath{\mathbb {C}}\xspace}
\newcommand{\BD}{\ensuremath{\mathbb {D}}\xspace}

\newcommand{\BF}{\ensuremath{\mathbb {F}}\xspace}
\newcommand{\BG}{\ensuremath{\mathbb {G}}\xspace}

\newcommand{\BL}{\ensuremath{\mathbb {L}}\xspace}
\newcommand{\BM}{\ensuremath{\mathbb {M}}\xspace}
\newcommand{\BN}{\ensuremath{\mathbb {N}}\xspace}

\newcommand{\BP}{\ensuremath{\mathbb {P}}\xspace}
\newcommand{\BQ}{\ensuremath{\mathbb {Q}}\xspace}
\newcommand{\BR}{\ensuremath{\mathbb {R}}\xspace}

\newcommand{\BV}{\ensuremath{\mathbb {V}}\xspace}
\newcommand{\BW}{\ensuremath{\mathbb {W}}\xspace}
\newcommand{\BX}{\ensuremath{\mathbb {X}}\xspace}

\newcommand{\BZ}{\ensuremath{\mathbb {Z}}\xspace}

\newcommand{\CA}{\ensuremath{\mathcal {A}}\xspace}
\newcommand{\CB}{\ensuremath{\mathcal {B}}\xspace}

\newcommand{\CE}{\ensuremath{\mathcal {E}}\xspace}
\newcommand{\CF}{\ensuremath{\mathcal {F}}\xspace}
\newcommand{\CG}{\ensuremath{\mathcal {G}}\xspace}
\newcommand{\CH}{\ensuremath{\mathcal {H}}\xspace}
\newcommand{\CI}{\ensuremath{\mathcal {I}}\xspace}

\newcommand{\CL}{\ensuremath{\mathcal {L}}\xspace}
\newcommand{\CM}{\ensuremath{\mathcal {M}}\xspace}

\newcommand{\CO}{\ensuremath{\mathcal {O}}\xspace}
\newcommand{\CP}{\ensuremath{\mathcal {P}}\xspace}

\newcommand{\CV}{\ensuremath{\mathcal {V}}\xspace}

\newcommand{\CX}{\ensuremath{\mathcal {X}}\xspace}
\newcommand{\CY}{\ensuremath{\mathcal {Y}}\xspace}

\newcommand{\ab}{{\mathrm{ab}}}

\newcommand{\ad}{{\mathrm{ad}}}

\DeclareMathOperator{\Aut}{Aut}

\DeclareMathOperator{\Gal}{Gal}
\newcommand{\GL}{\mathrm{GL}}

\newcommand{\GSp}{\mathrm{GSp}}

\newcommand{\id}{\ensuremath{\mathrm{id}}\xspace}

\DeclareMathOperator{\Ker}{Ker}

\newcommand{\loc}{\ensuremath{\mathrm{loc}}\xspace}

\newcommand{\red}{\ensuremath{\mathrm{red}}\xspace}

\DeclareMathOperator{\Res}{Res}

\newcommand{\RRZ}{{\rm {RZ}}}

\DeclareMathOperator{\Spa}{Spa\,}
\DeclareMathOperator{\Spec}{Spec\,}
\DeclareMathOperator{\Spd}{Spd\,}
\DeclareMathOperator{\Spf}{Spf\,}

\newcommand{\Sp}{{\mathrm{Sp}}}

\newcommand{\wt}{\widetilde}
\newcommand{\wh}{\widehat}

\newcommand{\norm}[1]{\|{#1}\|}

\newcommand{\ov}{\overline}
\newcommand{\incl}{\hookrightarrow}

\newcommand{\crys}{{\rm crys}}

\newcommand{\lps}{[\![}
\newcommand{\rps}{]\!]}
\newcommand{\llps}{(\!(}
\newcommand{\lrps}{)\!)}

\newtheorem{theorem}{Theorem}
\newtheorem{proposition}[theorem]{Proposition}
\newtheorem{prop/constr}[theorem]{Proposition/Construction}
\newtheorem{lemma}[theorem]{Lemma}
\newtheorem {conjecture}[theorem]{Conjecture}
\newtheorem{corollary}[theorem]{Corollary}

\theoremstyle{definition}

\newtheorem{definition}[theorem]{Definition}
\newtheorem{example}[theorem]{Example}

\newtheorem{remark}[theorem]{Remark}
\newtheorem{remarks}[theorem]{Remarks}

\newenvironment{altenumerate}
   {\begin{list}
      {\textup{(\theenumi)} }
      {\usecounter{enumi}
       \setlength{\labelwidth}{0pt}
       \setlength{\labelsep}{0pt}
       \setlength{\leftmargin}{0pt}
       \setlength{\itemsep}{\the\smallskipamount}
       \renewcommand{\theenumi}{\roman{enumi}}
      }}
   {\end{list}}
\newenvironment{altitemize}
   {\begin{list}
      {$\bullet$}
      {\setlength{\labelwidth}{0pt}
	   \setlength{\itemindent}{5pt}
       \setlength{\labelsep}{5pt}
       \setlength{\leftmargin}{0pt}
       \setlength{\itemsep}{\the\smallskipamount}
      }}
   {\end{list}}

\numberwithin{equation}{subsection}
\numberwithin{theorem}{subsection}

\setitemize[0]{leftmargin=0.2in,itemsep=\the\smallskipamount}
\setenumerate[0]{leftmargin=0.2in,itemsep=\the\smallskipamount}

\renewcommand{\to}{%
   \ifbool{@display}{\longrightarrow}{\rightarrow}%
   }
\let\shortmapsto\mapsto
\renewcommand{\mapsto}{%
   \ifbool{@display}{\longmapsto}{\shortmapsto}%
   }
\newlength{\olen}
\newlength{\ulen}
\newlength{\xlen}
\newcommand{\xra}[2][]{%
   \ifbool{@display}%
      {\settowidth{\olen}{$\overset{#2}{\longrightarrow}$}%
       \settowidth{\ulen}{$\underset{#1}{\longrightarrow}$}%
       \settowidth{\xlen}{$\xrightarrow[#1]{#2}$}%
       \ifdimgreater{\olen}{\xlen}%
          {\underset{#1}{\overset{#2}{\longrightarrow}}}%
          {\ifdimgreater{\ulen}{\xlen}%
             {\underset{#1}{\overset{#2}{\longrightarrow}}}
             {\xrightarrow[#1]{#2}}}}%
      {\xrightarrow[#1]{#2}}
   }
\makeatother
\newcommand{\xyra}[2][]{%
   \settowidth{\xlen}{$\xrightarrow[#1]{#2}$}%
   \ifbool{@display}%
      {\settowidth{\olen}{$\overset{#2}{\longrightarrow}$}%
       \settowidth{\ulen}{$\underset{#1}{\longrightarrow}$}%
       \ifdimgreater{\olen}{\xlen}%
          {\mathrel{\xymatrix@M=.12ex@C=3.2ex{\ar[r]^-{#2}_-{#1} &}}}%
          {\ifdimgreater{\ulen}{\xlen}%
             {\mathrel{\xymatrix@M=.12ex@C=3.2ex{\ar[r]^-{#2}_-{#1} &}}}
             {\mathrel{\xymatrix@M=.12ex@C=\the\xlen{\ar[r]^-{#2}_-{#1} &}}}}}%
      {\mathrel{\xymatrix@M=.12ex@C=\the\xlen{\ar[r]^-{#2}_-{#1} &}}}%
   }
\makeatletter
\newcommand{\xla}[2][]{%
   \ifbool{@display}%
      {\settowidth{\olen}{$\overset{#2}{\longleftarrow}$}%
       \settowidth{\ulen}{$\underset{#1}{\longleftarrow}$}%
       \settowidth{\xlen}{$\xleftarrow[#1]{#2}$}%
       \ifdimgreater{\olen}{\xlen}%
          {\underset{#1}{\overset{#2}{\longleftarrow}}}%
          {\ifdimgreater{\ulen}{\xlen}%
             {\underset{#1}{\overset{#2}{\longleftarrow}}}
             {\xleftarrow[#1]{#2}}}}%
      {\xleftarrow[#1]{#2}}
   }
\newcommand{\isoarrow}{%
   \ifbool{@display}{\overset{\sim}{\longrightarrow}}{\xrightarrow\sim}%
   }

\newcommand{\proet}{{\rm proet}}
\newcommand{\bfB}{{\mathbf B}}

\newcommand{\quash}[1]{}

 \usepackage{relsize} 
\usepackage[bbgreekl]{mathbbol} 
\usepackage{amsfonts} 
\DeclareSymbolFontAlphabet{\mathbb}{AMSb} 
\DeclareSymbolFontAlphabet{\mathbbl}{bbold} 
\newcommand{\Prism}{{\mathlarger{\mathbbl{\Delta}}}}

\newcommand{\pri}{{\scriptstyle\Prism}}
 
\newcommand{\LMint}{\wh{ L{\mathcal {M}}^{\rm int}_{\CG, b, \mu}}_{/x_0}}
\newcommand{\LG}{\wh{{L}^+_W\CG}}

\newcommand{\LMs}{\wh{ L{\mathcal {M}}^{\rm int}_{/x_0}}}

 \newcommand{\lpstperf}{{\llps t^{1/p^\infty}\lrps}}
  \newcommand{\pstperf}{{\lps t^{1/p^\infty}\rps}}

\newcommand{\Perf}{{\rm Perfd}}

\newcommand{\cred}{ }
\newcommand{\cmag}{}

\begin{document}

\title{$p$-adic shtukas and the theory of  global and local  Shimura varieties}
\author[G. Pappas]{Georgios Pappas}
\address{Dept. of Mathematics, Michigan State University, E. Lansing, MI 48824, USA}
\email{pappasg@msu.edu}

\author[M. Rapoport]{Michael Rapoport}
\address{Mathematisches Institut der Universit\"at Bonn, Endenicher Allee 60, 53115 Bonn, Germany, and University of Maryland, Department of Mathematics, College Park, MD 20742, USA}
\email{rapoport@math.uni-bonn.de}

\date{\today}

\begin{abstract}{We establish basic results on $p$-adic shtukas and apply them to 
 the theory of local and  global Shimura varieties, and on their interrelation.   We  construct canonical integral models for (local, and global) Shimura varieties of Hodge type 
 with parahoric level structure.}

\end{abstract}

\maketitle

\tableofcontents
 
\section{Introduction} 
Shimura varieties were formally defined  in Deligne's Bourbaki seminar talk \cite{DelBour}.
 Let $(\eG, X)$ be a Shimura datum, in the sense of Deligne. Attached to $(\eG, X)$ is a pro-system of quasi-projective algebraic varieties ${\rm Sh}_\eK(\eG, X)$ over $\BC$ whose members are enumerated by the open compact subgroups $\eK$ of the finite ad\`ele group $\eG(\BA_f)$.  These have canonical models  ${\rm Sh}_\eK(\eG, X)_\eE$ over the reflex field $\eE\subset \bar\BQ$ of $(\eG, X)$. Throughout the paper, we will impose the blanket assumption that the $\BQ$-split rank and the $\BR$-split rank of the center of $G$ coincide.

 Part of Deligne's philosophy of Shimura varieties is that (if the associated central character $w_X$ is defined over $\BQ$ and the connected center of $G$ splits over a CM-field \cite[Lem. B 3.9]{Reim})  the Shimura variety ${\rm Sh}_\eK(\eG, X)_\eE$ is a moduli space of motives over $\Spec(\eE)$ (with level-$\eK$-structure).  This philosophy has been a guiding principle behind much of the work on Shimura varieties in the last decades. However, a draw-back of this idea is that the concept of a motive is still eluding a precise definition. Scholze, in a lecture in Jan. 2019 in Essen, suggested that it might be profitable to instead view Shimura varieties as moduli spaces of certain shtukas \emph{over $\BZ$}, cf. also his ICM talk in Rio \cite{SchICM} and his Berkeley lectures \cite{Schber}. In particular, he suggested constructing  a ``universal" $\eG$-shtuka over the Shimura variety ${\rm Sh}_\eK(\eG, X)_\eE$. 
 Here a $\eG$-shtuka should be the number field analogue of the concept of $\eG$-shtuka for global function fields, cf. \cite[\S 11.1]{Schber}. However, at the moment it is also not clear how to define this concept. 

Let $p$ be a prime number. Scholze is able to define the concept of a shtuka \emph{over $\BZ_p$}, comp. \cite{Schber}. Therefore, after base change to the completion $E=\eE_v$ of $\eE$  at some $p$-adic place $v$, the Shimura variety ${\rm Sh}_\eK(\eG, X)_E={\rm Sh}_\eK(\eG, X)_\eE\otimes_\eE E$ should come with a family of $\BZ_p$-shtukas.  In fact, Scholze uses this insight to define \emph{local Shimura varieties} which are the analogues over $p$-adic fields of Shimura varieties. This puts into reality a hope spelled out in \cite{RV} and is a $p$-adic avatar of Scholze's idea on global Shimura varieties.  

Let $(G, b, \mu)$ be a local Shimura datum. We recall \cite{RV} that this means that $G$ is a reductive group over $\BQ_p$, that $b\in G(\breve\BQ_p)$, and that $\mu$ is a conjugacy class of minuscule cocharacters of $G_{\bar\BQ_p}$. Let $E\subset \bar\BQ_p$ be the local reflex field. Then Scholze associates to $(G, b, \mu)$ a pro-system of rigid-analytic spaces over the completion $\breve E$ of the maximal unramified extension of $E$ (with Weil descent datum down to $E$) whose members are enumerated by the open compact subgroups $K\subset G(\BQ_p)$.  These spaces are moduli spaces $\CM_{G, b, \mu}=(\CM_{G, b, \mu, K})_{ K\subset G(\BQ_p)}$ of $G$-shtukas (with level-$K$-structures). One important class of local Shimura varieties arises from Rapoport-Zink spaces: if the local Shimura datum arises from rational RZ data \cite{R-Z}, then the RZ tower associated with the corresponding RZ space is a local Shimura variety, cf. \cite[Cor. 24.3.5]{Schber}. In this special case, the theory of non-archimedean uniformization of Shimura varieties of PEL-type provides a link between global Shimura varieties and local Shimura varieties, cf. \cite[chap. 6]{R-Z}.

In this paper we develop Scholze's idea. We establish basic results on $p$-adic shtukas and apply them to 
 the theory of local and  global Shimura varieties, and on their interrelation.   We derive some 
 interesting consequences,
 such as a construction of canonical integral models for (local, and global) Shimura varieties of Hodge type 
 with parahoric level structure.

Let us now be more specific. In our formulations, we will use Scholze's language of diamonds, $v$-sheaves, etc.

\subsection{Results on shtukas} Let us start with stating our results\footnote{{\cmag Some of these results have been extended in very recent works of Gleason-Ivanov \cite{GlIv} and G\"uthge \cite{Guth}, which were completed during the refereeing process of this paper.}} on shtukas. Let $k$ be an algebraically closed field of characteristic $p$. For most of the time, $k$ is the residue field of $\breve E$. We recall from \cite{Schber} that a \emph{shtuka of height $h$} over a perfectoid space $S=\Spa (R, R^+)\in {\Perf}_k$ with leg at the untilt $S^\sharp$ of $S$ is a vector bundle of rank $h$ on the analytic adic space $S\bdtimes  \BZ_p$, together with a meromorphic Frobenius map $\phi_\sP$ which has a pole along the Cartier divisor $S^\sharp$ of $S\bdtimes  \BZ_p$. Here
$$
S\bdtimes  \BZ_p  =S\bdtimes {\rm Spa}(\BZ_p)= {\rm Spa}(W(R^+))\setminus \{[\varpi]= 0\} ,
$$
where $[\varpi]$ is the Teichm\" uller lift of a pseudo-uniformizer of $R^+$. 
For $\mu\in (\BZ^h)_{\geq}$, there is also the notion of a shtuka of rank $h$ bounded by $\mu$. Given a smooth group scheme $\CG$ over $\BZ_p$, there is also the variant notion of a $\CG$-shtuka over $S$ and, given a conjugacy class of cocharacters $\mu$ of the generic fiber $G$ of $\CG$, also the variant notion of a $\CG$-shtuka  over $S$ bounded by $\mu$. {\cmag Finally,  there are corresponding notions of families of such objects over   adic spaces or schemes, which are defined, roughly speaking, as 
sections of the $v$-stack of shtukas 
 over the diamond or $v$-sheaf attached to the adic space or to the scheme,  see \S  \ref{sss:vs}--\ref{sss:vsSch}, and Definitions \ref{deffamofsht}, \ref{def:shtsch} below.} As an illustration, a shtuka over $\Spec(K)$, for a perfect field $K$ of characteristic $p$, can be shown to correspond to a free $W(K)$-module $M$ of finite rank, equipped with an isomorphism $\phi_M\colon {\rm Frob}^*(M)[1/p]\isoarrow M[1/p]$, comp. Theorem \ref{FFisocrystal-Intro} below. For instance, a $p$-divisible group $\sG$ over a scheme of finite type over $O_E$ defines a shtuka of height equal to the height of $\sG$ and bounded by $\mu=(1^{(d)}, 0^{(h-d)})$, where $d$ is the dimension of $\sG$. 

When the adic space is in characteristic zero, there is a more explicit description of $\CG$-shtukas which is \'etale-sheaf theoretic in nature, provided that $\mu$ is minuscule:

 \begin{proposition}[see Proposition \ref{pairs2}]\label{pairs2Intro} 
 Let $X$ be a locally Noetherian adic space over $\Spa(E,O_E)$, with associated diamond $X^\diam$. Fix $(\CG, \mu)$, where $\CG$ is a  connected smooth model over $\BZ_p$ of a reductive group $G$ over $\BQ_p$, and where $\mu$ is a conjugacy class of minuscule cocharacters of $G_{\bar \BQ_p}$. 
 There is a functor $\sP\mapsto (\BP, {\rm HT}(\sP))$ which gives an equivalence between the categories of:
 
\noindent 1) $\CG$-shtukas $(\sP, \phi_\sP)$ over $X$ with one leg  bounded by $\mu$, and 
 
 \smallskip
 
 \noindent 2) pairs $(\BP, {\rm H})$ consisting 
 of a pro-\'etale $\underline{\CG(\BZ_p)}$-torsor $\BP$ over $X^\diam$ and a $\underline{\CG(\BZ_p)}$-equivariant map of sheaves
 ${\rm H}: \BP\to \CF^\diam_{G, \mu^{-1}}$ over $\Spd E$.\hfill
  \end{proposition}
 Here  $\CF_{G, \mu^{-1}}$ denotes a partial flag variety and ${\rm HT}$ is the sheaf-theoretic analogue of Scholze's Hodge-Tate period map.  The latter is also constructed by Hansen \cite{HansenPre}. 
 
 Let us now assume that $X$ is the adic space attached to a smooth scheme $\CX$. Then, under some hypotheses, a pro-\'etale $\CG(\BZ_p)$-cover defines a pair as in Proposition \ref{pairs2Intro} above:
 \begin{proposition}[see Proposition \ref{uniqueHT}]\label{detorsIntro}
 Let $\BP$ be a pro-\'etale $\CG(\BZ_p)$-cover over the smooth $E$-scheme $\CX$ which is de Rham and bounded by the minuscule cocharacter $\mu$. Then there is a natural $\CG$-shtuka bounded by $\mu$ over $\CX$ which under the correspondence of Proposition \ref{pairs2Intro} arises from the $\underline{\CG(\BZ_p)}$-torsor defined by $\BP$ and the Hodge-Tate period map defined by $\BP$. 
 \end{proposition}
 The property of being de Rham is defined by Scholze in \cite{SchpHodge}. It implies that at every classical point $x$ of $\CX$, the fiber $\BP_x$ is a Galois representation of $\Gal(\ov{E(x)}/E(x))$ in $G(\BQ_p)$ of \emph{de Rham type} and is associated by Fontaine's functor ${\rm D_{dR}}$ to a filtered $G$-isocrystal with filtration type given by $\mu$.  The map ${\rm H}$ maps $x$ to the point in the flag variety given by the corresponding filtration of the $G$-isocrystal. 
 
 Shtukas in characteristic $p$  
 are crystalline in nature. This is more transparent when the topology on the rings is discrete. 
 Let $\sX=\Spec(A)$ be a perfect $k$-scheme. A \emph{meromorphic Frobenius crystal} over $\sX$ is  a vector bundle  $\sM$ over $\Spec(W(A))$ equipped with an isomorphism
\begin{equation}
\phi_\sM\colon {\rm Frob}^*(\sM)[{1}/{p}]\xrightarrow{ \sim\ } \sM[{1}/{p}] .
\end{equation}

 \begin{theorem}[see Theorem \ref{FFisocrystal}]\label{FFisocrystal-Intro}
There is an exact tensor equivalence between the category of meromorphic Frobenius crystals over $\Spec(A)$ and the category of shtukas  over $\Spec(A)$.  
\end{theorem}
This statement is subtle, even when $A$ is an algebraically closed field. In this case an essential ingredient in the proof is the relation between Frobenius isocrystals and vector bundles over the Fargues-Fontaine curve (see \cite{FarguesFontaine}, \cite{An2}).

The relation between characteristic zero and characteristic $p$ is given by  the following shtuka analogue of Tate's theorem on homomorphisms of $p$-divisible groups.

\begin{theorem}[see Theorem \ref{vshtExt}]\label{vshtExtIntro}
Let $\sX$ be a separated scheme which is normal and of finite type and flat over $\Spec(O_E)$.  Denote by $\CX=\sX\times_{\Spec(O_E)}\Spec(E)$ the generic fiber. Let $(\sV, \phi_\sV)$ and $(\sV', \phi_{\sV'})$ be two shtukas over $\sX$. Any homomorphism   $\psi_\CX: {(\sV, \phi_\sV)}_{|\CX}\to {(\sV', \phi_{\sV'})}_{|\CX}$ between their restrictions to $\CX$ extends uniquely to a homomorphism $\psi: (\sV, \phi_\sV)\to (\sV', \phi_{\sV'})$. 
\end{theorem}
In particular, using the Tannakian formalism, there is at most one extension of a $\CG$-shtuka on $\CX$ to a $\CG$-shtuka on $\sX$.

An important ingredient of the proof of Theorem \ref{vshtExtIntro} is the full-faithfulness result of Proposition \ref{FFres} which allows us, under a certain condition,  to extend
homomorphisms of shtukas over the   divisor ``at infinity'' $[\varpi]=0$.

 \subsection{Results on local Shimura varieties}
 
 Let us now state our results on local Shimura varieties.  As mentioned above, Scholze defines, starting with local Shimura data $(G, b, \mu)$, the local Shimura variety $\CM_{G, b, \mu}=(\CM_{G, b, \mu, K})_{ K\subset G(\BQ_p)}$. It is instructive to compare local Shimura varieties and global Shimura varieties. For local Shimura varieties, it is easy to see that they support a $\CG$-shtuka. This follows from the definition of $\CM_{G, b, \mu}$ as a moduli space of $\CG$-shtukas. In contrast, for global Shimura varieties, constructing the $\CG$-shtuka over them is quite an effort. On the other hand, contrary to their global counterpart where one can write down the set of $\BC$-points, there does not seem to be an explicit description of the set of $\BC_p$-points of $\CM_{G, b, \mu}$. Just as global Shimura varieties, local Shimura varieties can be explicitly described when $G$ is a torus or when $\mu$ is central; however, contrary to global Shimura varieties, local Shimura varieties have good functorial properties, e.g., they obey \emph{pushout functoriality}, comp. Proposition \ref{pushoutprop}.

When $K=\CG(\BZ_p)$ is a parahoric subgroup (with $\CG$ a corresponding parahoric model of $G$ over $\BZ_p$), then Scholze defines an \emph{integral model} $\CM_{\CG, b, \mu}^{\rm int}$ of $\CM_{G, b, \mu, K}$ over $O_{\breve E}$ as a $v$-sheaf: just as its general fiber  $\CM_{G, b, \mu, K}$, it represents a moduli problem of $\CG$-shtukas. Scholze conjectures that it is always representable by a formal scheme $\sM_{\CG, b, \mu}$ over $\Spf (O_{\breve E})$. The conjecture holds true if the data $(\CG, b, \mu)$ come from integral RZ data in the sense of \cite{R-Z}  (this excludes the cases of type (D) since they yield non-connected groups). In fact, in this case $\CM^{\rm int}_{\CG, b, \mu}$ is represented by the corresponding RZ formal scheme, cf.  \cite[Cor. 25.1.3]{Schber}. We give the following characterization of the representing formal scheme. In the formulation, there appears the \emph{$(b, \mu^{-1})$-admissible locus} $X_{\CG}(b, \mu^{-1})$  inside the \emph{Witt vector partial flag variety} over the algebraic closure $k$ of the residue field of $E$ \cite{ZhuAfGr} and the \emph{specialization map} on the set of classical points ${\rm sp_{\CM}}\colon  |\CM_{G,  b, \mu, K}|^{\rm class}\to X_{\CG}(b,\mu^{-1})(k)$ defined by Gleason \cite{Gl}, {\cred \cite{Gl21}.}
\begin{proposition}[see Proposition \ref{propRZ}]\label{propRZIntro}
Assume that $\CM_{\CG, b, \mu}^{\rm int}$ is representable by the formal scheme $\sM_{\CG,b, \mu}$. Then $\sM_{\CG,b, \mu}$ is the  unique  normal formal scheme $\sR$ which is flat and locally formally of finite type over $\Spf(O_{\breve E})$ and is equipped with identifications
\begin{altenumerate}
\item $\sR^{\rm rig}=\CM_{G,  b, \mu, K}$,
\item ${\sR}_\red^{\rm perf}=X_{\CG}(b,\mu^{-1})$,
\end{altenumerate}
such that the following diagram is commutative: 
\begin{equation}\label{canlambdaIntro}
\begin{aligned}   \xymatrix{
        |\sR^{\rm rig}|^{\rm class} \ar[r]^{\rm sp_{\sR}} \ar[d]_{=} & {\sR_\red}(k)   \ar[d]^{=}\\
         |\CM_{G,  b, \mu, K}|^{\rm class}\quad\ar[r]^{\,{\rm sp_{\CM}}}  & X_{\CG}(b,\mu^{-1})(k) .
        }
        \end{aligned}
    \end{equation}
\end{proposition}
{\cmag Here $\sR^{\rm rig}$ is the generic fiber of the formal scheme $\sR$ in the sense of Berthelot \cite[chap.~5]{R-Z}. } That $\sM_{\CG,b, \mu}$ has the properties stated in the proposition follows from $p$-adic Hodge theory; the characterization follows from the fully faithfulness of the diamond functor, cf. \cite[18.4]{Schber}. 

 Gleason defines the {\em{$v$-sheaf formal completion}} of $\CM_{\CG, b, \mu}^{\rm int}$ at a point $x\in X_\CG(b, \mu^{-1})(k)$. If Scholze's conjecture on the representability of $\CM_{\CG, b, \mu}^{\rm int}$ is true, then the formal completion $\wh{{\CM}^{\rm int}_{\CG, b, \mu}}_{/x}$ is representable, i.e., is given by $\Spf(R)$ for a complete Noetherian local $O_{\breve E}$-algebra $R$. Indeed, if $\CM_{\CG, b, \mu}^{\rm int}$ is represented by $\sM_{\CG, b, \mu}$, then $R=\wh\CO_{\sM_{\CG, b, \mu}, x}$. We prove a kind of converse in the local Hodge type case. Here, we call the tuple $(p, G, \mu, \CG)$ of local Hodge type if the following conditions  are satisfied.
\begin{itemize}
 \item[1)] $(G, \mu)$ is of local Hodge type, i.e. there is a closed embedding $\rho\colon G\hookrightarrow \GL_n$ such that   $\rho\circ\mu$ is minuscule.
  \item[2)] $\CG$ is the Bruhat-Tits stabilizer group scheme 
 $\CG_x$ of a point $x$ in the extended Bruhat-Tits building of $G(\BQ_p)$ and  is connected, i.e., we have
 $
\CG=
\CG_x=
\CG_x^\circ$.
  \end{itemize}
 \begin{theorem}[see Theorem \ref{thmrepint}]\label{thmrepintIntro}
Let $(G, b, \mu)$ be a local Shimura datum and $\CG$ a parahoric model of $G$ such that $(p, G, \mu, \CG)$ is of local Hodge type. Assume that $\wh{{\CM}^{\rm int}_{\CG, b, \mu}}_{/x}$ is representable for all  $x\in X_{\CG}(b, \mu^{-1})(k)$. 
  Then $\CM_{\CG, b, \mu}^{\rm int}$ is representable by a normal formal scheme $\sM_{\CG, b, \mu}$ which is flat and locally formally of finite type over $\Spf O_{\breve E}$. 
   \end{theorem} 
   
 The proof of Theorem \ref{thmrepintIntro} consists in showing that the Hodge embedding $\rho$ induces a closed immersion of $v$-sheaves from $\CM_{\CG, b, \mu}^{\rm int}$ into $\CM_{\GL_n, \rho(b), \rho(\mu)}^{\rm int}$ and then showing, by imitating de Jong's construction of closed formal subschemes of formal schemes \cite{deJongCrys}, that this morphism is relatively representable. 
 
 Let us comment on the assumption appearing in the statement above. Scholze defines, using the Beilinson-Drinfeld style affine Grassmannian, for any local Shimura datum $(G, b, \mu)$ a $v$-sheaf $\BM^v_{\CG, \mu}$.  Scholze conjectures  in \cite [Conj. 21.4.1]{Schber} that this $v$-sheaf is representable by a normal scheme  $\BM^{\rm loc}_{\CG, \mu}$ proper  and flat over $O_E$ and with reduced special fiber. The scheme $\BM^{\rm loc}_{\CG, \mu}$ with its action by $\CG$ is called the {\it scheme local model\footnote{Under more restrictive hypotheses, 
 these local models also agree with the local models as defined in \cite{PZ}, and in \cite{HPR}.} }. This conjecture may be viewed as a linearized version of the representability of $\CM_{\CG, b, \mu}^{\rm int}$. Indeed, it may be conjectured that $\wh{{\CM}^{\rm int}_{\CG, b, \mu}}_{/x}$ is represented by the formal completion $\widehat{{\BM}^\loc_{\CG, \mu}}_{/y}$ of the scheme $\BM^{\rm loc}_{\CG, \mu}$, where $y$ is a point  of  $\BM^{\rm loc}_{\CG, \mu}(k)$ corresponding to $x$, cf. Conjecture \ref{conjtubeLM}. 
 
 The representability of $\BM^v_{\CG, \mu}$ is more accessible than the representability of $\CM_{\CG, b, \mu}^{\rm int}$. Indeed, Scholze's conjecture is now proved:  work  of Ansch\"utz, Gleason, Louren\c co and Richarz  \cite{AnRicLou}, cf. also Louren\c co's thesis \cite{LourencoThesis}, proves the conjecture in all cases,  except when $p=2$ and there is a  simple factor of $G_{\rm ad}\otimes_{\BQ_2}\breve\BQ_2$  of the form ${\rm Res}_{F/\breve\BQ_2}H$, where $H$  is the adjoint group corresponding to an odd ramified unitary group,
  or when $p=3$ and there is a  simple factor of $G_{\rm ad}\otimes_{\BQ_3}\breve\BQ_3$  of the form ${\rm Res}_{F/\breve\BQ_3}H$, where $H$  is the adjoint group corresponding to  a ramified triality group; the general case is  treated by Gleason and Louren\c co in \cite{GL}, basing themselves on \cite{AnRicLou}.

\begin{remark}
Consider the case when  $(p, G, \mu, \CG)$ is of local Hodge type which can be embedded, in the sense of Remark \ref{remlocglob}, into a  tuple $(p, \eG, X, \eK)$ of global Hodge type (see below for this terminology). If all $x\in X_{\CG}(b, \mu^{-1})(k)$ can be ``realized'' (in the sense of Remark \ref{remlocglob}) by  points $\bf x$ in the reduction of  $ {\rm Sh}_\eK(\eG, X)_E$,  then the assumption in Theorem \ref{thmrepintIntro} holds by Theorem \ref{extGshtShimIntro} below.  This realization hypothesis would follow from the resolution of the axioms on the reduction of Shimura varieties in \cite{HeR}, which is known in many cases, cf. \cite{Zhou,SYZ}. We therefore can view Theorem \ref{thmrepintIntro} as a blueprint to prove the representability of $\CM_{\CG, b, \mu}^{\rm int}$ by global methods. In particular, for $p\neq 2$, we can dispense with the assumption on the representability of $\wh{{\CM}^{\rm int}_{\CG, b, \mu}}_{/x}$ in Theorem \ref{thmrepintIntro} when $G$ is unramified \cite{Nie} or when $G$ is tamely ramified and residually split \cite{Zhou}. 
\end{remark}
\begin{remark}
In a sequel to this paper  \cite{PRintlsv}, we prove  the representability of $\CM_{\CG, b, \mu}^{\rm int}$ in greater generality by a local method, again using the method of proof of Theorem \ref{thmrepintIntro}. In particular, in the situation of Theorem \ref{thmrepintIntro}, the representability of ${\CM}^{\rm int}_{\CG, b, \mu}$ always holds when $p\neq 2$.
\end{remark}

 \subsection{Results on global Shimura varieties} Now let us state our results on global Shimura varieties. Start with Shimura data $(\eG, X)$ as above, and denote, as before, by ${\rm Sh}(\eG, X)_E$ its canonical model over $E=\eE_{v}$. We assume that $\eK\subset \eG(\BA_f)$ is of the form  $\eK=\eK_p\eK^p$, with $\eK_p=\CG(\BZ_p)$, where $\CG$ is a smooth model of $G=\eG\otimes_\BQ \BQ_p$ and where $\eK^p\subset \eG(\BA^p_f)$ is sufficiently small. Then there is a pro-\'etale 
 $\CG(\BZ_p)$-cover   
$\BP_\eK$ over $ {\rm Sh}_\eK(\eG, X)_E$
 obtained by the system of covers 
 \begin{equation}
  {\rm Sh}_{\eK'}(\eG, X)_E\to 
{\rm Sh}_{\eK}(\eG, X)_E,
 \end{equation} 
 where $\eK'=\eK'_p\eK^p\subset \eK=\eK_p\eK^p$, with $\eK'_p$ running over all compact open subgroups of $\eK_p=
\CG(\BZ_p)$, comp.  \cite[III]{MilneAA}, \cite[\S 4]{LZ} (note that $\CG(\BZ_p)=\varprojlim_{\eK'_p} \eK_p/\eK'_p$.) By Liu-Zhu \cite{LZ}, the pro-\'etale $\CG(\BZ_p)$-cover $
\BP_\eK$ over $ 
{\rm Sh}_\eK(\eG, X)_E$ is de Rham (and bounded by $\mu_X$). Using Proposition \ref{detorsIntro}, we obtain a $\CG$-shtuka as postulated by Scholze:

 \begin{proposition}[see Proposition \ref{ShimuraShtThm}]\label{ShimuraShtThmIntro}
 There exists a $\CG$-shtuka $\sP_{\eK, E}$ over ${\rm Sh}_\eK(\eG, X)_E$ with one leg bounded 
by $\mu_X$ which is associated with the pro-\'etale $\CG(\BZ_p)$-cover   $\BP_\eK$, in the sense of Proposition \ref{detorsIntro}. 
\end{proposition}

 Furthermore, $\sP_{\eK, E}$ are supporting prime-to-$p$ Hecke correspondences, i.e., for  $g\in \eG(\BA^p_f)$ and $\eK'^{ p}$ with $g\eK'^{ p}g^{-1}\subset \eK_p$, there are compatible
isomorphisms $[g]^*(\sP_{\eK, E})\simeq \sP_{\eK', E}$ which cover  the natural morphisms $[g]: {\rm Sh}_{\eK_p\eK'^p}(\eG, X)_E\to  {\rm Sh}_{\eK_p\eK^p}(\eG, X)_E $. 

 When $(\eG, X)$ is of abelian type and  $\eK=\eK_p\eK^p$, where $\eK_p$ is a parahoric, Kisin and the first author \cite{KP} have constructed integral models $\sS_\eK$ of ${\rm Sh}_{\eK}(\eG, X)_E$, provided that  $p$ is odd, $G=\eG\otimes_\BQ \BQ_p$ splits over a tamely ramified extension and $p\nmid |\pi_1(G_{\rm der})|$. By an integral model, here we mean a scheme which is separated and of finite type over $\Spec(O_E)$ and whose generic fiber is identified with 
${\rm Sh}_{\eK}(\eG, X)_E$.
 This construction has been extended by Kisin and Zhou  \cite{KZhou} to cover many more cases. Here we construct such integral models in even greater generality when $(\eG, X)$ is of Hodge type and  $\eK=\eK_p\eK^p$, where $\eK_p$ is a parahoric with corresponding parahoric group scheme $\CG$ over $\BZ_p$. 
   
We call a tuple $(p, \eG, X, \eK)$, with $\eK=\eK_p\eK^p$, of   global Hodge type if the following conditions are satisfied. 
 
 \begin{itemize}
 \item[1)] $(\eG, X)$ is of Hodge type, i.e. there is a closed embedding of Shimura data $(\eG, X)\hookrightarrow (\GSp_{2g}, S_{2g}^\pm)$ into a group of symplectic similitudes with its Siegel datum.   
 \item[2)] $\eK_p=\CG(\BZ_p)$, where $\CG$ is the Bruhat-Tits stabilizer group scheme 
 $\CG_x$ of a point $x$ in the extended Bruhat-Tits building of $G(\BQ_p)$ and $\CG$ is connected, i.e., we have
 $
\CG=
\CG_x=
\CG_x^\circ$.
  \end{itemize}
   Note that when  $(p, \eG, X, \eK)$ is of  global Hodge type, then $(p, G, \mu_X, \CG)$ is of local Hodge type.  
\begin{theorem}[see Theorem \ref{mainhodge}]\label{extGshtShimIntro}
 Let $(p, \eG, X, \eK)$ be of {global Hodge type}. Then there exists a  pro-system of  normal  and flat integral models $\sS_\eK$  with generic fiber ${\rm Sh}_{\eK}(\eG, X)_E$, with finite \'etale transition maps for varying $\eK^p$, with  the  following properties.
  \begin{itemize}
\item[a)] For every dvr $R$ of characteristic $(0, p)$ over $O_E$, 
\begin{equation}
(\varprojlim\nolimits_{\eK^p}{\rm Sh}_\eK(\eG, X)_E)(R[1/p])=(\varprojlim\nolimits_{\eK^p}\sS_\eK)(R).
\end{equation}
 
\item[b)] The $\CG$-shtuka $\sP_{\eK, E}$ extends to a $\CG$-shtuka $\sP_{\eK}$ on $\sS_{\eK}$.\footnote{{\cmag When  $\CG$ is reductive (i.e., in a good reduction case), an integral extension $\sP_{\eK}$ of the $\CG$-shtuka $\sP_{\eK, E}$  was also constructed by Wu \cite{Wu}. His extension lives over the integral model $\sS_{\eK}$  given in 
Kisin \cite{KisinJAMS} (which is canonical in the sense of Milne).}}

 \item[c)] For  $x\in \sS_\eK(k)$, with associated $b_x\in G(\breve \BQ_p)$, there exists an isomorphism of formal completions
 $$
\Theta_x\colon  \widehat{{{{\CM}}}^{\rm int}_{ \CG, b_x, \mu }}_{/x_0}\isoarrow (\widehat{\sS_{\eK_{/x}}})^\diam ,
$$ 
 such that the pullback shtuka $\Theta_{ x}^*(\sP_\eK)$ coincides with the tautological shtuka on $ {\CM}^{\rm int}_{ \CG, b_x, \mu }$
 that arises from the definition of $ {\CM}^{\rm int}_{ \CG, b_x, \mu }$ as a moduli space of shtukas.  Here $x_0$ denotes the base point of $ {\CM}^{\rm int}_{ \CG, b_x, \mu }$. In particular, $\widehat{{{{\CM}}}^{\rm int}_{ \CG, b_x, \mu }}_{/x_0}$ is representable. 
\end{itemize}
\end{theorem}
Here  the element $b_x\in G(\breve \BQ_p)$ is well-defined up to $\sigma$-conjugacy by $\CG(\breve\BZ_p)$, and is determined by the fiber of $\sP_\eK$ at $x$, as follows. 
The pull-back  $x^*(\sP_\eK)$ is a $\CG$-shtuka over $\Spec(k)$, and yields a $\CG$-torsor $\sP_x$ over $\Spec(W(k))$ with an isomorphism 
$$
\phi_{\sP_x}\colon \phi^*(\sP_x)[1/p]\isoarrow\sP_x [1/p].
$$
 The choice of a trivialization of the $\CG$-torsor $\sP_x$ then defines  $b_x\in G(\breve \BQ_p)$.  The $\sigma$-conjugacy class $[b_x]$ lies in the subset $B(G, \mu^{-1})$ of $B(G)$.

The construction of $\sS_\eK$ is quite straightforward. The Hodge embedding defines a closed embedding of Shimura varieties into the Siegel type Shimura variety,
$$
{\rm Sh}_{\eK}(\eG, X)_E\hookrightarrow {\rm Sh}_{\eK^\flat}(\GSp_{2g}, S_{2g}^\pm)_\BQ\otimes_\BQ E .
$$
 Here, we need to choose the Siegel moduli level structure $\eK^\flat$ appropriately, in particular so that $\eK=\eK^\flat\cap G(\BA_f)$. After identifying ${\rm Sh}_{\eK^\flat}(\GSp_{2g}, S_{2g}^\pm)_\BQ$ with the generic fiber of the Siegel moduli space $\CA_{\eK^\flat}$ over $\BZ_{(p)}$, one defines $\sS_\eK$ as the normalization of the Zariski closure of ${\rm Sh}_{\eK}(\eG, X)_E$ inside $\CA_{\eK^\flat}\otimes_{\BZ_{(p)}} O_E$. It turns out a posteriori that $\sS_\eK$ is independent of the choice of the Hodge embedding, cf. Theorem \ref{uniqIntro}. A bonus of this independence is that we obtain for an inclusion $\eK_p\subset\eK'_p$ of parahoric subgroups an extension $\sS_{\eK}\to \sS_{\eK'}$ of the natural morphism ${\rm Sh}_{\eK}(\eG, X)_E\to {\rm Sh}_{\eK'}(\eG, X)_E$ to the integral models. Here $\eK=\eK_p\eK^p$ and $\eK'=\eK'_p\eK'^p$, with $\eK'^p=\eK^p$. This is a consequence of the functoriality of the formation of our integral models with respect to embeddings $(G, X, \eK)\hookrightarrow (G', X', \eK')$ of Shimura data (see Theorem \ref{functorialThm} for a precise formulation).
 
 Let us compare this theorem (for Shimura varieties of Hodge type) with  the main result of \cite{KP} 
 and its generalization by Kisin and Zhou in \cite{KZhou} (these references consider, more generally, Shimura varieties of abelian type).  Unlike in \cite{KP}, for our result we do 
not need any tameness hypothesis and we also dispense with the assumption $p\nmid\pi_1(\eG_{\rm der})$. 
The tameness hypothesis is significantly relaxed in the paper \cite{KZhou}, which ends up covering essentially all cases when $p\geq 5$. Nevertheless, our construction is, at least conceptually, simpler.  In particular, in contrast to \cite{KP}, \cite{KZhou}, there is no need for a judicious choice of a Hodge embedding. By the independence of the Hodge embedding, we  see that our integral models agree with those in \cite{KP}, \cite{KZhou}, when the latter exist. (Without this independence that we show here, the extension of the morphism ${\rm Sh}_{\eK}(\eG, X)_E\to {\rm Sh}_{\eK'}(\eG, X)_E$ between Shimura varieties for different parahorics  to integral models as above, was known to exist only when the models were constructed using ``compatible" Hodge embeddings, cf. \cite[\S 7]{Zhou}.) On the other hand, we do not produce a local model diagram, and do not even assert that there is an isomorphism 
$$
\widehat{\sS_{\eK_{/x}}}\simeq \widehat\BM^{\rm loc}_{{\CG, \mu}_{/y}},
$$
although this last fact follows a posteriori by combining our results with the main theorem of \cite{PRintlsv}, at least when $p$ is odd.
 Based on the Hodge type case, one should be able to extend the theorem to Shimura varieties of abelian type using Deligne's theory of connected components and ``twisting", as in \cite{KisinJAMS}, \cite{KP}, \cite{KZhou}. 

Let us comment on the proof of Theorem \ref{extGshtShimIntro}. Property a)  for the model $\sS_\eK$ follows from the N\'eron-Ogg-Shafarevich criterion of good reduction for abelian varieties. For property b), one realizes the $\CG$-shtuka $\sP_{\eK, E}$ in the generic fiber through some tensors in the pull-back   to $\sS_{\eK, E}$ under the Hodge embedding of the shtuka defined by the universal abelian scheme. Using Theorem \ref{vshtExtIntro}, these tensors extend over $\sS_\eK$. The crux is now to show that these tensors indeed define a $\CG$-shtuka over $\sS_\eK$. Here the main tool is the theorem of Ansch\"utz  \cite{An} that for an algebraically closed non-archimedean field $C$ of characteristic $p$,  any $\CG$-torsor on the punctured spectrum $\Spec (W(O_C))\setminus \{ s\}$ is trivial.  Finally, property c) essentially comes down to  showing that the pull-back of the $\CG$-shtuka $\sP_\eK$ to the completed local ring $\widehat{\CO}_{\sS_\eK, x}$
admits a ``framing".

We note that Theorem \ref{extGshtShimIntro} c)  provides a link between local and global Shimura varieties. 
This link is a priori different from the one provided by the theory of non-archimedean uniformization of Shimura varieties of PEL-type mentioned earlier, since it holds only point-wise. {\cmag The following theorem is the non-archimedean uniformization statement in our context. 

\begin{theorem}[see Theorem \ref{HK}]\label{introHK}
 Let $(p,\eG, X, \eK)$ be of global Hodge type. Let $x\in \sS_\eK(k)$. Then $\CM_{\CG, b_x, \mu}^{\rm int}$ is representable by a formal scheme $\sM_{\CG,b_x,\mu}$ and there is non-archimedean uniformization along the isogeny class $\CI(x)$ in $\sS_\eK\otimes_{O_E} k$, i.e., an isomorphism of formal schemes over $O_{\breve E}$,
$$
 I_x(\BQ)\backslash ({\sM_{\CG,b_x,\mu}}\times G(\BA_f^p)/\eK^p)\isoarrow \wh{({\sS_\eK\otimes_{O_E}O_{\breve E})}}_{/\CI(x)} .
$$
\emph{This isomorphism is to be interpreted (especially when $x$ is non-basic) as for its PEL counterpart in \cite[Thm. 6.23]{R-Z}. }
\end{theorem} 

Here the relation to the previous theorem is the obvious one: If   $( x_0, g)\in {\sM_{\CG,b_x,\mu}}(k)\times G(\BA_f^p)/\eK^p$ is a point representing the base point $x$, the isomorphism $\Theta_{ x}$ from Theorem \ref{extGshtShimIntro}  c) coincides with the completion of the uniformization morphism  at $x$.

}

The proof of Theorem \ref{introHK} is modelled on \cite{HP, HamaKim}, and proceeds by relating the integral local Shimura variety $\CM^{\rm int}_{\CG,b_x,\mu}$ to  a Rapoport-Zink space obtained by a global construction using the Hodge embedding. {\cmag In the statement of Theorem \ref{HK}, we make the assumption  that a certain condition $({\rm U}_x)$ is satisfied. This condition guarantees that  the  abelian varieties that are obtained by isogenies from $k$-points of the LHS of the uniformization map actually do define points  on the Shimura variety in question (this is not obvious since the Hodge type Shimura variety lacks a global moduli-theoretic interpretation). In our original submitted version of our paper, we conjectured that Condition  $({\rm U}_x)$ is always satisfied.  (Note that condition $({\rm U}_x)$ is identical with Axiom A in \cite[\S 4.3]{HamaKim}; Hamacher and Kim also conjectured that, in the tame case, it is always satisfied.) During the refereeing period of our paper, this conjecture has been proved in general by Gleason-Lim-Xu \cite[Cor. 1.10]{GLX22}. We refer to Remark \ref{prevres} for previous partial results  of Zhou and Nie.}

We also prove that the integral models $\sS_\eK$ are uniquely determined by their characteristic properties. The following uniqueness theorem should be compared with the main theorem of \cite{PCan}, where an analogous uniqueness theorem is proved. In loc.~cit., instead of $\CG$-shtukas, ``$(\CG, \mu)$-displays" are used for this characterization. 
The characterization of \cite{PCan} seems more concrete since it is adapted to $p$-divisible groups and gives more information when it applies.
For example, it is directly linked to the existence of a local model diagram; this is a useful feature which is harder to see here (compare the discussion in \S \ref{ss:LMD}). On the other hand, it is more limited in its applicability (essentially to Shimura varieties of Hodge type). An ``a priori" relation between the characterization in \cite{PCan} and the one given here is not clear; and something similar is true for the relation to the characterization of the good reduction model in the hyperspecial case by the \emph{extension property}, comp. \cite[\S 2.3.7]{KisinJAMS}. However, in practice, the integral models that have been constructed in various cases satisfy all these characterizations, provided the assumptions for them to apply are met.
It is remarkable that to obtain such a characterization, even in classical PEL type cases, one needs this more sophisticated approach.

   \begin{theorem}[see Theorem \ref{uniq}]\label{uniqIntro}
  Let $(\eG, X)$ be an arbitrary Shimura datum and consider open compact subgroups $\eK\subset G(\BA_f)$  of the form  $\eK=\eK_p\eK^p$, with $\eK_p=\CG(\BZ_p)$, where $\CG$ is a (connected) parahoric model of $G=\eG\otimes_\BQ \BQ_p$. There is at most one pro-system of  normal  and flat integral models $\sS_\eK$  with generic fiber ${\rm Sh}_{\eK}(\eG, X)_E$, with finite \'etale transition maps for varying sufficiently small $\eK^p$, with   properties a)--c) above. 
\end{theorem}
 The main point in the proof of this theorem is a rigidity property of the isomorphisms $\Theta_x$ in (c). This essentially amounts to a rigidity property of the framing of the pull-back of the $\CG$-shtuka $\sP_\eK$ to $\widehat{\CO}_{\sS_\eK, x}$, which was  constructed in the proof of property (c) of Theorem \ref{extGshtShimIntro}. For this, we show that the diamond automorphism group of the $\CG$-shtuka given by the trivial $\CG$-torsor and the Frobenius $\phi_b=b\times {\rm Frob}$ has as its  global sections over the completed local ring $\widehat{\CO}_{\sS_\eK, x}$ only the obvious ones, i.e. the $\sigma$-centralizer group $J_b(\BQ_p)$.

  We conjecture that models $\sS_\eK$ with these properties exist for general Shimura varieties, cf. Conjecture \ref{par812}. This conjecture should be compared with Deligne's proposal to construct the models $\sS_\eK$ as moduli spaces of motives: whereas it seems quite difficult to make this precise, our conjecture, which is based on Scholze's idea of moduli spaces of shtukas, is quite concrete and seems accessible in many cases (e.g., for Shimura varieties of abelian type). The conjecture should be seen in conjunction with Scholze's conjecture on the  representability of the $v$-sheaf local model (which is now proved, see above), and the conjectured representability of the $v$-sheaf local model diagram, cf. Conjecture  \ref{conjLMD}. Also, the construction of the map $x\mapsto b_x$ above would define a map to the set of $\CG(\breve\BZ_p)$-$\sigma$-conjugacy classes  as predicted in \cite[Axiom 4]{HeR},
  $$
  \Upsilon_\eK\colon \sS_\eK(k)\to G(\breve \BQ_p)/\CG(\breve\BZ_p)_\sigma ,
  $$
  which would  define in the general case the Newton-stratification, the KR-stratification, the EKOR-stratification and the central leaves  in the special fiber,  without the use of abelian varieties. Finally, Theorem \ref{introHK} should be seen in conjunction with the Langlands-Rapoport conjecture enumerating the isogeny classes in the special fiber, comp. \cite{KisinLR, R}. It should hold true in general.

This picture is quite general, and may not be easy to implement.  The most immediate next task is to establish this picture for Shimura varieties of abelian type. Also, one can hope to extend the results to parahoric subgroups which are not the stabilizer of a point in the building, but only equal to the neutral component of a stabilizer. In addition, for applications to moduli spaces, it would be useful to also allow non-connected stabilizer subgroups (called ``quasi-parahoric" in \cite[\S 25.3]{Schber}) as the level, comp. \cite{KP}. 

 Our normalizations impose that for a Shimura datum $(\eG, X)$ of PEL-type, with universal abelian variety $\CA_\eK$ over the integral model $\sS_\eK$, the fiber $\sP_{\eK, x}$ at a point $x\in \sS_\eK(k)$ corresponds under the natural representation of $\eG$  to $({\rm Frob}^{-1})^*({\rm H}^1_{\rm crys}(\CA_x)^*)$, 
 the Frobenius descent  of the linear dual of the crystalline cohomology of $\CA_x$, i.e., the Frobenius descent  of the linear dual  of the contravariant Dieudonn\'e module  $\BD(\CA_x)^*$ of the $p$-divisible group of $\CA_x$, cf. Example \ref{pdivExample}. For the crystalline period map, we trivialize $\BD(\CA_x)^*$ and vary the Hodge filtration so that the Grothendieck-Messing period map has its target in the flag variety $\CF_{G, \mu}$. For the Hodge-Tate period map, we trivialize the associated local system so that the Hodge-Tate period map has its target in the flag variety $\CF_{G, \mu^{-1}}$. The relation with prismatic cohomology given by $\sP_{\eK, x}$ corresponding to $({\rm Frob}^{-1})^*({\rm H}^1_{\rm crys}(\CA_x)^*)={\rm H}^1_{\pri}(\CA_x/W(k))^*(W(k), (p))$  leads us  to expect that the models  $\sS_\eK$ support  a more refined object $\sP_{\eK, \pri}$
which is a deperfection of the  $\CG$-shtuka  $\sP_{\eK}$. Here, ``deperfection" is meant in the sense of the theory of prisms of Bhatt and Scholze \cite[\S 1]{BSPrism} and the object $\sP_{\eK, \pri}$ should be, roughly, a $\CG$-torsor 
with Frobenius over the prismatic site of the $p$-adic 
completion $\wh{\sS_\eK}$. (More precisely, it should be a {\it prismatic Frobenius crystal with $\CG$-structure} over $\wh{\sS_\eK}$, see \S \ref{ss:prism}). In the Hodge type case, it should be constructed from the prismatic cohomology of the universal abelian scheme.

\bigskip

\noindent{\bf Acknowledgements:}   We thank Peter Scholze for his  comments on a preliminary version of this paper, {\cmag Ian Gleason for corrections and comments, and the referee for his attentive reading of the paper.} We also thank Patrick Daniels, Laurent Fargues, Thomas Haines, Kazuhiro Ito, Mark Kisin, Eva Viehmann and  Yihang Zhu for additional helpful correspondence, discussions, corrections and comments.

The first author acknowledges support by NSF grant \#DMS-2100743. Also, his visit to Bonn in March 2020 (cut short by Covid-19) was financed by P.~Scholze's Leibniz prize. 

\bigskip

\noindent{\bf Notations:} The following general conventions are used.
\begin{altitemize}
\item We fix an algebraic closure $\bar\BQ_p$ of the $p$-adic numbers
$\BQ_p$. If $F/\BQ_p$ is a finite extension with $F\subset \bar\BQ_p$, we denote by $\breve F$ the completion of the maximal unramified extension of $F$ in $\bar\BQ_p$. 
The rings of integers are denoted by $O_F$, resp. by $O_{\breve F}$ or $\breve O_F$. 
\item We often write $X\otimes_A B$ for $X\times_{\Spec (A)} \Spec (B)$. 

\item For a Huber pair $(A, A^+)$, we write $\Spd(A, A^+)$ for the $v$-sheaf $(\Spa(A, A^+))^\diam$. 
If $A^+=A^\circ$, we simply write $\Spd(A)$.
\end{altitemize}

\section{Shtukas}
\subsection{$v$-sheaves and other preliminaries}\label{ss2.1}
We will use  heavily many of the constructions and results of \cite{Schber} and, in particular, the language and techniques of perfectoid spaces and $v$-sheaves of Scholze. Here, we recall some of the basic terms and notations. The reader is referred to \cite{Schber} and \cite{Sch-Diam} for more details and to \cite{KL} for more background and other important constructions. Let $k$ be a perfect field of characteristic $p$.

Let $(R, R^+)$ be perfectoid over $k$, so that $S=\Spa(R, R^+)\in {\rm Perfd}_k$ is affinoid perfectoid. Then, in particular, $R$ is a complete Tate Huber ring and is uniform, so we can take the subring of power-bounded elements $R^\circ$ as a ring of definition.  Now pick a pseudouniformizer $\varpi\in R^\circ$. We can regard $R$ as a (uniform) Banach ring by using the sub-multiplicative norm $|\ |: R\to \BR_{\geq 0}$ defined by
\[
|x|={\rm inf}_{\{n | \varpi^n x\in R^\circ\}} p^n.
\]
(See \cite[p. 11]{Schber}, \cite[Chapt. 2]{KL}.)
We will also need the spectral norm
\begin{equation}\label{def:spnorm}
\alpha(x)=\lim_{s\mapsto+\infty} |x^s|^{1/s} ,
\end{equation}
which is power-multiplicative and defines the same topology on $R$.
Then 
\[
R^\circ=\{x\in R\ |\ \alpha(x)\leq 1\}
\]
 and $R^+\subset R^\circ$.

As in \cite[\S 11.2]{Schber}, for $S=\Spa(R, R^+)\in{\rm Perfd}_k$ as above, we set
\begin{equation}
S\bdtimes  \BZ_p  =S\bdtimes {\rm Spa}(\BZ_p)= {\rm Spa}(W(R^+))\setminus \{[\varpi]= 0\},
\end{equation}
where $\varpi\in R^+$ is a pseudouniformizer of $(R, R^+)$ and $[\varpi]\in W(R^+)$ the Teichm\"uller representative.
By \cite[Prop. 11.2.1]{Schber} and its proof, this is a sousperfectoid (analytic) adic space covered by the affinoid subsets 
$\Spa(R_n, R_n^+)$, $n=1$, $2,\ldots $, 
of $\Spa(W(R^+))$, where
\begin{equation}
R_n=\Big\{\sum_{i\geq 0} [r_i] \left(\frac{p}{[\varpi^{1/p^n}]}\right)^i\ |\ r_i\in R, \ r_i\to 0\Big\}.
\end{equation}
Set also 
\begin{equation}
\CY(R, {R^+})={\rm Spa}(W(R^+))\setminus \{[\varpi]= 0, p=0\} .
\end{equation}
This is also a sousperfectoid (analytic) adic space (see \cite[\S 13.1]{Schber} for the case $(R,R^+)=(C, \CO_C)$, or the proof of Proposition 3.6 in \cite{Kedlaya}, in general).
We can define a continuous map (see  \cite[\S 15]{Schber} for the case $(R,R^+)=(C, \CO_C)$, and \cite[\S II.1.12]{FS} in general)
 \begin{equation}
\kappa:  |\CY(R, {R^+})| \to [0,\infty],
 \end{equation}
 taking any point $x\in \CY(R, {R^+})$ with rank-$1$-generalization $\tilde x$  to $\log(|[\varpi](\tilde x)|)/\log(|p(\tilde x)|)$. (This map depends on the choice of the pseudo-uniformizer $\varpi$). For an interval $I=[a, b]\subset [0,\infty]$ with rational endpoints, we denote by $\CY_I(R, R^+)$, or $\CY_I(S)$, the open of $\CY(R, {R^+})$ contained in $\kappa^{-1}(I)$, given by
 $$
 \CY_I(S)=\{ |p|^b\leq |[\varpi]|\leq|p|^a\} .
 $$ 
  Then, 
 \begin{equation}
 \begin{aligned}
 \CY_{[0,\infty)}(R, R^+)&={\rm Spa}(W(R^+))\setminus \{[\varpi]= 0\}=S\bdtimes \BZ_p\\
 \CY_{(0,\infty]}(R, R^+)&={\rm Spa}(W(R^+))\setminus \{p= 0\} .
 \end{aligned}
 \end{equation}
  We also have 
 $\CY_{[0, p^n]}(S)=\Spa(R_n, R_n^+)$ as above. In this, we can see that $R_n$ is independent of the choice of $R^+$.  This, 
 together with a result of Kedlaya-Liu \cite[Thm. 2.7.7]{KL},  implies that
 the (exact) category of vector bundles on $\CY_{[0,\infty)}(S)=S\bdtimes\BZ_p$ does not depend on the choice of $R^+$ in the pair $(R, R^+)$. In particular, we obtain:
 
 \begin{proposition}\label{ppProp}
 The natural restriction map gives an exact equivalence between vector bundles over $\Spa(R, R^+)\bdtimes \BZ_p$ 
and vector bundles over $\Spa(R, R^\circ)\bdtimes \BZ_p$. \hfill $\square$
\end{proposition}

Note that the construction of $\CY_{[0,\infty)}(R, R^+)$ ``globalizes" and one can make sense of $\CY_{[0,\infty)}(S)$ for a general perfectoid space $S$ over $k$, cf. \cite[Prop. II.1.13]{FS}. We also have the following  descent result, cf. \cite[Prop. 19.5.3]{Schber}. 

\begin{proposition}\label{vstack}
Sending a perfectoid space $S$ over $k$ to the groupoid of vector bundles over $\CY_{[0,\infty)}(S)$ gives a $v$-stack. \hfill $\square$
\end{proposition}

 For an interval $I\subset [0,\infty]$ and $S=\Spa(R, R^+)$ affinoid perfectoid, we will write
 \begin{equation}
 B_{R, R^+}^I=B_{S}^{I}=\Gamma(\CY_{I}(R, R^+),  \CO_{\CY}).
 \end{equation}
 If $r=m/n>0$, then
\[
B_{S}^{[m/n,\infty]}= {W(R^+)\Big\langle\frac{[\varpi]^{n}}{p^m}}\Big\rangle\Big[\frac{1}{p}\Big],
\]
where ${W(R^+)\big\langle\frac{[\varpi]^{n}}{p^m}}\big\rangle$ is the $p$-adic completion of ${W(R^+)\big[\frac{[\varpi]^{n}}{p^m}}\big]$.
 Recall also the $p$-adic completion of the universal divided power hull 
\[
A_{{\rm crys}}(R^{+}/\varpi)=W(R^+)\big\langle\big(\frac{[\varpi^n]}{n!}\big)_{n\geq 1}\big\rangle\to R^+/\varpi,
\] 
and set 
\[
B_{\rm crys}^+(R^{+}/\varpi)=A_{{\rm crys}}(R^{+}/\varpi)[1/p].
\]
 For $r\gg0$, there  are natural homomorphisms
\begin{equation}
W(R^+)\to B_{\rm crys}^+(R^{+}/\varpi)\to B_{S}^{[r,\infty]}=\Gamma(\CY_{[r,\infty]}(R, R^+), \CO_\CY) ,
\end{equation} 
see, for example, \cite[\S 6.2]{FarguesGeometrization}. {\cmag Note again that if $\infty\not\in I$, then $B^I_S=B^I_{R,R^+}$
does not depend on $R^+$ and we can denote it as $B^I_R$.}

The Frobenius of $R^+$ induces a ring homomorphism $W(R^+)\to W(R^+)$. This gives morphisms 
$\CY_{[r,\infty)}(R,R^+)\to \CY_{[pr,\infty)}(R,R^+)\hookrightarrow \CY_{[r,\infty)}(R,R^+)$ and also $\CY_{[r,\infty]}(R,R^+)\to \CY_{[pr,\infty]}(R,R^+)\hookrightarrow \CY_{[r,\infty]}(R,R^+)$, which we will denote
${\rm Frob}_S$ or simply $\phi$, if no confusion arises. Let us denote by ${\rm Vec}^\phi_{\CY_{[r,\infty)}(R,R^+)}$
the category of vector bundles $\sV$ over $\CY_{[r,\infty)}(R,R^+)$ with (isomorphism) $\phi$-structure, i.e. an isomorphism 
\[
\phi_{\sV}: \phi^*(\sV)\xrightarrow{\sim} \sV,
\]
and similarly for $\CY_{[r,\infty]}(R,R^+)$. {\cmag There is a similar definition of ${\rm Vec}^\phi_{\CY_{I}(R,R^+)}$, 
for any interval $I$ with endpoints $0$ and $\infty$. 

{\cmag Assume $r>0$. Since $\CY_{[r,\infty]}(R,R^+)$ is affinoid and sheafy (since it is sousperfectoid, see \cite[Prop. 3.6]{Kedlaya}), we can see by using \cite[Theorem 5.2.8]{Schber}, \cite[Theorem 2.7.7]{KL}, that taking global sections gives an equivalence of categories between ${\rm Vec}^\phi_{\CY_{[r,\infty]}(R,R^+)}$ and the category of projective finitely generated $ B_{S}^{[r,\infty]}$-modules
with a Frobenius semilinear map whose linearization is an isomorphism. Following \cite[\S 6]{KL}, \cite[\S 12.3, \S 13.2]{Schber}, we call them ``$\phi$-modules over  $ B_{S}^{[r,\infty]}$" or ``$\phi$-modules over $\CY_{[r,\infty]}(R,R^+)$''.}
\quash{We do not know if the corresponding fact is true for ${\rm Vec}^\phi_{\CY_{[r,\infty)}(R,R^+)}$ since vector bundles over the non-affinoid $\CY_{[r,\infty)}(R,R^+)$ are harder to describe. It does hold though for ${\rm Vec}^\phi_{\CY_{(0,\infty)}(R,R^+)}$ by \cite[Theorem 6.3.12]{KL} which gives an equivalence of categories between
${\rm Vec}^\phi_{\CY_{(0,\infty)}(R,R^+)}$ and the category of $\phi$-modules over $B^{(0,\infty)}_S$. }

The following result is implicitly stated in 
\cite[\S 6]{FarguesGeometrization}.  
\begin{proposition}\label{FFres}
Suppose that $R^+=R^\circ$. Then, for $r>0$,
the restriction functor
\[
{\rm Vec}^\phi_{\CY_{[r,\infty]}(R,R^+)}\to {\rm Vec}^\phi_{\CY_{[r,\infty)}(R,R^+)}
\]
is fully-faithful. It is an equivalence of categories when $(R, R^\circ)=(C, O_C)$ is an algebraically closed perfectoid field. 
\end{proposition}

 The result follows directly from the work in \cite[\S 11]{FarguesFontaine} when $(R, R^\circ)$ is a perfectoid field $(K, K^\circ)$. (See \cite[Rem. 11.1.11]{FarguesFontaine} and \cite[\S 13.3]{Schber}. In this case, the restriction functor is not an equivalence in general, unless $K$ is algebraically closed.) 
Also, as was pointed out to the authors by Scholze, full-faithfulness fails to hold if $R^+\neq R^\circ$, even when $R$ is a complete non-archimedean algebraically closed field. Indeed, by Proposition \ref{ppProp}, the target category does not depend on the choice of $R^+\subset R^\circ$ but the source category does.

\begin{remark}\label{Frobeniusamplify}
The usual argument of ``continuation by Frobenius"  shows that the restriction functors
\[
{\rm Vec}^\phi_{\CY_{(0,\infty]}(R,R^+)}\to  {\rm Vec}^\phi_{\CY_{[r,\infty]}(R,R^+)}
,\qquad
{\rm Vec}^\phi_{\CY_{(0,\infty)}(R,R^+)}\to  {\rm Vec}^\phi_{\CY_{[r,\infty)}(R,R^+)}
\]
are equivalences of categories for all $r>0$ (see \cite[Prop. 22.1.1]{Schber}). Using this we see that Proposition
\ref{FFres} is equivalent to the statement that, when $R^+=R^\circ$, the restriction functor
\[
{\rm Vec}^\phi_{\CY_{(0,\infty]}(R,R^+)}\to {\rm Vec}^\phi_{\CY_{(0,\infty)}(R,R^+)}
\]
is fully-faithful. Note that, by \cite[Theorem 6.3.12]{KL}, there is an equivalence of categories between
${\rm Vec}^\phi_{\CY_{(0,\infty)}(R,R^+)}$ and the category of $\phi$-modules over $B^{(0,\infty)}_{R}$,
with ``$\phi$-modules" defined as above. \quash{Using this and the above equivalence of continuation by Frobenius we can also see that, for $r>0$, 
there is an equivalence of categories between
${\rm Vec}^\phi_{\CY_{[r,\infty)}(R,R^+)}$ and the category of $\phi$-modules over $B^{[r,\infty)}_R$.}
\end{remark}

\begin{proof} (of Proposition \ref{FFres}) By applying the functor to the internal Hom between vector bundles with $\phi$-structure, we see that
it is enough to show that, for $\sV\in {\rm Ob}({\rm Vec}^\phi_{\CY_{[r,\infty]}(R,R^+)})$, restriction gives an isomorphism
\[
\Gamma(\CY_{[r,\infty]}(R,R^+), \sV)^{\phi_\sV=1}\xrightarrow{\sim}\Gamma(\CY_{[r,\infty)}(R,R^+), \sV)^{\phi_\sV=1}.
\]
The vector bundle $\sV$ is given by a $\phi$-module   over $B^{[r,\infty]}_{R,R^\circ}$. Set for simplicity $B^{r,+}_R=B^{[r,\infty]}_{R,R^\circ}$, $B^r_R=B^{[r,\infty)}_{R}$, $B^+_R=B_{R,R^\circ}^{(0,\infty]}$, $B_R=B_{R}^{(0,\infty)}$.  Recall that a $\phi$-module  over $B^{r,+}_R$ is a pair $(M,\phi_M)$ of a finite projective $B^{r,+}_R$-module $M$ and an isomorphism
\[
\phi_M: M\otimes_{B^{r,+}_R}B^{pr, +}_R\xrightarrow{\sim} M ,
\]
which is  linear over $\phi: B^{pr, +}_R\xrightarrow{\sim} B^{r,+}_R$. The statement amounts to the claim that
 the natural map
\begin{equation*}\label{bijectiveM}
M^{\phi_M=1}\to (M\otimes_{B^{r,+}_R}B^{r}_R)^{\phi_M=1}
\end{equation*}
is bijective. 
 
We first reduce to the case that the module $M=\Gamma(\CY_{[r,\infty]}(R,R^+), \sV)$ is finite free over $B^{r,+}_R$. This reduction follows the argument in the proof of \cite[Theorem II.2.6]{FS} (cf.  \cite[Cor. 1.5.3]{KL}): We pick a surjection $\psi: F=(B^{r,+}_R)^m\to M$ with $N=\ker(\psi)$ and choose a splitting $F=M\oplus N$. We would like to give the source $F$ a $\phi$-module structure 
\[
\phi_F: F\otimes_{B^{r,+}_R}B^{pr, +}_R\xrightarrow{\sim} F
\]
making $\psi$ equivariant for $\phi$. The desired $\phi_F$ is a lift of $\phi_M$ via $\psi$, which has to be an isomorphism. Hence, finding $\psi_F$ is possible
if there is an isomorphism
\[
\phi^*(N\otimes_{B^{r,+}_R}B^{pr, +}_R)= N\otimes_{B^{r,+}_R, \phi}B^{r, +}_R\simeq N
\]
of $B^{r,+}_R$-modules. Consider the Grothendieck group of finite projective $B^{r,+}_R$-modules. In this, the classes of both $N$ and 
$N\otimes_{B^{r,+}_R, \phi}B^{r, +}_R$ are given by $[(B^{r,+}_R)^m]-[M]$. Hence, there is such an isomorphism after possibly increasing $m$, i.e. after adding a finite free module to $F$.
We conclude that there is a $\phi$-module $(N,\phi_{N})$ such that $M\oplus N=F\simeq (B^{r,+}_R)^n$. Using this we see that it is enough to show that, for a $\phi$-module $(M,\phi_M)$ with $M$ free, the natural map
\[
M^{\phi_M=1}\to (M\otimes_{B^{r,+}_R}B^{r}_R)^{\phi_M=1}
\]
is bijective. After fixing a basis $M\simeq (B^{r,+}_R)^n$, we can write the map $\phi_M$ on $M$ by using an $n\times n$ matrix $A=(a_{ij})$ with entries in $B^{r,+}_R$,
i.e. $\phi_M(x)=A\cdot \phi(x)$. In fact, $A\in \GL_n(B^{r, +}_R)$ so the inverse $A^{-1}$ also has  entries in $B^{r,+}_R$.

We will use some more work of Kedlaya-Liu (see \cite[\S 5.1, \S 5.2]{KL}): For $r>0$ we have
\[
B^{[0,r]}_{R}[1/p]=\{\sum^{+\infty}_{i\gg -\infty}[x_i]p^i\in W(R)[1/p]\ |\ \lim_{i\to+\infty} p^{-i}\alpha(x_i)^r=0\}.
\]
For $x=\sum^{+\infty}_{i\gg-\infty}[x_i]p^i\in B^{[0,r]}_{R}[1/p]$,   and $r\geq s>0$, set
\[
\norm x_s=\norm {\sum^{+\infty}_{i\gg-\infty}[x_i]p^i}_s=\max_i \{p^{-i}\alpha(x_i)^s\}.
\]
Here, $\alpha$ is the spectral norm for $R$, cf. \eqref{def:spnorm}. (In \cite{KL}, $\norm x_s$ is denoted as $\lambda(\alpha^s)(x)$.) This induces a power-multiplicative norm on $B^{[0,r]}_{R}[1/p]$ which also extends on various other related rings described in \emph{loc. cit.}, see \cite[Prop. 5.1.2]{KL}. 
In particular, for each $s>0$ with $s\in [t,t']$, the norm $\norm{\  }_s$ extends to $B_S^{[t,t']}$. Therefore, for each $s$ with $r\leq s$, the norm $\norm{\  }_s$ extends to $B^r_R=B^{[r,\infty)}_S$ which contains $B^{r,+}_R=B^{[r,\infty]}_S$.
The following is an extension of \cite[Lem. 5.2.11 (a)]{KL} and \cite[Prop. 1.10.7]{FarguesFontaine}. Recall $R^+=R^\circ$ and $r>0$.

\begin{lemma}\label{KLlemmaext}
An element $x\in B^r_R$ belongs to $B^{r,+}_R$ if and only if 
\[
\limsup_{s\mapsto +\infty}\norm x_s^{1/s}\leq  1.
\]
\end{lemma}

\begin{proof}
Note that \cite[Lem. 5.2.11 (a)]{KL} gives the corresponding statement for $B_R$ and $B^+_R$. On the other hand, \cite[Prop. 1.10.7]{FarguesFontaine} implies the statement of the lemma, when $R=K$ is a perfectoid field. The following is a variation of the argument in the proof of \cite[Lem. 5.2.11 (a)]{KL}.

1) First consider these $x$ in $B^r_R$ which can be written in the form $x=\sum^{+\infty}_{i=a} [x_i]p^i$, with $x_i\in R$. Then 
\[
\limsup_{s\mapsto +\infty}\norm x_s^{1/s}=\limsup_{s\mapsto +\infty} \sup_i \{p^{-i/s}\alpha(x_i)\}.
\]
If $x_i\in R^\circ$ then $\alpha(x_i)\leq 1$. So, if $x_i\in R^\circ$ for all $i$, then $\sup_i \{p^{-i/s}\alpha(x_i)\}\leq p^{|a|/s}$, so  $\limsup_{s\mapsto +\infty}\ \norm x_s^{1/s}\leq 1$. On the other hand, if there is some $i$ such that $\alpha(x_i)>1$, then for all $\epsilon >0$ we can find $s\gg 0$ such that $p^{-i/s}\alpha(x_i)>1+\epsilon$ and then $\limsup_{s\mapsto +\infty}\ \norm x_s^{1/s}>1$. This gives that if $\limsup_{s\mapsto +\infty}\ \norm x_s^{1/s}\leq 1$, then $x_i\in R^\circ$ for all $i$, and so $x\in W(R^\circ)[1/p]\subset B^{r, +}_R$.

2) Next we show that if $x\in B^{r,+}_R$, then $\limsup_{s\mapsto +\infty} \norm x_s^{1/s}\leq  1$, which is the one direction of the implication. Suppose $r=m/n$. If $x\in B^{r,+}_R=W(R^\circ)\langle \frac{[\varpi]^n}{p^m}\rangle [1/p]$, then $p^ax=y\in W(R^\circ)\langle \frac{[\varpi]^n}{p^m}\rangle$ for some $a\geq 1$. We have $\norm{p^ax}_s^{1/s}=p^{-a/s} \norm x_s^{1/s}$. Hence,  it is enough to show $\limsup_{s\mapsto +\infty}\ \norm y_s^{1/s}\leq 1$ for 
$y\in W(R^\circ)\langle \frac{[\varpi]^n}{p^m}\rangle$. But  all such $y$ are $p$-adic limits of elements in $W(R^\circ)[1/p]$ (so also of the type appearing in (1) above), and so this holds by (1) and continuity. It follows that $\limsup_{s\mapsto +\infty}\norm x_s^{1/s}\leq 1$ for all $x\in B^{r,+}_R$.

3) We now deal with the converse in the case of a general element $x\in B^r_R$. 
We  observe that for any such $x$, there is $a\in \BZ$, so that we can write $x=y+z$ with $y=\sum^{+\infty}_{i=a} [y_i]p^i$, $y_i\in R$, as in (1) above, and 
with $z\in B^{r,+}_R$, cf. \cite[Lemma 5.2.8]{KL}. Our assumption and (2) applied to $z$ gives that $\limsup_{s\mapsto +\infty}\norm y_s^{1/s}\leq 1$. Hence by (1), $y\in B^{r,+}_R$ and so $x\in B^{r,+}_R$ also.
\end{proof}

Choose a basis $e_1,\ldots , e_n$ of $M$. 
For $x=(x_1,\ldots , x_n)^t\in (B^r_R)^n\simeq M\otimes_{B^{r,+}_R}B^r_R$, and $s\geq r$, set
\[
\norm x_s=\max_i\{ \norm {x_i}\mid i=1,\ldots, n\}.
\]
Hence, for $x\in M\otimes_{B^{r,+}_R}B^r_R$ with $\phi_M(x)=A\cdot \phi(x)=x$, i.e. with $\phi(x)=A^{-1}\cdot x$, we find
\[
  \norm {\phi(x)}_s^{1/s}=  \norm {A^{-1}\cdot x}_s^{1/s}\leq \norm {A^{-1}}^{1/s}_s\cdot \norm {x}_s^{1/s}
  \]
where $\norm B_s=\norm {(b_{ij})}_s=\max_{ij}\norm {b_{ij}}_s$. For $b\in B^{r,+}_R$, we have $\limsup_{s\mapsto+\infty}\norm b_s^{1/s}\leq 1$ by Lemma \ref{KLlemmaext}. Recall that the 
entries   of $A^{-1}$ are in $B^{r,+}_R$ so  $\limsup_{s\mapsto +\infty}\ \norm {A^{-1}}^{1/s}_s\leq 1$. Also observe that
\[
 \norm {\phi(x)}_s^{1/s}= \norm {x}_{ps}^{1/s}.
\]
So given $\epsilon>0$, there is $s_0\gg 0$ such that, for $s\geq s_0$, we have
\[
(\norm x_{ps}^{1/ps})^p\leq (1+\epsilon)\cdot \norm x_{s}^{1/s}.
\]
 It follows from this and \cite[Lemma 5.2.1]{KL} that 
 \[
\limsup_{s\mapsto +\infty} \norm x_s^{1/s}\leq  1.
\]
Now use Lemma \ref{KLlemmaext} to conclude that $x\in (B^{r,+}_R)^n=M$.

The second assertion, in the case when $(R, R^\circ)$ is an algebraically closed perfectoid field $(C, O_C)$, follows from \cite[Thm. 11.1.9, Cor. 11.1.13]{FarguesFontaine} (see also \cite[Theorem 13.2.1]{Schber}). 
\end{proof}
}

Recall that given an untilt $R^\sharp$ of $R$,  there is a canonical surjection  $W(R^+)\to R^{\sharp +}$ whose kernel is generated by an element $\xi=\xi_{R^\sharp}\in W(R^+)$ which is primitive of degree $1$, comp. \cite[\S 11.3]{Schber}.  Recall that $B^+_{\rm dR}(R^\sharp)$ is the $(\xi)$-adic completion of 
 $W(R^+)[[\varpi]^{-1}]$ and that $B_{\rm dR}(R^\sharp)=B^+_{\rm dR}(R^\sharp)[1/\xi]$. The element $\xi$ defines a closed Cartier divisor on $\CY(R,R^+)$, on $\CY_{[0,\infty)}(R,R^+)$, and  a Cartier divisor on the scheme $Y(R,R^+)$, where
\begin{equation}\label{Yscheme}
Y(R,R^+)=\Spec(W(R^+))\setminus V(p, [\varpi]).
\end{equation}
Consider  the open immersion $
j(R,R^+): Y(R,R^+)\hookrightarrow \Spec(W(R^+)).$
There is also a map of locally ringed spaces $\CY(R,R^+)\to Y(R,R^+)$.

\begin{theorem}[Kedlaya \cite{Kedlaya}]\label{FFKedlaya}
a) (GAGA)  Pull-back along the map of locally ringed spaces $\CY(R,R^+)\to Y(R,R^+)$ induces an exact tensor equivalence
between the corresponding categories of vector bundles.

b) The restriction $j(R,R^+)^*$ induces a fully-faithful tensor functor between the  categories of vector bundles on $Y(R,R^+)$ and on $ \Spec(W(R^+))$.

c) If $(R, R^+)=(K, K^+)$ is a perfectoid field, then the restriction $j(K, K^+)^*$ induces an equivalence between 
the corresponding tensor categories of vector bundles.
\end{theorem}

\begin{proof} Part (a) is \cite[Thm. 3.8]{Kedlaya}.
For part (b) see \cite[Rem. 3.11]{Kedlaya}, cf. {\cred \cite[Thm. 2.5]{Gl21}.} Part (c) follows from \cite[Thm. 2.7]{Kedlaya}, see also \cite[Prop. 14.2.6]{Schber}.
\end{proof}

\subsubsection{The $v$-sheaf associated to an adic space}\label{sss:vs} Let $\CY$ be an adic space over $\Spa(\BZ_p)$. The  $v$-sheaf    $\CY^\diam\to \Spd(\BZ_p)$   over ${\rm Spd}(\BZ_p)$ is the functor on ${\rm Perfd}_{k}$   which associates 
 to   $S={\rm Spa}(R, R^+)$  in ${\rm Perfd}_{k}$  the set of isomorphism 
 classes of pairs
 $
 (S^\sharp, x)
 $ 
where 
 \begin{itemize}
 \item[1)] $(S^\sharp, \iota)=(\Spa(R^\sharp, R^{\sharp +}), \iota)$ is an untilt of $S$ over $\BZ_p$,
 \item[2)] $x: S^\sharp \to \CY$ is a $\BZ_p$-morphism of adic spaces.
\end{itemize}
(See \cite[18.1]{Schber}. In fact, this construction of $\CY^\diam$ also applies to pre-adic spaces $\CY$ over $\Spa(\BZ_p)$; see \cite[App. to \S 3]{Schber} for the notion of pre-adic space. (There is also a similar construction for pre-adic spaces over
$\Spa(\BQ_p)$ producing $v$-sheaves over $\Spd(\BQ_p)$.)  

\subsubsection{The $v$-sheaf associated to a formal scheme}\label{sss:vfs} Let $\fkX$ be a formal scheme over $\Spf(\BZ_p)$ which is locally formally of finite type.  Consider
the corresponding adic space $\fkX^{\rm ad}$ over $\Spa(\BZ_p)$. We can then take the corresponding $v$-sheaf $(\fkX^{\rm ad})^\diam$
over  $\Spd(\BZ_p)$ as in \S\ref{sss:vs}; we denote this $v$-sheaf simply by $\fkX^\diam$.

\subsubsection{The $v$-sheaves associated to schemes}\label{sss:vsSch}
 Following \cite[\S 2.2]{AnRicLou}, we will   give two (different) constructions of $v$-sheaves associated to schemes over $\BZ_p$. These constructions are somewhat subtle. Suppose $\sX=\Spec(A)$ is an affine scheme over $\BZ_p$:
 \smallskip
 
 i) Let $\sX^{\bDiam}=\Spec(A)^{\bDiam}$ (``small diamond")
 to be\footnote{In \cite{AnRicLou}, the notation used is $\sX^{\diamond}$. Our notation is intended to make the distinction from $\sX^{\diam}$ more transparent.} the $v$-sheaf over $\Spd(\BZ_p)$ which associates to the perfectoid Tate Huber pair $(R, R^+)$ the set of isomorphism classes of triples $(R^\sharp, \iota, f)$, where $(R^\sharp, \iota)$ is an untilt over $\BZ_p$ and $f: A\to R^{\sharp,+}$ is a ring homomorphism.
 \smallskip
 
 ii) We let $\sX^\diam=\Spec(A)^\diam$ (``large diamond")
 to be the $v$-sheaf over $\Spd(\BZ_p)$ which associates to the perfectoid Tate Huber pair $(R, R^+)$ the set of isomorphism classes of triples $(R^\sharp, \iota, f)$, where $(R^\sharp, \iota)$ is an untilt over $\BZ_p$ and $f: A\to R^{\sharp}$ is a ring homomorphism.
 \smallskip
 
 Both these constructions ``glue" to give functors $(\ )\mapsto (\ )^{\bDiam}$, resp. 
 $(\ )\mapsto (\ )^\diam$, from the category of schemes over $\Spec(\BZ_p)$ to the category of $v$-sheaves
 over $\Spd(\BZ_p)$. There is a natural transformation $j: (\ )^{\bDiam}\mapsto (\ )^\diam$ such that
 \begin{equation}\label{jsX}
 j_\sX: \sX^{\bDiam}\to \sX^\diam
 \end{equation}
 is an open immersion of $v$-sheaves if $\sX$ is separated of finite type over $\BZ_p$. It is an isomorphism if $\sX$ is proper, cf. \cite[\S 2.2]{AnRicLou}.

For schemes separated of finite type over $\BZ_p$, both of these functors can be obtained by first going through the category of (pre-)adic spaces:
 If $\sX$ is such a scheme,
 we denote by $\wh\sX^{\rm ad}$ the adic space over ${\rm Spa}(\BZ_p)=\Spa(\BZ_p, \BZ_p)$ obtained from the formal scheme $\wh\sX$ given by the $p$-adic completion of $\sX$.  We can apply \S\ref{sss:vs} to $\CY=\wh\sX^{\rm ad}$. As in \S \ref{sss:vfs}, we then just write $(\wh\sX)^\diam$ instead of $(\wh\sX^{\rm ad})^\diam$.   By \cite[Rem. 2.11]{AnRicLou} we have a natural isomorphism   for the associated $v$-sheaves, 
\[
 \sX^{\bDiam}=(\wh\sX)^\diam.
 \]
 On the other hand, we can  consider the adic space over $\Spa(\BZ_p )$ given by the fiber product defined as in \cite[Prop. 3.8]{Huber} 
 \[
 \sX^{\rm ad}=\sX\times_{\Spec(\BZ_p)}\Spa(\BZ_p ).
 \]
We then have $\sX^\diam=(\sX^{\rm ad})^\diam$.
 There is a natural open embedding  of adic spaces 
 \[
 \wh\sX^{\rm ad}\lhook\joinrel\longrightarrow \sX^{\rm ad} ,
 \]
 which is an isomorphism if $\sX\to \Spec(\BZ_p)$ is proper, cf. \cite[ Rem. 4.6 (iv)]{Huber}.
 After applying the $\diam$-functor of \S\ref{sss:vs}, this gives the open immersion \eqref{jsX}.

Suppose now that $X$ is a separated scheme of finite type over $\BQ_p$ or a rigid analytic space over $\BQ_p$. We can consider the corresponding adic space $X^{\rm ad}$ over $\Spa(\BQ_p, \BZ_p)$ and then its associated $v$-sheaf $(X^{\rm ad})^\diam$ over $\Spd(\BQ_p)$ as in \S \ref{sss:vs} above, cf. \cite[10.2]{Schber}. We denote this simply by $X^\diam$. (This is a diamond in the sense of Scholze, \cite{Schber}, \cite{Sch-Diam}.) This gives functors $X\mapsto X^\diam$   (on the category of schemes over $\BQ_p$, resp. of rigid-analytic spaces over $\BQ_p$); in the case of schemes over $\BQ_p$ this agrees 
 with the functor $X\mapsto X^\diam$ defined above. 
  
 If $\sX$ is a separated scheme of finite type over $\BZ_p$,
 there is a natural isomorphism of $v$-sheaves over $\Spd(\BQ_p)$,
 \[
 (\sX^\diam)_\eta:=\sX^\diam\times_{\Spd(\BZ_p)}\Spd(\BQ_p)= (\sX\otimes_{\BZ_p}\BQ_p)^\diam .
 \]
Similarly, we can consider the $v$-sheaf $\sX^{\sdiam}$ over $\Spd(\BZ_p)$ and its ``generic fiber" which is a $v$-sheaf over $\Spd(\BQ_p)$,
\[
(\sX^{\sdiam})_\eta: =\sX^{\sdiam}\times_{\Spd(\BZ_p)}\Spd(\BQ_p)=(\wh\sX)^\diam\times_{\Spd(\BZ_p)}\Spd(\BQ_p) .
\]
 We have a natural isomorphism of $v$-sheaves over $\Spd(\BQ_p)$,
 \[
 (\sX^{\sdiam})_\eta= (\wh\sX_\eta)^\diam,
 \]
 where $\wh\sX_\eta$ is Berthelot's rigid analytic generic fiber of the formal scheme $\wh\sX$. 
 There is  an open immersion   of $v$-sheaves
 \begin{equation}
 j_{\sX_\eta}: (\sX^{\bDiam})_\eta= (\wh\sX_\eta)^\diam\lhook\joinrel\longrightarrow (\sX\otimes_{\BZ_p}\BQ_p)^\diam=  (\sX^\diam)_\eta,
 \end{equation}
obtained by taking the generic fiber of \eqref{jsX}.  {\cmag  It also arises by applying} the $\diam$-functor to the open embedding of rigid-analytic spaces $\wh\sX_\eta\hookrightarrow (\sX\otimes_{\BZ_p}\BQ_p)^{\rm rig}$.

 {\cmag
 \begin{proposition}\label{propFFvsheaf}  (\cite[Prop. 18.4.1]{Schber}, cf.  \cite{LourencoRH}) The functor $\fkX\mapsto \fkX^\diam$ from flat and normal formal schemes locally formally of finite type over 
$\Spf(\BZ_p)$ to $v$-sheaves over $\Spd(\BZ_p)$ is fully faithful. \qed
 \end{proposition} 
  \quash{
 \begin{proof} 
See \cite[Ch. 4, Thm. 4.6]{LourencoThesis} for 
the proof of a slightly stronger result.
This is obtained by combining  \cite[Prop. 10.2.3]{Schber} which gives the result for such adic spaces over $\Spa(\BQ_p, \BZ_p)$, and 
\cite[Prop. 18.4.1]{Schber} which uses \cite{LourencoRH} and gives the corresponding result for normal flat 
formal schemes over ${\rm Spf}(\BZ_p)$ which are locally formally of finite type. 
 \end{proof}}

 \begin{corollary}\label{corFF}
Let $\sX$, $\sY$ be normal schemes which are flat and separated of finite type over $\Spec(\BZ_p)$.
Let $f^v: \sX^{\bDiam}\to \sY^{{\bDiam}}$ be a morphism between the corresponding $v$-sheaves over $\Spd(\BZ_p)$
and $g: \sX\otimes_{\BZ_p}\BQ_p\to
 \sY\otimes_{\BZ_p}\BQ_p$ a morphism of schemes between their generic fibers over $\Spec(\BQ_p)$.
 Suppose that the diagram of $v$-sheaf maps 
 \begin{equation}
 \begin{aligned}
   \xymatrix{
         \sX^{\bDiam}\times_{\Spd(\BZ_p)}\Spd(\BQ_p) \ar[r]^{\ \ \ j_{\sX_\eta}} \ar[d]_{f^v\times_{\Spd(\BZ_p)}{\rm id}_{\Spd(\BQ_p)}} &  (\sX\otimes_{\BZ_p}\BQ_p)^\diam\ar[d]^{g^\diam}\\
         \sY^{\bDiam}\times_{\Spd(\BZ_p)}\Spd(\BQ_p)\ar[r]^{\ \ \ j_{\sY_\eta}} & (\sX'\otimes_{\BZ_p}\BQ_p)^\diam
        }
        \end{aligned}
    \end{equation}
commutes. 
Then 
 there is a unique morphism $f: \sX\to \sY$ of schemes over $\Spec(\BZ_p)$ such that $f^{\bDiam}=f^v$
 and $g=f\otimes_{\BZ_p} {\rm id}_{\BQ_p}$. 
\end{corollary}
\begin{proof}
Using Proposition \ref{propFFvsheaf} we obtain that there is $\wh f: \wh\sX\to \wh\sY$ between
formal $p$-adic completions which corresponds to $f^v$, i.e. with $f^v=(\wh f)^\diam=f^{\bDiam}$. By the full-faithfulness of the $\diam$-functor from normal rigid analytic spaces to $v$-sheaves  (see \cite[Prop. 10.2.3]{Schber}), $\wh f$ extends the morphism
between the rigid generic fibers induced by $g$. Note that the pair $(g, \wh f)$ defines a map of underlying topological spaces
$|g|\cup|\wh f|: |\sX|\to |\sY|$ which is continuous since it commutes with specializations. Using this and uniqueness, we see that  we can reduce showing the existence of $f$ to the case where $\sX$ and $\sY$ are affine.  But for a normal flat $\BZ_p$-algebra $A$ of finite type, we have $A=(A\otimes_{\BZ_p}\BQ_p)\cap \wh A$ (intersection in $\wh{A}\otimes_{\BZ_p}\BQ_p$). It now follows that $g$ respects the $\BZ_p$-integral structures and extends to $f$. 
\end{proof}
{\cmag The above statement motivates the following definition.
\begin{definition}\label{def:diam/}
For a scheme $\sX$ over $\Spec(\BZ_p)$, we let
\begin{equation}\label{notdiag}
\sX^{\diam/}:=\sX^\sdiam\sqcup_{(\sX^\sdiam\times_{\Spd(\BZ_p)}\Spd(\BQ_p))}(\sX\otimes_{\BZ_p}\BQ_p)^\diam
\end{equation}
be the coproduct $v$-sheaf over $\Spd(\BZ_p)$. Then $\sX\mapsto \sX^{\diam/}$ gives a functor from schemes over $\Spec(\BZ_p)$ to $v$-sheaves over $\Spd(\BZ_p)$. 
\end{definition}}

\begin{remark}
The considerations of \S\S \ref{sss:vs}--\ref{sss:vsSch} extend to the case where the ground ring $\BZ_p$ is replaced by a complete discrete valuation ring with perfect residue field. We apply this to the ring of integers in a finite extension $E$ of $\BQ_p$, or of $\breve\BQ_p$.
\end{remark}
 
}

\subsubsection{Products of points}\label{vcover} For future use we record the construction, following \cite{Sch-Diam}, \cite{Schber}, of a cover for the $v$-topology of the affinoid perfectoid $S=\Spa(R, R^+)$ over $k$: Consider a product $\prod_{i\in I} V_i$ of valuation rings $V_i$ with complete algebraically closed fraction field $K_i$ of characteristic $p$, where $I$ ranges over the set of points of $S$, each point given by $\Spa(K_i, V_i)\to \Spa(R, R^+)$ (see \cite[Prop. 4.2.5]{Schber}). Let $\varpi\in R^+$ be a pseudouniformizer for $R$ and denote by $\varpi_i\in V_i$ its image under $R\to K_i$; then $\varpi_i$ is 
 a pseudouniformizer for $K_i$. Set $\underline{\varpi}=(\varpi_i)_i$. Now let $B^+=\prod_{i\in I} V_i$ with the $\underline\varpi$-adic topology and set $B=B^+[1/\underline{\varpi}]\subset \prod_{i\in I} K_i$. The map $T=\Spa(B, B^+)\to S=\Spa(R, R^+)$ gives a $v$-cover of $S$. We call such a $T$ a \emph{product of points}. Note that  a product of points $T$ as above is a {\sl strictly totally disconnected} perfectoid space, in the sense of \cite{Sch-Diam}, {\cred see \cite[Prop. 1.6]{Gl}.}
Using $v$-covers given by products of points, we can often reduce various questions to the case $S=\Spa(C, C^+)$ with $C$ an algebraically closed field.

 \subsection{Shtukas} 
 
 We recall the definition of Scholze's $p$-adic mixed characteristic shtukas over perfectoid spaces.
 
 \begin{definition}\label{vectsht}
  Let $S=\Spa(R, R^+)\in{\rm Perfd}_k$, i.e. $S$ is affinoid perfectoid over $k$, and let 
 $S^\sharp=\Spa(R^\sharp, R^{\sharp +})$ be an untilt of $S$ over $  W(k)$. A  \emph{shtuka     over $S$} with one leg at $S^\sharp$ is a vector bundle $\sV$  over   the analytic adic space $S\bdtimes \BZ_p $ together with an isomorphism  
 $$
\phi_{\sV}: {\rm Frob}_S^*(\sV)|_{S\bdtimes \BZ_p\setminus S^\sharp}\xrightarrow{\sim} \sV|_{S\bdtimes \BZ_p\setminus S^\sharp}
 $$ 
 which is meromorphic along the closed Cartier divisor $S^\sharp$ of $S\bdtimes \BZ_p$.  
 \end{definition}

   The rank of the vector bundle $\sV$ is also called the height of the shtuka. Note that
 \begin{equation}\label{prodformulaEq}
 ( S\bdtimes  \BZ_p  )^\diamondsuit=S\times{\rm Spd}(\BZ_p)
 \end{equation}
 and that the untilt $S^\sharp$ corresponds to a section of $S\times{\rm Spd}(\BZ_p)\to S$, or equivalently to a morphism  $S\to  \Spd(\BZ_p)$, see \cite[11.2, 11.3]{Schber}. Hence, instead of saying ``a shtuka over $S$ with one leg at $S^\sharp$", we may equivalently say ``a shtuka over $S/\Spd(\BZ_p)$". 
 
 Let us remark here that, by using Proposition \ref{vstack}, one sees that 
 the notion of shtuka extends to general perfectoid spaces $S$ over $k$ with a morphism $S\to \Spd(\BZ_p)$ and that sending 
 $S/ \Spd(\BZ_p)$ to the groupoid of shtukas over $S/\Spd(\BZ_p)$ gives a stack for the $v$-topology.

 To simplify notations, we often write $\phi$ for ${\rm Frob}_S$.

\begin{definition}\label{vectshtMin} 
 A \emph{minuscule shtuka}  of height $h$ and dimension $d$ over $S$ with one leg at $S^\sharp$ is a shtuka $(\sV, \phi_{\sV})$  of height $h$ over $S$ with one leg at $S^\sharp$ such that
 \[
\sV\subset \phi_{\sV}({\rm Frob}_S^*(\sV))
\]
with $\phi_{\sV}({\rm Frob}_S^*(\sV))/\sV$ of the form $(i_{S^\sharp})_*(\sW)$, where $\sW$ is a vector bundle of rank $d$ over $S^\sharp$. 
 \end{definition}
 
We note that in the theory of Shimura varieties, it is not enough to  only consider minuscule shtukas because in the context of $\CG$-shtukas (cf. \S \ref{ss:CG-shtukas} below), even if $\mu$ is a minuscule cocharacter of $G$, there may not exist faithful representations $r\colon G\to \GL_h$ such that $r\circ\mu$ is a minuscule cocharacter of $\GL_h$. 

  \subsubsection{Shtukas and BKF-modules}\label{sss:BKFmod} We recall the related notion of a Breuil-Kisin-Fargues (BKF-) module, comp. \cite[Def. 11.4.3]{Schber} (for algebraically closed nonarchimedean fields).
  
  Recall (\cite[Def. 17.5.1]{Schber}) that an \emph{integral perfectoid ring} is a $p$-complete $\BZ_p$-algebra $R$ such that Frobenius is surjective on $R/p$, such that there is an element $\pi\in R$ with $\pi^p=pu$ for a unit $u\in R^\times$, and such that the kernel of $\theta\colon W(R^\flat)\to R$ is a principal ideal $(\xi)$. Here $R^\flat =\varprojlim\nolimits_{x\mapsto x^p} R$. 
    
 \begin{remarks}\label{remintperf} (i)  Let $S=\Spa(R, R^+)\in{\rm Perfd}_k$, i.e., $S$ is affinoid perfectoid over $k$, and let 
 $S^\sharp=\Spa(R^\sharp, R^{\sharp +})$ be an untilt of $S$. Then  $R^{\sharp +}$ is integral perfectoid (see \cite[Lem. 3.20, in combination with Lem. 3.9 and the discussion in Rem. 3.8]{BMS}).
 
 (ii) If $R$ is integral perfectoid with $pR=0$, then $R=R^\flat$ and $(\xi)=(p)\subset W(R)$, cf \cite[Lem. 3.10]{BMS}. Hence $R$ is a perfect ring. Conversely, a perfect ring
 in characteristic $p$ is integral perfectoid. 
 \end{remarks}
 
\begin{definition}\label{def:BKF}
Let $R$ be an integral perfectoid ring.  A \emph{Breuil-Kisin-Fargues (BKF-)module}
 over $R$  is a vector bundle $\sV$  over $\Spec(W(R^\flat))$ together with an isomorphism 
 \[
 \phi_{\sV}: \phi^*(\sV)[1/\xi]\xrightarrow{\ \sim\ } \sV[1/\xi].
 \]
 
If $S=\Spa(R, R^+)\in{\rm Perfd}_k$ and 
 $S^\sharp=\Spa(R^\sharp, R^{\sharp +})$ is an untilt of $S$, we also speak of a \emph{BKF-module over $S$ with leg along $S^\sharp$} instead of a BKF-module over $R^{\sharp +}$. 
  \end{definition}
  
 \begin{remark}
  In \cite{Schber}, \cite{BMS}, the terminology ``{Breuil-Kisin-Fargues (BKF-)module}
 over $R$"  is also used for a vector bundle $\sV_\inf$  over $\Spec(W(R^\flat))$ together with an isomorphism
 \[
 \phi_{\sV_\inf}: \phi^*(\sV_\inf)[1/\phi(\xi)]\xrightarrow{\ \sim\ } \sV_\inf[1/\phi(\xi)].
 \]
  For example, such a BKF-module is naturally associated to a $p$-divisible group over $R$ using Dieudonn\'e theory, see \cite[Thm. 17.5.2]{Schber}. To distinguish from the above, we will say that this has leg along $\phi(\xi)=0$ (or along $\phi^{-1}(S^\sharp)$). Since $\phi: W(R^\flat)\to W(R^\flat)$ is an isomorphism, there is an exact equivalence between the categories of the two types of BKF-module, which is obtained by twisting the $W(R^\flat)$-module structure: $\sV=(\phi^{-1})^*(\sV_\inf)=W(R^\flat)\otimes_{\phi^{-1}, W(R^\flat)}\sV_\inf$.
  (See also \cite[Rem. 11.4.6]{Schber}.)
 \end{remark}

 \begin{definition}\label{defassshtukaBKF}
 Let $S=\Spa(R, R^+)\in{\rm Perfd}_k$  and let
 $S^\sharp=\Spa(R^\sharp, R^{\sharp +})$ be an untilt of $S$. Let $\sV$ be a BKF-module over $S$ with leg at $S^\sharp$. The  \emph{shtuka over $S$ associated to $\sV$}  with one leg at $S^\sharp$ is the shtuka obtained by pull back along the map of locally ringed spaces 
 \begin{equation*}
 \CY_{[0,\infty)}(R,R^+)\to \Spec(W(R^+)) .
 \end{equation*}
We then also say that $\sV$ is an extension of the shtuka. In general, this extension is not uniquely determined by the shtuka.  
  \end{definition}
 \begin{proposition}\label{exttoW} 
 Let $S=\Spa(R, R^+)\in{\rm Perfd}_k$ and let
 $S^\sharp=\Spa(R^\sharp, R^{\sharp +})$ be an untilt of $S$. The restriction functor from the category of BKF-modules over $S$ with leg at $S^\sharp$ to the category of shtukas over $S$ with leg at $S^\sharp$ is faithful. It is fully faithful if $R^+=R^o$. It is an equivalence of categories  if $S=\Spa(C, O_C)$ for an algebraically closed non-archimedean extension $C$ of $k$. 
 \end{proposition}
 \begin{proof}
By Theorem \ref{FFKedlaya} (b), restriction along $j(R,R^+)\colon Y(R,R^+)\hookrightarrow\Spec (W(R^+))$ is a fully faithful functor from the category of vector bundles on $\Spec (W(R^+))$  to the category of vector bundles on $Y(R,R^+)$. By Theorem \ref{FFKedlaya} (a), there is an equivalence of categories between the categories of vector bundles on $\CY(R,R^+)$ and on $Y(R,R^+)$. The first assertion now follows since the restriction along $\CY_{[0,\infty)}(R,R^+)\hookrightarrow \CY(R,R^+)$ 
is also faithful. By Proposition \ref{FFres}, it is fully faithful when $R^+=R^\circ$, but not in general, by the comment after the statement of Proposition \ref{FFres}. 
{\cmag The last statement is Fargues' theorem. See \cite[Thm. 14.1.1]{Schber} which states this result (and more) when $C^\sharp$ has characteristic $0$. The statement in the case that $C^\sharp$ has characteristic $p$ is shown by the same argument, as outlined in the proof of \cite[Thm. 14.1.1]{Schber}.}
 \end{proof}

\subsubsection{The Fargues-Fontaine curve}\label{ss:FF} For $S=\Spa(R, R^+)$ affinoid perfectoid over $k$, we can consider the (adic) relative {\sl Fargues-Fontaine curve}
$X_{{\rm FF}, S}$ defined   as the quotient
\begin{equation}
X_{{\rm FF}, S}=\CY_{(0,\infty)}(S)/({\rm Frob}_S)^\BZ.
\end{equation}

Let $(\sV, \phi_\sV)$ be a shtuka over $S$ with one leg at $S^\sharp$. Then, there is
$r>0$ such that $\CY_{[r,\infty)}(S)$ does not intersect the divisor given by $S^\sharp$. Note that 
$\CY_{[r,\infty)}(S)\to X_{{\rm FF}, S}$ is surjective for all $r>0$. The restriction of
$\sV$ to $\CY_{[r,\infty)}(S)$ descends to a vector bundle $\sV_{{\rm FF}}$ on the quotient $X_{{\rm FF}, S}$.
Also, there is $r'>0$ such that
$\CY_{(0,r']}(S)$ does not intersect the divisor given by $S^\sharp$ and, as above, we
can descend the restriction of $(\sV, \phi_\sV)$ to $\CY_{(0,r']}(S)$ to a vector bundle $\sV'_{{\rm FF}}$ on $X_{{\rm FF}, S}$ (see \cite[\S 14.1 and Prop. 22.1.1]{Schber}). We can see, by pulling back to products of points, that both these functors
 $(\sV,\phi_\sV)\mapsto \sV_{\rm FF}$ and  $(\sV,\phi_\sV)\mapsto \sV'_{\rm FF}$ from the category of shtukas over $S$ to the category of vector bundles over $X_{\rm FF, S}$ are faithful. (They are not fully faithful.)
 
 Recall that $X_{{\rm FF}, S}$ makes sense for all perfectoid spaces $S$, and sending $S$ to the groupoid of vector bundles on $X_{{\rm FF}, S}$ defines a $v$-stack
(\cite[Prop. II.2.1]{FS}).

\subsection{Families of shtukas}\label{FamSht} We will also want to consider ``families" of  shtukas. This leads to the following definition.  

 {\cmag
 \begin{definition}\label{deffamofsht}
Let $\CF$ be a $v$-sheaf over ${\rm Spd}(\BZ_p)$. A  shtuka $(\sV, \phi_\sV)$ over $\CF$   is a section of the $v$-stack given by the groupoid of shtukas 
 over $\CF$. In other words, a shtuka over $\CF$ is a functorial rule  which to any point in  $x\in  \CF(S)$, where $S\in {\rm Perfd}_k$,  associates a  shtuka  $(\sV_S , \phi_{\sV_S })$   over $S$ with one leg at the untilt $S^\sharp$ given by $S\xrightarrow{x} \CF\to \Spd(\BZ_p)$.  
\end{definition} 

  As an example, let $\CX$ be an adic space  
  over $\Spa(\BZ_p)$ and denote as before, in Section \ref{sss:vs},  
  by $\CX^\diam\to \Spd(\BZ_p)$ the corresponding $v$-sheaf.  Then we obtain the notion of a shtuka over $\CX^\diam$.
 We can think of the  shtuka $\sV$ as having one leg  at the ``universal'' untilt given by $\CX^\diam \to {\rm Spd}(W(k))\to {\rm Spd}(\BZ_p)$. 
 
 \begin{definition}\label{def:shtsch}
 Let $O_E$ be a complete dvr of mixed characteristic with perfect residue field. Let $\sX$ be a scheme over $O_E$. A shtuka over $\sX$ is a shtuka over the $v$-sheaf (over $\Spd(O_E)$), 
 \[
 \sX^{\diam/} :=\sX^\sdiam\sqcup_{(\sX^\sdiam)_\eta} (\sX\otimes_{O_E}E)^\diam ,
 \]
cf. \eqref{notdiag}. In other words, a shtuka over $\sX$ is given by  a pair consisting  of a shtuka over the $v$-sheaf $\sX^\sdiam$ and a shtuka over the $v$-sheaf $(\sX\otimes_{O_E}E)^\diam$ together with an isomorphism between their pull backs to $(\sX^\sdiam)_\eta =\sX^\sdiam\times_{\Spd(O_E)}\Spd(E) $.
\end{definition}
}

{\cmag
\begin{remark}\label{rem:shtsch}
a) If $p$ is a unit in $\Gamma(\sX, O_{\sX})$, then  a shtuka over $\sX$ is the same as a shtuka over the $v$-sheaf $\sX^\diam$. By contrast, if   $p$ is nilpotent in $\Gamma(\sX,O_\sX)$,  
 then a shtuka over $\sX$ is a shtuka 
 over the $v$-sheaf $\sX^\sdiam$. Often, for clarity, we will still specify the $v$-sheaf ($\sX^{\diam/}$, $\sX^\diam$ or $\sX^\sdiam$), that we are using when talking of shtukas over schemes. 

b) By a shtuka over a formal scheme $\frak{X}$ over $\Spf(O_E)$ we mean a shtuka over the $v$-sheaf $(\frak{X}^{\rm ad})^\diam$. In the above description of Definition \ref{def:shtsch}, 
if $\sX\to \Spec(O_E)$ is separated of finite type, a shtuka over $\sX^\sdiam$ is given by a shtuka over the formal scheme $\wh\sX$ given by the $p$-adic completion, see \S \ref{sss:vsSch}.

 c) There is a natural map of $v$-sheaves over $\Spd(O_E)$
\[
\sX^{\diam/}\to \sX^\diam.
\]
 This is an isomorphism if $\sX\to \Spec(O_E)$ is proper, cf. \cite[\S 2.2]{AnRicLou}. It is not surjective in general, as one can see in the example of $\sX=\Spec(O_E[t])$.
One could also consider shtukas over the $v$-sheaf $\sX^\diam$;  the ``smaller" $v$-sheaf $\sX^{\diam/}$ is better suited to our application.
\end{remark}}

\begin{example}\label{pdivExample}
{\cmag Let $\sG$ be a $p$-divisible group of height $h$ and dimension $d$ over $\sX$, where $\sX$ is a scheme over $\Spec(W(k)) $. 
Then there is an associated shtuka $\CE(\sG)$ of height $h$ and dimension $d$ over $\sX^\sdiam$ with one leg, as follows. We may assume that $\sX=\Spec(A)$.
 Let $S=\Spa (R, R^+)\in {\rm Perfd}_{k}$. Suppose that $(S^\sharp, x)$ is a point of $\sX^\sdiam$ over $\Spd(\BZ_p)$, given by an untilt $S^\sharp=\Spa(R^\sharp, R^{\sharp, +})$ of $S$ and $x^*: A\to R^{\sharp +}$.}
 
 Recall that using \cite[Thm. 17.5.2]{Schber} we can associate to a $p$-divisible group $\sG$ over $\Spec(R^{\sharp +})$ a finite projective $W(R^+)$-module $M_\inf=M_\inf(\sG)$ together
 with an isomorphism
 \[
 \phi_{M_\inf}: \phi^* (M_\inf)[1/\phi(\xi)]\xrightarrow{\sim} M_\inf[1/\phi(\xi)].
 \]
 Note that, in this, the leg is along $\phi(\xi)=0$. We have
 \[
 M_\inf\subset  \phi_{M_\inf} (\phi^* (M_\inf))\subset M_\inf[1/\phi(\xi)].
 \]
  The module $M_\inf$ is obtained using Dieudonn\'e theory. If $pR^{\sharp +}=(0)$ so that $R^{\sharp +}=R^+$,
 then $M_\inf(\sG)$ is canonically the $W(R^+)$-linear dual of the value of the contravariant Dieudonn\'e crystal $\BD(\sG)$ of $\sG$ at $W(R^+)$, i.e. $M_\inf(\sG)=\BD(\sG)(W(R^+))^*$. For example, we have $M_\inf(\mu_{p^\infty})=W(R^+)$, with $\phi$ given by $p^{-1}$.
 
 Then the value of the shtuka $\CE(\sG)$ on $(S^\sharp, x)$ is given by the following shtuka $\CE_S$ on $S$ with one leg at $S^\sharp$. Namely, let $\CE_S$ be equal to the restriction  to the complement $\CY_{[0,\infty)}(S)=\Spa(W(R^+))\setminus \{[\varpi]=0\}$  of the pull-back $(\phi^{-1})^*(M_\inf(x^*(\sG)), \phi_{M(x^*(\sG))})$.  In other words, $\CE(\sG)$ is the shtuka associated to the BKF-module (with leg along $\xi=0$)
 $$
 (M(x^*(\sG)), \phi_{M(x^*(\sG))})=(\phi^{-1})^*(M_\inf(x^*(\sG)), \phi_{M_\inf(x^*(\sG))}) .
 $$ 
  \end{example}
 {\cmag
 \begin{remark}\label{redia}
  Note that the previous  construction does not apply to the ``large diamond" $\sX^\diam$. For example, we do not know how to define for a  $p$-divisible group over $\Spec(\BF_p[t])$ a shtuka over $\Spec(\BF_p[t])^\diam$ that extends the shtuka over $\Spec(\BF_p[t])^\sdiam$ given above.
  This is one reason that we mainly consider shtukas over $\sX^\sdiam$, or, in the mixed characteristic case, over $\sX^{\diam/}$.
  \end{remark}}

\subsubsection{Shtukas in characteristic $p$}

We will see that shtukas in characteristic $p$ have a Dieudonn\'e module like behaviour. We introduce the following definition. 
{\cmag \begin{definition}Let $\sX$ be a perfect {\cmag quasi-compact and quasi-separated} $k$-scheme. 
 A \emph{meromorphic Frobenius crystal} over $\sX$ is a pair $(\sM, \phi_\sM)$, where  $\sM$ is a finitely generated projective module over the sheaf of rings $W(\CO_\sX)$  and $\phi_\sM$ is an isomorphism
\begin{equation*}
\phi_\sM\colon \phi^*(\sM)[{1}/{p}]\xrightarrow{ \sim\ } \sM[{1}/{p}] .
\end{equation*}
Here $\phi$ denotes the Frobenius on $W(\CO_\sX)$.  The pair $(\sM[1/p], \phi_\sM)$ is the corresponding Frobenius isocrystal over $\sX$.
\end{definition} }
\begin{remark}\label{meroFrobBKF}
By Remark \ref{remintperf} (ii), we see that a meromorphic Frobenius crystal over the perfect $k$-scheme $\Spec(A)$ is the same as a BKF-module over the integral perfectoid ring $A$. 
\end{remark}

A meromorphic Frobenius crystal  $(\sM, \phi_\sM)$ over the perfect $k$-scheme $\sX$ gives a shtuka over  $\sX^\sdiam$, as follows. {\cmag Suppose $\sX=\Spec(A)$ is affine.} Let $S=\Spa(R, R^+)\in {\rm Perfd}_k$. Then a point of $\sX^\sdiam$ with values in $S$ is given by an  untilt $S^\sharp=\Spa(R^\sharp, R^{\sharp +})$, where $R^{\sharp +}$ is a $k$-algebra,  and a homomorphism $x^*\colon A\to R^{\sharp +}$. Then $R^\sharp=R$ and the  kernel of the natural map $W(R^+)\to R^{\sharp +}$ is generated by $\xi=p$. By extension of scalars, the map $W(A)\to W(R^+)$ defines therefore a BKF-module $\sM^\natural=\sM\otimes_{W(A)}W(R^+)$ over $S$ with leg along $S^\sharp$, which in turn defines a shtuka  over $S$ with leg along $S^\sharp$, cf. Definition \ref{defassshtukaBKF}.

\begin{theorem}\label{FFisocrystal}
The functor given by the construction above gives an exact fully faithful tensor equivalence
 from the category of meromorphic Frobenius crystals over the perfect $k$-scheme $\sX$ to the category of shtukas  over $\sX^\sdiam$. 
\end{theorem}

\begin{proof}\footnote{A different proof of this Theorem is given in \cite{GlIv}, see \emph{loc. cit.} Thm 10.4. The proof in \cite{GlIv} still uses Sen theory as in II) below.}We first show that the functor is fully faithful. {\cmag We can quickly reduce to the affine case $\sX=\Spec(A)$.}

Let $\psi: (\sM_{\rm sht}, \phi_{\sM_{\rm sht}})\to (\sN_{\rm sht}, \phi_{\sN_{\rm sht}})$ be a homomorphism between the shtukas corresponding to 
$(\sM, \phi_{\sM})$ and  $(\sN, \phi_{\sN})$. Set $(R, R^+)=(A\llps t^{1/p^\infty}\lrps, A\lps t^{1/p^\infty}\rps)$.
Recall that here $A\pstperf$ is the $(t)$-adic completion of the perfect algebra $A[t, t^{1/p}, t^{1/p^2}, \ldots ]$. The elements of 
$A\pstperf$ are represented as power series
\[
\sum_{i\in \BZ[1/p]_{\geq 0}} a_i t^i
\]
with $a_i\in A$ and with support (i.e. set of indices $i$ for which $a_i\neq 0$) which is either finite, or forms an increasing unbounded sequence. 
Then, $A\lpstperf=A\pstperf[1/t]$. Consider the $v$-cover $\pi\colon S=\Spa(R, R^+)\to \Spec(A)^\sdiam$. 

The homomorphism $\psi$ gives, by evaluating on $S$, a homomorphism of vector bundles  over $\CY_{[0,\infty)}(R,R^+)=\CY_{[0,\infty)}(S)$,
\[
\psi(S): \sM_{\rm sht}(S)\to \sN_{\rm sht}(S) .
\]
 These bundles come by restriction from vector bundles $\underline\sM(S)$ and $\underline\sN(S)$ over $\CY_{[0, \infty]}(R,R^+)=\CY(R,R^+)$. Note that $R^+=A\lps t^{1/p^\infty}\rps=R^\circ$, so we can apply Proposition \ref{FFres}  to $r>0$. Using this and then glueing, we see that $\psi(S)$ uniquely extends to a homomorphism of vector bundles over $\CY(R,R^+)$ which respects the Frobenius structures,
\[
\underline \psi(S): \underline\sM(S)\to \underline\sN(S) .
\]
 Recall $\underline\sM(S)$ is the pull-back of $\sM$ under
$\CY(R,R^+)\to \Spec(W(R^+))\to \Spec(W(A))$, and similarly for $\underline\sN(S)$. By Lemma \ref{FFKedlaya} (a), this
comes from a unique $W(R^+)$-linear homomorphism 
\[
\underline \psi(R^+): \sM\otimes_{W(A)}W(R^+)\to \sN\otimes_{W(A)}W(R^+).
\]
It remains to descend this to a $W(A)$-homomorphism by showing that the image of $\sM\subset \sM\otimes_{W(A)}W(R^+)$ under
$\underline \psi(R^+)$ lands in $\sN\subset \sN\otimes_{W(A)}W(R^+)$. {\cmag Since $\sN$ is finite projective, there is a $W(A)$-module $\sN'$ such that $\sN\oplus\sN'\simeq W(A)^n$. Using $\sN\subset \sN\oplus\sN'$ we can view $\underline \psi(R^+)$ as taking values in 
$
(\sN\oplus\sN')\otimes_{W(A)}W(R^+)\simeq W(R^+)^n$. Since 
\[
(\sN\oplus\sN')\cap (\sN\otimes_{W(A)}W(R^+))=\sN,
\]
 it is enough to show that the image of $\underline \psi(R^+)$ lies in $\sN\oplus\sN'\simeq W(A)^n$. Consider $k$-algebra homomorphisms $x^*: A\to K(x)$ with $K(x)$ a perfect field which induce $x^*: A\lps t^{1/p^\infty}\rps\to K(x)\lps t^{1/p^\infty}\rps$. Since $A$ is perfect, $A\hookrightarrow \prod_x K(x)$ for a set of $x$, and to show that an element $w(t)$ of $W(A\lps t^{1/p^\infty}\rps)=\prod_{i\geq 1} A\lps t^{1/p^\infty}\rps$ is ``constant", i.e. it lies in $W(A)=\prod_{i\geq 1} A$, it is enough to show that $x^*(w(t))\in W(K(x))$ for all $x$. Indeed,
 \[
 W(A\lps t^{1/p^\infty}\rps)\cap \prod_x W(K(x))=W(A)
 \]
 with the intersection taking place in $\prod_x W(K(x)\lps t^{1/p^\infty}\rps)$. This shows that it is enough to prove that
 $x^*\underline \psi(R^+)$ is in $W(K(x))$, for all $x$ as above, and it allows us to reduce to the case that $A$ is a field $K$. Now both $\sM$ and $\sN$ are finite free, so $\underline\psi(R^+)$ is given by a matrix $(r_{ij})$ with entries in $W(R^+)=W(K\lps t^{1/p^\infty}\rps)$ which we try to show are in $W(K)$. We use that $\psi(S)$ comes with descent data along $\pi$, i.e.
we have the identity
\begin{equation}\label{2projections}
p_1^*(\psi(S))=p_2^*(\psi(S))
\end{equation}
over $S\times_{\Spec(K)^\diam} S$. Note that 
\[
\CO^+(S\times_{\Spec(K)^\diam} S)=\CO^+(\widetilde{\bf D}^*_{K\lpstperf})=K\lps t_1^{1/p^\infty}, t_2^{1/p^\infty}\rps,
\]
where $\widetilde{\bf D}^*_{K\lpstperf}$ is the perfectoid punctured open disk over $K\lpstperf$ (compare to the proof of \cite[Prop. 18.3.1]{Schber}). Let $r_{ij}\in W(R^+)$ be an entry of the matrix giving $\underline \psi(R^+)$ as above. Using (\ref{2projections}) above, we see that the images of $r_{ij}$ under the two natural
maps
\[
W(R^+)\to \Gamma(\CY_{[0,\infty)}(S),\CO)\to \Gamma(\CY_{[0,\infty)}(S\times_{\Spec(K)^\diam} S), \CO)
\]
agree. Since 
\[
W(\CO^+(S\times_{\Spec(K)^\diam} S))=W(K\lps t_1^{1/p^\infty}, t_2^{1/p^\infty}\rps)\hookrightarrow \Gamma(\CY_{[0,\infty)}(S\times_{\Spec(K)^\diam} S), \CO),
\]
the images of $r_{ij}$ under the two natural
maps
\[
W(R^+)=W(K\lps t^{1/p^\infty}\rps)\to  W(K\lps t_1^{1/p^\infty}, t_2^{1/p^\infty}\rps)
\]
given by $t\mapsto t_1$, $t\mapsto t_2$, agree. This implies that $r_{ij}$ has, at the same time, only powers of $t_1$ and only powers of $t_2$, so it is constant, i.e. belongs to $W(K)$. By the above, this completes the proof of full-faithfulness.}

 We now proceed to show that the functor gives an equivalence of categories by showing it is also 
essentially surjective.

I) We first treat the case that $A=K$ is an algebraically closed field.
Let $(\sV, \phi_\sV)$ be a shtuka  of rank $d$ over $\Spec(K)^\sdiam$. This gives a corresponding vector bundle  $\sV_{{\rm FF}}$ of rank $d$ over the (absolute) Fargues-Fontaine curve $X_{{\rm FF}, \Spec(K)^\sdiam}$. 
This is meant in the sense described in \cite{An2} (since there is  not really  a relative FF curve $X_{{\rm FF}, \Spec(K)^\sdiam}$): there is a category whose objects we think of as the ``vector bundles over $X_{{\rm FF}, \Spec(K)^\sdiam}$". The objects are given by descent from a perfectoid $v$-cover, as follows. 
Set $L=K\llps t^{1/p^\infty}\lrps$, which is perfectoid.
  Then $\pi: T=\Spa(L, O_L)\to \sX^\sdiam=\Spec(K)^\sdiam$ is a $v$-cover and $\pi^*(\sV, \phi_{\sV})$ is 
  a shtuka over $T$ with descent data along $\pi$. By restriction from $\CY_{[0, \infty)}(L,O_L)$ to $\CY_{(0, \infty)}(L,O_L)$
  followed by descent, we obtain a corresponding vector bundle $\sE$ over the Fargues-Fontaine curve $X_{{\rm FF}, T}=X_{{\rm FF}, L}$. The descent datum along $\pi$ gives an isomorphism   over $X_{{\rm FF}, T\times_{\sX^\sdiam}T}$,
  \[
  p^*_1(\sE)\simeq p_2^*(\sE).
  \]
 This describes $\sV_{{\rm FF}}$. In particular, by definition, $\sE=\pi^*(\sV_{FF})$.
 
 In \cite{An2}, Ansch\"utz shows that the category of vector bundles over $X_{{\rm FF}, \Spec(K)^\sdiam}$ is canonically equivalent to the category of Frobenius isocrystals over $W(K)[1/p]$. Hence, there is a Frobenius isocrystal $(V, \phi_V)$ over 
  $W(K)[1/p]$ such that $\sE$ (with its descent datum) is obtained from
  $(V, \phi_V)$. It follows, by the construction of this equivalence, that the pull-back of $(V, \phi_V)$ under 
  \[
  \CY_{(0, \infty)}(L,O_L)\to \Spec(W(O_L)[1/p])\to \Spec(W(K)[1/p])
  \] 
  agrees with the restriction of $\pi^*(\sV, \phi_\sV)$ from $\CY_{[0,\infty)}(L,O_L)$ to $ \CY_{(0, \infty)}(L,O_L)$. 
  In particular, we can choose a framing
 of the shtuka $\pi^*(\sV, \phi_\sV)$ over $\Spa(L, O_L)$ which respects the $v$-descent data, and thus obtain a $\Spa(L, O_L)$-point
 of the moduli stack of shtukas with framing $\CM^{\rm int}_{\GL_d, b, \mu}$ together with $v$-descent data, i.e. a point of $\CM^{\rm int}_{\GL_d, b, \mu}$ with values in $\Spec(K)^\sdiam=\Spd(K)$.
  (See \S  \ref{defintSht} for the notation. Here, the element $b$ is determined by $(V,\phi_V)$ and we take $\mu$ sufficiently large.)  The argument in the proof of \cite[Prop. 2.30]{Gl21} now applies and implies that this point is given by a $\Spec(K)$-valued point of a corresponding affine Deligne-Lusztig variety (see \S \ref{Zhu}, and especially Theorem \ref{Gleason} (a)). This translates 
 to the fact that the shtuka over the whole $\CY_{[0,\infty)}(L, O_L)$ comes from a meromorphic Frobenius crystal $(M, \phi_M)$; here $M$ is a $W(K)$-lattice in $V$ and $\phi_M={\phi_V}_{|M[1/p]}$. This shows that our functor is essentially surjective.

 II) We now consider the general affine case $\sX=\Spec(A)$, with $A$ a perfect $k$-algebra.
Let $(\sV, \phi_\sV)$ be a shtuka  of rank $d$ over $\sX^\sdiam$.
We consider the $v$-cover
 \[
 S=\Spa(A\llps t^{1/p^\infty}\lrps
, A\lps t^{1/p^\infty}\rps)\to \Spd(A).
\]
We have a surjective map of $v$-sheaves
\[
S\times\Spd(\BZ_p)\to \Spd(A)\times\Spd(\BZ_p).
\]
 We also have an isomorphism of $v$-sheaves over $\Spd(\BZ_p)$, 
\begin{equation}\label{productformulaSpa}
\Spd(A)\times\Spd(\BZ_p)\simeq \Spa(W(A))^\diam ,
\end{equation}
obtained {\cmag by the argument in the proof of   \cite[Prop. 11.2.1]{Schber}.}
Here we write $\Spa(W(A))$ for $\Spa(W(A), W(A))$ where $W(A)$ is equipped with the $p$-adic topology; this is a pre-adic space in the sense of \cite[App. to \S 3]{Schber}. By the above $ \sV   $ gives a section of the $v$-stack of vector bundles over $\Spa(W(A))^\diam$. Set 
\[
(B, B^+)=(W(A)[1/p], W(A)).
\]
 Note that as in \cite[Rem. 13.1.2]{Schber},
the open $U=\Spa(B, B^+)$ is an sousperfectoid analytic adic space. In fact, we can obtain a 
 perfectoid cover 
 \[
 \hat U_\infty\to U^\diam=\Spa(B, B^+)^\diam
 \]
 of the $v$-sheaf $U^\diam$ as follows. 
 
 {\cmag Denote by $\hat\CO_\infty=\widehat {\BZ_p[\mu_{p^\infty}]}$ the $p$-adic completion of the 
  ring of integers $\BZ_p[\mu_{p^\infty}]$ in the $\BZ^\times_p$-extension $K_\infty=\BQ_p(\mu_{p^\infty})$ of $\BQ_p$. Write, as usual, $\BZ_p^\times=(\BZ/p\BZ)^\times\times (1+p\BZ_p)$ if $p>2$, and
  $\BZ_2^\times=\BZ/2\BZ\times (1+4\BZ_2)$ for $p=2$. For $p>2$ set 
   $\Gamma=\Gamma_0=1+p\BZ_p\simeq \BZ_p$ and $\Gamma_n=1+p^{n+1}\BZ_p\simeq p^n\BZ_p$, so that $K_n:=(K_\infty)^{\Gamma_n}=\BQ_p(\mu_{p^{n+1}})$, $\CO_n:=\BZ_p[\mu_{p^{n+1}}]$. For $p=2$, set $\Gamma_n=1+2^{n+2}\BZ_2$.}
  Then $\hat\CO_\infty$ supports
  a continuous $\BZ_p^\times$-action. We  now consider
 \[
 \hat U_\infty =\Spa(B\hat\otimes_{\BZ_p}\hat\CO_\infty, B^+\hat\otimes_{\BZ_p}\hat\CO_\infty).
 \]
This is perfectoid. Indeed, set $\hat B_\infty=B\hat\otimes_{\BZ_p}\hat\CO_\infty$;  this is a Tate ring with pseudo-uniformizer 
$1-\zeta_{p^2}$. Its subring of bounded elements is
\[
\hat B_\infty^\circ=\hat B_\infty^+=B^+\hat\otimes_{\BZ_p}\CO_\infty=W(A)\hat\otimes_{\BZ_p}\hat\CO_\infty .
\]
 {\cmag Since $\widehat {\BZ_p[\mu_{p^\infty}]}/(p)\simeq \BF_p[x^{1/p^\infty}]/(x^{p-1})$, we have}
 \[
 \hat B_\infty^\circ/(p)\simeq A[x^{1/p^\infty}]/(x^{p-1})
 \]
 and Frobenius is surjective on $\hat B_\infty^\circ/(p)$. Hence, $\hat B_\infty$ is perfectoid and $\hat U_\infty=\Spa(\hat B_\infty, \hat B_\infty^+)$ is 
 a perfectoid space.  The tilt $\hat U_\infty^\flat$ of $\hat U_\infty$ is
 \[
\hat U_\infty^\flat\simeq \Spa(A\lpstperf, A\pstperf).
 \]
The map $\hat U_\infty\to \Spa(W(A))^\diam=\Spd(A)\times\Spd(\BZ_p)$ is given by $\Spa(A\lpstperf, A\pstperf)\to \Spd(A)$
and the choice of $\hat U_\infty$ as an untilt of $\Spa(A\lpstperf, A\pstperf)$.

By restricting $\sV$ along the ``generic fiber"
 \[
U^\diam=\Spa(W(A))^\diam\times_{\Spd\BZ_p}\Spd\BQ_p\to  \Spa(W(A))^\diam
 \]
 we obtain a section of the $v$-stack of vector bundles over  $U^\diam$. 
By $v$-descent of vector bundles over perfectoid spaces, we see that its value at $\hat U_\infty\to U^\diam$ 
 is given by an actual vector bundle   over the affinoid perfectoid space $\hat U_\infty$, given, in turn, by 
 a finite projective $\hat B_\infty$-module $\wt M$. This comes with descent data for the $v$-cover $\hat U_\infty\to U^\diam$. Here, one needs to be careful: Since $U$ is not perfectoid,  the $v$-descent datum does not automatically 
 give a module over $B$.  
 However, we can use that $\hat U_\infty\to U$ is a pro-\'etale $\BZ_p^\times$-torsor. It is enough to first show that 
 the descent datum for $\hat U_\infty\to U_n$ is effective for some $\Gamma_n$; then we can proceed with usual \'etale descent
 for $U_n\to U$. We first note that using \cite[Cor. 5.4.42]{GabberRamero} 
 we see that there is $n$ and a finite projective $R_n$-module $P_n$ with an isomorphism 
 \[
 P_n\hat\otimes_{U_n}\hat U_\infty\simeq \wt M.
 \]
 This allows to express the descent datum along $\hat U_\infty\to U_n$
 by a continuous $1$-cocycle
 \[
 c: \Gamma_n\to {\rm Aut}_{\hat B_\infty}(P_n\hat\otimes_{B_n}\hat B_\infty)
 \]
 for the $\Gamma_n$-action on $\hat B_\infty$ coming from the $\Gamma_n$-action on $\CO_\infty$.
 For simplicity, we can assume $\Gamma=\Gamma_n$ by a base change and omit the subscript $n$ throughout the rest
 of the argument. After this change, 
 we would like to show that this descent datum is effective and gives an $B$-module $M$
 with a $\Gamma$-equivariant isomorphism
 \[
 M\hat\otimes_B \hat B_\infty\xrightarrow{\sim} \wt M.
 \]
 By ``usual" \'etale descent, we see that
 the effectivity is true if $c$ is cohomologous to a cocycle $c'$ which is trivial on a subgroup of finite index 
 of $\Gamma$. 
 
By (I)  the result holds (i.e. the module $\wt M$ obtained from a shtuka over $\Spd(A)$
 has effective descent datum) when $A$ is an algebraically closed field. Therefore, this is the case after base changing
 by $A\to k'$, where $k'$ is any algebraically closed field. In general, we can understand the set $ {\rm H}^1_{\rm cont}(\Gamma, {\rm Aut}_{\hat B_\infty}(P\hat\otimes_{B}\hat B_\infty))$ using the methods of Sen. In particular, the reference \cite[\S 2]{SenBSMF},
 details a version of this method for Banach $\BQ_p$-algebras, which is sufficiently general for our purposes.
(In fact, Sen considers the more complicated case of continuous cocycles for the Galois group ${\rm Gal}(\bar\BQ_p/\BQ_p)$ and tensors with $\BC_p={\bar\BQ_p}^\wedge$; here we only need to consider $\Gamma={\rm Gal}(K_\infty/\BQ_p)$ and tensor with $\hat K_\infty$.) 
 First of all, by loc.  cit. \S 2.3, there is a continuous cocycle
 \[
 c_\infty: \Gamma\to {\rm Aut}_{B_\infty}(P\otimes_{B}B_\infty)\subset {\rm Aut}_{\hat B_\infty}(P\hat\otimes_{B}\hat B_\infty)
 \]
  which is cohomologous to $c$.
 Using $c_\infty$, Sen defines in loc. cit. \S 2.3, an ``operator"
 $\psi_{c, B}\in {\rm End}_{B}(P)$ (denoted $\phi$ there) whose vanishing is equivalent to the triviality of the restriction 
 of the cohomology class of $c$ on a finite index subgroup of $\Gamma$ (see loc. cit. \S 2.5). The specialization of 
 $\psi_{c, B}$ by $B=W(A)[1/p]\to W(k')[1/p]$ is $\psi_{c, W(k')[1/p]}$. This is zero by the above, for all $A\to k'$, with $k'$ algebraically closed. Hence, since $A$ is reduced and $P$ is a projective module, $\psi_{c, B}=0$ and the
 descent is effective. The Frobenius structure on the module $\wt M$ respects the descent datum and so it descends to a 
 Frobenius structure on $M$.  
 
 So far our construction produced a  finite projective $W(A)[1/p]$-module $M$ with 
 \[
 \phi_M: \phi^*M\xrightarrow{\sim} M
 \]
which induces the pull-back (restriction) of the shtuka   by 
\[
U^\diam=\Spa(W(A)[1/p], W(A))^\diam\to \Spd(A)\times \Spd(\BZ_p).
\]
It remains to show that there is a lattice, i.e. a projective finite $W(A)$-submodule $\sM\subset M$ 
on which $\phi_M$ is meromorphic, such that the shtuka  is induced by $\sM$. Using the $v$-cover
 \[
 S=\Spa(A\lpstperf  
, A\pstperf )\to \Spd(A),
\]
we see that the shtuka over $S$ gives a finite projective $W(A\lpstperf)$-module 
\[
\sM_\infty\subset M\otimes_{W(A)[1/p]}W(A\lpstperf)[1/p]
\]
with descent data.\quash{This corresponds to a $A\lpstperf$-point of the Witt affine Grasmannian
parametrizing $W(A')$-lattices in $M\otimes_{W(A)[1/p]}W(A')[1/p]$, for a variable perfect $A$-algebra $A'$. The existence of the descent data on the shtuka implies that this point
is invariant under the $A$-automorphism $\delta$ of $A\lpstperf$ determined by $t\mapsto t^p$. Using the representability of the
Witt affine Grassmannian of \cite{BS} and $A\lpstperf ^{\delta=1}=A$, we can now see that this $A\lpstperf$-point is induced by a
 uniquely determined $A$-valued point; this gives the desired lattice $\sM\subset M$. {\cmag Indeed, since $\sM[1/p]=\sM\otimes_{W(A)}W(A)[1/p]=M$, the pair $(\sM, \phi_M)$ is
a meromorphic Frobenius crystal. By the construction above, this pair induces the shtuka $(\sV, \phi_\sV)$ over $\Spd(A)$.}}
 {\cmag We will show that the desired $W(A)$-module is $\sM=M\cap \sM_\infty$, the intersection taking place in 
$M\otimes_{W(A)[1/p]}W(A\lpstperf)[1/p]=\sM_\infty[1/p]$.  We first show that $\sM$ is a lattice in $M$, i.e. a finite projective $W(A)$-module
with $\sM[1/p]=M$. By \cite[Thm 4.1]{BS}, it is enough to prove this  after base changing to a (scheme-theoretic) $v$-cover of the perfect scheme $\Spec(A)$. By \cite[Thm 6.1]{IvanovArc}, there is such a $v$-cover, given by some algebra homomorphism $A\to \tilde A$, such that the base change $\tilde M:=M\otimes_{W(A)[1/p]}W(\tilde A)[1/p]$
is  finite free, $\tilde M\simeq  W(\tilde A)[1/p]^h$. The base change $\tilde \sM_\infty=\sM_\infty\otimes_{W(A\lpstperf)}W(\tilde A\lpstperf)$ corresponds to a $\tilde A\lpstperf$-point of the Witt affine Grasmannian of \cite{BS}
parametrizing $W(A')$-lattices in $M\otimes_{W(A)[1/p]}W(A')[1/p]\simeq W(A')[1/p]^h$, for a variable perfect $A$-algebra $A'$. 
The existence of the descent data on the shtuka implies that this point
is invariant under the $\tilde A$-automorphism $\delta$ of $\tilde A\lpstperf$ determined by $t\mapsto t^p$. Using the representability of the
Witt affine Grassmannian of \cite{BS} and $\tilde A\lpstperf ^{\delta=1}=\tilde A$, we can now see that this $\tilde A\lpstperf$-point is induced by a
 uniquely determined $\tilde A$-valued point; this corresponds to a finite projective $W(\tilde A)$-lattice $\tilde\sM\subset \tilde M$ which is then necessarily 
 $\tilde\sM=\tilde M\cap \tilde \sM_\infty$. We conclude as above by $v$-descent that $\sM=M\cap \sM_\infty$ is a lattice in $M$. Now we observe that, since $\sM[1/p]=\sM\otimes_{W(A)}W(A)[1/p]=M$, the pair $(\sM, \phi_M)$ is
a meromorphic Frobenius crystal. By the construction above, this pair induces the shtuka $(\sV, \phi_\sV)$ over $\Spd(A)$.}

Finally, we verify the exactness properties. The exactness of the original functor is easy since meromorphic Frobenius 
crystals are supported on finite projective modules. {\cmag Showing that we have an exact equivalence needs more care. We start with the following lemma.}

\begin{lemma}\label{exactLemma}
Let $B\to B'$ be an injective ring homomorphism with the property that the corresponding map $\Spec(B')\to \Spec(B)$ is surjective on closed points. A complex
 \[
 M_\bullet\ :\ 0\to M_1\to M_2\to M_3\to 0
 \]
 of finite projective $B$-modules is exact if and only if the base change $M_\bullet\otimes_{B}B'$ is exact.
 \end{lemma}
 
 \begin{proof}
 This follows by an argument as in \cite[Lem. 11.4]{An}  which we repeat here.  
 It is enough to show that the exactness of $M_\bullet \otimes_{B}B'$  implies that $M_\bullet$ is exact, so we assume
 $M_\bullet \otimes_{B}B'$ is exact.
 First observe that, since $M_1\otimes_B B'\to M_2\otimes_B B'$ is injective and the $M_i$ are projective and $B\to B'$ is injective, the map $M_1\to M_2$ is also injective. Now let $Q={\rm coker}(M_2\to M_3)$. Then
 $Q$ is finitely generated over $B$. If $\frak m\subset B$ is a maximal ideal then,  by our assumption, there is a maximal ideal $\frak m'\subset B'$ above $\frak m$. We have $Q\otimes_{B}B'_{\frak m'}=(0)$ which, by using Nakayama's lemma, gives $Q_{\frak m}=(0)$. Therefore, since $Q_{\frak m}=(0)$ holds for all $\frak m$, $Q=(0)$. We can now write $M_2=N\oplus M_3$ and apply the same argument to the cokernel 
 of the composition 
 \[
 M_1\to M_2=N\oplus M_3\xrightarrow{\rm pr} N
 \]
 to show it is also trivial. So, $M_1\to N$ is surjective and hence an isomorphism.
 \end{proof}
 
{\cmag  We will apply this Lemma to the natural map $B=W(A)\to B'=W(A\llps t^{1/p^\infty}\lrps)$, where $A$ is a perfect $k$-algebra. The condition on closed points is satisfied: Every maximal ideal of $W(A)$ is the inverse image $\frak m+(p)$ of a maximal ideal $\frak m$ of $A$ via $W(A)\to A=W(A)/(p)$. The kernel of $W(A\llps t^{1/p^\infty}\lrps)\to A\llps t^{1/p^\infty}\lrps\to (A/\frak m)\llps t^{1/p^\infty}\lrps$ is a maximal ideal that lies over $\frak m+(p)$.  

To show that we have an exact equivalence we can quickly reduce to the affine case $\sX=\Spec(A)$. We suppose we have a sequence $\sM_\bullet: 0\to \sM_1\to \sM_2\to \sM_3\to 0$ of finite projective $W(A)$-modules underlying a sequence of meromorphic Frobenius crystals over $A$. We assume that the induced sequence of shtukas over $\Spec(A)^\sdiam$ is exact and we would like to show $\sM_\bullet$ is exact. By our assumption, the induced sequence of shtukas over 
the $v$-cover $S=\Spa(A\llps t^{1/p^\infty}\lrps, A\lps t^{1/p^\infty}\rps)\to \Spec(A)^\sdiam$ is also exact. 
By restricting these shtukas from $\CY_{[0,\infty)}(S)$ to the affinoid $\CY_{[0,1]}(S)$, we obtain an exact sequence of finite projective $\Gamma(\CY_{[0,1]}(S), \CO)$-modules. By the definition of the functor, this sequence is obtained by base changing $\sM_\bullet$ along $W(A)\to \Gamma(\CY_{[0,1]}(S), \CO)$. Recall the natural ring homomorphisms
\begin{equation*}
\Gamma(\CY_{[0,\infty)}(S), \CO)\to \Gamma(\CY_{[0,1]}(S), \CO)\to W(A\llps t^{1/p^\infty}\lrps).
\end{equation*}
The composition $W(A)\to \Gamma(\CY_{[0,1]}(S), \CO)\to W(A\llps t^{1/p^\infty}\lrps)$ is the natural map $W(A)\to W(A\llps t^{1/p^\infty}\lrps)$. The above now shows that the sequence of 
$W(A\llps t^{1/p^\infty}\lrps)$-modules obtained by base changing $\sM_\bullet$ along $W(A)\to   W(A\llps t^{1/p^\infty}\lrps)$ is also exact. We now  apply  Lemma \ref{exactLemma} to $B=W(A)\to B'=W(A\llps t^{1/p^\infty}\lrps)$. 
We obtain that  $\sM_\bullet$ is exact, which 
implies the result.}
\end{proof}

\begin{remark}\label{pdivSpecialization}
Let $\sG$ be a $p$-divisible group over
 $\Spec(O_C)$, where $C$ is a complete non-archimedean 
 algebraically closed field. Let $\kappa=O_C/\mathfrak m_C$ be the residue
 field and set $\bar\sG=\sG\otimes_{O_C}\kappa$ and $i: \Spec(\kappa)\to \Spec(O_C)$.  Then, under the equivalence of Theorem \ref{FFisocrystal}, the pull-back $i^*(\CE(\sG))$ over $\Spec(\kappa)^\sdiam$ of the 
 shtuka $\CE(\sG)$ over $\Spec(O_C)^\sdiam$ 
 described in Example \ref{pdivExample},
 corresponds to the meromorphic Frobenius crystal over $\kappa$ given by
 \begin{equation}\label{Dnatural}
 \BD^\natural(\bar\sG)=(\phi^{-1})^*(\BD(\bar\sG)^*).
 \end{equation}
 Here, $\BD(\bar\sG)=\BD(\bar\sG)(W(\kappa))$ is the contravariant Dieudonn\'e $W(\kappa)$-module of $\bar\sG$,
 and $(\ )^*$ denotes the $W(\kappa)$-linear dual.
 \end{remark}

 \cmag

\subsubsection{Shtukas and BKF-modules, revisited}\label{exBKFgeneral}

Here we explain a certain extension of the shtuka associated to a BKF-module as in Definition \ref{defassshtukaBKF}.  
 
{\cmag Let $S=\Spa(R, R^+)$ be affinoid perfectoid over $\BF_p$ and let $S^\sharp=\Spa(R^\sharp, R^{\sharp +})$ be an untilt of $S$ over $\BZ_p$. 
Then $R^{\sharp +}$ is integral perfectoid; let $(\xi)$ be the kernel
of $W(R^+)\to R^{\sharp +}$, cf. \S \ref{sss:BKFmod}.}
\quash{Let $(R^\sharp, R^{\sharp +})$ be integral perfectoid\mar{\cred rephrase to start with $(R,R^+)$...} over $\BZ_p$ with tilt $(R, R^+)$ and let $(\xi)$ be the kernel
of $W(R^+)\to R^{\sharp +}$, cf. \S \ref{sss:BKFmod}.}  There is a  functor from BKF-modules over $R^{\sharp +}$
to shtukas over the $v$-sheaf $\Spd(R^{\sharp +})=\Spa(R^{\sharp +}, R^{\sharp +})^\diam\to \Spd(\BZ_p)$ which is given
as follows. Let  $(A, A^+)\in {\rm Perfd}_k$, and  let the $\Spa(A, A^+)$-point of $\Spd(R^{\sharp +})$ be given by  $\Spa(A^\sharp, A^{\sharp +})\to \Spa(R^{\sharp +}, R^{\sharp +})$ induced by $R^{\sharp +}\to A^{\sharp +}$.
This induces a map $W(R^+)\to W(A^+)$.  Now the functor is given by base change via $W(R^+)\to W(A^+)$ followed by
pullback along $\CY_{[0,\infty)}(A, A^+)\to \Spec(W(A^+))$.

\begin{proposition}\label{propBKFshtuka}
The above functor from BKF-modules over $R^{\sharp +}$
to shtukas over the $v$-sheaf $\Spd(R^{\sharp +})=\Spa(R^{\sharp +}, R^{\sharp +})^\diam\to \Spd(\BZ_p)$ 
is fully faithful. It is an equivalence of categories when $(R^\sharp, R^{\sharp +})=(K^\sharp, K^{\sharp+})$, where $K^\sharp$ is a perfectoid field with $K^{\sharp +}$ an open and bounded valuation ring.
\end{proposition}

\begin{proof}
 Let $(\sM, \phi_\sM)$ and $(\sN, \phi_\sN)$
be two BKF-modules over $R^{\sharp +}$, so $\sM$ and $\sN$ are finite projective
$W(R^+)$-modules. Let
\[
\psi: (\sM_{\rm sht}, \phi_{\sM_{\rm sht}})\to (\sN_{\rm sht}, \phi_{\sN_{\rm sht}})
\]
be a homomorphism between the corresponding shtukas over $\Spd(R^{\sharp +})\to\Spd(\BZ_p)$.
Then $\sM_{\rm sht}$, $\sN_{\rm sht}$, give global sections of the $v$-stack of vector bundles
over $\Spd(R^+)\times \Spd(\BZ_p)\simeq \Spa(W(R^+))^\diam$, cf. (\ref{productformulaSpa}). Indeed, let $T=\Spa(A, A^+)$ be affinoid perfectoid over $k$ and let $a: T\to \Spd(R^+)$ be given by a continuous $A^+\to R^+$. Then the untilt $R^{\sharp +}$ gives $\Spd(R^+)\to \Spd(\BZ_p)$  and we obtain
by composition an untilt $T^\sharp=\Spa(A^\sharp, A^{\sharp +})$ with $R^{\sharp +}\to A^{\sharp +}$
(\cite[Lemma 4.7]{Gl}); this gives 
$a': T\to \Spd(R^{\sharp +})$.
Then the restriction of the global section corresponding to $\sM_{\rm sht}$ under $T\times \Spd(\BZ_p)=\CY_{[0,\infty)}(A, A^+)^\diam\to \Spd(R^+)\times \Spd(\BZ_p)$ is given by the evaluation $(a')^*\sM_{\rm sht}$ of the shtuka at $a'$.

Consider now the pullbacks of $\sM_{\rm sht}$ and $\sN_{\rm sht}$
via $\CY(R, R^+)^\diam \to \Spa(W(R^+))^\diam$. 
To prove that $\psi$ is induced by a unique map $\sM\to \sN$ over $W(R^+)$, we first observe that by Theorem \ref{FFKedlaya}, it is enough to show the corresponding statement for the vector bundles over the sousperfectoid analytic adic space $\CY(R, R^+)$ which are obtained by pulling back $\sM$ and $\sN$ along $\CY(R, R^+)\to \Spec(W(R^+))$. Recall that the structure sheaf of a sousperfectoid analytic adic space is a sheaf for the $v$-topology (this follows by extending \cite[Thm. 17.1.3]{Schber} or \cite[Thm. 3.5.5]{KL}, see \cite{HansenKedlaya}). This implies that the global sections of $\sN\otimes\sM^{-1}$ over $\CY(R, R^+)$
agree with the global sections of $\sN_{\rm sht}\otimes\sM^{-1}_{\rm sht}$ over $\CY(R, R^+)^\diam$.
Since by the above, $\psi$ gives a global section of $\sN_{\rm sht}\otimes\sM^{-1}_{\rm sht}$ over 
$\CY(R, R^+)^\diam$ we obtain a unique corresponding global section of $\sN\otimes\sM^{-1}$ over $\CY(R, R^+)$. This implies the desired fully-faithfulness. 

The last statement giving an equivalence of categories in the case of $(K^\sharp, K^{\sharp+})$, follows from Proposition \ref{extensiontoY}, which we will show later, and the remark immediately after it.
\end{proof}

\begin{remark}\label{rem239}
We summarize as follows the functors introduced above. Let $R^{\sharp +}$ be an integral perfectoid ring, with tilt $R^+$. Let $S=\Spa(R, R^+)$ and $S^\sharp=\Spa(R^{\sharp }, R^{\sharp +})$. Then we have a commutative diagram of functors,
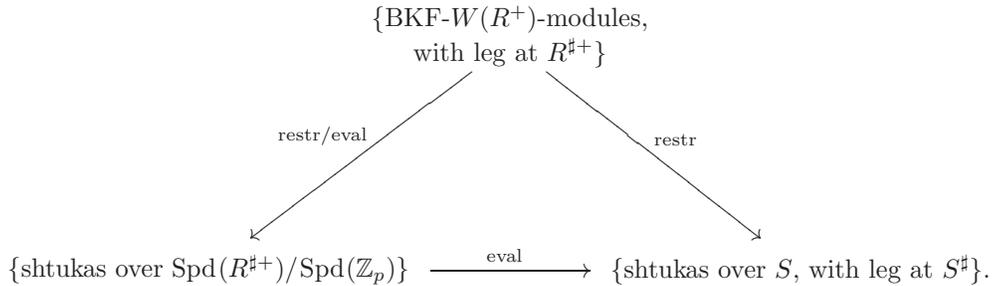
\begin{equation*}
   \xymatrix @R=30pt @C=22pt @M=8pt{ 
  & \save[]+<0cm, 1.2cm>*\txt<15pc>{\{BKF-modules over $R^{\sharp +}$\} }
\ar@<2pt>[dl]_{\text{restr/eval}}  \ar@<-2pt>[dr]^{\text{restr}} \restore   & \\
       \txt{\{shtukas over $\Spd(R^{\sharp +})/\Spd(\BZ_p)$\}}  \ar@<-3pt>[rr]^{\text{eval}}  &&  \txt{\{shtukas over $S$, with leg at $S^\sharp$\}.} 
   }
\end{equation*}
Here, the arrow ``restr'' is referred to in Proposition \ref{exttoW}, and is the restriction induced by the map $\CY_{[0, \infty)}(R, R^+)\to \Spec(W(R^+))$: this functor is faithful and,  when $R^{\sharp +}=R^{\sharp\circ}$, even fully faithful. The functor ``restr/eval" is referred to in  Proposition \ref{propBKFshtuka} and is given by the evaluation on $(A, A^+)\in {\rm Perfd}_k$ via restriction along $\CY_{[0, \infty)}(A, A^+)\to \Spec(W(R^+))$: this functor is fully faithful and an equivalence in the case of a perfectoid field. It is reasonable\footnote{\cmag This indeed follows from  \cite[Thm. 4.25]{Guth} for any integral perfectoid ring $R^{\sharp +}$ (use loc.~cit., Prop. 2.19, to relate \emph{perfect-prismatic $F$-crystals} of loc.~cit to our BKF-modules). As a consequence, the proof of Proposition \ref{Extperfd} can be simplified.} to expect that this functor  is often   an equivalence
of categories. The arrow ``eval" is the evaluation on $\Spa (R, R^+)\to \Spd(R^\sharp, R^{\sharp +})$.  In the proofs, the following various kinds of ``vector bundles" play a role: vector bundles over $\Spec(W(R^+))$, vector bundles over the adic spaces $\CY(R, R^+)$ and $\CY_{[0,\infty)}(R, R^+)$, and sections of the $v$-stack of vector bundles over the $v$-sheaf $\Spa(W(R^+))^\diam$.
\end{remark}

\smallskip

 \subsection{$\CG$-shtukas}\label{ss:CG-shtukas}
 
\subsubsection{Background}\label{sss:backg}

We give some preliminaries and fix notations and definitions.

We let $G$ be a connected reductive group over $\BQ_p$ and $\{\mu\}$ a $G(\bar\BQ_p)$-conjugacy class of  cocharacters  $\mu:{\BG_m}_{\bar\BQ_p}\to G_{\bar\BQ_p}$.
Denote by $E$ the field of definition of $\{\mu\}$, i.e. the fixed field
of all $\sigma\in {\rm Gal}(\bar\BQ_p/\BQ_p)$ for which $\sigma(\mu)$ is $G(\bar\BQ_p)$-conjugate to $\mu$. In the sequel, we will also write $\mu$ for the conjugacy class. 
Let $k$ be an algebraic closure of the residue field $\kappa=\kappa_E$ of $E$.

Let
\[
P_\mu=\{g\in G\ |\ \lim_{t\to \infty}{\rm ad}(\mu(t))g\ \hbox{\rm exists}\}
\]
be the parabolic associated to $\mu$. We consider also $P_{\mu^{-1}}$ which is the opposite
parabolic. We set 
\begin{equation}
\CF_{G,\mu}=G/P_\mu ,
\end{equation}
which is a smooth projective variety defined over the reflex field $E$.

Let $[b]$ be the $\sigma$-conjugacy class of an element $b \in G(\breve\BQ_p)$.  We will always assume that $[b]$ is neutral acceptable for $\mu^{-1}$, i.e., $[b]\in B(G, \mu^{-1})$.  In other words, $\nu_b\leq \overline{(\mu^{-1})}_{\rm dom}$ and $\kappa(b)=-\mu^\sharp$, with the notation of \cite{RV}.
If $\mu$ is minuscule, we will call
$(G, b, \mu)$ as above a \emph{local Shimura datum}, cf. \cite[Def. 24.1.1]{Schber}, comp. \cite{RV}.

\subsubsection{Local models}
Let $G$ be a reductive group over $\BQ_p$, with smooth parahoric model $\CG$ over $\BZ_p$. Consider the ``Beilinson-Drinfeld style affine Grassmannian" $v$-sheaves
 \begin{equation}
 {\rm Gr}_{G, {\rm Spd}(E)},\quad  {\rm Gr}_{\CG, {\rm Spd}(O_E)}
 \end{equation}
 over $\Spd(E)$, resp. over $\Spd(O_E)$, cf. \cite[Ch. 20]{Schber}. 
As in \cite[Prop. 20.2.3]{Schber} and if $\mu$ is minuscule, there is a closed immersion of the diamond associated to the flag variety of parabolic subgroups of type $\mu$, 
 \begin{equation*}
 \CF_{G,\mu}^\diamondsuit={\rm Gr}_{G, {\rm Spd}(E),\mu}\subset {\rm Gr}_{G, {\rm Spd}(E)} .
 \end{equation*}
Assume $\mu$ is minuscule. The following definition occurs implicitly in   \cite[Ch. 21]{Schber}:  we define  the \emph{$v$-sheaf local model} as the closure of  $\CF_{G,\mu}^\diamondsuit$ in ${\rm Gr}_{\CG, {\rm Spd}(O_E)}$, in the sense of $v$-sheaves (i.e., the minimal closed superset whose pullback to any perfectoid space is stable under generalizations, see \cite[2.1]{AnRicLou}). We  use the following notation, 
 \begin{equation}\label{LMBD}
 \BM_{\CG, \mu}^v= \BM_\mu^v={\rm Gr}_{\CG, {\rm Spd}(O_E),\mu} .
 \end{equation}
  The following theorem was conjectured by Scholze-Weinstein, cf. \cite {Schber}, Conj. 21.4.1].
  \begin{theorem}[Gleason-Louren\c co \cite{GL}]\label{LMconj}
 There exists a proper flat scheme ${\mathbb M}^{\rm loc}_{\CG, \mu}$ over $\Spec(O_E)$ with $\CG$-action and reduced special fiber such that its associated  $v$-sheaf over ${\rm Spd}(O_E)$ is equal to  $ \BM_{\CG, \mu}^v$. (Note that by \cite[Rem. 21.4.2]{Schber},  ${\mathbb M}^{\rm loc}_{\CG, \mu}$ is necessarily a normal scheme, and that, by \cite[Prop. 18.4.1]{Schber}, such a  scheme is unique if it exists.) 
  \end{theorem}
  
    We call ${\mathbb M}^{\rm loc}_{\CG, \mu}$ the \emph{(scheme) local model}.
    
    \begin{remark}\label{schlocmod}
 Before \cite{GL}, 
the most complete result on this conjecture was due to Ansch\"utz, Gleason, Louren\c co and Richarz \cite{AnRicLou}, comp. also \cite{LourencoThesis}.
 They proved that the conjecture  holds for all $(G,\mu)$ when $p\geq 5$ and in many cases when $p=2,3$.

 There is an a priori different general approach to (scheme) local models, if $G$ splits over a tamely ramified extension of $\BQ_p$, or more generally, if $G$ is {\cmag{\sl essentially tamely ramified}}
  as defined in Remark \ref{Gacc}, cf. \cite{PZ}, \cite{Levin}. Conjecturally (see \cite[Conj. 2.16]{HPR}), the local models of \cite{PZ} 
satisfy the conditions of Theorem \ref{LMconj} (provided they are slightly adjusted when $p$ divides $|\pi_1(G_{\rm der})|$ by using a $z$-extension as in \cite[\S 2.6]{HPR}.) Hence, this approach should also give ${\mathbb M}^{\rm loc}_{\CG, \mu}$ as above. This conjectural agreement is proven in \cite[Thm. 2.15]{HPR},
in almost all cases when $(G, \mu)$ is of abelian type, which is our main case of interest, see also  \cite{LourencoThesis}.  By definition \cite[Def. 9.6]{HLR}, ``of abelian type''  means that there is a central lift $(G_1, \mu_1)$ of $(G_\ad, \mu_\ad)$ which is of local Hodge  type, i.e., admits a closed embedding $\rho: G_1\hookrightarrow \GL_n$ with $\rho\circ \mu_1$ minuscule, cf. \cite[Rem. 5.5 (i)]{RV}. 
\end{remark}

\subsubsection{Witt vector affine  Grassmannian}\label{Conjadm}
Recall the 
Witt vector affine Grassmannian ${\rm Gr}^W_{\CG}$ of \cite{ZhuAfGr}, \cite{BS},
which is an ind-perfectly proper scheme over $k$. Suppose that $S=\Spa(R, R^+)\in {\rm Perfd}_k$ and take $S^\sharp=S$, i.e. the untilt is in characteristic $p$.
Then, since $B_{\rm dR}^+(R^{\sharp})=W(R)$ and $\xi=p$, we have a natural bijection
\[
{\rm Gr}_{\CG, {\rm Spd}(O_E)}(S)\xrightarrow{\ \sim\ } ({\rm Gr}^W_{\CG})^\sdiam(S)
\] 
functorial in $S$ (see \cite[\S 20.3]{Schber}, especially the passage after Prop. 20.3.2). Here, as before, 
$({\rm Gr}^W_{\CG})^\sdiam$ is the $v$-sheaf associated to ${\rm Gr}^W_{\CG}$. 
For a perfect (discrete) $k$-algebra $R$, we set $R^\sdiam:=\Spa(R, R)^\diam$.
We obtain
\[
{\rm Gr}_{\CG, {\rm Spd}(O_E)}(R^\sdiam)\xrightarrow{\ \sim\ }  ({\rm Gr}^W_{\CG})^\sdiam(R^\sdiam)={\rm Gr}^W_{\CG}(R).
\]
Here, the equality $({\rm Gr}^W_{\CG})^\sdiam(R^\sdiam)={\rm Gr}^W_{\CG}(R)$ is obtained by the full-faithfulness of the functor $Z\mapsto Z^\sdiam$
from perfect schemes to $v$-sheaves. 

Composing with $\BM_{\CG, \mu}^v(R^\sdiam)\hookrightarrow {\rm Gr}_{\CG, {\rm Spd}(O_E)}(R^\sdiam)$
gives
\begin{equation}\label{LMWAff}
\iota\colon\BM_{\CG, \mu}^v(R^\sdiam)\to {\rm Gr}^W_{\CG}(R).
\end{equation}
If $K$ is a (discrete) algebraically closed field of characteristic $p$, this gives
\[
 \BM_{\CG, \mu}^v(K^\sdiam)\to {\rm Gr}^W_{\CG}(K)=G(W(K)[1/p])/\CG(W(K)).
\]
If  ${\mathbb M}^{\rm loc}_{\CG, \mu}$ above exists, then $\BM_{\CG, \mu}^v(K^\sdiam)={\mathbb M}^{\rm loc}_{\CG, \mu}(K)$.

By \cite[Thm. 6.16]{AnRicLou}, the map \eqref{LMWAff} 
identifies $ \BM_{\CG, \mu}^v(K^\sdiam)$ with the union inside $ G(W(K)[1/p])/\CG(W(K))$,
\begin{equation}\label{ptsofvLMo}
\BM_{\CG, \mu}^v(K^\sdiam)=\bigcup\nolimits_{w\in {\rm Adm}(\mu)_\CG}{\rm Gr}^W_{\CG, w}(K) . 
\end{equation}
Here, ${\rm Adm}(\mu)_\CG$ denotes the \emph{admissible set} in the double quotient $W_\CG\backslash\widetilde W/W_\CG$ of the Iwahori-Weyl group $\widetilde W$ of $G$, and ${\rm Gr}^W_{\CG, w}$ is the Schubert cell in ${\rm Gr}^W_{\CG}$ corresponding to $w$.  
(For the abelian type case, see also \cite[Chapt. 4, Cor. 4.24]{LourencoThesis},  
which also gives this, provided that 
the conditions, for $p=2$, $3$, in Remark \ref{schlocmod} are satisfied.)

\subsubsection{$\CG$-shtukas}
We recall Scholze's notion of a $\CG$-shtuka over a perfectoid space, cf. \cite{Schber}. 
  \begin{definition}\label{defGsht}
   Let $S\in{\rm Perfd}_k$, i.e., $S$ is a perfectoid space over $k$, and let 
 $S^\sharp$ be an untilt of $S$ over $  O_{\breve E} $. A \emph{$\CG$-shtuka  over $S$} with one leg at $S^\sharp$
is a pair 
\begin{equation}
(\sP , \phi_{\sP }) ,
\end{equation}
 where
\begin{itemize}
\item[1)] $\sP $ is a $\CG$-torsor  over the analytic adic space $ S\bdtimes  \BZ_p $,
\item[2)] $\phi_{\sP }$   is a $\CG$-torsor isomorphism
\begin{equation}
\phi_{\sP }:  {\rm Frob}^*_S(\sP )_{| S\bdtimes  \BZ_p  \setminus S^\sharp}  \xrightarrow{\sim} \sP _{| S\bdtimes  \BZ_p  \setminus S^\sharp} 
\end{equation}
which is meromorphic along the closed Cartier divisor $S^\sharp\subset  S\bdtimes  \BZ_p  $.
\end{itemize}
We say that the $\CG$-shtuka $\sP$ over $S$ with one leg at $S^\sharp$ is \emph{bounded by $\mu$} if the relative position of $\phi_{\sP }({\rm Frob}^*_S(\sP) )$ and $\sP $ at $S^\sharp$  (in this order!)  is bounded by $\BM_{\CG, \mu}^v$. 
 \end{definition}
   
Let us explain the meaning of the term ``bounded by $\BM_\mu^v$", comp. \cite[Def. 20.3.5]{Schber}.  
For $S=\Spa(C, O_C)$, with $C$ an non-archimedean complete algebraically closed field, the complete local ring $\widehat\CO_{\CY_{[0,\infty)}(S), x}$ at the point $x$ that corresponds to the untilt $S^\sharp=\Spa(C^\sharp, O_{C^{\sharp}})$ is  a discrete valuation ring with residue field $C^\sharp$. In fact, 
$\widehat\CO_{\CY_{[0,\infty)}(S), x}$ is isomorphic to $B_{\rm dR}^+(C^\sharp)$ (recall that this is  $W(C^\sharp)$ if $C^\sharp\simeq C$ has characteristic $p$.) This ring   is strictly henselian so every $\CG$-torsor over its spectrum is trivial.  Let $S=\Spa(R, R^+)$ be affinoid perfectoid over $k$, and let $(\sP, \phi_\sP)$ be a $\CG$-shtuka over $S$ with one leg at $S^\sharp$. Let  
$x: T=\Spa(C, O_C)\to S$ be a point over $\Spd(O_E)$ with $C$ algebraically closed, as above. By the above, we can choose a trivialization
\[
\beta: {\rm Frob}^*_T(\sP_T)_{|\Spec(\widehat\CO_{\CY_{[0,\infty)}(T), x})}\xrightarrow{\sim} \CG\times \Spec(\widehat\CO_{\CY_{[0,\infty)}(T), x})
\]
of the pull-back of the $\CG$-torsor $ {\rm Frob}^*_T(\sP_T)$ to the completion of $T\bdtimes  \BZ_p $ along the corresponding untilt $T^\sharp$. Fixing such a trivialization, we can consider the pair 
\[
(   \sP_T ,\   \beta\cdot \phi_{\sP_T}^{-1} )
\]
where
\[
\phi_{\sP_T}:  {\rm Frob}^*_T (\sP_T)_{| T\bdtimes  \BZ_p  \setminus T^\sharp} \xrightarrow{\sim}  (\sP_T)_{| T\bdtimes  \BZ_p  \setminus T^\sharp}. 
\] 
This pair gives a $\CG$-torsor and a trivialization
 of its restriction to $\Spec ({\rm Frac}(\widehat\CO_{\CY_{[0,\infty)}(T), x}))$.   By \cite[Def. 20.3.1, Prop. 20.3.2]{Schber}, such a pair  (together with the untilt $T^\sharp$), corresponds to an $T$-valued point  of ${\rm Gr}_{\CG, {\rm Spd}(O_E)}$. The boundedness condition on the relative position is that, for all $T\to S$ as above, this point factors through $\BM^v_\mu={\rm Gr}_{\CG, {\rm Spd}(O_E),\mu}\incl {\rm Gr}_{\CG, {\rm Spd}(O_E)}$. To be clear, the point $ {\rm Frob}^*_T(\sP_T)$ is the base point of the affine Grassmannian, and $(\phi_\sP)^{-1}(\sP_T)$ is considered as a variable torsor to be compared with $ {\rm Frob}^*_T(\sP_T)$.

 Let us compare the above definition with Definition \ref{vectsht}.
  
 \begin{lemma}
 Let $\CG=\GL_h$ and $\mu_d=(1^{(d)}, 0^{(h-d)})$. There is a functorial equivalence of categories between  $\CG$-shtuka over $S/\Spd(\BZ_p)$ bounded by $\mu_d$ and shtuka over $S/\Spd(\BZ_p)$ of height $h$ and dimension $d$. 
  \end{lemma} 
    \begin{proof} We use the equivalence of categories between the category of $\GL_h$-bundles and the category of vector bundles of rank $h$. The statement then follows from the description of ${\rm Gr}_{\GL_h, \Spd(\BZ_p), \mu_d}$
    in \cite[10, 20]{Schber}, see in particular \cite[Prop. 19.4.2]{Schber} and its proof. 
    \end{proof} 
    
    By Proposition \ref{ppProp}, for $S=\Spa(R, R^+)$  affinoid perfectoid, restriction gives an exact equivalence between vector bundles over $\Spa(R, R^+)\bdtimes \BZ_p$ 
and vector bundles over $\Spa(R, R^\circ)\bdtimes \BZ_p$. Using the Tannakian interpretation of $\CG$-torsors 
as in \cite[Thm. 19.5.1]{Schber}, it also follows that restriction gives an equivalence between $\CG$-shtukas
over $\Spa(R, R^+)$ and $\Spa(R, R^\circ)$.

 The notion of a $\CG$-shtuka bounded by $\mu$ generalizes to any perfectoid space $S$ over $k$ 
with a map $S\to \Spd(\CO_E)$. Also, by \cite[Prop. 19.5.3]{Schber}, sending $S/\Spd(\CO_E)$ to the groupoid of $\CG$-shtukas over $S/\Spd(\CO_E)$ (bounded by $\mu$)
gives a $v$-stack.
    
    \subsubsection{$\CG$-BKF-modules} We can define the related notion of a $\CG$-Breuil-Kisin-Fargues (BKF-) module.

\begin{definition} 
Let $R$ be an integral perfectoid ring, cf. \S \ref{sss:BKFmod}.  A \emph{$\CG$-Breuil-Kisin-Fargues} (BKF-)\emph{module} over $R$  is a 
$\CG$-torsor $\sP$ over $\Spec(W(R^\flat))$ together with an isomorphism
 \[
 \phi_{\sP}: \phi^*(\sP)[1/\xi]\xrightarrow{\ \sim\ } \sP[1/\xi].
 \]
 Here, again, $\xi$ is a primitive generator of the kernel of the map $W(R^\flat)\to R$. 
 
If $S=\Spa(R, R^+)\in{\rm Perfd}_k$, i.e., $S$ is affinoid perfectoid over $k$, and 
 $S^\sharp=\Spa(R^\sharp, R^{\sharp +})$ is an untilt of $S$, we also speak of a $\CG$-\emph{BKF-module over $S$ with leg along $S^\sharp$} instead of a $\CG$-BKF-module over $R^{\sharp +}$, comp. Definition \ref{def:BKF}. 
\end{definition}

As in Definition \ref{defassshtukaBKF},  a $\CG$-BKF module over $S\in{\rm Perfd}_k$ with leg along the untilt $S^{\sharp }$ of $S$
  defines a $\CG$-shtuka $\sP$ over $S$ with one leg at $S^\sharp$. We then say that the shtuka $\sP$ extends. Note that since the restriction functor from $ \CY_{[0,\infty]}(S)$ to $ \CY_{[0,\infty)}(S)$ is not fully faithful in general (cf. the comment after Proposition \ref{FFres}), such an extension may not be unique.  
 If $\sP$ is bounded by $\mu$, in the sense above, then we will call the extension  a \emph{$(\CG,\mu)$-Breuil-Kisin-Fargues module}.

 \subsubsection{Specializations of $\CG$-Shtukas}\label{overpoint}  
 Let $C$ be a non-archimedean complete algebraically closed field over $O_E$ with a valuation of rank $1$ and let $O_C$ be its valuation ring. Set $S=\Spa(C^\flat, O^\flat_C)$ and let  $\sP$ be a $\CG$-shtuka over $S$ with leg at $\Spa (C, O_C)$. 
 We can obtain a corresponding Frobenius $G$-isocrystal over the residue field $\kappa$ of $O_C$ as follows.  For $r\gg 0$, the $\phi$-$G$-torsor given by $\sP_{|\CY_{[r,\infty)}(S)}$
 descends to a $G$-torsor over the Fargues-Fontaine curve $\CY_{(0,\infty)}(S)/\phi^\BZ$, cf. Proposition \ref{exttoW}. By Fargues' theorem \cite[Thm. 14.1.1]{Schber}, this corresponds to  a Frobenius $G$-isocrystal over $\kappa$.  
  
  \begin{proposition}\label{AnExtension} Let $(\sP, \phi_\sP)$ be a $\CG$-shtuka over $\Spa(C^\flat, O_C^\flat)/\Spd(O_E)$ with one leg at $\Spa(C, O_C)$.
The $\CG$-shtuka $\sP$
extends, in the sense of Definition \ref{defassshtukaBKF}, to a unique $\CG$-BKF-module $\sP^{ \natural}$ over $O_C$, i.e., a module over $W(O^\flat_C)$.
 Base changing  $\sP^{ \natural}$ via $W(O^\flat_C)\to W(O^\flat_C/\mathfrak m_{C^\flat})=W(\kappa)$ gives a $\CG$-torsor over $\Spec(W(\kappa))$,
 \[
 \sP_0=\sP^{ \natural}\otimes_{W(O^\flat_C)}W(\kappa) ,
 \] 
  with Frobenius
  \[
 \phi_{\sP_0}:  {\rm Frob}^*(\sP_0)[1/p]\xrightarrow{\ \sim \ }\sP_0[1/p].
  \]
 \end{proposition}
 
 \begin{proof}
 Using Fargues' theorem as above, we also see that, for $r\gg0$, the $\phi$-$\CG$-torsor $\sP_{|\CY_{[r,\infty)}(S)}$ extends to
 a $\phi$-$\CG$-torsor over $\CY_{[r,\infty]}(S)$. By the Extension Conjecture \ref{paraextconj} proved by Ansch\"utz \cite{An},
 $\sP$ extends to a $\CG$-torsor $\sP^{ \natural}$ over $\Spec(A_\inf)=\Spec(W(O^\flat_C))$ (this torsor is actually trivial). The Frobenius $\phi_\sP$ extends to a meromorphic map
 \[
 \phi_{\sP^{ \natural}}: \phi^*(\sP^{ \natural})[1/\xi]\xrightarrow{\ \sim \ }\sP^{ \natural}[1/\xi]
 \] 
 which then defines the $\CG$-BKF-module. The second part of the statement follows since $\xi\equiv p\, \mod   
 W(\mathfrak m_{C^\flat})$.
  \end{proof}
  
  \begin{remark}
 This proposition is the basis for the construction of the specialization map for moduli of local shtuka as in \cite{Gl}, \cite{Gl21},
 see Theorem \ref{Gleason} below. 
 We will explain this in detail later, see Remark \ref{ideaSpecialization}. In fact, we can see that if $\sP$ is bounded by $\mu$, then
 the map $\phi_{\sP_0}: {\rm Frob}^*(\sP_0)[1/p]\xrightarrow{\ \sim \ }\sP_0[1/p]$ also has pole bounded by $\mu$. \end{remark}

  \subsubsection{Families of $\CG$-Shtukas} We will also want to consider ``families" of $\CG$-shtukas, cf. \S \ref{FamSht}. {\cmag The following definition  is modelled on the corresponding definition of a family of ``vector space'' shtukas  of \S \ref{deffamofsht}. 
  
  \begin{definition}\label{deffamofshtG}
  Let $\CF$ be a $v$-sheaf over ${\rm Spd}(O_E)$. 
  A $\CG$-shtuka $(\sP, \phi_{\sP})$ over $\CF$, resp. a $\CG$-shtuka $(\sP, \phi_{\sP})$ over $\CF$
 with leg bounded by $\mu$,   is a  section of the $v$-stack given by the groupoid of  $\CG$-shtukas  
 over $\CF$, resp.  is a  section of the $v$-stack given by the groupoid of  $\CG$-shtukas  
 over $\CF$ with leg bounded by $\mu$. In particular, it is a functorial rule   which to any point  $x$ of $\CF$ with values in $S\in {\rm Perfd}_k$  associates a $\CG$-shtuka  $(\sP_S , \phi_{\sP_S })$ over $S$  
with one leg at the untilt $S^\sharp$ given by $  S\xrightarrow{x} \CF\to \Spd(O_E)$, in the resp. case bounded by $\BM_\mu^v$.
\end{definition} 
  
 We can now define $\CG$-shtuka over adic spaces, schemes, or formal schemes over $O_E$, as $\CG$-shtukas over the corresponding $v$-sheaves, by generalizing the ``vector space'' shtuka definitions of \S \ref{FamSht}.   For example, suppose that $\sX$ is a scheme over $\Spec(O_E)$. Following \S \ref{def:shtsch}, we can define a $\CG$-shtuka over $\sX$ to be a $\CG$-shtuka over the $v$-sheaf $\sX^{\diam/}$ over $\Spd(O_E)$. Note that the corresponding notion of a vector space shtuka $(\sV, \phi_\sV)$ of rank $n$  of \S \ref{FamSht} corresponds to taking $\CG=\GL_n$.    
   }
 
 \quash{
 \begin{definition}\label{deffamofshtG}
 A $\CG$-shtuka $(\sP, \phi_{\sP})$ over $\CX^\diam \to {\rm Spd}(O_E)$, resp. a $\CG$-shtuka $(\sP, \phi_{\sP})$ over $\CX^\diam \to {\rm Spd}(O_E)$
 with leg bounded by $\mu$,   is a  section of the $v$-stack given by the groupoid of  $\CG$-shtukas  
 over $\CX^\diam / {\rm Spd}(O_E)$, resp.  is a  section of the $v$-stack given by the groupoid of  $\CG$-shtukas  
 over $\CX^\diam / {\rm Spd}(O_E)$ with leg bounded by $\mu$. In particular, it gives a map which to any point $(S^\sharp, x)$ of $\CX^\diam $ with values in $S\in {\rm Perfd}_k$  associates a $\CG$-shtuka  $(\sP_S , \phi_{\sP_S })$ over $S$  
with one leg at $S^\sharp$, in the resp. case bounded by $\BM_\mu^v$, and this association comes with $v$-descent data, cf. Definition \ref{defGsht}.
\end{definition} 

 In the above definition, we can think of the $\CG$-shtuka $\sP$ as having one leg  at the ``universal'' untilt given by $\CX^\diam \to {\rm Spd}(O_E)\to {\rm Spd}(\BZ_p)$.  The corresponding notion of a (vector space) shtuka $(\sV, \phi_\sV)$ of rank $n$ over $\CX^\diam/\Spd(O_E)$ of \S \ref{FamSht}  corresponds to taking $\CG=\GL_n$. 
 }

 Applying the Tannaka formalism to Theorem \ref{FFisocrystal}, we obtain the following example.
  \begin{example}\label{G-AlgClosedPoint}
 Let $\sX=\Spec(A)$, with $A$ a perfect $k$-algebra. Then there is an equivalence of categories between the category of $\CG$-shtukas over $\sX^\sdiam$ and the category of $\CG$-torsors $\sP$ on $\Spec (W(A))$ equipped with an isomorphism 
 $$
 \phi_\sP\colon \phi^*(\sP)[1/p]\isoarrow \sP[1/p] .
 $$ 
  \end{example}

\subsection{Shtukas and local systems}

In this subsection, we consider shtukas in characteristic zero. We explain a relation of shtukas with local systems which support a suitable sheaf theoretic Hodge-Tate period map.
 
 \subsubsection{$\CG$-shtukas and de Rham-Tate period maps}  Let $S=\Spa(R, R^+)\in{\rm Perfd}_k$, i.e. $S$ is affinoid perfectoid over $k$, and let 
 $S^\sharp=\Spa(R^\sharp, R^{\sharp +})$ be an untilt of $S$ over $  O_{\breve E} $. 
 Suppose now that the leg $S^\sharp$ is over $\Spa(E,O_E)$, i.e. that $S\to \Spd(O_E)$ factors through $S\to\Spd(E)$. Then, there exists $r>0$ such that $\CY_{[0,r]}(S)\subset  S\bdtimes  \BZ_p  \setminus S^\sharp$. 
As before, in what follows, we will often denote ${\rm Frob}_S$ by $\phi$, for simplicity.

Let $(\sP, \phi_\sP)$ be a $\CG$-shtuka over $S$, with leg at $S^\sharp$. The restriction of $(\sP, \phi_\sP)$ to $\CY_{[0,r]}(S)$ defines a $\varphi^{-1}$-equivariant $\CG$-torsor on  $\CY_{[0,r]}(S)$.
By \cite[Prop. 22.6.1]{Schber}, this defines a pro-\'etale $\underline{\CG(\BZ_p)}$-torsor $\BP$ over $S$ with an isomorphism of $\varphi^{-1}$-equivariant $\CG$-torsors over ${\CY_{[0,r]}}(S)$,
\[
\sP_{|\CY_{[0,r]}(S)}\simeq {\BP}\times^{\underline{\CG(\BZ_p)}}(\CG\times_{\Spa(\BZ_p)}{\CY_{[0,r]}(S)}) .
\]
 We say that the $\underline{\CG(\BZ_p)}$-torsor $\BP$
is \emph{ associated} to $\sP$. 

The $\phi^{-1}$-$\CG$-torsor $\sP_{|\CY_{[0,r]}(S)}$ descends to a $\CG$-torsor $\sP_0$ over the relative Fargues-Fontaine curve $\CX_{{\rm FF}, S}=\CY_{(0,\infty)}(S)/\varphi^\BZ$. This gives, after pullback and extension, a $\phi$-$\CG$-torsor over $\CY_{[0,\infty)}(S)=S\bdtimes \BZ_p$,
\begin{equation}
\sP_0={\BP}\times^{\underline{\CG(\BZ_p)}}(\CG\times_{\Spa(\BZ_p)}{\CY_{[0,\infty)}(S)}) .
\end{equation}
 In other words, $\sP_0$ is a $\CG$-shtuka over $S$ with no legs.

By  construction, the  $\CG$-torsors $\sP_0$ and $\sP$ over $S\bdtimes\BZ_p$ agree away from $\bigcup_{n\geq 1}\phi^n_S(S^\sharp)$.
In fact, by the argument in the proof of \cite[Prop 12.4.1]{Schber}, there is a unique $\phi^{-1}$-isomorphism
\begin{equation}
i_\sP: \sP|_{\CY_{[0,\infty)}(S)\setminus\bigcup_{n\geq 1}\phi^n(S^\sharp)}\simeq \sP_0|_{\CY_{[0,\infty)}(S)\setminus\bigcup_{n\geq 1}\phi^n(S^\sharp)} .
\end{equation}
 In particular, $i_\sP$ gives an isomorphism 
of $\sP$ and $\sP_0$ over the completion of $\CY_{[0,\infty)}(S)$ along the leg $S^\sharp$.

Suppose that the pro-\'etale $\underline{\CG(\BZ_p)}$-torsor ${\BP}$ on $S$ is trivial and choose a
trivialization 
$
a: {\BP}\simeq \underline{\CG(\BZ_p)}\times S.
$
This gives a $\phi$-trivialization of the $\CG$-shtuka with no legs,
\[
\alpha_0: \CG\times_{\Spa(\BZ_p)}{(S\bdtimes \BZ_p)} \xrightarrow{\sim} \sP_0.
\]
 By the above, $\alpha_0$ composed with $i_\sP^{-1}$   induces a trivialization $\alpha$ of the pullback of $\sP$ to the completion of $S\bdtimes \BZ_p$ along $S^\sharp$. 
 Now consider the pair 
 \[
 ({\rm Frob}_S^*(\sP),\ \alpha^{-1}\circ \phi_\sP ) .
 \]
 This gives an $S$-point of ${\rm Gr}_{G, \Spd(E)}$ over $\Spd(E)$, cf. \cite[Def. 20.2.1, Prop. 20.2.2]{Schber}, which we denote by 
 ${\rm DRT}(\sP)(a)$. This construction produces 
 a $\underline{\CG(\BZ_p)}$-equivariant 
map of $v$-sheaves, the de Rham-Tate  map, 
\begin{equation}\label{DefDRT}
{\rm DRT}(\sP):  {\BP}\to {\rm Gr}_{G, \Spd(E)} .
\end{equation}

\begin{proposition}\label{pairs0}
Let $S/ \Spd(E)$ be perfectoid. The functor 
 \[
 (\sP, \phi_\sP)\mapsto (\BP, {\rm DRT}(\sP))
 \]
 sending the $\CG$-shtuka $(\sP, \phi_\sP)$ over $S/ \Spd(E)$  with one leg
to the pair consisting 
of a pro-\'etale $\underline{\CG(\BZ_p)}$-torsor ${\BP}$ on $S$ and a $\underline{\CG(\BZ_p)}$-equivariant 
map ${\rm DRT}(\sP): {\BP}\to  {\rm Gr}_{G, \Spd(E)}$ over ${\Spd(E)}$ gives an equivalence
between the category of $\CG$-shtukas over $S/\Spd(E)$ with one leg 
and the category of such pairs $(\BP, {\rm D})$.
\end{proposition}
\begin{proof}
The case of $S=\Spa(C^\flat)$, where $C^\flat\in {\rm Perfd}_k$ is an algebraically closed field, and where $\CG=\GL_n$, is the content of \cite[Prop. 12.4.6]{Schber}. In this case, a pair $(\BP, {\rm D})$ corresponds to a finite free $\BZ_p$-module $T$ equipped with a $B^+_{\rm dR}$-lattice $\Xi\subset T\otimes_{\BZ_p}B_{\rm dR}$. In the general case, by pro-\'etale descent, we may assume that $\BP$ is trivial.  Given the description of ${\rm Gr}_{G, \Spd(E)}$ provided by 
\cite[Prop. 20.3.2]{Schber} (which uses Beauville-Laszlo glueing), we then see that 
the proof of \cite[Prop. 12.4.6]{Schber}
 goes over without change, see also 
 \cite[3.5]{CaraSch}. (But beware that, compared to \cite{Schber}, we have the leg on a different Frobenius twist.) \end{proof}

 Suppose now that the $\CG$-shtuka $(\sP, \phi_\sP)$ has leg bounded by $\mu$, which  we assume minuscule.
Then we can see that ${\rm DRT}(\sP)(a)$ lies in\footnote{Note the inverse in $\mu^{-1}$ here.  This is due to the fact that $\sP$ is the base point of the affine Grassmannian   and  $ \phi_\sP({\rm Frob}^*(\sP))$  is considered variable.} ${\rm Gr}_{G, \Spd(E), \mu^{-1}}$. 
The Bia{\l}ynicki-Birula construction gives an isomorphism
\[
{\rm Gr}_{G, \Spd(E), \mu^{-1}}\xrightarrow{\sim} \CF^\diam_{G, \mu^{-1}}.
\]
Composing ${\rm DRT}(\sP)$ with this isomorphism, we obtain the \emph{Hodge-Tate map}, 
\begin{equation}\label{DefHT}
{\rm HT}(\sP): {\BP}\to \CF^\diam_{G, \mu^{-1}} .
\end{equation}

\begin{proposition}\label{pairs1}
Let $S/ \Spd(E)$ be perfectoid. Fix $(\CG, \mu)$, where $\mu$ is minuscule. 
The functor 
 \[
( \sP, \phi_\sP)\mapsto (\BP, {\rm HT}(\sP))
 \]
 sending the $\CG$-shtuka $\sP$ over $S/ \Spd(E)$  with one leg bounded by $\mu$
to the pair consisting 
of a pro-\'etale $\underline{\CG(\BZ_p)}$-torsor ${\BP}$ on $S$ and a $\underline{\CG(\BZ_p)}$-equivariant 
map ${\rm HT}(\sP): {\BP}\to \CF^\diam_{G, \mu^{-1}}$ over ${\Spd(E)}$ gives an equivalence
between the category of $\CG$-shtukas over $S/\Spd(E)$ with one leg bounded by $\mu$ 
and the category of such pairs $(\BP, {\rm H})$.\qed
\end{proposition}

\subsubsection{Variant for adic spaces}
Let $Y$ be a locally Noetherian adic space over $\Spa(E,O_E)$. Later, we will apply the set-up to $Y=X^{\rm ad}$, the adic space associated to a scheme $X$ of finite type over $E$.

  Let $(\sP, \phi_\sP)$ be a $\CG$-shtuka over $Y^\diam/\Spd(E)$.
 For $S\in {\rm Perfd}_{k}$, with $S$-valued point $(S^\sharp, x: S^\sharp\to Y)$ of $Y^\diam$, the $v$-sheaf $\sP$ gives $\sP_S$, a $\CG$-shtuka
 over $S$ with leg at $S^\sharp$. Since the construction of the previous paragraph,
 \[
 \sP_S\mapsto (\BP(S), {\rm DRT}(\sP)(S))
 \]
  is functorial in $S$, it associates to $\sP$ a pair $(\BP, {\rm DRT}(\sP))$, consisting  of a
 $\underline{\CG(\BZ_p)}$-torsor $\BP$ over $Y^\diam$ and a $\underline{\CG(\BZ_p)}$-equivariant map of $v$-sheaves
 ${\rm DRT}: \BP\to {\rm Gr}_{G, \Spd(E)}$. If $\sP$ has leg bounded by the minuscule $\mu$ then, by composing  
 with the Bia{\l}ynicki-Birula morphism, we obtain
 \begin{equation}
 {\rm HT}: \BP\to \CF^\diam_{G, \mu^{-1}}.
 \end{equation}
 The previous proposition immediately gives:

 \begin{proposition}\label{pairs2} 
 Let $Y$ be a locally Noetherian adic space over $\Spa(E,O_E)$. Fix $(\CG, \mu)$, where $\mu$ is minuscule. 
 The functor $\sP\mapsto (\BP, {\rm HT}(\sP))$ gives an equivalence between the categories of:
 
 1) $\CG$-shtukas $(\sP, \phi_\sP)$ over $Y^\diam\to {\rm Spd}(E)$ with one leg  bounded by $\mu$, and 
 
 2) pairs $(\BP, {\rm H})$ consisting 
 of a pro-\'etale $\underline{\CG(\BZ_p)}$-torsor $\BP$ over $Y^\diam$ and a $\underline{\CG(\BZ_p)}$-equivariant map of $v$-sheaves
 ${\rm H}: \BP\to \CF^\diam_{G, \mu^{-1}}$ over $\Spd E$.\hfill
 $\square$
  \end{proposition}

  \subsection{Shtukas and de Rham local systems}\label{ss:dR}
  
  Here, as an application of the last subsection, we show  how to construct a shtuka which is ``associated" to a de Rham $p$-adic local system over a smooth scheme (or smooth rigid analytic variety) which is defined
  over a finite extension $E$ of $\BQ_p$ or of $\breve\BQ_p$. More generally, the construction works when
  $E$ is replaced by any discretely valued complete non-archimedean field extension of $\BQ_p$
  with perfect residue field. 
      
  \subsubsection{de Rham local systems}
  
  Let $X$ be a smooth scheme (of finite
  type) or a smooth rigid-analytic variety over $E$. Let $Y=X^{\rm ad}$ be the corresponding analytic adic space over $\Spa(E, O_E)$.

Recall  from 
 \cite[\S 6]{SchpHodge} the definitions of the
 sheaves of rings $\CO_Y$, $\widehat\CO_Y$, $\widehat\CO^+_Y$, $\widehat\CO_{Y^\flat}$, $\widehat\CO^+_{Y^\flat}$, $\BA_{\inf, Y}=W(\widehat\CO^+_{Y^\flat})$, $\BB_{\rm dR}^+$, $\BB_{\rm dR}$, $\CO\BB_{\rm dR, Y}^+$,  $\CO\BB_{\rm dR, Y}$, on the pro-\'etale site $Y_{\rm proet}$.  
 Consider the morphisms of sites $\nu: Y_{\proet}\to Y_{\et}$ and $\lambda: Y_{\et}\to Y_{\rm an}$, where $Y_{\rm an}$ is the site
 of open subsets of $Y$.
  There is a homomorphism of sheaves of rings
 \begin{equation}
 \theta: \BA_{\inf, Y}\to \widehat\CO^+_Y
 \end{equation}
  whose kernel is locally principal on $Y_{\rm proet}$, i.e. locally generated by a single element, and hence the sheaf of rings $\BA_{\inf, Y}[1/\ker(\theta)]$
 makes sense.  The reader is referred to \cite[\S 4, \S 6, \S 7]{SchpHodge} for more details.
 In particular, \cite[Lem. 7.3]{SchpHodge} shows how pulling back by $\nu$, resp. $\lambda$, allows us to identify between 
 various possible notions of ``vector bundles" over $Y$.
 
 If $\BL$ is a lisse $\BZ_p$-local system over $Y_{\et}$, denote by $\hat\BL$ the corresponding 
 $\hat\BZ_p$-local system over $Y_{\proet}$.   We set
 \begin{equation}
 {\rm D}_{\rm dR}(\BL)=\nu_*(\hat\BL\otimes_{\hat\BZ_p} \CO\BB_{\rm dR, Y}).
 \end{equation}
 By \cite{LZ}, this is a vector bundle $\CE= {\rm D}_{\rm dR}(\BL)$ over $Y$ with a (separated and exhaustive decreasing) filtration by locally direct summands ${\rm Fil}^i$, $i\in\BZ$, and an integrable connection  $\nabla$ satisfying Griffiths transversality with respect to the filtration.
 There is a $\CO\BB_{\rm dR, Y}$-linear ``comparison" map  
 \begin{equation}
 a_{\rm dR}(\BL): \nu^*{\rm D}_{\rm dR}(\BL)\otimes_{\CO_Y}\CO\BB_{\rm dR, Y}\to \hat\BL\otimes_{\hat\BZ_p} \CO\BB_{\rm dR, Y}.
 \end{equation}
 
 \begin{definition}(\cite[Def. 8.3]{SchpHodge}, \cite{LZ})
The local system $\BL$ is \emph{de Rham} if $a_{\rm dR}(\BL)$ is an isomorphism.
 \end{definition}
 
We note that this property only depends on  the $\BQ_p$-local system $\BV:=\BL[1/p]$. Then $\BL$ is associated to $({\rm D}_{\rm dR}(\BL), {\rm Fil}^i, \nabla)$, i.e.
 \[
\hat \BL\otimes_{\hat\BZ_p}\BB_{\rm dR}^+\cong {\rm Fil}^0\big({\rm D}_{\rm dR}(\BL)\otimes_{\CO_Y} \CO\BB_{\rm dR, Y}\big)^{\nabla=0}.
 \]
  Set $\BM:=\hat\BL\otimes_{\hat\BZ_p}\BB_{\rm dR}^+$ and also
 \[
 \BM_0:=({\rm D}_{\rm dR}(\BL)\otimes_{\CO_Y} \CO\BB^+_{\rm dR, Y})^{\nabla=0} ,
 \]
 which is also a $\BB_{{\rm dR}, Y}^+$-local system.
The $\BB_{{\rm dR}, Y}^+$-local systems $\BM$ and $\BM_0$ have filtrations ${\rm Fil}^j\BM=\ker(\theta)^j\BM$ and  ${\rm Fil}^j\BM_0=\ker(\theta)^j\BM_0$ obtained from the filtration on $\BB_{\rm dR, Y}^+$. 
 For this filtration, we have
  \[
   {\rm gr}^0 \BM=\hat\BL\otimes_{\hat\BZ_p} \widehat\CO_Y.
  \]
Both $\BM$ and $\BM_0$ are ``$\BB_{\rm dR, Y}^+$-lattices in the same $\BB_{\rm dR, Y}$-space" via a canonical comparison isomorphism
\[
c: \hat\BL\otimes_{\hat\BZ_p}\BB_{\rm dR, Y}=\BM\otimes_{\BB_{\rm dR, Y}^+}\BB_{\rm dR, Y}\xrightarrow{\ \ \simeq\ \ } \BM_0\otimes_{\BB_{\rm dR, Y}^+}\BB_{\rm dR, Y}.
\]
By \cite[Prop. 7.9]{SchpHodge} we have
\[
(\BM\cap {\rm Fil}^i\BM_0)/(\BM \cap {\rm Fil}^{i+1}\BM_0) = {\rm Fil}^{-i}{\rm D}_{\rm dR}(\BL)\otimes_{\CO_Y} \widehat\CO_Y(i) 
\subset {\rm gr}^i\BM_0 \cong {\rm D}_{\rm dR}(\BL)\otimes_{\CO_Y} \widehat {\CO}_Y(i)
\]
 for all $i\in \BZ$.

\begin{example}
 Let  $f:A\to X$ be an abelian scheme. We consider the ${\BZ}_p$-local system on $Y=X^\ad$, given by the $p$-adic \'etale cohomology  $\BL=R^1f_{*, \et}({\BZ}_p)$.  We refer to \cite[Thm. 2.2]{CaraSch} which uses \cite{SchpHodge}; by loc. cit. (or \cite{LZ}) this ${\BZ}_p$-local system is de Rham. We have
 \[
 {\rm D}_{\rm dR}(\BL)={\rm H}^1_{\rm dR}(A/Y)
 \]
 equipped with its Gauss-Manin connection and its Hodge filtration. Also
  \[
  \BM:=\hat\BL\otimes_{\hat\BZ_p} \BB_{\rm dR, Y}^+={\rm H}^1(A, \BB_{\rm dR, A}^+) ,
  \]
  which is compared to 
  \[
  \BM_0=\big({\rm H}^1_{\rm dR}(A/Y)\otimes_{\CO_Y}\CO\BB_{\rm dR, Y}^+\big)^{\nabla=0}.
  \]
  The comparison $c$ identifies
\[
\BM\cap {\rm Fil}^j\BM_0/(\BM\cap {\rm Fil}^{j+1}\BM_0)={\rm Fil}^{-j}{\rm H}^1_{\rm dR}(A/Y)\otimes_{\CO_Y}\hat\CO_Y(j)
\subset
{\rm gr}^j\BM_0={\rm H}^1_{\rm dR}(A)\otimes_{\CO_Y}\hat\CO_Y(j).
\]
In particular, we obtain $\BM_0\subset\BM$. This also gives an ascending filtration 
on 
 $
   {\rm gr}^0 \BM=\hat\BL\otimes_{\hat\BZ_p} \widehat\CO_Y
  $
defined by ${\rm Fil}_{-j}(\hat\BL\otimes_{\hat\BZ_p}\widehat\CO_Y)=\BM\cap {\rm Fil}^j\BM_0/({\rm Fil}^1\BM\cap {\rm Fil}^{j}\BM_0)$. So, we have
\[
\ker(\theta) \BM\subset \BM_0\subset \BM ,
\]
with the corresponding filtration on $ \BM/\ker(\theta) \BM={\rm gr}^0(\BM)=\BL\otimes_{\BZ_p} \widehat\CO_Y$, with  graded pieces
\[
{\rm gr}_{j}(\hat\BL\otimes_{\hat\BZ_p}\widehat\CO_Y)={\rm gr}^j {\rm H}^1_{\rm dR}(A/Y)\otimes_{\CO_Y}\widehat\CO_Y(-j).
\]
This is the Hodge-Tate filtration:
\begin{align*}
\BM/\BM_0&=R^0f_*(\Omega^1_{A/Y})\otimes_{\CO_Y}\widehat\CO_Y(-1),\\
\BM_0/\ker(\theta)\BM&=R^1f_*(\CO_A)\otimes_{\CO_Y}\widehat\CO_Y. 
\end{align*}
\end{example}
\smallskip

 We continue with a general de Rham $\BZ_p$-local system $\BL$ over $X$.
 Since $Y=X^{\rm ad}$, we write, for simplicity, $X^\diam/\Spd(E)$ instead of $Y^\diam/\Spd(E)=(X^{\rm ad})^\diam/\Spd(E)$.
 Let $n$ be the $\BZ_p$-rank of $\BL$. Then the 
 sheaf of trivializations (``frames") 
 \[
\underline{\rm Isom}(\BZ_p^n\times Y, \BL)
 \]
  is a $\underline{\GL_n(\BZ_p)}$-torsor over $Y_{\proet}$.  For an $S/\Spd(E)$-valued point of $X^\diam/\Spd(E)$ given by $x: S^\sharp\to Y=X^{\rm ad}$ over $E$, we set
 \[
 \BP(S)=\{\alpha: (\underline\BZ_p^n)_{S^\sharp}\xrightarrow{\ \simeq\ } x^*(\BL)\}.
 \]
 This gives a  $v$-sheaf which is a pro-\'etale $\underline{\GL_n(\BZ_p)}$-torsor
 over $X^\diam$.
 \begin{proposition}\label{uniqueHT}
 Let $\BL$ be a de Rham $\BZ_p$-local system of rank $n$ over the smooth $E$-scheme $X$, with associated pro-\'etale $\underline{\GL_n(\BZ_p)}$-torsor $\BP$ over $X^\diam.$ There is a unique $\underline{\GL_n(\BZ_p)}$-equivariant
 map of $v$-sheaves over $\Spd(E)$,
 \begin{equation*}
 {\rm DRT}(\BL): \BP\to {\rm Gr}_{\GL_n, \Spd(E)} ,
 \end{equation*}
such that:

A point $\alpha\in \BP(S)$ with values in an  affinoid perfectoid $S=\Spa(R, R^+)$ over $\kappa_E$ with $x: S^\sharp=\Spa(R^\sharp, R^{\sharp +})\to Y$
 over $E$ is mapped to the point of ${\rm Gr}_{\GL_n, \Spd(E)}$ given by the $\BB_{\rm dR}^+(R^\sharp)$-lattice equal to the inverse image under $\alpha$ of $x^*(\BM_0)\subset x^*(\BM_0)\otimes_{\BB_{\rm dR}^+(R^\sharp)}\BB_{\rm dR}(R^\sharp)=x^*(\BL)\otimes_{\BZ_p}\BB_{\rm dR}(R^\sharp)$: 
 \[
 \alpha^{-1}(x^*(\BM_0))\subset \BB_{\rm dR}(R^\sharp)^n\xrightarrow{\alpha\otimes 1} x^*(\BL)\otimes_{\BZ_p}\BB_{\rm dR}(R^\sharp).
 \]
  \end{proposition}
  \begin{proof}
 Here, the pull-backs $x^*(\BM_0)$ and $x^*(\BM)$ via $x: S^\sharp=\Spa(R^\sharp, R^{\sharp +})\to Y=X^{\rm ad}$
 are first considered as sheaves for the pro-\'etale topology (as defined in \cite[\S 3]{SchpHodge}), over the affinoid perfectoid $S^\sharp=\Spa(R^\sharp, R^{\sharp +})$. By descent (see the proof of \cite[Thm. 3.4.5]{CaraSch}) these sheaves are given by finite projective
 $\BB_{\rm dR}^+(R^\sharp)$-modules which we also denote by $x^*(\BM_0)$ and $x^*(\BM)$, and the construction above produces a point of 
 ${\rm Gr}_{\GL_n, \Spd(E)}$ as claimed. Note that there is a pro-\'etale cover 
 $U=\varprojlim_i U_i\to Y$ such that $\hat U=\Spa(R^\sharp, R^{\sharp +})$ is an affinoid perfectoid 
 mapping to $Y$ as above  (see \cite[Def. 4.3, Prop. 4.8]{SchpHodge}), over which $\BP$ has a point;
 then $\hat U^\diam\to X^\diam$ is $v$-surjective and the uniqueness follows.   
  \end{proof}

  \begin{definition}\label{LSVectorShtuka}
 Let $\BL$ be a de Rham $\BZ_p$-local system of rank $n$ over the smooth $E$-scheme $X$. 
 The  shtuka  $(\sV, \phi_\sV)$ of rank $n$ over $X^\diam/\Spd(E)$ \emph{corresponding to} $\BL$ is the vector space shtuka whose corresponding $\GL_n$-shtuka is given by the pair $(\BP, {\rm DRT}(\BL))$ via  Proposition \ref{pairs0}.
  \end{definition}
 
 Suppose that ${\rm DRT}(\BL)$ factors through
 ${\rm Gr}_{\GL_n, \Spd(E), \mu^{-1}}$, where $\mu$ is a minuscule coweight of $\GL_n$.
 Equivalently, we have, pro\'etale locally on $S$, 
 \begin{equation}\label{minuPos}
  \alpha^{-1}(\BM_0)=g\cdot \BB^+_{\rm dR}(R^\sharp)^n,\quad  g\in \GL_n(\BB^+_{\rm dR}(R^\sharp))\mu(\xi)^{-1} \GL_n(\BB^+_{\rm dR}(R^\sharp)) .
 \end{equation}
 We have a $\underline{\GL_n(\BZ_p)}$-equivariant map  of $v$-sheaves over $\Spd(E)$,
 \begin{equation}\label{HdgdR}
 {\rm HT}(\BL):  \BP\to {\rm Gr}_{\GL_n, \Spd(E), \mu^{-1}}\simeq \CF_{\GL_n,\mu^{-1}}^\diam .
 \end{equation}
 Then, in the above, the  shtuka  $(\sV, \phi_\sV)$ has leg bounded by $\mu$ and also corresponds to the pair $(\BP, {\rm HT}(\BL))$ in the sense of  Proposition \ref{pairs1}.

 \subsubsection{Generalization to torsors}
 By the Tannakian formalism we can extend the previous construction from local systems to torsors. More precisely, fix $G$ over $\BQ_p$ as before with parahoric model $\CG$, and let  $\BP$ be a pro-\'etale ${\CG(\BZ_p)}$-cover over the smooth scheme $X$ over $E$. For simplicity, we denote also by $\BP$ the corresponding pro-\'etale $\underline{\CG(\BZ_p)}$-torsor over
  $Y^\diam=X^\diam$, cf. \cite[\S 9.3]{Schber}. This amounts to giving, functorially, a pro-\'etale $\CG(\BZ_p)$-cover $\BP(S)$ over $S$, for every $S\in {\rm Perfd}_{k}$ and  every $S$-valued point of $X^\diam/\Spd(E)$, i.e.,
  $(S^\sharp, x)$ with  $x:S^\sharp\to X^{\rm ad}$, as above. Here, 
 for $S=\Spa(R, R^+)$, the torsor $\BP(S)$ is obtained
by pulling-back $\BP$ along $x$, followed by the tilting equivalence.

  If $\rho: G\to \GL(W)$ is a finite dimensional $\BQ_p$-rational representation, then there
  exists a $\BZ_p$-lattice $\Lambda\subset W$ such that $\rho(\CG(\BZ_p))\subset \GL(\Lambda)$. Set $\BL_{\rho,\Lambda}$ 
  for the corresponding $\BZ_p$-local system over $Y_{\proet}$ whose torsor of frames is given by
  $
  \underline{\GL(\Lambda)}\buildrel{\underline{\CG(\BZ_p)}}\over\times \BP.
  $
  \begin{definition}\label{defDR}
We say that the pro-\'etale $\CG(\BZ_p)$-cover $\BP$ is \emph{de Rham}, if for each  $(\rho,\Lambda)$ as above the local system $\BL_{\rho,\Lambda}$ is de Rham.

  This definition is independent of $\Lambda$. It is enough to check this property on one single faithful representation $\rho$.
    \end{definition}
  
  Suppose that $\BP$ is a de Rham pro-\'etale ${\CG(\BZ_p)}$-cover over $X$. Then $\CE_\rho:={\rm D}_{\rm dR}(\BL_{\rho,\Lambda})$ with its connection and filtration only depends on $\BL[1/p]$ and 
  $\rho: G\to \GL(W)$, and not on the lattice $\Lambda\subset W$. For 
  each classical point $x$ of $X$ with residue field $E(x)$  finite over $E$, this compares with Fontaine's theory in the sense that 
  the fiber 
  $
  {\rm D}_{\rm dR}(\BL_{\rho,\Lambda})_{x}
  $
  is canonically isomorphic to the value ${\rm D}_{\rm dR}((\BL_{\rho})_{\bar x})$ of Fontaine's functor applied to the 
  representation of ${\rm Gal}(\ov {E(x)}/E(x))$ given by the fiber $(\BL_{\rho})_{\bar x}$. 
  
  The functor $\rho\mapsto (\CE_{\rho, x}, {\rm Fil}^i(\CE_{\rho, x}))$ from representations of $G$ to filtered vector spaces defines a  conjugacy class $\{\mu_{\BP, x}\}$, defined over $E(x)$, of a coweight $\mu_{\BP, x}$ of $G_{\ov {E(x)}}$. 
 If $x$, $x'$ are in the same connected component of $X$, then $\{\mu_{\BP, x}\}=\{\mu_{\BP, x'}\}$. 
Assume that for all classical points $x$ of $X$, the conjugacy class $\{\mu_{\BP, x}\}$ is constant, equal to $\{\mu\}$. Then 
\cite[Prop. 7.9]{SchpHodge} implies that the above construction gives a $\underline{\CG(\BZ_p)}$-equivariant map of $v$-sheaves over $\Spd E$,
\begin{equation}\label{HdgGdR}
{\rm HT}: \BP\to {\rm Gr}_{G, \Spd(E), \mu^{-1}}=\CF_{G, \mu^{-1}}^\diam .
\end{equation}

  \begin{definition}\label{LSGShtuka}
 Let $\BP$ be a de Rham pro-\'etale ${\CG(\BZ_p)}$-cover 
   over the smooth $E$-scheme $X$ and assume that $\{\mu_{\BP, x}\}=\{\mu\}$
   is constant.
 The  $\CG$-shtuka  $(\sP, \phi_\sP)$ over $X^\diam/\Spd(E)$ with one leg bounded by $\mu$ \emph{corresponding to} $\BP$
is the $\CG$-shtuka corresponding to $\BP$ and the morphism \eqref{HdgGdR} via Proposition \ref{pairs2}. 
  \end{definition}
  
   \begin{remark}
  The map ${\rm HT}: \BP\to \CF_{G,\mu^{-1}}^\diam$ obtained from a de Rham pro-\'etale ${\CG(\BZ_p)}$-cover can be thought of as a sheaf analogue of Scholze's Hodge-Tate period map for Shimura varieties. This construction is also given by Hansen in \cite{HansenPre}. Hansen actually shows that $\BP$ is given by a diamond. Here, we only consider $\BP$ as a $v$-sheaf and we do not really need any additional geometric structure.
 \end{remark}

 \subsection{Maps of $\CG$-shtukas}\label{ss:MapsGsht}

 Our main goal is to show that maps between shtukas suitably extend and show Theorem \ref{vshtExt} of the introduction. In most of this section we discuss the linear case, i.e. take $\CG=\GL_d$. The results in the general case are obtained from this by applying the Tannakian equivalence.

\subsubsection{Maps between bundles over the curve}\label{ss:BC} 

Recall that for $S=\Spa(R, R^+)$ affinoid perfectoid over $k$, we have the (adic) relative Fargues-Fontaine curve
$X_{{\rm FF}, S}$.  Recall also that sending $S$ to the groupoid of vector bundles on $X_{{\rm FF}, S}$ defines a $v$-stack
(\cite[Prop. II.2.1]{FS}).

Let $\sE_1$ and $\sE_2$ be two vector bundles over $X_{{\rm FF}, S}$.
 Sending a perfectoid space $T$ over $k$ to the set
 \begin{equation}
  \underline{\CH}_{S}(\sE_1, \sE_2)(T)=\{(\alpha, f) \mid \alpha\in  T\to S, f\in {\rm Hom}_{X_{{\rm FF}, T}}(\alpha^*(\sE_1), \alpha^*(\sE_2))\}
 \end{equation}
 gives a $v$-sheaf $\underline{\CH}_{S}(\sE_1, \sE_2)$ with a morphism 
 \[
 H: \underline{\CH}_{S}(\sE_1, \sE_2)\to S.
 \]
 
 \begin{proposition}\label{BCspace} 
 a) The $v$-sheaf $\underline{\CH}_{S}(\sE_1, \sE_2)$ is represented by a locally spatial
 diamond over $S$. 
 
 b) The morphism $H: \underline{\CH}_{S}(\sE_1, \sE_2)\to S$
 is partially proper (in the sense of \cite[17.4.7]{Schber}).
 \end{proposition}
 
 \begin{proof} Since
 \[
 {\rm Hom}_{X_{{\rm FF}, T}}(\alpha^*(\sE_1), \alpha^*(\sE_2))={\rm H}^0(X_{{\rm FF}, T}, \alpha^*(\sE_2\otimes \sE_1^\vee)) ,
 \]
 both assertions follow immediately from 
 \cite[Prop.  II.2.15]{FS} applied to the bundle $\sE_2\otimes\sE_1^\vee$. Note here that the (quasi-)separatedness of $H$ follows by using the group structure, and showing that the zero section $0_S: S\hookrightarrow \underline{\CH}_{S}(\sE_1, \sE_2)$ is a closed immersion. 
\end{proof}

In general, if $Y$ is a  $v$-sheaf over $\Spd(k)$, {\cmag and $\sE_1$ and $\sE_2$ vector bundles ``over
$X_{{\rm FF}, Y}$" (in the sense described in \cite{An2}),} we can consider the $v$-sheaf given by
\[
\underline{\CH}_{Y}(\sE_1, \sE_2)(T)=\{(\alpha, f)\mid \alpha\in Y(T), f\in \underline{\rm Hom}_{X_{{\rm FF}, T}}(\alpha^*(\sE_1), \alpha^*(\sE_2))\} ,
\]
which affords a map $H: \underline{\CH}_{Y}(\sE_1, \sE_2)\to Y$.

{\cmag
 \begin{proposition}\label{FormalSep} If $Y$ is formally separated (in the sense of Gleason \cite[Def. 3.27]{Gl}), then the map $H:  \underline{\CH}_{Y}(\sE_1, \sE_2)\to Y$ is also formally separated.
 \end{proposition}

\begin{proof}
 For simplicity, set $\CF=\underline{\CH}_{Y}(\sE_1, \sE_2)$. 

1) First we observe that $H:\CF\to Y$ is separated, i.e. the diagonal $\CF\to \CF\times_Y\CF$ is a closed immersion: This follows from \cite[Prop. II.2.16]{FS}, see also Proposition \ref{BCspace} above.

2) We next prove that the diagonal 
$
\CF\to \CF\times_Y\CF
$
is formally adic in the sense of \cite[Def. 3.20]{Gl}. Then, by (1), $\CF\to \CF\times_Y\CF$ is formally closed, so, by definition, $H:\CF\to Y$ is formally separated.

Since $Y$ is formally separated, $\Delta: Y\to Y\times Y$ is formally adic. By \cite[Prop. 3.24]{Gl},  
we see that the base change of $Y\to Y\times Y$ by $H\times H: \CF\times\CF\to Y\times Y$, which is 
\[
  (\CF\times\CF)\times_{Y\times Y}Y\simeq \CF\times_Y\CF\to \CF\times\CF,
\]
is formally adic. Now, if $\CF\to\CF\times \CF$ is formally adic, it follows by the definition and 
standard properties of Cartesian diagrams that $\CF\to \CF\times_Y\CF$ is also formally adic.

 It remains to show that $\CF\to \CF\times \CF$ is formally adic. 
 By \cite[Lem.  3.30]{Gl}, it is enough 
to show that the homomorphism
\[
(\CF_{\red})^\sdiam\to \CF
\]
induced by adjunction, is an injective map of $v$-sheaves. Here, $\CF_{\red}$ is the reduction of the $v$-sheaf defined in \cite[Def. 3.12]{Gl}, see also \S \ref{par331}.  By the proof of \cite[Lem. 3.30]{Gl},   $(\CF_{\red})^\sdiam$
is the $v$-sheaf associated to the presheaf which sends an affinoid perfectoid $\Spa(R,R^+)$ to 
$
\CF(\Spec(R^+)^{\sdiam}).
$
The adjunction map
\[
{\rm adj}(R, R^+): \CF(\Spec(R^+)^{\sdiam})\to \CF(\Spa(R, R^+))
\]
is given by evaluating at $ \Spa(R, R^+)\to \Spec(R^+)^\sdiam=\Spd(R^+)$ given by $(R^+, R^+)\hookrightarrow (R, R^+)$. It will be enough to show that ${\rm adj}(R,R^+)$ is injective, for all $(R,R^+)$.

Set 
\[
(A, A^+)=(R^+\llps t^{1/p^\infty}\lrps
, R^+\lps t^{1/p^\infty}\rps)
\]
Recall that  $A^+=R^+\pstperf$ is the $t$-adic completion of the perfect algebra $R^+[t, t^{1/p}, t^{1/p^2}, \ldots ]$ and $A=R^+\lpstperf=R^+\pstperf[1/t]$.
 The elements of 
$R^+\pstperf$ are represented as power series
\[
\sum_{i\in \BZ[1/p]_{\geq 0}} r_i t^i
\]
with $r_i\in R^+$ and with support (i.e. set of indices $i$ for which $r_i\neq 0$) which is either finite, or forms an increasing unbounded sequence. Then $T=\Spa(A, A^+)$ is an affinoid perfectoid (with the $t$-topology) and
 \[
 c: T=\Spa(A, A^+)\to \Spd(R^+).
\]
is a $v$-cover.
 
By $v$-descent
\[
\CF(\Spec(R^+)^{\sdiam})\hookrightarrow \CF(T)=\CF(\Spa(A, A^+)).
\]
Choose a pseudouniformizer $\varpi$ of $R^+$. For any such choice, the morphism $a: \Spa(R, R^+)\to \Spec(R^+)^\sdiam=\Spd(R^+)$ is equal to the composition
\[
\Spa(R, R^+)\xrightarrow{t=\varpi} \Spa(A, A^+)\xrightarrow{c} \Spd(R^+).
\]
So, it will be enough to show that
\[
\CF(\Spd(R^+))\to \CF(\Spa(A, A^+))\to \prod_{\varpi} \CF(\Spa(R, R^+))
\]
is injective, where the product is over all pseudouniformizers of $R^+$.
We will use:

\begin{lemma}\label{BCinjectiveLemma}
\[
\bigcap_{\varpi}\ker(R^+\llps t^{1/p^\infty}\lrps\xrightarrow{t\mapsto \varpi} R)=(0).
\]
\end{lemma}

 \begin{proof}
 Suppose $f$ is in the intersection above. There is $N\geq 0$ such that 
 $t^Nf\in R^+\lps t^{1/p^\infty}\rps$ and so, without loss of generality, we can suppose that $f=f(t)\in R^+\lps t^{1/p^\infty}\rps$. Assuming $f\neq 0$ we will obtain a contradiction. There exists a perfectoid
 $\Spa(K, V)\to \Spa(R, R^+)$, with $V$ a  valuation ring and $K$ its fraction field, such that $f\neq 0$ in $V\lps t^{1/p^\infty}\rps$. Write 
 \[
 f=f(t)=\sum_{i=0}^\infty a_i t^{m_i}
 \] 
 with $a_i\in V$,
$
 0\leq m_0<m_1<\cdots <m_n<\cdots 
$,
 all in $\BZ[1/p]$, and $a_0\neq 0$. Choose a pseudouniformizer $\varpi\in R^+$, it is topologically nilpotent and a unit in $R$, and $\varpi\neq 0$ in $V$. Since $\varpi$ is topologically nilpotent, there is $m\geq 1$  
 such that 
 \[
 |\varpi |^{m(m_1-m_0)}<|a_0|.
 \]
 Since $\varpi^m$ is still a pseudouniformizer of $(R, R^+)$, we have $f(\varpi^m)=0$ in $R$ and so also in $V$. This implies 
 \[
|a_0||\varpi|^{mm_0}= |a_1 \varpi^{mm_1}+ \cdots +a_n\varpi^{mm_n}+\cdots |\leq {\rm sup}_{n\geq 1}|\{|a_n||\varpi|^{mm_n}\}\leq |\varpi|^{mm_1}
 \]
 which contradicts our choice of $m$. \end{proof}
 
Take two elements of $\CF(\Spec(R^+)^\sdiam)$ given by $(\alpha_i, f_i)$, $i=1$, $2$,  in $\CF(\Spa(A, A^+))$, with the same image in $\prod_\varpi\CF(\Spa(R, R^+))$. Since $Y$ is formally separated, the diagonal $Y\to Y\times Y$ is formally adic. Hence, by \cite[Lem. 3.30]{Gl} and an argument as above, we have $\alpha_1=\alpha_2$ 
and $\alpha:=\alpha_1=\alpha_2$   factors as $\Spa(A, A^+)\xrightarrow{c} \Spec(R^+)^\sdiam\to Y$. Write $f=f_1-f_2$.

Choose $I=[1, p]$ and consider the affinoid $Y_{T,I}=\CY_{T, I}$ (notation as in \cite{FS}), after fixing the choice of $t$ as a pseudouniformizer of $(A,A^+)$.  We will use the map 
$\pi: Y_{T,I}\rightarrow  X_{FF, T}=Y_{T, I}/\phi$ 
in the definition of the FF curve (\cite[Prop. II.1.6]{FS}). We see that
\[
{\rm Hom}_{X_{FF, T}}(\alpha^*\sE_1, \alpha^*\sE_2)\subset {\rm Hom}_{Y_{T,I}}(\pi^*(\alpha^*\sE_1), \pi^*(\alpha^*\sE_2)).
\]
Now enlarge the projective modules corresponding to $\sE_{T, i}=\alpha^*\sE_i$ over $X_{FF, T}$ as in the proof of \cite[Thm. II.2.6]{FS}, i.e. find $\CG'_i$ over $X_{FF, T}$, such that $\pi^*\sE_{T,i}\oplus \pi^*\CG'_i$ are the sheaves given by finite free $\CO(Y_{T, I})=\Gamma(Y_{T,I}, \CO_{Y_{T,I}})$-modules of rank $n_i$. Then, since
\[
 {\rm Hom}_{Y_{T,I}}(\pi^*\sE_{T,1}, \pi^*\sE_{T,2})\subset  {\rm Hom}_{Y_{T,I}}(\pi^*\sE_{T,1}\oplus \pi^*\CG'_1, \pi^*\sE_{T,2}\oplus \pi^*\CG'_2)\simeq {\rm Mat}_{n_1\times n_2}(\CO({Y_{T, I}}) ),
\]
an element $(\alpha, f)$ of $\CF(T)=\CF(\Spa(A, A^+))$ is determined by $\alpha\in Y(T)$ and a matrix $M=M(f)$ with entries in $\CO({Y_{T, I}})$, i.e. functions on $Y_{T,I}$. 

Assume that $f$ is in the kernel of the map given by $\CF(\Spa(A, A^+))\to\prod_{\varpi} \CF(\Spa(R, R^+))$; then $f$ is also zero
after pulling back by $Y^\varpi_{(R, R^+), I}\to Y_{T, I}$, given by $t\mapsto \varpi$, to all such $Y^\varpi_{(R, R^+), I}$.
(Note here that we include the superscript on $Y^\varpi_{(R, R^+), I}$ to emphasize that this depends on the choice of $\varpi$, see \cite[II.1.2]{FS}.)
The locus $|Z(f)|$ on $Y_{T,I}$ where $M(f)$ is zero is a ``Zariski closed'' subset of the topological space $|Y_{T,I}|$ underlying  $Y_{T,I}$; by our assumption,
this contains the images of $|Y^\varpi_{ (R, R^+), I}|$ for all $t\mapsto \varpi$.

We now  claim that $|Z(f)|=|Y_{T,I}|$; this implies that $M(f)=0$ and hence  
$f=0$ in $\CF(\Spa(A, A^+))$. In turn, this implies the desired injectivity. By \cite[Lem. IV.4.23., p. 142]{FS} (this uses that ``Zariski closed is strongly Zariski closed", shown in \cite[Rem. 7.5]{BS}), the complement of the
image of $|Y_{T,I}|-|Z(f)|$ under $v: |Y_{T,I}|\to |T|$ is Zariski closed. Hence
\[
|T|-v(|Y_{T,I}|-|Z(f)|)
\]
is a Zariski closed subset $|W|$ of $|T|$ underlying an affinoid perfectoid $W=\Spa(B, B^+)\to T$ defined by an ideal $J\subset R^+\llps t^{1/p^\infty}\lrps$. By our assumption
and the same reference,  all the morphisms $\Spa(R, R^+)\to T$ given by $t\mapsto\varpi$ factor through $W$ and we have 
\[
J\subset \bigcap_{\varpi}\ker(R^+\llps t^{1/p^\infty}\lrps\xrightarrow{t\mapsto \varpi} R).
\]
Lemma \ref{BCinjectiveLemma} gives $J=(0)$ which implies $|Y_{T,I}|=|Z(f)|$.
As we just saw above, this concludes the proof.
\end{proof}
}

\subsubsection{Maps between shtukas}

 Let $R^{\sharp +}$ be an integral perfectoid flat $\BZ_p$-algebra in the sense of \cite[Def. 17.5.1]{Schber}, cf. \S \ref{sss:BKFmod}; in particular, it is $p$-adically complete. We can take an element of the form $\pi=p^{1/p}\cdot(\hbox{\rm unit})\in R^{\sharp +}$ as a pseudouniformizer. 
 Set $R^\sharp=R^{\sharp +}[1/p]=R^{\sharp +}[1/\pi]$ and also consider the tilt 
 $\Spa(R, R^+)$ of $\Spa(R^\sharp, R^{\sharp +})$. 
 We consider the $v$-sheaves  $S^{+}=\Spa(R^{\sharp +}, R^{\sharp +})^\diam$ and 
 $S=\Spa(R^{\sharp}, R^{\sharp +})=\Spa(R^{\sharp}, R^{\sharp +})^\diam$ over $\Spd(\BZ_p)$.

 \begin{prop/constr}\label{extensiontoY}
 Let $(\sV, \phi_\sV)$ be a shtuka over $S^{+}/ \Spd(\BZ_p)$. There is a 
  vector bundle with meromorphic Frobenius $(\sV^+, \phi_{\sV^+})$
 over  $\CY(R,R^+)$ which extends the  shtuka $(\sV, \phi_\sV)_{|S}$ given as the evaluation
of $(\sV, \phi_\sV)$ on the affinoid perfectoid
 \[
 S=\Spa(R^{\sharp}, R^{\sharp +})\to S^+.
 \]
 \end{prop/constr}

\begin{remark}\label{rem275}
We expect that, in many cases at least, $(\sV^+, \phi_{\sV^+})$ ``further extends along $[\varpi]=p=0$", i.e. it is obtained
by restriction along the morphism of locally ringed spaces
\[
\CY(R, R^+)\to \Spec(W(R^+))
\]
from a uniquely determined $W(R^+)$-module
with meromorphic Frobenius structure, more precisely from a BKF module over $R^{\sharp +}$.
For example, when $(R, R^+)=(K, K^+)$ with $K^+$ a valuation ring of the
perfectoid field $K$, this follows from the result of Kedlaya (Theorem \ref{FFKedlaya}).
Combined with Proposition \ref{extensiontoY} above, this shows that in this case of a perfectoid field, the functor of Proposition \ref{propBKFshtuka}
is essentially surjective and hence an equivalence of categories. {\cmag In fact, the same conclusion holds 
when $(R, R^+)=((\prod_{i\in I}K_i)[1/(\varpi_i)], \prod_{i\in I}K^+_i)$ is a product of points, when restricting to shtukas of fixed rank. Indeed, we can apply the argument above, by using the extension of Kedlaya's result given in \cite[Prop. 2.7]{Gl21}.}
\end{remark}
 
 \begin{proof}
 Note that the shtuka $(\sV, \phi_\sV)_{|S}$ is a vector bundle with meromorphic Frobenius over $\CY_{[0,\infty)}(R, R^+)$.
 Let $\varpi=\pi^\flat\in R^+$ be the pseudo-uniformizer of $R^+$ that corresponds to $\pi$. Consider the  sousperfectoid adic space
  \[
U= \Spa\left(W(R^+)\left\langle \frac{[\varpi]}{p^a}\right\rangle[1/p], W(R^+)\left\langle \frac{[\varpi]}{p^a}\right\rangle\right)
 \]
 for $a\gg 0$. Here, $W(R^+)\langle \frac{[\varpi]}{p^a}\rangle$ is the completion of $W(R^+)[ \frac{[\varpi]}{p^a}]$ for the $p$-adic topology. (Note that $[\varpi]^n\mapsto 0$ in the $p$-adic topology). We can consider $B=W(R^+)\langle \frac{[\varpi]}{p^a}\rangle[1/p]$ as a Banach $\BQ_p$-algebra.  As in the proof of Theorem \ref{FFisocrystal}, we can pull back by
 \[
 U^\diam\to \Spa(W(R^+))^\diam=S^+\times\Spd(\BZ_p)
 \] 
 and then use the $\BZ^\times_p$-pro-\'etale perfectoid cover
 \[
 \hat U_\infty\to U=\Spa\left(W(R^+)\left\langle \frac{[\varpi]}{p^a}\right\rangle[1/p], W(R^+)\left\langle \frac{[\varpi]}{p^a}\right\rangle\right)
 \]
 constructed by complete tensoring with $\hat\CO_\infty$: Here
 \[
  \hat U_\infty=\Spa(\hat B_\infty, \hat B_\infty^+)
 \]
 with 
\[
\hat B_\infty:=W(R^+)\left\langle \frac{[\varpi]}{p^a}\right\rangle[1/p]\hat\otimes_{\BZ_p} \hat\CO_\infty, \quad
\hat B_\infty^+:=W(R^+)\left\langle \frac{[\varpi]}{p^a}\right\rangle\hat\otimes_{\BZ_p} \hat\CO_\infty.
\]
We obtain  a finite projective $\hat B_\infty$-module $\wt M$ with $\BZ^\times_p$-pro-\'etale descent data.
 The argument in the proof of Theorem \ref{FFisocrystal} now applies for the Banach algebra $B=W(R^+)\langle \frac{[\varpi]}{p^a}\rangle[1/p]$: This gives that $\wt M$ descends to a $B$-module $M$ if there is such descent for the base changes after 
 \[
 W(R^+)\left\langle \frac{[\varpi]}{p^a}\right\rangle\to W(O_C)\left\langle \frac{[\varpi]}{p^a}\right\rangle
 \]
 given by all the continuous
 \[
f:  (R, R^+)\to (C, O_C), \qquad f(\varpi)=\varpi_C.
 \]
 This allows us to reduce the proof of descent to the case of a shtuka over $\Spd(O_{C^\sharp})$, where $O_{C^\sharp}$
 has the analytic topology and $C^\sharp$ is algebraically closed. In this case, we know by Fargues-Fontaine 
 (cf. \cite[13.2]{Schber}) 
 that the restriction of the shtuka
 via $\Spa(C^\sharp, O_{C^\sharp})\to \Spd(O_{C^\sharp})$ uniquely extends to a BKF module $M(O_C)$ over $W(O_C)$.
 By full-faithfulness of the restriction ``away from $\infty$" (Proposition \ref{FFres}, see also the proof of Proposition \ref{propBKFshtuka}) we see that
 this $W(O_C)\langle \frac{[\varpi_C]}{p^a}\rangle$-module $M(O_C)\otimes_{W(O_C)}W(O_C)\langle \frac{[\varpi_C]}{p^a}\rangle $ provides the descent module; so descent holds in this situation. 
 
 By the above we see that there is a vector bundle with Frobenius structure over the (sousperfectoid) analytic adic   space $U$
 which descends the vector bundle with Frobenius structure over $\hat U_\infty$ obtained by the shtuka $(\sV, \phi_\sV)$
 over $S^+$. Note that $\CY(R, R^+)$ is a (sousperfectoid) analytic adic space which is covered by the open subspaces $U$ and $\CY_{[0, b]}(R, R^+)\subset \CY_{[0,\infty)}(R, R^+)$. We can now obtain, by glueing, the desired vector bundle $(\sV^+, \phi_{\sV^+})$ over $\CY(R, R^+)$
 that extends $(\sV, \phi_{\sV})_{|S}$.
 \end{proof}

  We will now assume in addition that    $R^{\sharp +}=R^{\sharp\circ}$. Then
 we also have $R^+=R^\circ$.

 \begin{proposition}\label{Extperfd}
 For $R^{\sharp +}$ integral perfectoid as above with $R^{\sharp +}=R^{\sharp \circ}$, 
consider the $v$-sheaves  $S^{+}=\Spa(R^{\sharp +}, R^{\sharp +})^\diam$ and 
 $S=\Spa(R^{\sharp}, R^{\sharp +})^\diam$ over $\Spd(\BZ_p)$. Let $(\sV, \phi_\sV)$ and  $(\sV', \phi_{\sV'})$ be two shtukas over
 $S^{+}/ \Spd(\BZ_p)$. Every homomorphism
 \[
 \psi_{S}: (\sV, \phi_\sV)_{|S}\to (\sV', \phi_{\sV'})_{|S}
 \]
between their restrictions to $S/ \Spd(\BZ_p)$, extends uniquely to a homomorphism
\[
\psi: (\sV, \phi_\sV)\to (\sV', \phi_{\sV'})
\]
 over $S^+/ \Spd(\BZ_p)$.
 \end{proposition}

 \begin{proof} {\cmag Since $S^+$ is formally separated, the uniqueness follows from Proposition \ref{FormalSep} and \cite[Prop. 4.9]{Gl}, together with the fact that the functor $(\sV, \phi_{\sV})\mapsto \sV_{\rm FF}$ is faithful. (In fact, this argument implies that if an extension exists, it is unique, even when $R^{\sharp +}\neq R^{\sharp \circ}$.)}
 
 Let us discuss the existence.  By Proposition \ref{extensiontoY} above, $(\sV, \phi_\sV)$ and $(\sV', \phi_{\sV'})$ give vector bundles $\sV^+$ and  $\sV'^+$  over $\CY(R,R^+)$
 with meromorphic 
 Frobenius. Suppose that $r>0$ is large enough so that $\CY_{[r, \infty)}(R,R^+)$ does not intersect the divisor $S$ where the Frobenius is not an isomorphism.
 By Proposition \ref{FFres}, since $R^+=R^\circ$,
  restriction  from $\CY_{[r, \infty]}(R,R^+)$ to $\CY_{[r, \infty)}(R,R^+)$ gives a fully-faithful functor from the category of vector bundles with $\phi$-structure over $\CY_{[r, \infty]}(R,R^+)$ to the category of  vector bundles with $\phi$-structure over $\CY_{[r, \infty)}(R,R^+)$.
This implies that the homomorphism $\psi_S$ over $\CY_{[0,\infty)}(R,R^+)$ uniquely extends to a homomorphism of vector bundles $\psi_{\CY(R,R^+)}\colon \sV^+\to \sV'^+$ over the adic space $\CY_{[0,\infty]}(R,R^+)=\CY(R,R^+)$. Hence, by the GAGA-equivalence of Theorem \ref{FFKedlaya}, this corresponds to a homomorphism of vector bundles over the scheme $Y(R,R^+)$, 
\[
\psi_{Y(R, R^+)}: V^+\to V'^+ .
\]
{\cmag Note the open immersion
\[
j: \Spec(W(R^+)[1/p])\hookrightarrow Y(R, R^+).
\]
We set $\psi_{W(R^+)[1/p]}:=j^*\psi_{Y(R, R^+)}$.

Now let $x: T=\Spa(B, B^+)\to S^+=\Spd(R^{\sharp +})$, with $T$ affinoid perfectoid over $k$, given by 
an untilt $T^\sharp=\Spa(B^\sharp, B^{\sharp +})$ and a continuous $R^{\sharp +}\to B^{\sharp +}$ which corresponds to $R^+\to B^+$.
We want to construct a homomorphism $\psi_T\colon x^*(\sV)\to x^*(\sV')$ over $T\bdtimes \BZ_p$. We will first assume that $T=\Spa(C, C^+)$, with $C$ a complete nonarchimedean algebraically closed field over $k$.

1) Suppose that $x$ is adic, so that $\pi^\flat$ maps to a pseudouniformizer of $C^{+}$. 
Then $x$ factors through $S=\Spa(R^\sharp, R^{\sharp +})^\diam$ and we set $\psi_T$ to be the pull-back of $\psi_S$ by $x: T\to S$.

Note that in this case,  we also have a  homomorphism $\psi(x): x^*(\sV)\to x^*(\sV')$ over $\CY(C,C^+)=\Spa(W(C^+))\setminus \{[\varpi_{C}]=0, p=0\}$ obtained by pulling back $\psi_{\CY(R,R^+)}$ by $x$. Then $\psi_T$ is also given by the restriction of 
$\psi(x)$ to $\CY_{[0,\infty)}(C,C^+)$. 

By results of Kedlaya (see \cite[Thm. 3.8]{Kedlaya} and \cite[Prop. 14.2.6]{Schber}),  
$\psi(x)$ uniquely extends to a homomorphism $\psi(x)^+: (x^*\sV)^+\to (x^*\sV')^+$ of corresponding $W(C^+)$-modules. The homomorphism $x^*: R^+\to C^+$ also gives $x^*: W(R^+)\to W(C^+)$ and we have
\begin{equation}\label{pullb1}
\psi(x)^+[1/p]=x^*\psi_{W(R^+)[1/p]},
\end{equation}
as maps $(x^*\sV)^+[1/p]\to (x^*\sV')^+[1/p]$.
 
2)   Suppose that $x$ is not adic. Then $\pi^\flat$ maps to $0$ in $C^+$ and $C^{\sharp +}=C^+$, i.e.
the untilt $T^\sharp=T$ is in characteristic $p$. Set $R^+_{\rm red}:=(R^{\sharp +}/(\pi))_{\rm red}=(R^+/(\pi^\flat))_{\rm red}$. The point $x$ factors as
\[
x: T=\Spa(C, C^+)\to \Spec(R^+_{\rm red})^\sdiam\to \Spd(R^{\sharp +}),
\]
and the shtukas $x^*(\sV, \phi_\sV)$ and $x^*(\sV, \phi_{\sV'})$, are pull-backs of shtukas over $\Spd(R^+_{\rm red})=\Spec(R^+_{\rm red})^\sdiam$. \quash{By Theorem \ref{FFisocrystal}, these are given by meromorphic Frobenius  crystals underlying corresponding $W(R^+_{\rm red})$-modules $\sM $ and $\sM' $.}

Now, using that $R^{\sharp +}$ is $\BZ_p$-flat, we can
find a point $\tilde x: \Spa(\tilde C , \tilde C^{+})\to S^+$ as in case 1), 
i.e. with untilt $(\tilde C^\sharp, \tilde C^{\sharp +})$ over $(\BQ_p, \BZ_p)$, and such that the corresponding $\tilde x^*: R^{\sharp +}\to \tilde C^{\sharp +}$ lifts
$x^*: R^{\sharp +}\to R^{\sharp +}/(\pi)=R^+/(\pi^\flat)\to C^+$. This is meant in the sense 
that there is a map $k(\tilde C^{\sharp})=O_{\tilde C^{\sharp}}/\fkm_{\tilde C^{\sharp}}\hookrightarrow C$ restricting to 
$\tilde C^{\sharp +}/\fkm_{\tilde C^{\sharp}}\to C^+$ which, when composed with $ \tilde x^*\, {\rm mod}\ \fkm_{\tilde C^{\sharp}}$, gives $x^*$. There is a commutative diagram
 \begin{equation}\label{diagramM}
 \begin{aligned}
   \xymatrix{
          \Spd(C^+) \ar[r]^{} \ar[d]_{} &  \Spd(R^+_{\rm red})\ar[d]_{}\\
         \Spd(\tilde C^{\sharp +})\ar[r]^{} &  \Spd(R^{\sharp +}).
        }
        \end{aligned}
    \end{equation}
We have two shtukas $\tilde x^*(\sV, \phi_\sV)$ and $\tilde x^*(\sV, \phi_{\sV'})$ over $\Spd(\tilde C^{\sharp +})$ obtained by pulling back via $\tilde x: \Spd(\tilde C^{\sharp +})\to S^+=\Spd(R^{\sharp +})$. Pulling back these by $\Spd(C^+)\to \Spd(\tilde C^{\sharp +})$ given by $\tilde C^{\sharp +}\to C^+$, recovers the shtukas $x^*(\sV, \phi_\sV)$ and $x^*(\sV, \phi_{\sV'})$.
Note that there are equivalences of categories between shtukas over $\Spd(C^+)$ and BKF-modules (with leg at $p$, i.e. meromorphic Frobenius crystals) over $W(C^+)$, shtukas over $\Spd(\tilde C^{\sharp +})$ and BKF-modules over $W(\tilde C^+)$ (with leg at $\tilde C^{\sharp +}$), see Proposition \ref{propBKFshtuka}, also Remarks \ref{rem239}, \ref{rem275},
 and also shtukas over $\Spd(R^+_{\rm red})$ and BKF-modules (with leg at $p$, i.e. meromorphic Frobenius crystals) over $W(R^+_{\rm red})$ (Theorem \ref{FFisocrystal}); these equivalences are compatible with pull-backs. Using these equivalences, we can now define
\[
\psi(x)^+:=\psi(\tilde x)^+\, {\rm mod}\, (W(\fkm_{\tilde C}))
\]
as a map $(x^*\sV)^+\to (x^*\sV')^+$ between the corresponding $W(C^+)$-modules. We can now see, using 
the commutativity of (\ref{diagramM}) and the above, that $\psi(x)^+[1/p]$, and therefore also $\psi(x)^+$, is independent of the lift $\tilde x$. Indeed, $\psi(x)^+[1/p]$ is also given as the pull-back
\begin{equation}\label{pullb2}
\psi(x)^+[1/p]=x^*\psi_{W(R^+)[1/p]}
\end{equation}
of $\psi_{W(R^+)[1/p]}=j^*\psi_{Y(R, R^+)}$ under
\[
x: \Spec(W(C^+)[1/p])\to \Spec(W(R^+_{\rm red})[1/p])\to \Spec(W(R^+)[1/p]).
\]}
 {\cmag We can now consider the general case in which $T=\Spa(B, B^+)$ is affinoid perfectoid with $x:T\to S^+$. 
\quash{We have $x^*:  W(R^+)[1/p]\to  W(B^+)[1/p]$ giving
 \[
 \psi_{W(B^+)[1/p]}=x^*\psi_{W(R^+)[1/p]},
 \]
 and, roughly speaking, we would like to use this to obtain $\psi_T: x^*\sV\to x^*\sV'$.}
 Choose a pseudo-uniformizer $\varpi$ of $B^+$ and form a corresponding product of points $f: Z=\Spa((\prod_{i\in I}C^+_i)[1/(\varpi_i)], \prod_{i\in I}C^+_i)\to T$ which is a $v$-cover of $T$. The composition $g=x\cdot f$ is given by $R^{\sharp +}\to \prod_{i\in I}C^{\sharp +}_i$ which gives $g^+: Z^+:=\Spd(\prod_{i\in I}C^{\sharp +}_i)\to S^+$. We have shtukas $(g^+)^*\sV$, $(g^+)^*\sV'$ over $Z^+$; by Remark  \ref{rem275} these correspond to BKF modules $\prod_i (x_i^*\sV)^+$ and $\prod_i (x_i^*\sV')^+$ over $W(\prod_{i\in I}C^+_i)=\prod_{i\in I}W(C^+_i)$. These modules are well-defined (up to canonical isomorphism) and only depend on $R^+\to \prod_{i\in I}C^+_i$ and the shtukas $\sV$, $\sV'$ over $S^+$. For each $i\in I$, the composition $x_i: \Spa(C_i, C^+_i)\to Z\to S^+$ 
gives a point of $S^+$ to which we can apply the construction above. We obtain $\psi(x_i)^+: (x_i^*\sV)^+ \to (x_i^*\sV')^+$ which gives a homomorphism $\psi_Z^+: \prod_i (x_i^*\sV)^+ \to \prod_i (x_i^*\sV')^+$ over $W(\prod_i C^+_i)$; this restricts to give a morphism $\psi_Z: (x\cdot f)^*\sV\to (x\cdot f)^*\sV'$ of shtukas over $Z$. Using the above we will see that 
$\psi_Z$ satisfies the $v$-descent condition and hence gives a homomorphism $\psi_T: x^*\sV\to x^*\sV'$ of shtukas over $T$.  We want to check the equality $p^*_1\psi_Z=p^*_2\psi_Z$ of the two pull-backs of the morphism $\psi_Z$ by the two projections $p_i: Z\times_TZ\to Z$, $i=1,2$. Consider a point $t=(t_1, t_2): \Spa(C, C^+)\to Z\times_T Z$. The two points 
$t_1$, $t_2$ give by composition the same map 
 $\Spa(C, C^+)\to  T\to S^+$. By the above, the pull-backs over $W(C^+)[1/p]$ of the two maps $p_1^*(\psi^+_Z[1/p])$, $p_2^*(\psi^+_Z[1/p])$ by $t$ are given by the base change of $\psi_{W(R^+)[1/p]}$ by the map $R^+\to C^+$ given by the composition $\Spa(C, C^+)\to  S^+$, and so they are equal. Hence,   the pull-backs of the two maps $p_1^*(\psi_Z)$, $p_2^*(\psi_Z)$ by $t$ are also equal as maps between vector bundles over $\CY_{[0,\infty)}(C, C^+)$. 
 But the restriction along all points $\Spa(C, C^+)\to Z\times_TZ$
is faithful on maps of vector bundles over $\CY_{[0,\infty)}(Z\times_TZ)$, cf. \cite[Lem. 2.3]{FarguesAbelJacobi}, and the descent property follows. Now by \cite[Prop. 19.5.3]{Schber}, the map $\psi_Z$ descends to $\psi_T: x^*\sV\to x^*\sV'$.

Finally we need to show that the maps $\psi_T: x^*\sV\to x^*\sV'$ for variable $x: T\to S^+$ give a homomorphism of shtukas $\psi: (\sV,\phi_\sV)\to (\sV',\phi_{\sV'})$ over the $v$-sheaf $S^+$: Consider a $v$-cover $W\to S^+$ by a perfectoid space $W$. By the work above, we obtain $\psi_W$ which we would like to show satisfies descent along $W\to S^+$. This is done by an argument similar to the one above. Let $T$ be a product of points with a map $T\to S^+$ and set $W'=W\times_{S^+}T$. We want to check equality of the pull-backs over the perfectoid $W'\times_T W'=(W\times_{S^+}W)\times_{S^+}T$.
Consider a point $t=(t_1,t_2): \Spa(C,C^+)\to W'\times_T W'$. The two points $t_i: \Spa(C,C^+)\to W'$ give the same map $\Spa(C, C^+)\to T\to S^+$ after composition, and by the argument above we see $t^*_1\psi_{W'}=t^*_2\psi_{W'}$ where $\psi_{W'}$ is the pull-back of $\psi_W$. This is enough to deduce that $\psi_W$ satisfies descent. }
\end{proof}

\subsubsection{Extending maps between shtukas}
In this subsection, we establish a relation between shtukas in characteristic zero and characteristic $p$. Let $\sX$ be a separated scheme of finite type and flat over $\Spec ( O_E)$.  Denote by $X=\sX\times_{\Spec (O_E)}\Spec (E)$ the generic fiber.

\begin{theorem}\label{vshtExt} Assume that $\sX$ is normal. 
Let $(\sV, \phi_\sV)$ and $(\sV', \phi_{\sV'})$ be two shtukas  over $\sX$. Any homomorphism   $\psi_X: {(\sV, \phi_\sV)}_{|X}\to {(\sV', \phi_{\sV'})}_{|X}$ between their restrictions to $X$ extends uniquely to a homomorphism $\psi: (\sV, \phi_\sV)\to (\sV', \phi_{\sV'})$ of shtukas over $\sX$.
\end{theorem}

\begin{proof}  {\cmag Recall that a shtuka over $\sX$ is, by Definition \ref{def:shtsch}, a shtuka over the $v$-sheaf $\sX^{\diam/}$ over $\Spd(O_E)$.} We can easily see that we may assume that $\sX$ is affine, $\sX=\Spec(A)$ with $A$ a normal domain. 
Consider the $v$-sheaf $H(\sV, \sV')\to \sX^\sdiam$ whose points with values in the affinoid perfectoid $T=\Spec(B, B^+)$
is the set 
\begin{equation}
H(\sV, \sV')(T)=\{(\alpha, f)\ |\ \alpha\in \sX^\sdiam(T), f: \alpha^*(\sV, \phi_\sV)\to \alpha^*(\sV', \phi_{\sV'})\} .
\end{equation}
As usual, we denote by $\hat A$ the $p$-adic completion of $A$ which is integrally closed in $R=\hat A[1/p]$.
Applying \cite[Lem. 10.1.6]{Schber} to the Tate ring $R=\hat A[1/p]$ (with $R^\circ=\hat A$) we obtain
 an affinoid perfectoid $\Spa(\tilde R, \tilde R^+)$ over $\Spa(E, O_E)$ with a morphism
 \[
  \Spa(\tilde R, \tilde R^+)\to \Spa(\hat A[1/p], \hat A)\to \Spa(\hat A, \hat A).
 \]
 In this situation, $(\tilde R, \tilde R^+)$ can be obtained as follows: Fix an algebraic closure $\bar F$
 of the fraction field $F$ of $R$ and consider the filtered direct system $\varinjlim\nolimits_j R_j$ over all finite
\'etale extensions $R_j/R$ contained in $\bar F$. Let $B_j$ be the integral closure of $\hat A$ in $R_j$ and
consider the $p$-adic completion 
 \[
 \tilde R^+:=\widehat{\varinjlim\nolimits_j B_j}.
 \]
Finally, take $\tilde R=\tilde R^+[1/p]$. Note that  $\tilde R^+$ is the ring of power bounded elements $\tilde R^\circ$
in the Tate ring $\tilde R$.
 Set 
 \begin{equation*}
 \begin{aligned}
 Y=\Spa(\hat A[1/p], \hat A)^\diam, \quad Y^+=\Spd(\hat A, \hat A), \quad
 \tilde Y=\Spa(\tilde R, \tilde R^+)^\diam,\quad  \tilde Y^+=\Spd(\tilde R^+, \tilde R^+).
 \end{aligned}
 \end{equation*}
  They all come with morphisms to $\Spd(\BZ_p)$. 
 
 \begin{lemma}\label{vCover} The morphism
 \[
 \beta: \tilde Y^+\to Y^+=\sX^\sdiam=(\wh\sX)^\diam
 \]
  is a surjective morphism of $v$-sheaves.
  \end{lemma}
  
 \begin{proof} {\cmag The morphism $\tilde Y\to Y$ is a surjective morphism of $v$-sheaves and the same result for $\tilde Y^+\to Y^+$ then follows from the more general result \cite[Prop. 2.31]{AnRicLou}. Here we give a more direct argument.} We can apply \cite[Lem. 17.4.9]{Schber} (see also the comment below that lemma):
  The morphism $\beta$ is quasi-compact  
   and for every complete non-archimedean algebraically closed field $C$,
  a morphism $\Spa(C, C^+)\to \Spa(\hat A, \hat A)$ given by $\hat A\to C^+\subset C$ factors as $A\to B_j\to C^+$
  (since $B_j/\hat A$ is integral and $C^+$ is a valuation ring) and so to
  \[
  \hat A\to \widehat {\varinjlim\nolimits_j B_j}=\tilde R^+\to C^+.
  \]
  This gives $\Spa(C, C^+)\to \Spa(\tilde R^+, \tilde R^+)$, i.e. a $\Spa(C, C^+)$-point of $\Spd(\tilde R^+)$.
  \end{proof}

The given homomorphism $\psi_X$ gives a morphism of $v$-sheaves $f: Y \to H(\sV, \sV')$. We consider the composition
\[
\tilde f: \tilde Y\to Y\to H(\sV, \sV') ,
\]
which corresponds to the shtuka homomorphism $\psi_{\tilde Y}$
obtained by pulling back $\psi_X$ to $\tilde Y$. 
Since $\tilde R^+=\tilde R^\circ$, we can apply Proposition \ref{Extperfd} to $\psi_{\tilde Y}$. This gives an extension of $\tilde f$ to $\tilde f^+: \tilde Y^+\to H(\sV, \sV')$. We would like to descend this to $f^+: Y^+\to H(\sV, \sV')$
along the $v$-cover $\beta: \tilde Y^+\to Y^+$.  Let
$$
Z^+=\tilde Y^+\times_{Y^+}   
 \tilde Y^+,\quad  Z=\tilde Y\times_{Y}   
 \tilde Y .
$$
To show descent of $\tilde f^+$ to $f^+\in H(\sV, \sV')(Y^+)$, we have to check the equality of the two pull-backs $p_1^*(\tilde f^+)$, resp. $p_2^*(\tilde f^+)$,
 in $H(\sV, \sV')(Z^+)$;
  this equality is true   in $ H(\sV, \sV')(Z)$.
  
 {\cmag Let us set
  \[
D^+=\widehat{\tilde R^+\otimes_{\hat A }\tilde R^+},\quad  D=(\widehat{\tilde R^+\otimes_{\hat A }\tilde R^+})[1/p],
\]
where the hat denotes the $p$-adic completion. Also let $\tilde D^{+}$ be the $p$-adic completion of the integral closure of the image
of $\tilde R^+ \otimes_{\hat A}\tilde R^+$
in $(\tilde R^+ \otimes_{\hat A}\tilde R^+)[1/p]$. Then we have $Z=\Spa(D,\tilde D^+)^\diam$ and $Z^+=\Spa(D^+, D^+)^\diam$. We also
have $\tilde Z^+:=\Spa(\tilde D^+, \tilde D^+)^\diam\to Z^+$ which is a $v$-cover.}
\quash{ Recall that $Z=\Spa(D,D^+)^\diam$, resp. $Z^+=\Spa(D^+, D^+)^\diam$, with
 \begin{equation*}
\begin{aligned}
D&=\tilde R\hat\otimes_{\hat A[1/p]}\tilde R\\
D^+&=\text{ the completion of the integral closure of the image
of $\tilde R^+\otimes_{\hat A}\tilde R^+$
in $D$.}
\end{aligned}
\end{equation*}
}
{\cmag By its construction, $D=\tilde D^+[1/p]$ is perfectoid and $\tilde D^+$ is flat over $\BZ_p$. 
By Proposition \ref{Extperfd},  
a point of $H(\sV, \sV')(Z)$ has at most one extension in $H(\sV, \sV')(\tilde Z^+)$. It follows that $p_1^*(\tilde f^+)=p_2^*(\tilde f^+)$ over $\tilde Z^+$ and so also over $Z^+$. This concludes the proof of the existence of the extension to $\sX^\sdiam=(\wh\sX)^\diam$. We can now see, by uniqueness, that we obtain the extension to $\sX^{\diam/}=\sX^\sdiam\sqcup_{\sX^\sdiam\times_{\Spd(O_E)}\Spd(E)}X^\diam$.} 
\end{proof}

\quash{To extend to $\sX^\diam$ now use the cover\mar{\cred change/remove}
\[
X^\diam \sqcup (\wh\sX)^\diam \to \sX^\diam
\]
and the construction above. The uniqueness follows by an argument as above, using that $|X^\diam|$ is dense in $|\sX^\diam|$, cf. Lemma \ref{topflat}.}

\begin{remark}
The previous theorem bears a formal resemblance to the theorem of de Jong-Tate on extending a homomorphism given outside a divisor between $p$-divisible groups over a normal base scheme. It would be interesting to extend Theorem \ref{vshtExt} to this general setting. 
\end{remark}

 Using the Tannakian equivalence,  Theorem \ref{vshtExt} immediately implies:

\begin{corollary}\label{shtExt}
Let $(\sP, \phi_\sP)$ and $(\sP', \phi_{\sP'})$ be two $\CG$-shtukas over $\sX$. Any isomorphism   $\psi_X: {(\sP, \phi_\sP)}_{|X}\xrightarrow{\ \sim\ }{(\sP', \phi_{\sP'})}_{|X}$ between their restrictions to $X$ extends uniquely to an isomorphism $\psi: (\sP, \phi_\sP)\xrightarrow{\ \sim\ }(\sP', \phi_{\sP'})$ over $\sX$. \hfill$\square$
\end{corollary}

 \section{Local Shimura varieties and their integral models}\label{s:LSV}

\subsection{Local Shimura varieties} 
We recall Scholze's  local Shimura varieties \cite[\S 24]{Schber}  and list some functorial properties of them. 

\subsubsection{Definitions} \label{ss:LSV}Let $\CG$ be a smooth affine group scheme over $\BZ_p$ with generic fiber $G$ a reductive group over $\BQ_p$, and $b\in G(\breve \BQ_p)$, and $\mu$  a conjugacy class of cocharacters of $G$. It is assumed that the $\sigma$-conjugacy class of $b$ lies in $B(G, \mu^{-1})$  and that $\CG$ has connected special fiber. Then Scholze associates to a triple $(\CG, b, \mu)$   a moduli space of shtukas ${\rm Sht}_{\CG, b, \mu}$. It is given as a ``diamond moduli space" of certain $\CG$-shtukas with one leg bounded by $\mu$ with a fixed associated Frobenius element, cf.  \cite[\S\S 23.1, 23.2, 23.3]{Schber}. 

More precisely, consider the functor on ${\rm Perfd}_k$ that sends $S$ to the set of isomorphism classes of quadruples 
\begin{equation}
 (S^\sharp,  \sP , \phi_{\sP }, i_r) ,
 \end{equation}
 where
 \begin{itemize}
 
\item[1)] $S^\sharp$ is an untilt of $S$ over $\Spa({\breve E})$,
\item[2)] $(\sP , \phi_{\sP })$ is a $\CG$-shtuka over $S$ with one leg along $S^\sharp$ bounded 
by $\mu$,
\item[3)] $i_r$ is an isomorphism of $G$-torsors 
\begin{equation}
i_r: G_{\CY_{[r,\infty)}(S)}\xrightarrow{\sim} \sP_{\, |\CY_{[r,\infty)}(S)}
\end{equation}
for large enough $r$ (for an implicit choice of pseudouniformizer $\varpi$), under which
$\phi_{\sP }$ is identified with $\phi_{G, b}$. We call $i_r$ a \emph{framing}.
 \end{itemize}
 Here we have denoted by $G_{\CY_{[r,\infty)}(S)}$ the trivial $G$-torsor over ${\CY_{[r,\infty)}(S)}$ (denoted $G\times {\CY_{[r,\infty)}(S)}$ in \cite[App. to \S 19]{Schber}), and by 
 \begin{equation}\label{phi_G}
 \phi_b=\phi_{G, b}\colon \phi^*(G_{\CY_{[r,\infty)}(S)})=G_{\CY_{[\frac{1}{p}r,\infty)}(S)}\xrightarrow{b\phi} G_{\CY_{[r,\infty)}(S)}
 \end{equation}
  the  $\phi$-linear isomorphism induced by right multiplication by $b$, denoted by $b\times {\rm Frob}$ in \cite[Def. 23.1.1]{Schber}.
In 3) we mean more precisely an equivalence class, where $i_r$ and $i'_{r'}$ are called equivalent if there exists $r''\geq r, r'$ such that ${i_r}_{\, |\CY_{[r'',\infty)}(S)}={i'_{r'}}_{\, |\CY_{[r'',\infty)}(S)}$. 

 Note that the definition above makes sense even when we do not make the hypothesis that $\mu$ is minuscule.  One of the main results of \cite{Schber} is that, in this generality,  
${\rm Sht}_{\CG, b, \mu}$ is a $v$-sheaf and is represented by a locally spatial diamond over $\Spd(E)$; this is shown by employing the crystalline period morphism (also called the ``Grothendieck-Messing period morphism"),
\begin{equation}\label{GMmap}
\pi_{GM}: {\rm Sht}_{\CG, b, \mu}\to {\rm Gr}_{G, \Spd(\breve  E), \leq\mu},
\end{equation}
which is \'etale. We set
\begin{equation}
J_b(\BQ_p)=\{g\in G(W(k)[1/p])\ |\ gb\sigma(g)^{-1}=b\}.
\end{equation}
The group $J_b(\BQ_p)$ acts on ${\rm Sht}_{\CG, b, \mu}$ by changing the framing, 
\[
g\cdot  (S^\sharp,  \sP , \phi_{\sP }, i_r)= (S^\sharp,  \sP , \phi_{\sP }, i_r\circ g^{-1}).
\]

 Let $K=\CG(\BZ_p)$. Using the period morphism, one sees that ${\rm Sht}_{\CG, b, \mu}={\rm Sht}_{G, b, \mu, K}$ only depends on $\CG$ via $K$.
 Then $K$ varies through the open compact subgroups of $G(\BQ_p)$ and  there is an action of $G(\BQ_p)$ on the tower  
$({\rm Sht}_{G,b, \mu, K})_K$.

There is a Weil descent datum on the tower. Let $\tau$ be the relative Frobenius automorphism of $\breve E$ over $E$, and define the $v$-sheaf ${\rm Sht}_{\CG, b, \mu}^{(\tau)}$ by
$$
{\rm Sht}_{\CG, b, \mu}^{(\tau)}(S)={\rm Sht}_{\CG, b, \mu}(S\times_{\Spa (k), \tau}\Spa (k)) .
$$ If $S=\Spa (R, R^+)$, with structure morphism $\epsilon\colon k\to R$, denote by $R_{[\tau]}$ the same ring with the $k$-algebra structure defined by $k\buildrel{x\mapsto x^q}\over\to k\buildrel{\epsilon}\over\to R$. Here $q=|\kappa_E|$. Then $S\times_{\Spa k, \tau}\Spa k=\Spa (R_{[\tau]}, R^+_{[\tau]})$.  We define the Weil descent datum 
\begin{equation}\label{Weildes}
\omega\colon {\rm Sht}_{\CG, b, \mu}\to {\rm Sht}_{\CG, b, \mu}^{(\tau)}
\end{equation}
by sending a point $ (S^\sharp,  \sP , \phi_{\sP }, i_r)$ of ${\rm Sht}_{\CG, b, \mu}$ with values in $S=\Spa (R, R^+)$ to the point of ${\rm Sht}_{\CG, b, \mu}((R, R^+)_{[\tau]})$ given by $ (S^\sharp,  \sP , \phi_{\sP }, i'_{r'})$. Here $r'=q{r}$, and we use the identification 
$$
\phi\colon \CY_{[r, \infty)}(R, R^+)=\CY_{[r', \infty)}((R, R^+)_{[\tau]}) . 
$$
Then $i'_{r'}$ is defined as the composition
\begin{equation*}
 G_{\CY_{[r',\infty)}((R, R^+)_{[\tau]})}\xrightarrow{\phi_b^f}G_{\CY_{[r,\infty)}(R, R^+)}\xrightarrow{\sim} \sP_{\, |\CY_{[r,\infty)}(R, R^+)}= 
 \sP_{\, |\CY_{[r',\infty)}((R, R^+)_{[\tau]})},
\end{equation*}
where $q=p^f$ and where $\phi_b^f=(\phi_b)^f$ denotes the $f$-fold iteration of the $\phi$-linear automorphism of the trivial $G$-torsor in \eqref{phi_G}.

Any group homomorphism $\rho\colon G\to G'$  compatible with the other  data $(G, b, \mu)\to (G',  b', \mu')$ induces a  morphism of towers of $v$-sheaves,
\begin{equation}\label{functLSV}
{\rm Sht}_{G, b, \mu}\to {\rm Sht}_{G', b', \mu'}\times_{\Spd(\breve{E'})}\Spd(\breve E) ,
\end{equation}
compatible with descent data. 
Here $ E\supset E'$ denote the corresponding reflex fields.  

 The \emph{local Shimura varieties} (LSV) correspond to the cases that $\mu$ is minuscule, i.e., when $(G, b, \mu)$ is a local Shimura datum, cf. \S \ref{sss:backg}. This is the main case of interest here, and we will concentrate on it.  In this case, 
${\rm Gr}_{G, \Spd(\breve  E), \leq \mu}=\CF^\diamondsuit_{G, \mu, \breve  E}$ 
and as in \cite[\S 24.1]{Schber}, 
\begin{equation}
{\rm Sht}_{G,b, \mu, K}=\CM_{G,b,\mu, K}^\diamondsuit ,
\end{equation}
for a uniquely-determined smooth rigid analytic space $\CM_{G,b,\mu, K}$ over $\breve  E$. 
{\cmag
\begin{proposition}\label{LSVqcqs}
\noindent(i)  Assume that  $\rho^{-1}(Z_{G'})\subset Z_G$, where $Z_G$ and $Z_{G'}$ denote the centers of $G$ and $G'$. 
  Then the induced morphism of towers of rigid-analytic spaces with $G'(\BQ_p)$-action is  a closed immersion,
\begin{equation*}
{\CM}_{G, b, \mu}\times^{G(\BQ_p)}G'(\BQ_p)\to {\CM}_{G', b', \mu'}\times_{\Sp(\breve{E'})}\Sp(\breve E) .
\end{equation*}
 (ii) Suppose that $\rho:G\to G'$ has finite kernel. Then, for $\rho(K)\subset K'$, the morphism  \eqref{functLSV} induces a qcqs morphism,
$$
{\CM}_{G, b, \mu, K}\to {\CM}_{G', b', \mu', K'}\times_{\Sp(\breve{E'})}\Sp(\breve E).
$$

\end{proposition}
\begin{proof}
For (i), we claim that   the assumption implies
\begin{equation}\label{eqadm}
(\CF_{G, \mu, \breve  E})^{\rm adm}=\rho^{-1}((\CF_{G', \mu', \breve  E'}\times_{\Sp(\breve E')}\Sp(\breve E))^{\rm adm}) .
\end{equation}
Here there appear the admissible sets, which are by definition the images under the crystalline period maps, comp. \eqref{GMmap}. This then implies that
$$
{\CM}_{G, b, \mu}\times^{G(\BQ_p)}G'(\BQ_p)\simeq  ({\CM}_{G', b', \mu'}\times_{\Sp(\breve{E'})}\Sp(\breve E))\times_{(\CF_{G', \mu', \breve  E'}\times_{\Sp(\breve E')}\Sp(\breve E))} \CF_{G, \mu, \breve  E} ,
$$
hence the LHS is a closed subspace of ${\CM}_{G', b', \mu'}\times_{\Sp(\breve{E'})}\Sp(\breve E)$. This means that for any $K\subset G(\BQ_p)$, there exists $K'\subset G'(\BQ_p)$ with $K'\supset \rho(K)$ such that the induced map ${\CM}_{G, b, \mu, K}\to {\CM}_{G', b', \mu', K'}\times_{\Sp(\breve{E'})}\Sp(\breve E)$ is a closed immersion.  

To prove the claim, consider a $C$-valued point $x$ of $\CF_{G, \mu, \breve E}$. Then $x$  lies in the admissible locus if and only if the corresponding modification $\CE_{b,x}$ of $\CE_b$ at $\infty$ is a trivial  $G$-bundle on the FF curve, cf. \cite[Thm. 22.6.2]{Schber}.  Note that, as we are assuming that $[b]\in B(G, \mu^{-1})$, we have that $\kappa(b)=-\mu^\natural$, and hence $\CE_{b,x}$ is trivial if and only if  $\CE_{b,x}$ is semi-stable, or, equivalently,   the corresponding element in $B(G)$ is basic. We note that this characterization shows that the admissible locus is open, as follows from the upper semi-continuity of the Newton point (\cite[Cor. 22.5.1]{Schber}) and the local constancy of the Kottwitz invariant in $\pi_1(G)_\Gamma$ (\cite[Thm. III.2.7]{FS}). Now the image of $x$ lies in the admissible set in $\CF_{G', \mu', \breve E'}$ if and only if the  $G'$-bundle $\CE_{b', \rho_*(x)}=\rho_*(\CE_{b, x})$ is a semi-stable $G'$-bundle on the FF curve. If $\CE_{b, x}$ is the trivial $G$-bundle then, obviously, also $\rho_*(\CE_{b, x})$ is the trivial $G'$-bundle. Conversely, if $\rho_*(\CE_{b, x})$ corresponds to a basic element in $B(G')$ then, since  $\rho^{-1}(Z_{G'})\subset Z_G$, the same holds for $\CE_{b, x}$. In fact, under our assumption the natural map $\rho_*:B(G)\to B(G')$ satisfies $\rho_*^{-1}(B(G')_{\rm basic})=B(G)_{\rm basic}$. Indeed, $\nu_b$ is central in $G$ if and only $\nu_{\rho(b)}=\rho\circ\nu_b$ is central in $G$.  

For (ii), we note that the assumption of (i) is satisfied: indeed, any element $\tilde z\in \rho^{-1}(Z_{G'})$ defines a morphism $G\to\ker(\rho): g\mapsto \tilde z g \tilde z^{-1}g^{-1}$; since $G$ is connected and $\ker(\rho)$ finite, the morphism is constant, i.e., $\tilde z\in Z_G$. The morphism in (ii) factors as 
\begin{equation*}
{\CM}_{G, b, \mu, K}\to {\CM}_{G, b, \mu}\times^{G(\BQ_p)}G'(\BQ_p)/K'\to {\CM}_{G', b', \mu', K'}\times_{\Sp(\breve{E'})}\Sp(\breve E) .
\end{equation*}
By (i), the second morphism is qsqc. The first two  spaces map by \'etale morphisms to $(\CF_{G, \mu, \breve  E})^{\rm adm}$ (with fibers $G(\BQ_p)/K$, resp. $G'(\BQ_p)/K'$). Hence the first morphism is qsqc, and therefore also the composed morphism.
\end{proof}
}

\quash{
We use the alternative description of ${\rm Sht}_{G, b, \mu, K}(S)$ in terms\mar{\cred  Fix this proof? ``I don't understand''.}
 of $\underline{K}$-lattices in a pro-\'etale $\underline{G(\BQ_p)}$-torsor over the admissible set $(\CF_{G, \mu, \breve  E})^{\rm adm}$, which is used in Scholze's proof of the fact that these moduli functors are locally spatial diamonds, cf. \cite[Prop. 23.3.3]{Schber}.  The  proof of \cite[Prop. 23.3.3]{Schber} shows that if the lattices in question are allowed to vary in a bounded way,  the corresponding sub-$v$-sheaf is a spatial diamond, and in particular qcqs. This proves (i). 
}

   \subsubsection{Pushout functoriality}\label{ss:pushout}
Let $\rho\colon G\to G'$ be a group homomorphism compatible with the other data $(G, b, \mu)\to (G',  b', \mu')$ of local Shimura varieties. We assume that the kernel is a central subgroup and that the cokernel is a torus, i.e., $\rho$ induces an isomorphism
\begin{equation}\label{pushoutad}
\rho_\ad\colon G_\ad\isoarrow G'_\ad .
\end{equation}
  In \cite{PRintlsv} such $\rho$ are called    ad-isomorphisms (following Kottwitz). 
\begin{proposition}\label{pushoutprop}
Assume \eqref{pushoutad}.  The  morphism \eqref{functLSV} of pro-systems
 \[
{\CM}_{G, b, \mu}\to {\CM}_{G', b', \mu'}\times_{\Spa(\breve{E'})}\Spa(\breve E)
\]
 with action of $G(\BQ_p)$, resp. $G'(\BQ_p)$, induces an isomorphism of pro-systems with $G'(\BQ_p)$-action,
\begin{equation}\label{mapGG'}
{\CM}_{G, b, \mu}\times^{G(\BQ_p)}G'(\BQ_p)\isoarrow {\CM}_{G', b', \mu'}\times_{\Spa(\breve{E'})}\Spa(\breve E) .
\end{equation}
\end{proposition}
\begin{proof} 
By  Proposition \ref{LSVqcqs} (ii), the morphism of pro-systems is qcqs. Therefore it remains  to show the bijectivity of \eqref{mapGG'} on $C$-valued points, for any algebraically closed non-archimedean field $C$. Consider the following commutative diagram of (pro-systems of) rigid analytic spaces over $\breve E$, in which the vertical arrows are the crystalline period maps,
 \begin{displaymath}
   \xymatrix{
          {{\CM}}_{G, b, \mu} \ar[r]^{} \ar[d] & {{\CM}}_{G', b', \mu'}\times_{\Spa(\breve{E'})}\Spa(\breve E)  \ar[d]\\
         \CF_{G_\ad, \mu_\ad,  E_\ad}\times_{\Spa({E_\ad})}\Spa(\breve E)\ar[r]^{} & \CF_{G'_\ad, \mu'_\ad,  E'_\ad}\times_{\Spa({E'_\ad})}\Spa(\breve E).
        }
    \end{displaymath}
   The images of the vertical maps are the admissible sets. By our assumption, the lower horizontal arrow is an isomorphism. Under this isomorphism, the admissible sets correspond to each other, as follows from the argument in the proof of (ii) of Proposition \ref{LSVqcqs}.  Now consider the following diagram, where in the lower line appear the admissible sets,
    \begin{displaymath}
   \xymatrix{
          {{\CM}}_{G, b, \mu}\times^{G(\BQ_p)}G'(\BQ_p) \ar[r]^{} \ar[d] & {{\CM}}_{G', b', \mu'}\times_{\Spa(\breve{E'})}\Spa(\breve E)  \ar[d]\\
        (\CF_{G_\ad, \mu_\ad,  E_\ad}\times_{\Spa(E_\ad)}\Spa(\breve E))^{\rm adm}\ar[r]^{\simeq} &  (\CF_{G'_\ad, \mu'_\ad,  E'_\ad}\times_{\Spa(\breve{E'}_\ad)}\Spa(\breve E))^{\rm adm} .
        }
    \end{displaymath}
Now the fibers of the left vertical arrow are identified with $G(\BQ_p)\times^{G(\BQ_p)}G'(\BQ_p)=G'(\BQ_p)$, and hence map bijectively to the fibers of the right vertical arrow.
\end{proof}

\subsubsection{LSV of dimension zero}
Recall that $\dim {\CM}_{G, b, \mu}=\dim\CF_{G, \mu}=\langle\mu,2\rho\rangle$. Hence 
$$
\dim {\CM}_{G, b, \mu}=0 \iff \mu\colon \BG_{m, {\bar\BQ_p} }\to G_{\bar\BQ_p} \text{ is central. }
$$
Let us assume this. Denote by $Z^o$ the connected center of $G$, and by $\mu_Z\in X_*(Z^o)$ the element corresponding to $\mu$. Then $E(G, \mu)=E(Z^o, \mu_Z)$. Let $b_Z\in Z^o(\breve\BQ_p)$ be a representative of the unique $\sigma$-conjugacy class in $B(Z^o, \mu^{-1})$. Via push-out of torsors along the map $Z^o\to G$, we obtain a morphism of rigid-analytic spaces over $\breve E$,
\begin{equation}\label{zerodim}
{\CM}_{Z^o, b_Z, \mu_Z}\to {\CM}_{G, b, \mu} .
\end{equation}

\begin{proposition} Let $\dim {\CM}_{G, b, \mu} =0$. The morphism \eqref{zerodim} induces an isomorphism of towers of rigid-analytic spaces over $\breve E$, compatible with Weil descent down to $E=E(G, \mu)$, 
$$
{\CM}_{Z^o, b_Z, \mu_Z}\times^{Z^o(\BQ_p)}G(\BQ_p)\isoarrow {\CM}_{G, b, \mu} .
$$
\end{proposition}

\begin{proof}
Both $\CF(Z^o, \mu)$ and $\CF(G, \mu)$ reduce to a point. The assertion follows because the map is qsqc and since for any algebraically closed non-archimedean extension $C$ of $E$, the $C$-points of both source and target are identified with $G(\BQ_p)$. 
\end{proof}

\subsubsection{The torus case}
Let $G=T$ be a torus. In this case, there is a unique $[b]\in B(T, \mu^{-1})$, and ${\CM}_{T, b, \mu}$ is zero-dimensional. Let $\BC_p$ be the completion of $\bar\BQ_p$. 
\begin{proposition} The rigid-analytic space ${\CM}_{T, b, \mu}$  has a natural model over $\Sp (E)$, compatible with its Weil descent datum. 
 There is an identification 
$$
{\CM}_{T, b, \mu, K}(\BC_p)=T(\BQ_p)/K, \quad K\subset T(\BQ_p), 
$$
such that $\gamma\in\Gal(\bar\BQ_p/E)$ acts through its quotient $\Gal(\bar\BQ_p/E)^\ab$; furthermore, $\gamma$ with preimage $\epsilon\in E^\times$ 
 under the reciprocity map ${\rm rec}_E\colon E^\times \to \Gal(\bar\BQ_p/E)^\ab$ acts as 
$$
xK\mapsto {\rm N}_\mu(\epsilon)xK . 
$$
Here ${\rm N}_\mu$ is the composition ${\rm N}_\mu\colon  \Res_{E/\BQ_p}(\BG_m)\xrightarrow{\Res_{E/\BQ_p}(\mu)} \Res_{E/\BQ_p}(T)\xrightarrow{{\rm N}_{E/\BQ_p}} T$. 
\end{proposition}
\begin{proof}
By push-out functoriality, we are reduced to the case where $T=T_0=\Res_{E/\BQ_p}(\BG_m)$ and $\mu=\mu_0$, where $\mu_0\in X_*(T)={\rm Ind}_{\BQ_p}^E(\BZ)$ is the canonical element given by 
$$
 (\mu_0)_\varphi=\begin{cases}
    \begin{aligned}
      1,  &\text{ if $\varphi=\id$ }\\
      0  , &\text{ otherwise.}     
    \end{aligned}
  \end{cases}
$$
Consider the rational RZ-data of EL-type given by the semi-simple $\BQ_p$-algebra $E$, the standard $E$-vector space $V$ of dimension $1$ and the cocharacter $\mu_0$, comp. \cite[\S 4.1]{RV}. The associated algebraic group over $\BQ_p$ associated to these rational RZ-data is $T_0$. Let $(M_{(E, V, \mu_0)})_{ K}, K\subset T_0(\BQ_p)$ be the associated RZ-tower over $\breve E$, with its descent datum to $E$. Then $(M_{(E, V, \mu_0)})_{ K}$ coincides with ${\CM}_{T_0, b_0, \mu_0, K}$, cf. \cite[Cor. 24.3.5]{Schber}. The result follows after identifying $(M_{(E, V, \mu_0)})_{ K}$ with the Lubin-Tate tower, cf. \cite[\S 3]{MChen}. More precisely, let $X_\pi$ be the Lubin-Tate group over $E$ corresponding to a uniformizer $\pi$ of $E$. Then the $p$-adic Tate module $T_p(X_\pi)$ is independent of $\pi$ (\cite[\S 3.7]{Ser}),  the space
 of trivializations of $T_p(X_\pi)$ is $(M_{(E, V, \mu_0)})_{ K}$, and the Galois action by $\Gal_E$ is given by Lubin-Tate theory, cf. \cite[\S 3.4]{Ser}. 
\end{proof}

 \subsection{Integral models}
 Suppose in addition that  $K$ is parahoric with $\CG$ the corresponding Bruhat-Tits group scheme. Then Scholze gives in \cite[Def. 25.1.1]{Schber} also  a construction of an ``integral model" $\CM^{\rm int}_{\CG, b, \mu}$  of ${\rm Sht}_{G, b, \mu, K}$ over $\Spd(O_{\breve E})$. 
 
 \begin{definition}\label{defintSht}
  Let $(G, b, \mu)$ be a local Shimura datum, {\cmag and let $\CG$ be a parahoric group scheme which is a model of $G$ over $\BZ_p$.}  The \emph{integral moduli 
  space of local shtuka} $\CM^{\rm int}_{\CG, b, \mu}$ is
the  functor  that sends $S\in {\rm Perfd}_k$ to the  set of isomorphism classes of tuples 
 \begin{equation}
 (S^\sharp,  \sP , \phi_{\sP }, i_r) ,
 \end{equation}
 where
 \begin{itemize}
 
\item[1)] $S^\sharp$ is an untilt of $S$ over $\Spa( O_{\breve E})$,
\item[2)] $(\sP , \phi_{\sP })$ is a $\CG$-shtuka over $S$ with one leg along $S^\sharp$ bounded 
by $\mu$,
\item[3)] $i_r$ is a framing, i.e. an isomorphism of $G$-torsors
\begin{equation}
i_r: G_{\CY_{[r,\infty)}(S)}\xrightarrow{\sim} \sP_{\, |\CY_{[r,\infty)}(S)}
\end{equation}
for large enough $r$ (for an implicit choice of pseudouniformizer $\varpi$), under which
$\phi_{\sP }$ is identified with $\phi_b=b\times {\rm Frob}_S$.
 \end{itemize}
 \end{definition}

  This definition makes sense even if $\mu$ is not minuscule. By loc. cit.,   $\CM^{\rm int}_{\CG, b, \mu}$ is a $v$-sheaf whose generic fiber 
 is {\cmag ${\rm Sht}_{G, b, \mu, K}$.}  In fact, by \cite[Prop. 2.23]{Gl21},  $\CM^{\rm int}_{\CG, b, \mu}$ is a small $v$-sheaf. 
 If $\mu$ is minuscule, we can think of the $v$-sheaf $\CM^{\rm int}_{\CG, b, \mu}$ as an integral model 
 of the rigid-analytic local Shimura variety $\CM_{G,b,\mu, K}$ and call it \emph{the integral local Shimura variety} (integral LSV). The Weil descent datum \eqref{Weildes} on $\CM_{G,b,\mu, K}$ extends to a Weil descent datum on $\CM^{\rm int}_{\CG, b, \mu}$ from $O_{\breve E}$ down to $O_E$.  Also note that if $b'=g^{-1} b \sigma(g)$, then  the association $(S^\sharp, \CP, \phi_\CP, i_r)\mapsto (S^\sharp, \CP, \phi_\CP, i_r\circ g)$ induces an isomorphism $\CM^{\rm int}_{\CG, b, \mu}\isoarrow\CM^{\rm int}_{\CG, b', \mu}$.

  The following conjecture is implicit in \cite[\S 25.1]{Schber}.
 \begin{conjecture}[Scholze]\label{repMint}
 Assume that $\mu$ is minuscule. There exists a normal formal scheme $\sM_{\CG,b, \mu}$, flat and  locally formally of finite type over $\Spf O_{\breve E}$, whose associated $v$-sheaf is equal to $\CM^{\rm int}_{\CG, b, \mu}$. (Note that by \cite[Prop. 18.4.1]{Schber} this formal scheme is unique if it exists.)
 \end{conjecture}

\begin{remark}
 The conjecture holds true in many cases when the data $(\CG, b, \mu)$ come from integral RZ data in the sense of \cite{R-Z}  
 (this excludes the cases of type (D) since they yield non-connected groups). In fact, in this case $\CM^{\rm int}_{\CG, b, \mu}$ is represented by the corresponding RZ formal scheme, cf.  \cite[Cor. 25.1.3]{Schber}. More precisely, one has to define the RZ formal scheme using the flat closure local model instead of the naive local model, cf. \cite[\S 21.6]{Schber}. This is conditional on showing that the flat closure local model is normal. Conjecture \ref{repMint} is proved in \cite{PRintlsv} when the local Shimura datum $(G, b, \mu)$ is of abelian type, if $p\neq 2$ and if $p=2$ and $G_\ad$ is a product of simple factors of type A or C.
 \end{remark}  
  
  Just as LSV, so also the formation of their integral models is functorial. More precisely, let $G\to G'$ be a group homomorphism compatible with local Shimura data
  $(G,b, \mu)\rightarrow (G', b', \mu')$. Then there is an inclusion of corresponding reflex fields, $E\supset E'$. Let $\CG$ and $\CG'$ be parahoric models of $G$, resp. $G'$ such that $G\rightarrow G'$ extends to $\CG\to \CG'$. Push-out of torsors under $\CG\to \CG'$ gives a $v$-sheaf morphism
  \begin{equation}\label{functmor}
  \rho: \CM^{\rm int}_{\CG,b, \mu}\to \CM^{\rm int}_{\CG', b', \mu'}\times_{\Spd( O_{\breve E'})}\Spd( O_{\breve E}),
  \end{equation}
compatible with Weil descent data.

\subsection{The reduced locus of integral LSV and specialization}\label{Zhu}

Recall that $k$ denotes the algebraic closure of the residue field $\kappa=\kappa_E$ of the reflex field $E$ of $(G, \mu)$.  
 
\begin{definition}\label{defadmloc} Let $X_{\CG}(b,\mu^{-1})$ be  the functor which to a perfect $k$-algebra $R$ associates the set of isomorphism classes
of pairs $(\CP, \alpha)$ where
\begin{itemize}
\item[1)] $\CP$  is a $\CG$-torsor over $\Spec(W(R))$, 
\item[2)]   $\alpha$ is a $\CG$-torsor isomorphism
$$
\alpha: \CG\times{\Spec(W(R)[1/p])}\xrightarrow{\sim} \CP[1/p]
$$ 
such that $\phi_\CP=\alpha\circ\phi_b\circ\phi^*(\alpha)^{-1}$ (where  $\phi_b=b\times {\rm Frob}$) 
 defines the structure of a meromorphic Frobenius crystal
  $$\phi_\CP: {\rm Frob}^*(\CP)[1/p]\xrightarrow{\sim} \CP[1/p].
  $$ It is required that the corresponding $\CG$-shtuka has leg along the  divisor $p=0$   with pole bounded by $\mu$. 
\end{itemize}
\end{definition}

\begin{remark}\label{rem332}
The above definition makes sense also when $\mu$ is not minuscule.
Explicitly, the term  ``pole bounded by $\mu$" comes down to the following condition.  Using the trivialization $\alpha$ we can view the pair
 $( \CP, \phi_\CP\circ \phi^*(\alpha))$ as a pair of a $\CG$-torsor over $\Spec(W(R))$ together with a trivialization of the restriction of this torsor to $\Spec(W(R)[1/p])$. By \cite[3.1]{ZhuAfGr} (see also \cite{BS}), this gives an $R$-valued point of the  Witt vector affine partial flag variety ${\rm Gr}^W_{\CG}$. Then the pole condition  means that this $R$-valued point  factors through the map $\iota\colon \BM^v_{\CG, \mu}(R^\sdiam)\to {\rm Gr}^W_{\CG}(R)$ of \eqref{LMWAff}. By work of Ansch\"utz-Gleason-Louren\c co-Richarz \cite[Thm 6.16]{AnRicLou} (see \S \ref{Conjadm}), this condition is equivalent to asking that for all $K$-valued points of $\Spec (R)$ with $K$ algebraically closed, the corresponding point
in ${\rm Gr}^W_{\CG}(K)=G(W(K)[1/p])/\CG(W(K))$ lies in 
$
{\rm Gr}^W_{\CG, {\rm Adm}(\mu)}(K)=\bigcup\nolimits_{w\in {\rm Adm}(\mu)_\CG}{\rm Gr}^W_{\CG, w}(K). 
$
\end{remark}

By  \cite[Thm 6.16]{AnRicLou}, the image 
$
(\BM^v_{\CG, \mu})(K^\sdiam)\to {\rm Gr}^W_{\CG}(K)
$
is equal to the set of points of some finite union of affine Schubert varieties in the ind-perfectly proper Witt affine Grassmannian ${\rm Gr}^W_{\CG}$. Using this, we can see as in  \cite[\S 1]{ZhuAfGr} that
the functor $X_{\CG}(b,\mu^{-1})$ is represented by a perfect $k$-scheme which is, in fact, locally perfectly
of finite type over $k$. We call $X_{\CG}(b,\mu^{-1})$ the \emph{$(b, \mu^{-1})$-admissible locus} inside the Witt affine Grassmannian. 
The group $J_b(\BQ_p)$ acts on $X_{\CG}(b,\mu^{-1})$ by
\[
g\cdot (\CP, \alpha)=(\CP, \alpha\circ g^{-1}).
\]

 \subsubsection{Specialization for $v$-sheaves}\label{par331}
 We now recall some constructions and results of Gleason \cite{Gl}, \cite{Gl21}.  
 Denote by $\widetilde {\rm Perfd}$ the category of small $v$-sheaves on ${\rm Perfd}_k$.
 Recall the functor $S\mapsto S^\sdiam$ from perfect affine $k$-schemes to $\widetilde {\rm Perfd}$.
 If $\CF$ is a small $v$-sheaf on ${\rm Perfd}_k$, then its reduction $\CF_{\red}\in \wt{\rm SchPerf}$
 is the small \emph{scheme-theoretic $v$-sheaf}  (see \cite[Rem.  3.13]{Gl}) with
 \[
 \CF_{\rm red}(S)={\rm Hom}_{\widetilde{\rm Perfd}}(S^\sdiam, \CF),
 \]
for $S$ a perfect $k$-scheme.
 It makes sense to consider the corresponding $v$-sheaf  $(\CF_{\red})^\sdiam$ and
 there is a natural adjunction morphism of $v$-sheaves
 \begin{equation}\label{adjred}
 (\CF_{\red})^\sdiam\to \CF.
 \end{equation}
 Gleason explains certain conditions on $\CF$ that allow for the construction of
a continuous \emph{specialization map}
\begin{equation}
{\rm sp}_\CF\colon |\CF|\to |\CF_{\rm red}|.
\end{equation}
Here, $|\CF|$ and $|\CF_{\rm red}|$ are the topological spaces associated to these sheaves as in
 \cite{Sch-Diam}, comp. \cite[\S 1]{Gl}, i.e., the equivalence classes of maps $\Spa(K, K^+)\to \CF$, where $K$ is a perfectoid field of characteristic $p$ and $K^+$ a bounded open valuation subring. If $\CF$ comes with a map of sheaves $\CF\to \Spd(\BZ_p)$, we set
 $\CF_{\eta}=\CF\times_{\Spd(\BZ_p)}\Spd(\BQ_p)$. Then we can consider the
 composition of the natural map $|\CF_{\eta}|\to |\CF|$ with the specialization map ${\rm sp}_\CF$,
 \begin{equation}
 {\rm sp}_{\CF_\eta}\colon |\CF_{\eta}|\to |\CF_{\rm red}|.
 \end{equation}
Assume in addition that $\CF_{\rm red}$ is represented by a scheme and that
the natural adjunction map \eqref{adjred}  is a closed immersion.
 In this situation, for a closed point $x\in \CF_{\rm red}$, Gleason \cite[Def. 4.18]{Gl} defines  a sub-$v$-sheaf $\widehat\CF_{/x}$ of $\CF$
that  he calls the \emph{tubular neighborhood}  of $x$ in $\CF$.  We will use the more traditional name \emph{formal completion}.  Namely,
\begin{equation}\label{functcompl}
\widehat\CF_{/x}(S)=\{ y\colon S\to \CF\mid {\rm sp}_\CF\circ y(|S|)\subset \{x\}\}. 
\end{equation}
One can then take its generic fiber  (which is traditionally called the \emph{tube} over $x$)
 \[
 (\widehat\CF_{/x})_\eta:=\widehat\CF_{/x}\times_{\Spd(\BZ_p)}\Spd(\BQ_p).
 \] 
 
Let us exemplify these definitions when $\CF=\fkX^\diam$, where $\fkX$ is a formal scheme which is flat, separated and formally locally of
 finite type over ${\rm Spf}(\breve \BZ_p)$. Then $\CF_{\red}$ is represented by the perfection of the underlying
 reduced $k$-scheme $\fkX_{\rm red}$ of $\fkX$. In this case, there is a specialization map
 which is a continuous map of topological spaces
 \[
 {\rm sp}: |\fkX_{\eta}|\to |\fkX_{\rm red}|,
 \]
where $\fkX_{\eta}$ is the generic fiber of $\fkX$ considered as an adic space over $\breve\BQ_p$.
There is also a surjective map
\[
 {\rm sp}: |\fkX^{\rm rig}|^{\rm class}\to \fkX_{\rm red}(k) ,
\]
where $|\fkX^{\rm rig}|^{\rm class}$ denotes the classical points of the rigid space $\fkX^{\rm rig}$
given by the generic fiber of $\fkX$, defined as by Berthelot, cf. \cite[chap. 5]{R-Z}.
For $\CF=\fkX^\diam$, with $\fkX$ a formal scheme as above, there are natural identifications
 \[
 |\fkX_\eta|=|\CF_\eta|,\quad |\CF_{\rm red}|=|\fkX_{\rm red}|,
 \]
 and, under these, ${\rm sp}_{\CF_\eta}$ agrees with
 the specialization map ${\rm sp}: |\fkX_{\eta}|\to |\fkX_{\rm red}|$ (see \cite{Gl}, \S   4).
 For a closed point $x\in \fkX_{\rm red}$, by \cite[Prop.  4.19]{Gl}, we can also identify formal completions
 \[
 \widehat\CF_{/x}=(\widehat{\fkX}_{/x})^\diam=\Spa(\widehat\CO_{\fkX, x}, \widehat\CO_{\fkX, x})^\diam.
 \]

  \subsubsection{Specialization for integral models of LSV}One of the main results of \cite{Gl21}, is that, under certain assumptions, $\CF=\CM^{\rm int}_{\CG, b, \mu}$ affords a specialization map to $X_\CG(b, \mu^{-1})$. 
  
  \begin{theorem}\label{Gleason}(Gleason \cite{Gl21})

a) The $v$-sheaf ${\CM}^{\rm int}_{\CG, b, \mu}$ is small and its reduced locus   $({\CM}^{\rm int}_{\CG, b, \mu})_{\rm red}$ 
  is represented by the perfect $k$-scheme $X_\CG(b, \mu^{-1})$. The identification with the $(b, \mu^{-1})$-admissible locus, 
  \[
  ({\CM}^{\rm int}_{\CG, b, \mu})_{\rm red}\xrightarrow{\ \sim\ } X_{\CG}(b,\mu^{-1})
  \]
   is functorial in $(\CG, b, \mu)$.
   
   b) The adjunction morphism 
   \[
  ({\CM}^{\rm int}_{\CG, b, \mu})_{\rm red}^ \sdiam\simeq  X_{\CG}(b,\mu^{-1})^\sdiam\to {\CM}^{\rm int}_{\CG, b, \mu}
   \]
  is a closed immersion. 
  
  c) The $v$-sheaf ${\CM}^{\rm int}_{\CG, b, \mu}$ is ``specializing" in the terminology of loc. cit., \S 1.4. Hence,
  there is a continuous specialization map
  \[
  {\rm sp}: |{\rm Sht}_{G, b, \mu, K}|\to |X_{\CG}(b,\mu^{-1})|
  \]
  functorial in $(\CG, b, \mu)$, which also defines
  \[
  {\rm sp}: |{\rm Sht}_{G, b, \mu, K}|^{\rm class}\to X_{\CG}(b, \mu^{-1})(k).
  \]
  \end{theorem}
  
\begin{proof} {\cmag This follows from \cite[Thm. 2]{Gl21} and its proof, see also \cite[Prop. 2.30, Lem. 2.31]{Gl21}.}
\quash{In the case that $\CG$ is reductive, (a) and (b) follow from\mar{\cred read this to update Gleason refs}
 \cite[Prop. 2.3.5, Prop. 2.3.10, Lem. 2.3.11]{Gl}. More precisely, the cited proofs use the reductivity of $\CG$ only through  \emph{loc. cit.}, Thm. 2.1.18. But this last theorem extends to the case of arbitrary parahoric group schemes by Ansch\"utz \cite{An}, comp. section \ref{s:parext}.  For part (c), the specializing property is proved in  \emph{loc. cit.}, Prop. 2.3.10 (again, the reductivity of $\CG$ only enters through Thm. 2.1.18). That this implies the facts on the specialization map is  \cite[Prop. 1.4.17]{Gl}. }
\end{proof}
  Theorem \ref{Gleason} above holds for all $\mu${\cmag, not necessarily minuscule}. However, in the sequel, we will return to our blanket assumption that $\mu$ is minuscule. The following conjecture is the local analogue of Scholze's Conjecture \ref{repMint}.
 \begin{conjecture}\label{conjtubeMint}
 Assume that $\mu$ is minuscule.   Let $x\in X_{\CG}(b, \mu^{-1})(k)$. Then the formal completion  $\wh{{\CM}^{\rm int}_{\CG, b, \mu}}_{/x}$ is representable, i.e., there exists a normal complete Noetherian local ring $R$ such that 
   $$\wh{{\CM}^{\rm int}_{\CG, b, \mu}}_{/x}\simeq \Spd(R).
   $$
 \end{conjecture}
  It is clear that Conjecture \ref{conjtubeMint} follows from Conjecture \ref{repMint}: indeed, if ${\CM}^{\rm int}_{\CG, b, \mu}$ is representable by ${\sM}_{\CG, b, \mu}$, then $\wh{{\CM}^{\rm int}_{\CG, b, \mu}}_{/x}\simeq \Spd(\wh{\CO}_{{\sM}_{\CG, b, \mu}, x})$. In \S \ref{ss:repintLSV} we prove a kind of converse under some conditions. 
  
  The following conjecture makes Conjecture \ref{conjtubeMint} more precise  (see also \cite[Conj. 1]{Gl21}).
  \begin{conjecture}\label{conjtubeLM}
    Let $x\in X_{\CG}(b, \mu^{-1})(k)$ and let $y\in \BM^{\rm loc}_{\CG, \mu}(k)={\BM}^v_{\CG, \mu}(k)\subset {\rm Gr}^W_{\CG}(k)$ be the point obtained after 
  fixing a trivialization of the corresponding $\CG$-torsor over $W(k)$. 
Then there is an
  isomorphism
  \begin{equation}\label{GleasonCan}
  \widehat{{\CM}^{\rm int}_{\CG, b, \mu}}_{/x}\simeq \widehat{{\BM}^\loc_{\CG, \mu}}_{/y}=\Spd(\widehat{\CO}_{\BM, y}) ,
    \end{equation}
    where, for simplicity of notation, $\BM:=\BM^{\rm loc}_{\CG, \mu}$. 
      \end{conjecture}

    \begin{remark}
a) In the case of RZ-spaces, Conjecture \ref{conjtubeLM} holds true, as follows from the ``classical'' local model diagram, cf. \cite{R-Z}.

 b)  The isomorphism in Conjecture \ref{conjtubeLM} cannot be expected to be canonical. In fact, the isomorphism  
between the completions of local models which are induced using (\ref{GleasonCan}) and the functoriality of integral LSV's is not always the one obtained by the functoriality of local models.
To explain this statement, consider $(\CG, b, \mu)\to (\CG', b', \mu')$. This induces natural morphisms
    \begin{equation}\label{Natfun}
  {\BM}^v_{\CG, \mu} \to {\BM}^v_{\CG',  \mu'}\times_{\Spd(O_{E'})}\Spd( O_E),  \qquad  {\CM}^{\rm int}_{\CG, b, \mu}\to  {\CM}^{\rm int}_{\CG', b', \mu'}\times_{\Spd( O_{\breve E'})}\Spd( O_{\breve E}) .
    \end{equation}
    Let $x\in X_\CG(b, \mu^{-1})(k)$ and let $y\in {\BM}^v_{\CG, \mu}(k)\subset {\rm Gr}^W_{\CG}(k)$ be the point obtained after 
  fixing a trivialization of the corresponding $\CG$-torsor over $W(k)$. Let $x'\in X_{\CG', \mu'}(b')(k)$ be the image of $x$ and let $y'\in {\BM}^v_{\CG', \mu'}(k)\subset {\rm Gr}^W_{\CG'}(k)$ be the point obtained from the corresponding  trivialization of the corresponding $\CG'$-torsor over $W(k)$. Then, assuming the conjecture, we obtain a diagram 
   \begin{equation*} \xymatrix{
   \widehat {\CM_{\CG, b, \mu }^{\rm int}}_{/x}\ar[d]^{\simeq}\ar[r]  &\ar^{\simeq}[d]  
   \widehat {\CM_{\CG', b', \mu'}^{\rm int}}_{/x'}\times_{\Spd( O_{\breve E'})}\Spd( O_{\breve E})\\
      \widehat {\BM^v_{\CG, \mu }}_{/y} \ar[r]& \wh{\BM^v_{\CG', \mu'}}_{/y'}\times_{\Spd( O_{\breve E'})}\Spd( O_{\breve E}),}
  \end{equation*}
where the upper and the lower morphism are derived from \eqref{Natfun}, and where the vertical maps are the isomorphisms appearing in \eqref{GleasonCan}. This diagram will in general not commute, no matter how the isomorphisms in \eqref{GleasonCan} are chosen. An example is given by $\GL_2$, when $\mu=(1, 0)$ and $b$ is basic, and $\CG$ is the Iwahori model of $\GL_2$ and $\CG'$ is the hyperspecial model of $\GL_2$. In this case, the upper horizontal arrow is represented by a finite morphism but the lower horizontal morphism is induced by a blow-up morphism. 
    \end{remark}
  \begin{remark}\label{ideaSpecialization}
  The basic idea behind the construction of the specialization map for ${\rm Sht}_{G, b, \mu, K}$
  can be displayed as follows, see also Proposition \ref{AnExtension} and its proof. 
  
  Let
  $S=\Spa(C, C^+)\to  {\rm Sht}_{G, b, \mu, K}$ be a point over $\Spd(O_E)$ which corresponds to a $\CG$-shtuka $(\sP, \phi_\sP)$
  with trivialization $i_r$ over $\CY_{[r, \infty)}(S)$. Using the trivialization $i_r$ we can extend $(\sP, \phi_\sP)$ over $\CY_{[r, \infty]}(S)$
  and hence, by glueing, over $\CY_{[0, \infty]}(S)$. By the Extension Conjecture \ref{paraextconj} shown by Ansch\"utz \cite{An}, 
  this extends to a $\CG$-BKF-module $\sP^\natural$, i.e., a module
  over $\Spec(W(C^+))$.   The trivialization extends to a trivialization $\tilde i_r$ of the pullback of $\sP^\natural$ via 
  $\CY_{[r, \infty]}(S)\to \Spec(W(C^+))$. Let $\kappa$ be the residue field of $C^+$. Now the base change by $W(C^+)\to W(\kappa)$ gives a $\CG$-torsor 
  \[
  \sP_0=\sP^\natural\otimes_{W(C^+)}W(\kappa)
  \] 
  over $\Spec(W(\kappa))$ with Frobenius $\phi_0$ defined on $\sP_0[1/p]$. Base-changing the trivialization $\tilde i_r$ 
  under
  \[
  B^{[r,\infty]}_{(C, C^+)}\to W(\kappa)[1/p]
  \]
  defines a trivialization $\alpha$ of $\sP_0[1/p]$.
  We can see that the Frobenius on $\sP_0[1/p]$ is bounded by $\mu$. Hence $(\sP_0, \phi_0, \alpha)\in X_\CG(b, \mu^{-1})(\kappa)$.  
  {\cmag Unravelling the definition of the specialization map (\cite[Def. 4.12, Rem. 4.13]{Gl}) gives} 
  \[
  {\rm sp}(\sP, \phi_\sP,  i_r)=(\sP_0, \phi_0, \alpha).
  \]
  \end{remark}
  \subsection{Structure of the formal completions of integral LSV}\label{ss:structure}
 The aim of this subsection is to prove the following fact.
 \begin{proposition}\label{tubestruc}
Let $(G, b, \mu)$ be a local Shimura datum and let $\CG$ be a parahoric group scheme for $G$. Let $x\in X_{\CG}(b, \mu^{-1})(k)$.

1) The  formal completion  $\wh{ \CM^{\rm int}_{\CG, b, \mu}}_{/x}$ of the integral local Shimura
variety $ \CM^{\rm int}_{\CG, b, \mu}$ at $x$ is  topologically flat.

2) The topological space $|(\wh{{\mathcal M}^{\rm int}_{\CG, b, \mu}}_{/x})_\eta|$ given by the generic fiber  $(\wh{{\mathcal M}^{\rm int}_{\CG, b, \mu}}_{/x})_\eta$ of the formal completion is connected.
 \end{proposition}
 
  For the proof we need some preparation. For a small $v$-sheaf $X$ over $\Spd(\BZ_p)$ we write $X_\eta=X\times_{\Spd(\BZ_p)}\Spd(\BQ_p)$
for its ``generic fiber". This is a small $v$-sheaf which comes with a map $X_\eta\to X$. We will say that $X$
is \emph{topologically flat}, if $|X_\eta|$ is dense in $|X|$. We need the following lemma.
{\cmag
\begin{lemma} \label{topflat} 
Suppose $\fkX$ is a formal scheme which is flat and formally of finite type over $\Spf(\BZ_p)$. Then the corresponding $v$-sheaf $\fkX^\diam$ over $\Spd(\BZ_p)$ is 
topologically flat. \qed 
\end{lemma}

\begin{proof}
This is given in \cite[Lem. 4.4]{LourencoThesis}, a similar argument also appears in the proof of \cite[Lemma 2.17]{AnRicLou}.
\end{proof}
}
In \cite[Def. 2.34]{Gl}, Gleason defines a $v$-sheaf of groups $\LG$ over $\Spd(\BZ_p)$
given   by 
\begin{equation}
\LG(S)=\{((S^\sharp, y), g)\ |\ g\in \CG(W(R^+)),\ g\equiv 1\, {\mathrm {mod}}\, [\varpi_g] \},
\end{equation}
where $S=\Spa(R, R^+)$ and $(S^\sharp, y)$ is an $S$-valued point of $\Spd(\BZ_p)$, and where $\varpi_g$ is a pseudouniformizer of $R^+$ (that depends on $g$).  We can think of $\LG$ as the formal completion at  the identity of the $v$-sheaf of groups 
$\BW^+\CG\times \Spd(\BZ_p)$ of loc. cit., with $\BW^+\CG$ given by 
\[
\BW^+\CG(\Spa(R, R^+))=\CG(W(R^+)).
\]
 Consider the $v$-sheaf $\wh \BW^+$ which associates to the affinoid perfectoid $S=\Spa(R, R^+)$ over $k$
the set 
\begin{equation}
\wh \BW^+(R, R^+)=\{a\in W(R^+) \ |\ a\in ([\pi_a]), \ \pi_a \hbox{\rm\ a pseudo-uniformizer depending on $a$}\}.
\end{equation}

\begin{lemma}\label{Wideal}
The subset $\wh \BW^+(R, R^+)$ is an ideal of the ring $W(R^+)$.
\end{lemma}

\begin{proof}
If $\pi_1$, $\pi_2$ are two pseudo-uniformizers in $R^+$, 
we have $\pi_1^{1/p^n}|\pi_2$ and $\pi^{1/p^n}_2|\pi_1$ for some $n$.
Then, $[\pi_1]b_1+[\pi_2]b_2=[\pi_1]b_1+[\pi^{1/p^n}_1 c]b_2=[\pi_1^{1/p^n}]([\pi^{1-1/p^n}_1]b_1+[c]b_2)$.
This shows that $\wh \BW^+(R, R^+)$ is a subgroup and, hence, also an ideal of $W(R^+)$.
\end{proof}
Recall the identity of Witt vectors (in characteristic $p$)
\[
(b_0, b_1, \ldots , b_n, \ldots )\cdot (\pi_a, 0, \ldots , 0,  \ldots  )=(b_0\pi_a, b_1\pi^p_a, b_2 \pi^{p^2}_a, \ldots ).
\]
Now choose a pseudo-uniformizer $\pi$ and hence a norm $|\ |$ on $R$.

 So $a=(a_0, a_1, \ldots , a_n, \ldots )\in W(R^+)$ belongs to $\wh \BW^+(R, R^+)$, if and only if
\[
 a_{n-1} \in \BB_{r^{p^n}, S}(R, R^+)  
\]
with $r=r_a=|\pi_a|$. Here, we have the ball of radius $r^{p^n}$ over $S$
\[
\BB_{r^{p^n}, S}=\Spa(R\langle t , t/\pi^{p^n}_a \rangle, R^+\langle t , t/\pi^{p^n}_a\rangle)^\diam,
\]
so for $(R, R^+)\to (A, A^+)$, $\BB_{r^{p^n}, S}(A, A^+)= \pi^{p^n}_a A^+$. 
The Frobenius on $R$ extends to a ring isomorphism $\phi: R\langle t , t/\pi_a \rangle\xrightarrow{\sim} R\langle t, t/\pi^p_a\rangle$
with $\phi(t)=t$. This gives
\[
\phi: \BB_{r, S}\xrightarrow{\sim} \BB_{r^p, S}.
\]
Hence, we see
\[
\wh \BW^+(R, R^+)=\bigcup_{r<1} (\BB_{r, S}\times \BB_{r^p, S}\times \BB_{r^{p^2}, S}\times\cdots )(R, R^+)\xrightarrow{\simeq } (\bigcup_{r<1}(\BB_{r, S})^{\BN})(R, R^+),
\]
the isomorphism given by $(1,\phi^{-1},\phi^{-2},\ldots )$. 
This isomorphism  gives
\begin{equation}\label{isoball}
\wh \BW^+\times S\simeq \bigcup_{r<1}(\BB_{r, S})^{\BN}.
\end{equation}
\begin{proposition}\label{Wfs}
The map $\wh \BW^+\to *$ is formally smooth in the sense of \cite[Def. IV.3.1]{FS}. 
\end{proposition}  
\begin{proof} The proof of  \cite[Prop. IV.3.3]{FS}, gives that
$\BB^{\BN}\to *$ is formally smooth and so is
\[
\bigcup_{r<1}(\BB_{r, S})^{\BN}\to S
\]
for all $S$. The result then follows from (\ref{isoball}).
\end{proof}
Note here that 
  \[
  \BB^{\BN}\simeq \varprojlim_{d}\BB^d,
  \]
where $\BB^d$ is the $d$-dimensional ball. 
If $U\subset \BB^{\BN}_S$ is a quasi-compact
open then, using \cite[Prop. 6.4]{Sch-Diam}, one sees that there is some $d\geq 1$ such that $U$ is
the inverse image of an open $V\subset \BB^d_S$ by the projection $\BB^{\BN}_S\to \BB^d_S$.

We write
\[
\CG(\wh \BW^+(R, R^+)):=\{    g\in \CG(W(R^+)),\ g\equiv 1\,  {\mathrm {mod}}\, [\varpi_g] \}.
\]
Set $\CO(\CG)=\Gamma(\CG, O_{\CG})$ and let $\wh{\CO(\CG)}_{e}$ be its completion 
along the kernel $I_\CG$ of the identity section. Since $\CG$ is smooth over $\Spec(\BZ_p)$, we have
\[
\wh{\CO(\CG)}_{e}\simeq \BZ_p\lps u_1,\ldots , u_d\rps,
\]
non-canonically, with $\wh I_\CG\simeq (u_1,\ldots , u_d)$.
An element $g\in \CG(\wh \BW^+(R, R^+))$ 
is given by an algebra map
\[
g^*: \CO(\CG)\to W(R^+)
\]
such that $g^*(I_\CG)\subset [\pi_g] W(R^+)\subset \wh \BW^+(R, R^+)$. Since $W(R^+)$ is $[\pi_g]$-complete, 
$g^*$ factors uniquely through ${\CO(\CG)}\to \wh{\CO(\CG)}_{e}=\BZ_p\oplus \wh I_\CG$. Using that
$\wh \BW^+(R, R^+)\subset W(R^+)$ is an ideal (Lemma \ref{Wideal}), we see that $g^*$ is uniquely given
by assigning the values $g^*(u_i)$, $i=1,\ldots, d$.
This gives
\[
\LG(R, R^+)\simeq \wh \BW^+(R, R^+)^d\times \Spd(\BZ_p)(R, R^+),
\]
and so
\begin{equation}\label{present}
\LG \simeq (\wh \BW^+) ^d\times \Spd(\BZ_p), 
\end{equation}
non-canonically (as $v$-sheaves of sets). 
 
\begin{proposition}\label{fsmooth}
The map $\LG\to \Spd(\BZ_p)$ is formally smooth in the sense of \cite[IV.3]{FS}. 
\end{proposition}

\begin{proof}
This 
follows from (\ref{present}) and Proposition \ref{Wfs}. 
\end{proof}

\begin{corollary}\label{uopen}
The map $\LG\to \Spd(\BZ_p)$ is universally open and topologically flat. 
\end{corollary}

\begin{proof}
Proposition \ref{fsmooth} together with \cite[Prop. IV.3.2, p. 129]{FS} implies that $\LG\to \Spd(\BZ_p)$ is universally open. Hence, an open $U\subset |\LG|$ maps to an open in $|\Spd(\BZ_p)|$. Since $\Spd(\BZ_p)$ is topologically flat (Lemma \ref{topflat} above), the image of $U$ intersects the generic fiber $|\Spd(\BZ_p)_\eta|=|\Spd(\BQ_p,\BZ_p)|$. Now topological flatness also follows: indeed,
\[
|U\times_{\LG}(\LG)_\eta|=|U\times_{\Spd(\BZ_p)}\Spd(\BQ_p)|\to  |U|\times_{|\Spd(\BZ_p)|}|\Spd(\BQ_p)|
\]
is a surjection by \cite[Prop. 12.10]{Sch-Diam} and, by the above, the target is non-empty. Hence the source is non-empty. This source maps
\[
|U\times_{\LG}(\LG)_\eta|\to |U|\times_{|\LG|}|(\LG)_\eta|
\]
and so $|U|\times_{|\LG|}|(\LG)_\eta|=|U|\cap |(\LG)_\eta|\neq\emptyset$.
\end{proof}

We will also need:

\begin{proposition}\label{LGconnect}
For any affinoid perfectoid field $(K, K^+)$ over $k$ with $\Spa(K, K^+)\to \Spd(\BZ_p)$, the topological space $| \LG\times_{\Spd(\BZ_p)}\Spa(K, K^+)|$
is connected.
\end{proposition}

\begin{proof}
It is enough to show this for $(K, K^+)=(C, C^+)$ algebraically closed over $k$.
By (\ref{present}) and (\ref{isoball}), 
\[
| \LG\times_{\Spd(\BZ_p)}\Spa(C, C^+)|\simeq |(\bigcup_{r<1}\BB^{\BN}_{r, \Spa(C, C^{+})})^d|.
\]

\begin{lemma}
Suppose that $(C, C^+)$ is an algebraically closed affinoid field over $k$. Then $|\BB^{\BN}_{(C, C^+)}|:=|\BB^{\BN}_{\Spa(C, C^+)}|$ is connected.
\end{lemma}
\begin{proof}
Note that
\[
\BB^{\BN}_{(C, C^+)}\simeq \Spa(C\langle t_1, t_2, \ldots \rangle, C^+\langle t_1, t_2, \ldots \rangle)^\diam
\]
 is quasi-compact. Suppose that $|\BB^{\BN}_{(C, C^+)}|=U\sqcup V$, with $U$ and $V$ both clopen subsets.
Since $|\BB^{\BN}_{(C, C^+)}|$ is quasi-compact, the closed subsets $U$ and $V$ are also quasi-compact. Hence, by
\cite[Prop. 6.4]{Sch-Diam}, they
are the inverse images of open subsets $U'$, $V'\subset \BB^d_{(C, C^+)}$ by the 
surjective projection
$pr_d: \BB^\BN_{(C, C^+)}\to \BB^d_{(C, C^+)}$, for some $d$.
We have $|\BB^d_{(C, C^+)}|=pr_d(U\sqcup V)=pr_d(U)\cup pr_d(V)=U'\cup V'$.
However, $|\BB^d_{(C, C^+)}|$ is connected, so $U'\cap V'\neq \emptyset$. This 
contradicts that $U$ and $V$ are disjoint.\end{proof}

Since $\BB^{\BN}_{r, \Spa(C, C^{+})}\simeq \BB^{\BN}_{\Spa(C, C^{+})}$ the lemma implies that 
$|\BB^{\BN}_{r, \Spa(C, C^{+})}|$ is also connected. The proof follows.
\end{proof}

Assume now that $b$ is $\mu$-admissible for $\CG$, i.e. $b\in \CG(\breve\BZ_p)w\CG(\breve\BZ_p)$, for some $w\in {\rm Adm}(\mu^{-1})_\CG$, cf. Remark \ref{remAdm423}. Then $X_\CG(b, \mu^{-1})(k)=\CM^{\rm int}_{\CG, b, \mu}(\Spd(k))\subset  G(W(k)[1/p])/\CG(W(k))$ has a ``base point" $x_0$
given by the image of the unit element.

 Gleason {\cred (\cite[Thm. 3]{Gl21})} constructs a diagram of $v$-sheaves over $\Spd( O_{\breve E})$
\begin{equation}\label{GDia}
\begin{gathered}
   \xymatrix{
	     &\LMint \ar[dl]_-{\text{$\pi_\bullet$}} \ar[dr]^-{\text{$\pi_\star$}}\\
	  \wh{ \CM^{\rm int}_{\CG, b, \mu}}_{/x_0}  & & \ \wh{\BM^v_{\CG, \mu}}_{/x_0} \ .
	}
\end{gathered}
\end{equation}

Let us explain this diagram.  For more details, the reader is referred to \cite[\S 3.6]{PRintlsv}, {\cred \cite[\S 2.4]{Gl21}.}
\smallskip

a) The $v$-sheaf $\LMint$ assigns to an affinoid perfectoid $S=\Spa(R, R^+)$ over $k$, the set
\[
\LMint(S)=\{((S^\sharp, y), h)\ |\ h\in \CG(W(R^+)[1/\xi_{R^\sharp}]),\ h
\equiv b\, {\mathrm {mod}}\, [\varpi_h], \ [h^{-1}]\in \BM^v_{\CG, \mu}(S)\},
\]
where $(S^\sharp=\Spa(R^{\sharp}, R^{\sharp +}), y)$ is an untilt of $S$ over $ O_{\breve E}$ and where $[h^{-1}]$ is the $S$-point of the ${\BB}_{\rm dR}$-affine Grassmannian ${\rm Gr}_{\CG, \Spd( O_{\breve E})}$ defined by the coset $h^{-1} \CG(\BB^+_{\rm dR}(R^\sharp))$.

\smallskip

b) The map $\pi_*$ is projection to the coset $[h^{-1}]=h^{-1} \CG(\BB^+_{\rm dR}(R^\sharp))$.
\smallskip

c) The map $\pi_\bullet$ sends $((S^\sharp, y), h)$ to the pair $((\sP,\phi_\sP), i_r)\in \CM^{\rm int}_{\CG, b, \mu}(S)$, where the $\CG$-shtuka over $S$ with leg at $y$ is given by the trivial $\CG$-torsor $\sP=\CG$ with
 Frobenius defined by $\phi_\sP=h\phi:  (\phi^*\CG)[1/\xi_{R^\sharp}]\xrightarrow{\sim} \CG[1/\xi_{R^\sharp}]$, and the framing $i_r$ is given by the unique lift (as in \cite[Lem. 2.1.28]{Gl}) of the identity trivialization modulo $[\varpi_h]$.

 Hence, for $((S^\sharp, y), h)\in \LMint(S)$ as above, the framing in (c) is given by the unique element $i(h)=i_r(h)\in G(B^{[r,\infty)}_{(R, R^+)})$ with $i(h)
\equiv 1\, {\mathrm {mod}}\, [\varpi_h]$ and the property 
 \[
 h=i(h)^{-1}\cdot b\cdot \phi(i(h)). 
 \]

d) The point $x_0\in \CM^{\rm int}_{\CG, b, \mu}(\Spd(k)) $ is the base point given as above. Similarly  $x_0\in \BM^v_{\CG,  \mu}$  denotes the base point of $\BM^v_{\CG, \mu}(\Spd(k))\subset {\rm Gr}^W_{\CG}(k)$ given by the coset   $b^{-1}\CG(W(k))$.

The $v$-sheaf of groups $\LG$ acts on $\wh{ L{\mathcal {M}}^{\rm int}_{\CG, b, \mu}}_{/x_0}$ on the right by $h\star g=g^{-1}h$ and also by $h\bullet g=\phi(g)^{-1}h   g $.  

By {\cred\cite[Lem. 2.35, Lem. 2.36, proof of Thm. 2.33]{Gl21}}, both $\pi_{\bullet}$, $\pi_\star$ are $\LG$-torsors  (for the $v$-topology)
for the two corresponding actions. 
\smallskip

{\it Proof of Proposition \ref{tubestruc}:}\label{proofofprop}   We first reduce the statement to the case where $x$ is the base point. 
Let $x\in {{\CM}^{\rm int}_{\CG, b, \mu}}(k)$ given by $(\CP, \phi_\CP, i_r)$, where $\CP$ is a $\CG$-bundle on $W(k)$. Choose a trivialization of $\CP$. 
Then $\phi_\CP$ is given as
\begin{equation}\label{bxbase}
\phi_{b_x}\colon \CG[1/p]\to \CG[1/p] ,
\end{equation}
which is bounded by $\mu$. Furthermore, $i_r$ is given by $g\colon \CG[1/p]\to \CG[1/p]$ with $g^{-1} b_x \sigma(g)=b$. We obtain an isomorphism 
\begin{equation}\label{translatesigma}
\tau_g: {{\CM}^{\rm int}_{\CG, b, \mu}}\xrightarrow{\sim} {{\CM}^{\rm int}_{\CG, b_x, \mu}}, \quad 
\tau_g((\sP, \phi_\sP), i_r)=((\sP, \phi_\sP), i_r\cdot g^{-1}),
\end{equation}
which sends $x$ to the base point $x_0$ of ${{\CM}^{\rm int}_{\CG, b_x, \mu}}$. Then $ \wh{ \CM^{\rm int}_{\CG, b, \mu}}_{/x}$ is isomorphic to $ \wh{ \CM^{\rm int}_{\CG, b_x, \mu}}_{/x_0}$. It is thus enough to consider the case of the base point in Proposition \ref{tubestruc}. 

 Recall that, by Theorem \ref{LMconj}, $\BM^v_{\CG, \mu}$ is representable by $\BM^{\rm loc}_{\CG, \mu}$. Hence,   
 \[
 \wh{\BM^v_{\CG, \mu}}_{/x_0}=\Spd(A, A),
 \]
 where $A$ is a complete normal local Noetherian flat $\CO_{\breve E}$-algebra. By Lemma \ref{topflat},
 $\wh{\BM^v_{\CG, \mu}}_{/x_0}$ is topologically flat; 
 similarly for $\BM^v_{\CG, \mu}$.  Proposition \ref{tubestruc} follows from the next proposition.

\begin{proposition} \label{topplatt}
Let $(G, b, \mu)$ be a local Shimura datum and let $\CG$ be a parahoric group scheme for $G$.  Assume $b$ is $\mu$-admissible for $\CG$ and let $x_0\in X_{\CG}(b, \mu^{-1})(k)$ be the base point.

1) The $v$-sheaf $\LMint$ is topologically flat. 

2) The  formal completion  $\wh{ \CM^{\rm int}_{\CG, b, \mu}}_{/x_0}$ of the integral local Shimura
variety $ \CM^{\rm int}_{\CG, b, \mu}$ at $x_0$ is  topologically flat.

3) The topological space $|(\wh{{\mathcal M}^{\rm int}_{\CG, b, \mu}}_{/x_0})_\eta|$ underlying the generic fiber  $(\wh{{\mathcal M}^{\rm int}_{\CG, b, \mu}}_{/x_0})_\eta$ of the formal completion is connected.
\end{proposition}

\begin{proof} We recall the diagram (\ref{GDia}). 
Note that $\pi_\bullet$ induces a continuous map $|\pi_\bullet|$ between the corresponding
topological spaces. Since $\pi_\bullet$   splits $v$-locally, $|\pi_\bullet|$ is surjective.
 For simplicity, we will drop the subscripts $\CG$, $\mu$, $b$
from the notation.

We first show (1). For simplicity set $X=\LMs$, $Y=\wh{\BM^v_{/x_0}}=\Spd(A)$ and set
\[
\pi=\pi_* : X\to Y
\] 
for the $\LG$-torsor. Let $U\subset |X|$ be a non-empty open subset; we identify it in notation with the corresponding open $v$-subsheaf of $X$
(\cite[Prop. 12.9]{Sch-Diam}). We would like to show
$U\cap |X_\eta|\neq\emptyset$. It is enough to show that $|\pi|(U) $ is  
   open in $|Y |=|\Spd(A) |$ . Then by Lemma \ref{topflat}, 
$|\pi|(U)$ intersects $|Y_\eta|$. Since $|X_\eta|=|\pi|^{-1}(|Y_\eta|)$, then $U$ intersects $|X_\eta|$. 

{\cmag Since $\pi$ is a $\LG$-torsor for the $v$-topology (cf. \cite[Lem. 2.35]{Gl21}), there is a $v$-cover by a perfectoid space $T$,
\[
q: T\to \Spd(A)=Y ,
\]
 such that the 
base change of $\pi$ by $q$ splits,
\[
\LG\times_{\Spd(\BZ_p)}T\simeq X\times_{Y}T.
\]}
By Corollary \ref{uopen}, the image $\pi(|U\times_{Y}T|)$ is open in $|T|$.
By \cite[Prop. 12.10]{Sch-Diam} the canonical map
\[
|U\times_Y T|\to |U|\times_{|Y|}|T|
\]
is surjective. In the fibered product of sets
\begin{equation}\label{d1}
 \begin{aligned}
   \xymatrix{
     |U|\times_{|Y|}|T|  \ar[r]^{\ \ |q|}  \ar[d]_{|\pi|}  &  |U|\ar[d]^{|\pi|}\\
    |q|^{-1}(|\pi|(|U|)) \ar[r]^{\ \ |q|} &  |\pi|(|U|),
        }
        \end{aligned}
    \end{equation}
    the vertical maps are surjective.  Hence,
    \[
   |\pi|:  |U\times_Y T|\to |q|^{-1}(|\pi|(|U|))
    \]
is surjective and so
\[
|q|^{-1}(|\pi|(|U|))=|\pi|(|U\times_{Y}T|)
\]
 Hence, 
$|q|^{-1}(|\pi|(|U|))$ is open in $T$. By \cite[Prop. 12.9]{Sch-Diam}, the $v$-cover $q$ gives a quotient map
$|q|: |T|\to |Y|$. It follows  
that $|\pi|(|U|) $ is open in $|Y|$.  

Part (2) follows quickly from (1) since $|\pi_\bullet|$ is continuous and surjective.
 
 Finally, we show (3), i.e. that $|(\wh{{\mathcal M}^{\rm int}_{/x_0}})_\eta|$ is connected.
By continuity and surjectivity of $|\pi_\bullet|$ in the diagram (\ref{GDia}),
it is enough to show that the source $|(\LMs)_\eta|$ is connected.
We will use the following standard lemma:
  
  \begin{lemma}\label{easy}
 Let $f: Z\to W$ be a continuous map of topological spaces which is surjective and open. Assume that $W$ is connected and that for each $w\in W$, the fiber $Z\times_W \{w\}=f^{-1}(w)\subset Z$ is connected with the subspace
  topology. Then $Z$ is connected. \qed  
  \end{lemma}

  We will apply Lemma \ref{easy} to
  \[
  |\pi_*|: |(\LMs)_\eta|\to |(\wh{\BM^v_{/x_0}})_\eta|.
  \]
  Note that $|\pi_*|$ is surjective and open by the argument in the proof of (1) above.

\begin{proposition}\label{Mconnect}
The topological space $|(\wh{\BM^v_{/x_0}})_\eta|$ is connected.
\end{proposition}

\begin{proof}
Since  $\BM:=\BM^{\rm loc}_{\CG, \mu}$ and its strict completion $A$ are normal, the  (Berthelot) rigid analytic fiber 
$(\wh{\BM_{/x_0}})_\eta^{\rm rig}={\rm Spf}(A)^{\rm rig}$ is connected, cf.   \cite[Lem. 7.3.5]{deJongCrys}. 
In fact, by \cite[Prop. 6.1.1]{deJongCrys}, this 
is path connected in the sense considered in loc. cit. It 
then follows that the corresponding Berkovich space and then also the corresponding analytic adic space, and the
topological space $|(\wh{\BM^v_{/x_0}})_\eta|$ for the corresponding  $v$-sheaf are connected. 
For this last step one can use a construction of \cite[13.7-13.12]{Sch-Diam}: By \cite[Prop. 13.10]{Sch-Diam}, there is a functor $X\mapsto |X|^B=$
``Berkovich topological space of $X$", defined for small $v$-sheaves
$X$. This extends Berkovich's construction for rigid analytic spaces.  There is a functorial continuous quotient
map $|X|\to |X|^B$, and each fiber of the map has a generic point, so it is connected. Hence, if $|X|^B$ is connected, so is $|X|$,
by a simple variation of Lemma \ref{easy}. (Here, again recall that the topological space  for a $v$-sheaf which corresponds to an analytic adic space  agrees with that space.)
\end{proof}

Next, we show that the fibers of $|\pi_*|$ are connected. Set again $X_\eta=(\LMs)_\eta$ and 
  $Y_\eta=(\wh{\BM^v_{/x_0}})_\eta$ and $\pi=\pi_*$. Suppose that $y\in |Y_\eta|$ is represented by
  $\Spa(K, K^+)\to Y_\eta$, with $(K, K^+)$ affinoid perfectoid over $k$. Then $y$ is the image of the unique closed point in $|\Spa(K, K^+)|$
  and $|\Spa(K, O_K)|\to |\Spa(K, K^+)|$ gives a generization.
  Choose $(C, C^+)$ algebraically closed
  with $(K, K^+)\to (C, C^+)$. 
  We have a continuous surjective map
  \[
 |X_\eta\times_{Y_\eta}\Spa(C, C^+)|\to  |X_\eta\times_{Y_\eta}\Spa(K, K^+)|\to 
 |X_\eta|\times_{|Y_\eta|}|\Spa(K, K^+)|.
  \]
 By the torsor property, $X_\eta\times_{Y_\eta}\Spa(C, C^+)\simeq \LG\times_{\Spd(\BZ_p)}\Spa(C, C^+)$.
 This is connected by Proposition \ref{LGconnect}. Hence, 
  the fiber $|\pi_*|^{-1}(y)$ over $y$ is also connected. 
  
  The proof  of (3) now follows from Lemma \ref{easy}.
  \end{proof}

 \begin{remark} Using some more of the results of \cite{AnRicLou} we can see   that 2) in Proposition \ref{topplatt} (the formal completion $\wh{ \CM^{\rm int}_{\CG, b, \mu}}_{/x_0}$ is topologically flat), also holds for general $\mu$, i.e. not necessarily minuscule.

Indeed, the argument in the proof above shows that we can deduce this by knowing that  the formal completion $\wh {\BM^v_{\CG, \mu}}_{/x_0}$ is topologically flat. By \cite[Prop. 4.13]{AnRicLou}, see also \cite[Cor. 4.14]{LourencoThesis}, the general $v$-sheaf local model $\BM^v_{\CG, \mu}$ is  topologically flat, without assuming representability, and  this holds even for general $\mu$ (not necessarily minuscule). By \cite[Prop. 4.14]{AnRicLou}, ${\BM^v_{\CG, \mu}} $ is a ``prekimberlite"
and $\{x_0\}$ is constructible, also for general $\mu$. Then, {\cred by \cite[Prop. 4.22]{Gl21}},
 \[
 \wh {\BM^v_{\CG, \mu}}_{/x_0}\to \BM^v_{\CG, \mu}
 \]
is an open immersion. Hence the topological flatness of $\BM^v_{\CG, \mu}$ implies the topological flatness of  $\wh {\BM^v_{\CG, \mu}}_{/x_0}$. 
 \end{remark}

  \subsection{Characterization of the formal scheme $\sM_{\CG, b, \mu}$.}
  
 In this subsection, we provide a more ``classical" characterization of the formal scheme $\sM_{\CG, b, \mu}$
 of Conjecture \ref{repMint}. In what follows we assume that $\sM_{\CG, b, \mu}$ as in Conjecture \ref{repMint} actually exists.
 
Our main point is that the specialization map {\cmag in (iii) of Gleason's Theorem \ref{Gleason} }
 \[
{\rm sp}: |{\rm Sht}_{G,  b, \mu, K}|^{\rm class}\to X_{\CG}(b, \mu^{-1})(k) 
 \]
can be interpreted using the theory of Breuil-Kisin modules 
as follows. 
Given a classical point $y\in {\rm Sht}_{G,  b, \mu, K}(F)$ with $F/\breve E$ finite, we consider
 the corresponding Galois representation
 \begin{equation*}
 \rho_y: {\rm Gal}(\bar E/F)\to \CG(\BZ_p)\subset G(\BQ_p)
 \end{equation*}
obtained by evaluating the local system given by the period morphism over the point $\pi_{GM}(y)$.
This representation is crystalline: indeed, the point $y$ directly provides the corresponding admissible filtered Frobenius $G$-isocrystal $D_{\rm crys}(\rho_y)$ which is associated to $\rho_y$ by Fontaine's functor, cf. \cite[\S 1.6]{R-Z}. In fact,
the point $y$ also gives an isomorphism of the underlying Frobenius $G$-isocrystal over $k$ with the 
Frobenius $G$-isocrystal given by $b$.

Fix  a uniformizer  $\pi=\pi_F$  of $F$. Then by Breuil-Kisin theory (comp. \cite[Thm. 3.3.2]{KP}), to any Galois stable lattice in a crystalline  representation of  $ {\rm Gal}(\bar E/F)$, there is an associated Breuil-Kisin module $(\fkM, \phi_\fkM)$ over $O_{F}$. Let us recall this notion, cf. \cite[\S 4.1]{BMS}. There is a natural surjection of $W(k)$-algebras 
 \begin{equation}\label{BKtheta}
\theta\colon \frak{S}=W(k)\lps T\rps\to O_{F} ,
\end{equation}
sending $T$ to $\pi$. Its kernel is generated by an Eisenstein polynomial $E(T)$. There is a Frobenius $\phi$ on $\frak{S}$, which is the Frobenius on $W(k)$ and sends $T$ to $T^p$. 

A \emph{Breuil-Kisin module} over $O_{F}$ is a vector bundle $\fkM$ over $\Spec(\frak{S})$ equipped with an isomorphism
\begin{equation}
\phi_\fkM\colon\phi^*(\fkM)[\frac{1}{E(T)}]\isoarrow  \fkM[\frac{1}{E(T)}] .
\end{equation}

As explained in  \cite[\S 3.3, Cor. 3.3.6]{KP} (see also \cite[\S 4.2]{PCan}), one can use the extension result  \cite[Cor. 1.2]{An} to upgrade the Breuil-Kisin construction and obtain from a crystalline representation $\rho_y: {\rm Gal}(\bar E/F)\to \CG(\BZ_p)$ a $\CG$-torsor 
$\sP_{\rm BK}$ over $\frak S$ with a $\CG$-torsor isomorphism
\[
\phi_{\sP_{\rm BK}}: \phi^*(\sP_{\rm BK})[E(T)^{-1}]\xrightarrow{\sim} \sP_{\rm BK}[E(T)^{-1}]
\]
over $\frak S[E(T)^{-1}]$. The pair $(\sP_{\rm BK}, \Phi)$ is called a \emph{$\CG$-Breuil-Kisin module} in \cite{PCan} ($\CG$-BK module). 
By base changing via $\frak S\to W(k)$, given by $T\mapsto 0$,
we obtain a $\CG$-BKF-module  $(\sP_0, \phi_{\sP_0})$ over $k$. By the properties of the Breuil-Kisin functor (e.g. \cite[Thm. 3.3.2 (1)]{KP}), the Frobenius $G$-isocrystal $(\sP_0[1/p],  \phi_{\sP_0}[1/p])$ over $k$   is canonically the Frobenius $G$-isocrystal
underlying $D_{\rm crys}(\rho_y)$. Hence the point $y\in {\rm Sht}_{G,  b, \mu, K}(F)$ also provides the data of an isomorphism $\alpha$
of the Frobenius $G$-isocrystal $(\sP_0[1/p],  \phi_{\sP_0}[1/p])$ over $k$ with the one given by $b$.  We set 
\begin{equation}
{\rm sp}_{\rm BK}(y)=(\sP_0,  \phi_{\sP_0}[1/p]\circ \phi^*( \alpha))\in  {\rm Gr}^W_\CG(k).
\end{equation}
 This, as we will see in the next proposition, 
belongs to the subset $X_{\CG}(b, \mu^{-1})(k)\subset {\rm Gr}^W_\CG(k)$.

\begin{proposition}\label{propRZ}
\begin{altenumerate}
\item There is an identification of the {\cmag generic fiber $\sM_{\CG,b, \mu}^{\rm rig}$ }with \\ ${\CM}_{G,  b, \mu, K}$, as rigid analytic varieties over $\breve E$,

\item There is an identification of the perfection of the reduced special fiber $(\sM_{\CG,b, \mu})_\red^{\rm perf}$ with $X_{\CG}(b, \mu^{-1})$. 

\item The above identifications make the {\cmag specialization map 
\begin{equation}
{\rm sp}: |\sM_{\CG,b, \mu}^{\rm rig}|^{\rm class}\to (\sM_{\CG,b, \mu})_\red^{\rm perf}(k)
\end{equation}
agree} with the map 
\begin{equation}
{\rm sp}_{\rm BK}: |{\rm Sht}_{G,  b, \mu, K}|^{\rm class}\to X_{\CG}(b, \mu^{-1} )(k) .
\end{equation}
given above.

\end{altenumerate}
\end{proposition}

\begin{proof}  
Parts i) and ii) follow from Theorem \ref{Gleason} and it remains to show iii).

Recall that the specialization map ${\rm sp}$ is described  in Remark \ref{ideaSpecialization}.
We extend the natural homomorphism 
$W(k)\to W(O_{C^\flat})$ to 
\begin{equation}\label{injWA}
i: \frak{S}\to W(O_{C^\flat})
\end{equation}
 by sending $T$ to $[\pi^\flat]\in W(O_{C^\flat})$. Here $\pi^\flat$ is given by a choice of roots $(\pi^{1/p^n})_n$ in $C$.  
Then $i$ is compatible with the Frobenius homomorphisms and the map $\theta$ of \eqref{BKtheta}, resp. the map $\theta\colon W(O_{C^\flat})\to O_C$, in the sense that there is a commutative diagram
\begin{displaymath}
   \xymatrix{
         \frak{S}   \ar[r]^{i} \ar[d]_{\theta\ } &  W(O_{C^\flat})  \ar[d]^{\theta}\\
        O_{\breve F}\ar[r]^{\ } &  \ O_C.
        }
    \end{displaymath}
Furthermore, the image of $E(T)$ is a generator of the kernel of $\theta$ since it is primitive of degree $1$, comp. \cite[proof of Prop. 4.32]{BMS}.  Hence, $(i(E(T)))=(\xi)$ as ideals in $W(O_{C^\flat})$. 
Consider the map of locally ringed spaces $j: \CY_{[0, \infty)}(C^\flat, O_{C^\flat})\to \Spec(W(k)\lps T\rps)$ induced by $i\colon W(k)\lps T\rps\to W(O_{C^\flat})$, cf. \eqref{injWA}.  
  To show that ${\rm sp}$ agrees with ${\rm sp}_{\rm BK}$, it suffices to 
 show that the pull-back $j^*(\sP_{\rm BK}, \Phi)$ of the $\CG$-BK module attached to $\rho_y$ by the Breuil-Kisin functor (\cite{KFcrys})
 is isomorphic to the $\CG$-shtuka $(\sP, \phi_\sP)$
which is attached to the $(C, O_C)$-point of ${\rm Sht}_{G,  b, \mu, K}$  given by pre-composing $y$
with $\Spa(C, O_C)\to \Spa(F, O_F)$.  

Let first $\CG=\GL_n$. Then the  BKF-module $\frak M\otimes_{W(k)\lps T\rps}W(O_{C^\flat})$ given by the Breuil-Kisin module
$(\frak M, \Phi)$
extends the shtuka $(\sP, \phi_\sP)$.  Indeed, this shtuka is obtained by the de Rham $\BZ_p$-\'etale local system 
over $\Spec(F)$ given by $\rho_y$, as in Definition \ref{LSVectorShtuka}. Hence the assertion follows from \cite[Prop. 4.34]{BMS}, which  shows that
$\frak M\otimes_{W(k)\lps T\rps}W(O_{C^\flat})$ gives the ``correct" pair $(T, \Xi)$ (notation as in loc. cit., see also the proof of Proposition \ref{pairs0} above).
(This pair determines the shtuka by Fargues' theorem \cite[Thm. 14.1.1]{Schber}). 

This handles the case $\CG=\GL_n$. The case of general $\CG$ then follows by a standard argument by writing $\CG$ as the closed subgroup scheme 
of $\GL_n$ {\cmag given as the stabilizer of a family of tensors;}
see also the discussion in \S \ref{sss:torsorstensors}.
(In particular, this shows that the construction of ${\rm sp}_{\rm BK}$ is independent of the choice of the uniformizer $\pi_F$
and of $\pi_F^\flat$.)  
\end{proof}

\begin{proposition}
$\sM_{\CG,b, \mu}$ is the  unique  normal formal scheme $\sR$ which is flat and locally formally of finite type over $\Spf(O_{\breve E})$ and is equipped with identifications
\begin{altenumerate}
\item {\cmag $\sR^{\rm rig}=\CM_{G,b, \mu, K}$,}
\item ${\sR}_\red^{\rm perf}=X_{\CG}(b,\mu^{-1})$,
\end{altenumerate}
such that the following diagram is commutative: {\cmag
\begin{equation}\label{canlambda}
\begin{aligned}   \xymatrix{
        |\sR^{\rm rig}|^{\rm class} \ar[r]^{\rm sp_{\sR}} \ar[d]_{=} & {\sR_\red}(k)   \ar[d]^{=}\\
         |{\rm Sht}_{G,  b, \mu, K}|^{\rm class}\quad\ar[r]^{\,{\rm sp_{\rm Sht}}}  & X_{\CG}(b,\mu^{-1})(k) .
        }
        \end{aligned}
    \end{equation}
    }
\end{proposition}
\begin{proof} Indeed, by Louren\c co \cite{LourencoRH}, see \cite[18.4.2]{Schber}, the triple {\cmag $(\sR^{\rm rig}, {\sR}_\red^{\rm perf}, {\rm sp}_{\sR})$}
 characterizes such a formal scheme $\sR$. 
\end{proof}

  \subsection{Functoriality of integral LSV}
  
  Let $G\hookrightarrow G'$ be a group embedding compatible with local Shimura data
  $(G,b, \mu)\hookrightarrow (G', b', \mu')$. Then there is an inclusion of corresponding reflex fields, $E\supset E'$. Let $\CG$ and $\CG'$ be parahoric models of $G$, resp. $G'$, such that 
   \begin{equation}\label{equpara}
   \CG(\breve\BZ_p)=\CG'(\breve\BZ_p)\cap G(\breve\BQ_p)
   \end{equation}
   (intersection in $G'(\breve \BQ_p)$). Then, by \cite[Prop. 1.7.6]{BT2},
   $G\hookrightarrow G'$ extends to $\CG\to \CG'$. 
   
   \begin{lemma}\label{groupdil}
   Under the above assumptions, $\CG\to \CG'$  
   identifies
   $\CG$ with the group smoothening (in the sense of \cite{BLR}) of the Zariski closure $\bar\CG$ of $G$ in 
   $\CG'$, 
   \[
   \CG=\bar\CG^{\rm sm}\to \bar\CG\hookrightarrow \CG' .
   \]
   
   \end{lemma}
   
   \begin{proof} By the universal property of the group smoothening, we have a morphism $\CG\to\bar\CG^{\rm sm}$. By (\ref{equpara}) 
   we have $\CG(\breve\BZ_p)=\CG'(\breve\BZ_p)\cap G(\breve\BQ_p)=\bar\CG(\breve\BZ_p)$.
  The group smoothening $\bar\CG^{\rm sm}\to \bar\CG$ satisfies 
  $\bar\CG^{\rm sm}(\breve\BZ_p)=\bar\CG(\breve\BZ_p)$, so the morphism $\CG\to\bar\CG^{\rm sm}$ induces a bijection $\CG(\breve\BZ_p)=\bar\CG^{\rm sm}(\breve\BZ_p)$. The characterization of smooth integral models of 
  $G$ by their $\breve\BZ_p$-points given by the extension property \cite[Prop. 1.7.6]{BT2}
now  implies $\CG\simeq \bar\CG^{\rm sm}$.
   \end{proof}
  In the above situation, $ \CG=\bar\CG^{\rm sm}\to \bar\CG$ is a dilation (see \cite{BLR}) and we will call the group scheme morphism $\CG\to \CG'$ a \emph{dilated immersion}.

  Consider the corresponding morphism \eqref{functmor} arising by functoriality, 
    \begin{equation}
  \rho: \CM^{\rm int}_{\CG,b, \mu}\to \CM^{\rm int}_{\CG', b', \mu'}\times_{\Spd( O_{\breve E'})}\Spd( O_{\breve E}).
  \end{equation}  
  \begin{proposition}\label{prop332}
  Under the  assumption \eqref{equpara}, $\rho$ is a closed immersion in the sense of \cite[Def. 17.4.2]{Schber}.
  \end{proposition}

  \begin{proof} By \cite[Cor. 17.4.8]{Schber}, it is enough to show that the morphism is quasi-compact, quasi-separated, satisfies the valuative extension
criterion for properness as in loc. cit. and that for any algebraically closed field $C$ of characteristic $p$, the induced map 
\[
 \rho(C, O_C): \CM^{\rm int}_{\CG,b, \mu}(C,O_C)\to \CM^{\rm int}_{\CG', b', \mu'}(C, O_C)
\]
is injective. First note that using Proposition \ref{ppProp} and the Tannakian equivalence, 
$\CM^{\rm int}_{\CG,b, \mu}$ and $\CM^{\rm int}_{\CG',b', \mu'}$ satisfy ``partial properness":
$\CM^{\rm int}_{\CG,b, \mu}(R,R^\circ)=\CM^{\rm int}_{\CG,b, \mu}(R,R^+)$ and
$\CM^{\rm int}_{\CG',b', \mu'}(R,R^\circ)=\CM^{\rm int}_{\CG',b', \mu'}(R,R^+)$,
and so the valuation
criterion of properness is satisfied for $\rho$. 

Note that if $R$ is a flat $\BZ_p$-algebra, we have
\[
\CG(R)\subset \CG'(R).
\]
Arguing as in the proof of {\cred\cite[Prop. 2.25]{Gl21}}, we see that
the injectivity of $\rho(C, O_C)$ follows from
\begin{equation}\label{insideG'}
\CG(B^{[r,\infty)}_{(C, O_C)})\cap \CG'(W(O_C))=\CG(W(O_C)).
\end{equation}
To see this equality, observe that 
\[
\CG(B^{[r,\infty)}_{(C, O_C)})=G(B^{[r,\infty)}_{(C, O_C)}) ,
\]
 since $p$ is a unit in $B^{[r,\infty)}_{(C, O_C)}$. 
 For simplicity, we will set
 \[
 B^{[r, \infty)}_C=B^{[r,\infty)}_{(C, O_C)}=B^{[r,\infty)}_{(C, C^+)}.
 \]
Since $W(O_C)[1/p]\subset B^{[r,\infty)}_{C}$ and
 $G\hookrightarrow G'$ is a closed immersion,
 \[
 \CG(B^{[r,\infty)}_{C})\cap \CG'(W(O_C))\subset G(B^{[r,\infty)}_{C})\cap G'(W(O_C)[1/p])=G(W(O_C)[1/p]).
 \]
 Now by the property of group smoothening, it follows from \eqref{equpara} that 
 \[
 G(W(C)[1/p])\cap \CG'(W(C))=\CG(W(C)).
 \]
 Hence the intersection $\CG(B^{[r,\infty)}_{C})\cap \CG'(W(O_C))$ is contained in $\CG(W(C))\cap \CG(W(O_C)[1/p])$.
 Since $\CG$ is affine and $W(C)\cap W(O_C)[1/p]=W(O_C)$, this last intersection is $\CG(W(O_C))$, which 
 proves \eqref{insideG'}. The same argument shows that $\rho(R, R^+)$ is injective when $(R, R^+)$ is obtained as a product of points $(C_i, C_i^+)$, cf \S\ref{vcover}. (Quasi-)~Separateness now follows by the argument in {\cred\cite[Prop. 2.25]{Gl21}}.

 It remains to show that $\rho$ is quasi-compact.  Note that every $\Phi\in \CG(W(R^+)[1/\xi_{R^\sharp}])$ with pole bounded by $\mu$ defines a $\CG$-shtuka $\sP_\Phi$ over $(R, R^+)$
 with leg at $(R^\sharp, R^{\sharp +})$ by taking the trivial $\CG$-torsor with Frobenius
 given by the element $\Phi$. {\cred As in \cite[Def. 2.22]{Gl21}} we consider the small $v$-sheaf $\CL\CM^{\rm int}_{\CG,b,\mu}$ over ${\rm Spd}(O_E)$ which classifies
 pairs $(\Phi, i_r)$, where $\Phi\in \CG(W(R^+)[1/\xi_{R^\sharp}])$ is bounded by $\mu$ and $i_r$ is a trivialization of the restriction of
 $\sP_\Phi$ to $\CY_{[r,\infty)}(R, R^+)$. The forgetful morphism of $v$-sheaves 
 \[
 \pi: \CL\CM^{\rm int}_{\CG,b,\mu}\to \CM^{\rm int}_{\CG,b,\mu}
 \]
is surjective for the $v$-topology and is a torsor for the $v$-topology and for the $v$-sheaf in groups $(R, R^+)\mapsto \CG(W(R^+))$, {\cred see \cite[Prop. 2.23]{Gl21}.}
It will be enough to show that the natural morphism
\[
\rho: \CL\CM^{\rm int}_{\CG,b,\mu}\to \CL\CM^{\rm int}_{\CG',b',\mu'}
\]
 is quasi-compact. 
 
  The argument is inspired by the proof of \cite[Thm. 21.2.1]{Schber}.
 Let $S=\Spa(A, A^+)$ be affinoid perfectoid and  let $S\to
 \CL\CM^{\rm int}_{\CG',b',\mu'}$, given by $(\Phi', i'_r)$.
 Consider the fibered product (small) $v$-sheaf
 \[
  T:=\CL\CM^{\rm int}_{\CG,b,\mu}\times_{\CL\CM^{\rm int}_{\CG',b',\mu'}}S ,
 \]
which classifies $(\Phi, i_r)$ such that $(\rho(\Phi), \rho(i_r))=(\Phi', i'_r)$. 
It is enough to show that $T$ is quasi-compact for all such $S\to
 \CL\CM^{\rm int}_{\CG',b',\mu'}$. Each point $t\in |T|$ is in the image of some $\Spa(C_t, C^+_t)\to T$ given by 
 $(\Phi_t, i_{t,r})$. So, 
 for each $t\in |T|$, we have 
\[
\Phi_t\in \CG(W(C^+_t)[1/\xi_t]), \quad i_{t, r}\in \CG(B^{[r_t,\infty)}_{C_t}).
\]
 We set $\xi=(\xi_t)\in W(D^+)=\prod_t W(C^+_t)$.

The composition $\Spa(C_t, C^+_t)\to T\to S$ gives $A\to C^+_t$. Choose a pseudo-uniformizer $\varpi_A\in A^+$ and denote by $\varpi_t\in C^+_t$ its image  under $A\to C^+_t$. Now consider the product of points 
\[
(D, D^+)=((\prod_t C^+_t)[1/\varpi], \prod_t C^+_t)
\]
with $\varpi=(\varpi_t)_t$, see \S\ref{vcover}. We have $\Spa(D, D^+)\to S$ which extends $\Spa(C_t, C^+_t)\to S$, for each $t$.
 
 Observe that each $\Phi_t$ has pole bounded by $\mu$, which is the same for all $t$.
Choose a closed group scheme immersion $j: \CG\hookrightarrow \GL_m$. Then
 the  entries of the matrices $j(\Phi_t)$, $j(\Phi_t)^{-1}$ lie in $\xi_t^{-N}W(C^+_t)$ 
 with $N$ bounded above, 
uniformily in $t$. It follows that
\[
(\Phi_t)_t \in \prod_t \CG(W(C^+_t)[1/\xi_t])=\CG(\prod_t W(C^+_t)[1/\xi_t])
\]
actually lies in $\CG(W(\prod_t C^+_t)[1/\xi])=\CG(W(D^+)[1/\xi])$. Set $\Phi=(\Phi_t)_t\in \CG(W(D^+)[1/\xi])$.
Now observe that 
\[
(\rho(i_{t, r}))_t\in \CG'(\prod_t B^{[r_t,\infty)}_{C_t}) 
\]
comes from  $\CG'(B^{[r',\infty)}_{(A, A^+)})$ via
$
 B^{[r',\infty)}_{(A, A^+)}\to \prod_t  B^{[r_t,\infty)}_{C_t}.
$ Hence, $(i_{t, r})_t$ lies in the intersection
\[
\CG'(B^{[r', \infty)}_{(D, D^+)})\cap \CG(\prod_t B^{[r_t,\infty)}_{C_t}) .
\]
This intersection is $G(B^{[r', \infty)}_{(D, D^+)})$ since $\CO_{G}=\CO_{G'}/I$ is a quotient 
and we have an injection
\[
B^{[r', \infty)}_{(D, D^+)}\hookrightarrow \prod_t B^{[r_t,\infty)}_{C_t}.
\]
Hence, $i_{r'}=(i_{t, r_t})_t$ is in $G(B^{[r', \infty)}_{(D, D^+)})$. We now consider
the pair $(\Phi, i_{r'})$ which gives a point $\Spa(D, D^+)\to\CL\CM^{\rm int}_{\CG,b,\mu}$. Combined with $\Spa(D, D^+)\to S$ as above gives $\Spa(D, D^+)\to T$ which is surjective.
This shows that $T$ is quasi-compact, as required.
\end{proof}

\subsection{Representability of integral LSV}\label{ss:repintLSV}  
  The aim of this subsection is to prove the following confirmation of  Conjecture \ref{repMint} under  some assumptions.
  \begin{theorem}\label{thmrepint}
Let $(G, b, \mu)$ be a local Shimura datum and $\CG$ a parahoric group scheme. The following assumptions are imposed. 
  \begin{itemize}
  \item[1)] $(G, \mu)$ is of local Hodge  type, i.e., there
  is a closed group embedding $\rho: G\hookrightarrow \GL_n$ such that $
  \rho\circ \mu$ is minuscule.
  \item[2)] $\CG$ is the Bruhat-Tits stabilizer group $\CG_x$ of a point in the extended Bruhat-Tits building of $G(\BQ_p)$, i.e., $\CG=\CG_x=\CG_x^\circ$. 
  \item[3)] Conjecture \ref{conjtubeMint} on the representability of formal completions $\wh{{\CM}^{\rm int}_{\CG, b, \mu}}_{/x}$  is true for all points $x\in\CM_{\CG, b, \mu}^{\rm int}(k)$.
  \end{itemize}
  Then $\CM_{\CG, b, \mu}^{\rm int}$ is representable by a normal formal scheme $\sM$ which is flat and locally formally of finite type over $\Spf O_{\breve E}$. 
  \end{theorem}
  \begin{remark}\label{remlocglob}
  Recall from \eqref{bxbase} that to $x$ we can associate $b_x\in G(\breve \BQ_p)$, well-defined up to $\sigma$-conjugacy by $\CG(\breve\BZ_p)$, which is $\sigma$-conjugate to $b$. Then $\CM_{\CG, b_x, \mu}^{\rm int}$ has a natural base point $x_0$ and there is an isomorphism as in \eqref{translatesigma} that induces an isomorphism 
  \[
  \wh{{\CM}^{\rm int}_{\CG, b, \mu}}_{/x}\simeq\wh{{\CM}^{\rm int}_{\CG, b_x, \mu}}_{/x_0}.
  \]
  Assume  that the local  data $(p, G, b, \mu, \CG)$ come from a   global datum $(p,\eG, X, \eK^p, \CG)$ of Hodge type and that there is a point ${\bf x}\in \sS_\eK(k)$ in the reduction of the integral model  in  \S \ref{ss:shimhodge} of the Shimura variety ${\rm Sh}_\eK(\eG, X)$  such that $b_x=b_{\bf x}$ (up to $\sigma$-conjugacy by $\CG(\breve\BZ_p)$). Here $b_{\bf x}$ is defined as in \eqref{bxbase}, using the extension of the $\CG$-shtuka to $\sS_\eK$, comp.  the passage before Remark \ref{Upsilonmap}. Then the hypothesis in 3) can be eliminated. Indeed, it is automatically satisfied, as  follows from Theorem \ref{mainhodge} below.

  The assumption can be formulated in the framework of \cite{HeR} as follows. Let
  $$
  C(\CG, \mu^{-1})=\big(\CG(\breve \BZ_p) {\rm Adm}(\mu^{-1})\CG(\breve \BZ_p)\big)/\CG(\breve \BZ_p)_\sigma ,
  $$
 and let $C(\CG, \mu^{-1})_{[b]}$ be the inverse image of $[b]$ under the natural map $C(\CG, \mu^{-1})\to B(G, \mu^{-1})$.  Then $b_x\in C(\CG, \mu^{-1})_{[b]}$. Consider the map from Remark \ref{Upsilonmap}
 \begin{equation}
 \Upsilon_\eK\colon \sS_\eK(k)\to C(\CG, \mu^{-1}) .
 \end{equation}
It may be conjectured that $\Upsilon_\eK$  is surjective (this follows from the system of axioms in \cite{HeR}, cf. \cite[Cor. 4.2]{HeR}).  If this conjecture holds true, then the assumption made above is satisfied.  In \cite{Zhou}  this conjecture is proved in the Hodge type case if $G$ is tamely ramified and residually split; and the same proof works when $G$ is unramified by using the results of Nie \cite{Nie}; see also  \cite{SYZ} for more classes of Shimura varieties.

We note that this approach to Conjecture \ref{repMint} is global, as it makes use of the theory of Shimura varieties. By contrast,   in \cite{PRintlsv} we pursue a purely local approach, with more general results. Indeed, in \cite{PRintlsv} we prove   Conjecture \ref{repMint} in the case when $(G, \mu)$ is of abelian type and $p\neq 2$ or $p=2$ and $G_\ad$ is of type A or C. However, even in this other approach, we use Theorem \ref{thmrepint} and its proof.
\end{remark}

  \subsubsection{A construction of formal subschemes}\label{parconstruction371}
  
  By the main theorem of \cite{La},  there is an equivariant embedding of extended buildings $\rho_*: \sB^e(G,\BQ_p)\hookrightarrow\sB^e(\GL_n,\BQ_p)$ which is associated to the embedding $\rho: G\hookrightarrow H:=\GL_n$. Set $\CH=\CH_{\rho_*(x)}$ for the stabilizer Bruhat-Tits group scheme
given by the point $\rho_*(x)\in \sB^e(H,\BQ_p)$. In this case of $\GL_n$ this group scheme is 
the stabilizer of a periodic lattice chain 
 \[
 \Lambda_i : \cdots \subset p\Lambda_0\subset \Lambda_{r} \subset \Lambda_{r-1}\subset\cdots \subset \Lambda_0\subset p^{-1}\Lambda_r\subset\cdots
 \] 
in $V=\BQ_p^n$ and is connected, so it is a 
parahoric group scheme.
We have 
\[
\CG(\breve\BZ_p)=\CH(\breve\BZ_p)\cap G(\breve\BQ_p).
\]
In this, we can assume at the cost of adjusting $\rho$, that 
$\CH=\GL(\Lambda)$ is hyperspecial and given by a single lattice. 
Indeed, we can replace $V$ by $V'=V^{\oplus r}$ and the lattice chain $\Lambda_\bullet$ by the lattice chain 
given by the multiples $p^\bullet \Lambda'$ of the single lattice
$\Lambda'=\Lambda_0\oplus\cdots \oplus \Lambda_r$.

Let $\ov\CG$ be the flat closure of $G$ in $\GL(\Lambda)$.
Here, $\CG_x=\ov\CG^{\rm sm}$ is the Neron group smoothening of $\ov\CG$, comp. Lemma \ref{groupdil}. Then
  \begin{equation}
  \CG(\breve\BZ_p)= \CG_x(\breve\BZ_p)= \ov\CG(\breve\BZ_p)= \CH(\breve\BZ_p)\cap G(\breve\BQ_p) .
 \end{equation}
 {\cmag The ``Grothendieck-Messing"  period morphisms
  fit in a commutative diagram of $v$-sheaves (in fact, diamonds) over $\Spd(\breve E)$,
 \begin{equation*} \xymatrix{
   {\rm Sht}_{\CG,b, \mu} \ar[r]^{\rho\ \ \ \ \ \ \ \ \ \ \ \  } \ar[d]_{\pi_{\rm GM}} &\ar^{\pi_{\rm GM}\times 1}[d]  {\rm Sht}_{\CH, \rho(b), \rho(\mu)}\times_{\Spd(\breve\BQ_p)}\Spd(\breve  E)\\
       {\rm Gr}_{G, \Spd(\breve  E), \mu}   \ar[r]^{\rho\ \ \ \ \ \ \ \ \ \ \ \  } & {\rm Gr}_{H, \Spd(\breve\BQ_p), \rho(\mu)}\times_{\Spd(\breve\BQ_p)}\Spd(\breve  E) ,}
  \end{equation*}
  {\cmag comp. \S \ref{ss:LSV}.}}
  
  The diamonds  ${\rm Sht}_{\CG,b, \mu}$ and  
  $ {\rm Sht}_{\CH,\rho(b), \rho(\mu)}$ are represented by the (smooth) rigid analytic varieties $\CM_{\CG,b, \mu}$ and $\CM_{\CH,\rho(b), \rho(\mu)}$ over $\breve E$; the period morphisms are \'etale with fibers $G(\BQ_p)/\CG(\BZ_p)$ and $H(\BQ_p)/\CH(\BZ_p)$
  respectively and the bottom horizontal morphism is induced by the closed immersion 
  of homogeneous spaces $\CF_{G,\mu}\hookrightarrow \CF_{H, \rho(\mu)}\otimes_{\breve\BQ_p}\breve E$. Since $\rho$ also gives $G(\BQ_p)/\CG(\BZ_p)\hookrightarrow H(\BQ_p)/\CH(\BZ_p)$, it follows that the top horizontal morphism is a closed immersion. (This is in agreement with
  Proposition \ref{prop332} which extends this to a closed immersion of the corresponding $v$-sheaf integral models.)
  
  By \cite[Cor. 24.3.5]{Schber}, $\CM^{\rm int}_{\CH,\rho(b), \rho(\mu)}$ is represented by a Rapoport-Zink formal scheme $\sM_{\CH,\rho(b),\rho(\mu)}$. It is normal, separated flat and locally formally of finite type over ${\rm Spf}({\breve\BZ_p})$.

  For simplicity of notation, set $\breve O=O_{\breve E}$, let $\pi$ be a uniformizer so that $k=\breve O/(\pi)$.
  
  Our construction of $\sM_{\CG, b, \mu}$ is based on the following ``formal descent" statement inspired by a similar statement
  in \cite{deJongCrys}.
  Suppose that $\fkX$ is a formal scheme over ${\rm Spf}(\breve O)$ which is separated, locally formally of finite type and flat over $\breve O$.
  Let {\cmag $\fkX^{\rm rig}$ be} the rigid analytic generic fiber of $\fkX$ over $\breve E$ in the sense of Berthelot, cf. \cite[chap. 5]{R-Z}. Let $\fkX_{\rm red}$
  the reduced locus of $\fkX$ which is a scheme locally of finite type over $\Spec(k)$. Let
  \begin{equation}
 {\rm sp}:  |\fkX^{\rm rig}|^{\rm class}\to \fkX_{\rm red}(k)
  \end{equation}
be the specialization map.

  \begin{proposition}\label{deJong}
Let the following data be given:
\begin{itemize}
\item[1)] A closed rigid analytic subvariety $Z\subset \fkX^{\rm rig}$,

\item[2)] A closed reduced $k$-subscheme $T\subset \fkX_{\rm red}$.
 
\item[3)] For each $t\in T(k)$, a closed formal subscheme $V_t\subset \widehat\fkX_{/t}$.
\end{itemize}
It is assumed that $
{\rm sp}(|Z|^{\rm class})\subset T(k)$, and that for all $t\in T(k)$, 
\begin{equation}\label{compatible}
(\widehat\fkX_{/t})^{\rm rig}\cap Z=(V_t)^{\rm rig}.
\end{equation}
Then, there is a unique closed formal subscheme $\fkZ\subset \fkX$ such that:
\begin{itemize}
\item[$\alpha$)] $\fkZ^{\rm rig}=Z$,
\item[$\beta$)]  $\fkZ_{\rm red}=T$,
\item[$\gamma$)] For all $t\in T(k)$, $\widehat \fkZ_{/t}=V_t$ as closed formal subschemes
of $\widehat\fkX_{/t}$.
\end{itemize}
\end{proposition}
  
  \begin{proof}
  This is given by an argument as in the proof of \cite[Prop. 7.5.2]{deJongCrys}.  By considering formal open subschemes ${\rm Spf}(A)\subset \fkX$ we can reduce to the case that
  $\fkX={\rm Spf}(A)$. Let  $I$ be a maximal ideal of definition 
  of $A$ such that $\fkX_{\rm red}=\Spec(A/I)$. As in \emph{loc. cit.}, 7.1.1,  we set
  \[
  B_n=\widehat{A[{I^n}/{\pi}]} ,
  \]
  where the hat denotes $I$-adic completion.  The natural map
  \[
  A[I^{n+1}/\pi]\hookrightarrow A[I^n/\pi]
  \]
 induces a continuous map $B_{n+1}\to B_n$. Then the association $A\mapsto (B_n(A))_n$ is functorial.
 Set $C_n=B_n[1/\pi]$ which is an affinoid $\breve E$-algebra. Then
 \[
 \fkX^{\rm rig}=\bigcup_{n} {\rm Sp}(C_n)=\varinjlim\nolimits_n {\rm Sp}(C_n) ,
 \]
 where ${\rm Sp}(C_n)\hookrightarrow {\rm Sp}(C_{n+1})$ is an affinoid subdomain.
 Let $U_n={\rm Sp}(C_n)$. The intersection $U_n\cap Z$ is given by an ideal $I_n\subset B_n$; by replacing
 $I_n$ by its saturation, we can suppose that $B_n/I_n$ is $\breve O$-flat. (In fact, we can choose $I_n$ to be the kernel of $B_n\to \Gamma(U_n\cap Z, \CO)$.)
 As in \emph{loc. cit.} 7.1.13, there is $c=c(A, I)\geq 0$ and compatible surjective $\breve O$-algebra homomorphisms
 \[
 \beta_n : B_n\to A/I^{n-c}
 \]
 for all $n\geq c$.
 
 For what follows, we will need some information on $c$. 
 Choose generators $(f_1,\ldots , f_r)$ of the ideal $I\subset A$ and consider
 the graded homomorphism
 \[
 \psi: (A/\pi A)[x_1,\ldots , x_r]\to  \bigoplus_{n=0} (I^n/\pi I^n) =A/\pi A\oplus I/\pi I\oplus I^2/\pi I^2\oplus\cdots 
 \]
 sending $x_i$ to $(0, \bar f_i, 0, \ldots )$. The kernel of $\psi$ is a homogenous ideal generated
 by homogeneous polynomials $(\bar P_1,\ldots , \bar P_s)$. In the construction of \emph{loc. cit.},  we
 can take
 \[
 c=c(\psi):={\rm max}(\deg(\bar P_i))_{i=1,\ldots ,s}.
 \]
Hence, in our arguments we can always take $c=c(A, I)$ to be 
the smallest such integer $c(\psi)$ among the possible presentations as above.

 Now given $x\in \Spec(A/I)(k)$, let $I\subset \fkM_x\subset A$ be the corresponding maximal ideal.
 For simplicity, write $\hat A_x$ for the completed local ring $\hat A_{\fkM_x}$.
 
 \begin{lemma}
Under our assumption on $(A, I)$, there is $C$ such that 
\[
c(\hat A_x, \hat\fkM_x)\leq C
\]
 for all $x\in \Spec(A/I)(k)$.
\end{lemma}

\begin{proof} This follows by an application of the theory of normal flatness (\cite{Hironaka}, see \cite[(2.2)]{Bennett}).
Set $R=A/\pi A$, $S=\Spec(A/I)$. We can find a sequence of reduced closed subschemes $S_{\rm red}=S_0\supset S_1\supset \cdots\supset  S_n=\emptyset$, such that $U_i=S_i\setminus S_{i-1}$ is regular and which has the following property: Let $J_i$ be the ideal of 
$(A/\pi A)\otimes_k\CO_{U_i}$ that corresponds to
the diagonal section $U_i\hookrightarrow \Spec(A/\pi A)\times_{\Spec(k)} U_i\to U_i$. Then, for all $i$,
\[
{\rm Gr}_{J_i}((A/\pi A)\otimes_k\CO_{U_i})=\oplus_{m\geq 0} J_i^m/J_i^{m+1}
\]
is flat over $\CO_{U_i}$. This implies that we can calculate the blow-up 
${\rm Proj}_{\CO_{U_i}}(\oplus_{m\geq 0} J_i^m)$ of $\Spec(A/\pi A)\times_k U_i$ along $U_i\hookrightarrow
\Spec(A/\pi A)\times_{\Spec(k)} U_i$ by an exact sequence
\[
0\to K_i\to ((A/\pi A)\otimes_k\CO_{U_i})[x_1,\ldots , x_r]\xrightarrow{\underline \psi} \oplus_{m\geq 0} J_i^m\to 0 ,
\]
in which all terms are flat over $\CO_{U_i}$. Write $K_i=(\bar P_{i1},\ldots , \bar P_{is})$,
where $\bar P_{ij}$ are homogeneous polynomials in $x_1,\ldots , x_r$. Take $x\in U_i(k)$ given by $\CO_{U_i}\to k$.
By flatness, the base change of the above exact sequence by $\CO_{U_i}\to k$ gives a presentation
of the Rees algebra for the blow-up of $A/\pi A$ at $x$. Hence, the specialization
of the polynomials $\bar P_{ij}$ at {\sl all} $x\in U_i(k)$ can be used to calculate the blow-up of $A_x/\pi A_x$ 
at $\fkM_x$, but also the blow-up of $\hat A_x/\pi \hat A_x$ at $\hat\fkM_x$. Therefore, 
\[
C_i={\rm max}(\deg(\bar P_{ij}))_{j=1,\ldots ,s}\geq c(A_x, \fkM_x)\geq  c(\hat A_x, \hat\fkM_x)
\]
for all $x\in U_i(k)$. We can now take $C={\rm max}(C_i)_{i=1,\ldots, n}$.
 \end{proof}
 
 Recall that for $t\in T(k)\subset \Spec(A/I)(k)$ we are given $V_t\subset \widehat{\fkX}_{/t}={\rm Spf}(\hat A_{t})$, which is given by an ideal
 \[
 J_t\subset \hat A_{t}.
 \]
 We also take $J_x=(1)$, for $x\in \Spec(A/I)(k)\setminus T(k)$.
 
 By \cite{deJongCrys}, identity (1) on p. 93, for all  $t\in T(k)$, there is $c(t)$ such that for every $n\geq c(t)$, we have
 \begin{equation}\label{deJongid}
 \beta_n(I_n)\, \hat A_t/ (I\hat A_t)^{n-c(t)}= J_t\, {\rm mod}\, (I\hat A_t)^{n-c(t)}.
 \end{equation}
 In fact, by the proof of identity (1) of loc.~cit.,  we see that we can take
 \[
 c(t)={\rm max}(c(A, I), c(\hat A_t, \hat\fkM_t)).
 \]
 Therefore, the arguments below work for $n\geq c={\rm max}(c(A, I), C)$, where $C$ is as in the lemma above.
   
 Also, since ${\rm sp}(|Z|^{\rm class})\subset T(k)$, we have  for $x\not\in T(k)$,
 \[
 \beta_n(I_n)\, \hat A_x/ (I\hat A_x)^{n-c}=\hat A_x/ (I\hat A_x)^{n-c} .
 \]
Now note that the natural homomorphism 
 $A\to \prod_x \hat A_x$  is faithfully flat and so is
 \[
 A/I^{n-c}\to \prod\nolimits_x \hat A_x/(I\hat A_x)^{n-c}.
 \]
 Using this, descent and (\ref{deJongid}), we see that, for all $n\geq c$,
 \[
 \beta_{n+1}(I_{n+1})\, {\rm mod}\, I^{n-c}=\beta_n(I_n) ,
 \]
 (an equality of ideals in $A/I^{n-c}$.) This implies that
 \[
 J:=\varprojlim\nolimits_n \beta_n(I_n)
 \]
 gives an ideal in $\varprojlim\nolimits_n A/I^{n-c}=A$.
 We set
 \[
 \fkZ={\rm Spf}(A/J)
 \]
 which satisfies the desired ($\alpha$), ($\beta$), ($\gamma$). 
  \end{proof}
 
  \subsubsection{Application to integral LSV} For simplicity, in the rest of the paragraph, we will omit $b$ and $\mu$ from  the notation,  and write
  $\CM^{\rm int}_\CG$, $X_\CG$, etc. instead of $\CM^{\rm int}_{\CG,b, \mu}$, $X_{\CG}(b, \mu^{-1})$, etc. and also
  $\CM^{\rm int}_\CH$ instead of $\CM^{\rm int}_{\CH,\rho(b),\rho(\mu)}$, etc.

  Note that the perfection
  $(\sM_{\CH}\hat\otimes_{\breve\BZ_p}\breve O)_\red^{\rm perf}$
  of $(\sM_{\CH}\hat\otimes_{\breve\BZ_p}\breve O)_{\rm red}$
  is identified (see Theorem \ref{Gleason} (a)) with $X_{\CH}$.
  The closed immersion of $v$-sheaves of Proposition \ref{prop332}
  \[
  \rho: \CM^{\rm int}_{\CG}\hookrightarrow \CM^{\rm int}_{\CH}
  \]
  gives, after applying the functor $(\ )_{\rm red}$ and the identification of 
  Theorem \ref{Gleason} (a), a morphism  of perfect $k$-schemes
  \[
  \rho: X_{\CG}\rightarrow X_{\CH } =(\sM_{\CH }\hat\otimes_{\breve\BZ_p}\breve O)_\red^{\rm  perf} .
  \]
Note that $X_\CG\subset {\rm Gr}^W_\CG$, $X_\CH\subset {\rm Gr}^W_\CH$, are closed and that 
 $\rho: X_{\CG}\rightarrow X_{\CH }$ is given by restricting the obvious
$\rho: {\rm Gr}^W_\CG\to {\rm Gr}^W_\CH$ which is a morphism of (ind-) perfectly proper 
schemes over $k$ such that $\rho(k)$ is injective.

 We now apply Proposition \ref{deJong} to: 
  
  \begin{altitemize}
  \item $\fkX=\sM_{\CH}\hat\otimes_{\breve\BZ_p}\breve O$,  
  
  \item $Z=\CM_{\CG}\subset \CM_{\CH}\hat\otimes_{\breve\BQ_p}\breve E$,
  
  \item
  $T=$ scheme theoretic image of the composition
  \[
  \rho: X_{\CG}\rightarrow X_{\CH}= (\sM_{\CH}\hat\otimes_{\breve\BZ_p}\breve O)_\red^{\rm perf}\to (\sM_{\CH}\hat\otimes_{\breve\BZ_p}\breve O)_{\rm red}
  \]
  where the last arrow is the natural morphism. We have $\rho: X_{\CG}\to T^{\rm perf}$ which gives 
  $T(k)=\rho(X_{\CG}(k))\subset X_{\CH}(k)$. In fact, this is a universal homeomorphism and so, by \cite[Lemma 3.8]{BS}, $\rho: X_{\CG}\xrightarrow{\sim} T^{\rm perf}$ is an isomorphism. 
    
  \item  Let $t\in T(k)=\rho( X_{\CG}(k))\simeq X_{\CG}(k)$. By the local model diagram  for the RZ formal scheme $\sM_{\CH, \rho(b), \rho(\mu)}$,  there exists an isomorphism  
  $\widehat{\sM_{\CH, \breve O }}_{/\rho(t)}\simeq \widehat {{\BM_{\CH, \breve O } }}_{/\widetilde{\rho(t)}}$, where the target is the formal completion of the base-change  $\BM_{\CH, \breve O }=\BM_\CH\otimes_{\breve \BZ_p}\breve O$ of the local model $\BM_\CH$ at a suitable point $\widetilde{\rho(t)}$. 
  Since we are assuming Conjecture \ref{conjtubeMint} for $t$, we have a normal complete  Noetherian local ring $R_t$ such that  $\widehat{\CM^{\rm int}_{\CG }}_{/t}\simeq \Spd(R_t)$ and
  \begin{equation}\label{injective4209}
  \Spd(R_t)\to \widehat{\CM^{\rm int}_{\CH,\breve O }}_{/\rho(t)}.
  \end{equation}
 By the full faithfulness of the $\diam$-functor, we obtain a morphism of affine formal schemes
  \begin{equation}\label{moronLM}
  \wt{\rho}: \Spf(R_t)\rightarrow \widehat {\sM_{\CH, \breve O }}_{/{\rho(t)}} ,
  \end{equation}
  which corresponds to a homomorphism of local rings $\wt \rho^*\colon \hat\CO_{\sM_{\CH, \breve O},{\rho(t)}}\to R_t$. The injectivity of (\ref{injective4209}) implies that the fiber of $\wt\rho^*$ over the closed point
   is given by an Artin local ring; hence $\wt\rho^*$ is finite.
  We define 
  \[
  {V_t}\hookrightarrow \widehat {\sM_{\CH, \breve O }}_{/{\rho(t)}}
  \]
   to be the formal closed subscheme of $\widehat {\sM_{\CH, \breve O }}_{/{\rho(t)}}$ {\cmag which corresponds 
   to the scheme-theoretic image of the corresponding morphism 
   of affine schemes induced by \eqref{moronLM}, i.e. defined by the kernel of $\wt\rho^*$. 
   Note that, by construction, the morphism $\wt \rho$ factors through  a morphism 
\[
 \Spf(R_t) \to  V_t.
\]
Since $R_t$ is assumed to be  normal, this   identifies $R_t$ with the
normalization of $\hat\CO_{V_t}:=\hat\CO_{\sM_{\CH, \breve O},{\rho(t)}}/\ker(\wt\rho^*)$. The corresponding morphism  of $v$-sheaves 
$\Spf(R_t)^\diam \to  V_t^\diam$ is then surjective, but, since  (\ref{injective4209}) is injective, it is also injective.
Hence,  
\[
\Spf(R_t)^\diam= \wh{\CM^{\rm int}_{\CG, b, \mu}}_{/t}\xrightarrow{\ \sim\ } V_t^\diam
\]
is an isomorphism of $v$-sheaves.}
\end{altitemize}
\quash{
the corresponding $V_t$ is unibranch.   What is the translation between $v$-sheaves and unibranchedness? Even so, the conclusion is in the world of $v$-sheaves, hence I don't see the usefulness of this. I also didn't understand the reference to Lourenco.}

  We now verify that these choices satisfy ${\rm Sp}(|Z|^{\rm class})\subset T(k)$ and (\ref{compatible}). 
  We will again use the $\diam$-functor $(\ )\mapsto (\ )^\diam$ into $v$-sheaves over $\Spd(\breve O)$, or over $\Spd(\breve E)$, and the fact that it is fully faithful
  when the source category is either the category of flat normal formal schemes locally of finite type over $\breve O$, or the category of 
  smooth rigid-analytic spaces over $\breve E$.

  Note that classical points
  of $Z=\CM_{\CG}$ over a finite extension $F/\breve E$ are uniquely given by morphisms $\Spa(F, O_F)^\diam\to Z^\diam$.
  Such a point specializes to a point in $T(k)=X_{\CG}(k)$, by Theorem \ref{Gleason} and the compatibility of Gleason's
  specialization map with the ``classical" specialization map.
  It  remains to show
  that (\ref{compatible}) holds. Again, it is enough to show this for the associated $v$-sheaves. 
  The $\diam$-functor commutes with formal completions (\cite[Prop. 4.19]{Gl}).    It also commutes with 
  taking generic fibers: This follows from the fact that, under the functor from rigid spaces over $\breve E$ to adic spaces over
  $\Spa(\breve E, O_{\breve E})$, Berthelot's generic fiber ${\rm Spf}(A)^{\rm rig}$ of a formal scheme ${\rm Spf}(A)$ 
  corresponds to the generic fiber of the corresponding adic space, i.e. to $\Spa(A, A)\times_{\Spa(O_{\breve E}, O_{\breve E})}\Spa(\breve E, O_{\breve E})$.  Hence we have
  \[
 ((\widehat\fkX_{/\rho(t)})^{\rm rig})^\diam= ((\widehat\fkX_{/\rho(t)})^{\diam})_\eta=((\widehat\fkX^{\diam})_{/\rho(t)})_\eta=(\widehat{\CM^{\rm int}_{\CH, \breve O}}_{/\rho(t)})_\eta.
  \]
  Since $Z^\diam=\CM^\diam_\CG={\rm Sht}_\CG$ and $\rho: \CM^{\rm int}_\CG\to \CM^{\rm int}_\CH$
  is a closed immersion (and hence injective)
  \[
  ((\widehat\fkX_{/\rho(t)})^{\rm rig})^\diam\times_{\CM^{\rm int}_{\CH}} Z^\diam=(\widehat{\CM^{\rm int}_{\CH}}_{/\rho(t)})_\eta\times_{\CM^{\rm int}_{\CH}} \CM^\diam_\CG=
  (\widehat{\CM^{\rm int}_{\CG}}_{/t})_\eta.
  \]
  In this, the RHS is the $v$-sheaf given as the generic fiber of the formal completion
 $\widehat{\CM^{\rm int}_{\CG}}_{/t}$.
  This is equal to the $v$-sheaf $(V_t^{\rm rig})^\diam$ associated to $V_t^{\rm rig}$, by our choice of $V_t$.

  Applying Proposition \ref{deJong} gives a formal scheme $\sM_{\CG,b, \mu}^-$. Denote by   $\sM_{\CG,b, \mu}$
  its normalization. Then 
  $\sM_{\CG,b, \mu}$ is as in the statement of Conjecture
  \ref{repMint} and, hence, this proves Theorem \ref{thmrepint}.  \hfill $\square$
  
   \section{Global Shimura varieties and their universal $\CG$-shtukas}

\subsection{Shimura varieties}
Let $(\eG, X)$ be a Shimura datum. Let $\{\mu\}$ be the $\eG(\bar\BQ)$-conjugacy class of the
corresponding minuscule cocharacter $\mu=\mu_X$. 
 We make the following blanket assumption on the split ranks of the connected center of $G$: 
 \begin{equation}\label{blanket}
 \begin{aligned}
 {\rm rank}_\BQ(Z^o)={\rm rank}_\BR(Z^o) .
  \end{aligned}
  \end{equation}
\begin{remark}
Assumption \eqref{blanket} is equivalent to the condition that there is no non-trivial subtorus of $Z^o$ which is anisotropic over $\BQ$ but splits over $\BR$. Imposing  this condition allows the construction of the natural pro-\'etale torsor on the Shimura variety ${\rm Sh}_\eK(\eG, X)$ used below. The main point is to ensure that an arithmetic subgroup of $Z^o(\BQ)$ is finite, and this holds true if and only if Assumption \eqref{blanket} is satisfied, cf. \cite[Prop., Ch. II, A. 2]{SerreAb}.  When  $Z^o$ splits over a CM-extension (Milne adds this to the axioms of a Shimura variety, cf. \cite[II. equ. (2.1.4)]{MilneAA}), then this assumption holds if $G$ is replaced by its quotient by the maximal anisotropic subtorus of $Z^o$ which splits over $\BR$. 
\end{remark}

Let $\eE=\eE(\eG, X)\subset\bar\BQ\subset \BC$ be the reflex field.
 For an open compact subgroup
 $\eK\subset \eG(\BA_f)$ of the finite adelic points of $\eG$,  the Shimura variety
 \begin{equation}
 {\rm Sh}_\eK(\eG, X)=\eG(\BQ)\backslash  (X\times \eG(\BA_f)/\eK)
 \end{equation} 
 has a canonical model ${\rm Sh}_\eK(\eG, X)_\eE$ over $\eE$. 
 
 Now fix a prime $p$. Suppose that $\eK=\eK_p\eK^p$
 with $K:=\eK_p\subset \eG(\BQ_p)$ and $\eK^p\subset \eG(\BA_f^p)$ compact open. 
 We always assume 
 that $\eK^p$ is sufficiently small. Fix a  parahoric group scheme $\CG$ over $\BZ_p$ with $\CG(\BZ_p)=K=\eK_p$.
 
 Choose a place $v$ of $\eE$ over $(p)$ given by $\eE\subset \bar\BQ_p$. Let $E=\eE_v$ be the completion of $\eE$ at $v$ and consider $\mu$ as also giving a conjugacy class of cocharacters of $G=\eG\otimes_{\BQ}\BQ_p$. We also denote this conjugacy class by $\{\mu\}$; then $E$ is the local reflex field
 of $\{\mu\}$. Recall that $k$ denotes the algebraic closure of the residue field; we also sometimes write $k_v$ for $k$.

Consider the pro-\'etale $\CG(\BZ_p)$-cover 
$\BP_\eK$ over $ {\rm Sh}_\eK(\eG, X)_E:= {\rm Sh}_\eK(\eG, X)_\eE\otimes_\eE E$
 obtained by the system of covers 
 \begin{equation}
  {\rm Sh}_{\eK'}(\eG, X)_E\to 
{\rm Sh}_{\eK}(\eG, X)_E,
 \end{equation} 
 where $\eK'=\eK'_p\eK^p\subset \eK=\eK_p\eK^p$, with $\eK'_p$ running over all compact open subgroups of $\eK_p=
\CG(\BZ_p)$. (See \cite[III]{MilneAA}, \cite[\S 4]{LZ}. Note that $\CG(\BZ_p)=\varprojlim_{\eK'_p} \eK_p/\eK'_p$.)  By our condition on the smallness of $\eK^p$, this is a tower of smooth varieties, with \'etale transition maps.

 \begin{proposition}\label{ShimuraShtThm}
Assume that $(\eG, X)$ satisfies \eqref{blanket}. There exists a $\CG$-shtuka $\sP_{\eK, E}$ over ${\rm Sh}_\eK(\eG, X)^\diam_E\to {\rm Spd}(E)$ with one leg bounded 
by $\mu$ which is associated to the pro-\'etale $\CG(\BZ_p)$-cover $\BP_\eK$, in the sense of Section \ref{ss:dR}. 
Furthermore, $\sP_{\eK, E}$ are supporting prime-to-$p$ Hecke correspondences, i.e., for  $g\in \eG(\BA^p_f)$ and $\eK'^{ p}$ with $g\eK'^{ p}g^{-1}\subset \eK_p$, there are compatible
isomorphisms $[g]^*(\sP_{\eK, E})\simeq \sP_{\eK', E}$ which cover  the natural morphisms $[g]: {\rm Sh}_{\eK_p\eK'^p}(\eG, X)_E\to  {\rm Sh}_{\eK_p\eK^p}(\eG, X)_E $. 

 \end{proposition}
 
 \begin{proof} This follows by combining the results in section \ref{ss:dR}  with the ridigity theorems of Liu-Zhu \cite{LZ}. 
  By \cite[Thm. 1.2]{LZ} combined with \cite[Thm. 3.9 (iv)]{LZ}, the pro-\'etale $\CG(\BZ_p)$-torsors $\BP_\eK$ 
 over ${\rm Sh}_{\eK, E}$ is de Rham in the sense of Definition \ref{defDR}. 
  The desired shtuka is the one associated to 
the de Rham pro-\'etale $\CG(\BZ_p)$-torsor $\BP_\eK$
as in Definition \ref{LSGShtuka}.
\end{proof}

\subsection{Integral models} We continue to assume Condition \eqref{blanket} on $(\eG, X)$.
 Suppose that $K_p$ is parahoric and let $\CG$ be the corresponding Bruhat-Tits group scheme over $\BZ_p$.
Then we conjecture the existence of a system of normal integral models $\sS_{\eK}$ of ${\rm Sh}_{\eK, E}$ that support a ``universal" $\CG$-shtuka $\sP_\eK$  over $\sS_\eK$ with one leg bounded by $\mu$ that extends $\sP_{\eK, E}$. The  $v$-sheaves $\sS_{\eK}^\diam$ corresponding to these models should  formally locally coincide with 
Scholze's integral local Shimura varieties $\CM^{\rm int}_{\CG,  b, \mu}$, for varying $b$. More precisely, let $x\in  \sS_{\eK}(k)$.

The pull-back  $x^*(\sP_\eK)$ is a $\CG$-shtuka over $\Spec(k)$, i.e., yields by Example \ref{G-AlgClosedPoint} a $\CG$-torsor $\sP_x$ over $\Spec(W(k))$ with an isomorphism 
$$
\phi_{\sP_x}\colon \phi^*(\sP_x)[1/p]\isoarrow\sP_x [1/p].
$$
 The choice of a trivialization of the $\CG$-torsor $\sP_x$ defines an element $b_x\in G(\breve \BQ_p)$. This element is independent up to $\sigma$-conjugacy by an element of $\CG(\breve \BZ_p)$ of the choice of the trivialization and only depends on the point of $\sS_\eK$ underlying $x$. Since the shtuka $\sP_\eK$ is bounded by $\mu$, the $\sigma$-conjugacy class $[b_x]$ under $G(\breve \BQ_p)$ lies in the subset $B(G, \mu^{-1})$ of $B(G)$. 
\begin{remark}\label{Upsilonmap}
We obtain maps 
$$
\Upsilon_\eK\colon \sS_\eK(k)\to G(\breve\BQ_p)/ \CG(\breve\BZ_p)_\sigma,\quad \text{ resp. }\quad \delta_\eK\colon \sS_\eK(k)\to B(G)=G(\breve\BQ_p)/ G(\breve\BQ_p)_\sigma
$$
 The fibers of these maps define the \emph{central leaves, resp. the Newton stratification} of $\sS_\eK\otimes_{O_E}\kappa$, cf. \cite[Rem. 3.4, (3)]{HeR}. We recall from \cite{HeR} the map 
 $$
\ell_\eK\colon G(\breve\BQ_p)/ \CG(\breve\BZ_p)_\sigma\to \CG(\breve\BZ_p)\backslash G(\breve\BQ_p)/ \CG(\breve\BZ_p)=W_\eK\backslash \tilde W/W_\eK ,
 $$
 where $\tilde W$ denotes the Iwahori-Weyl group of $G(\breve \BQ_p)$ and $W_\eK$ the parabolic subgroup corresponding to $\CG(\breve\BZ_p)$. Let $\lambda_\eK=\ell_\eK\circ \Upsilon_\eK$,
 $$
 \lambda_\eK\colon \sS_\eK(k)\to W_\eK\backslash \tilde W/W_\eK .
 $$
 The fibers of this map define the \emph{Kottwitz-Rapoport stratification} of $\sS_\eK\otimes_{O_E}\kappa$, cf. \cite[eq. (3.4)]{HeR}.   The map $\Upsilon_\eK$ also allows to define the \emph{EKOR-stratification}, cf. \cite[\S 6]{HeR}.
\end{remark}
  
 The $v$-sheaf $ \CM^{\rm int}_{\CG,b_x, \mu}$ comes with a base point
 \begin{equation}\label{basepointM}
 x_0\in  \CM^{\rm int}_{\CG,b_x, \mu}(k) .
 \end{equation}
 This base point associates to $S\in{\rm Perfd}_k$ the tuple $(S^\sharp,  \sP_0 , \phi_{\sP_0 }, i_r)$, where $S^\sharp=S$, and $(\sP_0 , \phi_{\sP_0 }, i_r)=(\CG_{ \CY_{[0, \infty)}(S)}, \phi_{b_x}, \id)$. Indeed, the pair $(\sP_0, \id)$ lies in $X_{\CG}(b_x, \mu^{-1})\simeq (\CM^{\rm int}_{\CG, b_x, \mu})_\red$, cf. Theorem \ref{Gleason}, a). 
 This follows since the shtuka $(\sP_x, \phi_{\sP_x})$ has leg bounded by $\mu$. 
 
We have the following  conjecture.

\begin{conjecture}\label{par812} Let $\eK_p$ be parahoric, with corresponding parahoric model $\CG$ over $\BZ_p$.  

There exist normal flat models  $\sS_\eK$ of ${\rm Sh}_{\eK}(\eG, X)_E$ over $O_{E}$, for $\eK=\eK_p\eK^p$ with variable sufficiently small $\eK^p$, with the following properties. 

\begin{itemize}

\item[a)] For every dvr $R$ of characteristic $(0, p)$ over $O_E$, 
\begin{equation}\label{extprop}
(\varprojlim\nolimits_{\eK^p}{\rm Sh}_\eK(\eG, X)_E)(R[1/p])=(\varprojlim\nolimits_{\eK^p}\sS_\eK)(R).
\end{equation} If  ${\rm Sh}_{\eK}(\eG, X)_E$ is proper over $\Spec (E)$, then $\sS_\eK$ is proper  over $\Spec (O_{E})$.  In addition, the system  $\sS_\eK$ supports prime-to-$p$ Hecke correspondences, i.e., for  $g\in \eG(\BA^p_f)$ and $\eK'^{ p}$ with $g\eK'^{ p}g^{-1}\subset \eK^p$, there are finite \'etale  morphisms
$[g]: \sS_{\eK'}\to \sS_\eK $ which extend the natural maps $[g]: {\rm Sh}_{\eK_p\eK'^p}(\eG, X)_E\to  {\rm Sh}_{\eK_p\eK^p}(\eG, X)_E $.

\item[b)] The $\CG$-shtuka $\sP_{\eK, E}$ extends to a $\CG$-shtuka $\sP_{\eK}$ on $\sS_{\eK}$.

 \item[c)] For  $x\in \sS_\eK(k)$, let $b_x\in G(\breve \BQ_p)$ be defined by $\sP_{\eK}$, as explained above. 
There is an isomorphism of completions
 $$
\Theta_x\colon  \widehat{{{{\CM}}}^{\rm int}_{ \CG, b_x, \mu }}_{/x_0}\isoarrow (\widehat{\sS_{\eK_{/x}}})^\diam ,
$$ 
under which the pullback shtuka $\Theta_{ x}^*(\sP_{\eK})$ coincides with the tautological shtuka on $ {\CM}^{\rm int}_{ \CG, b_x, \mu }$ that arises from the definition of $ {\CM}^{\rm int}_{ \CG, b_x, \mu }$ as a moduli space of shtukas. Here $x_0$ denotes the base point of ${\CM}^{\rm int}_{ \CG, b_x, \mu }$, cf. \eqref{basepointM}. 
\end{itemize}
\smallskip 

\emph{We note that by Corollary \ref{shtExt}, the extension $\sP_{\eK}$ of $\sP_{\eK, E}$ is uniquely determined. In addition, the  prime-to-$p$ Hecke correspondences $[g]: \sS_{\eK'}\to \sS_\eK $ for  $g\in \eG(\BA^p_f)$ and $\eK'^{ p}$ with $g\eK'^{ p}g^{-1}\subset \eK^p$ in (a)  induce  by uniqueness compatible isomorphisms $[g]^*(\sP_\eK)\simeq \sP_{\eK'}$.}

\emph{ Replacing $b_x$ by $b'_x=gb_x\sigma^{-1}(g)$, where $g\in\CG(\breve \BZ_p)$, defines an isomorphism $ {{{\CM}}^{\rm int}_{ \CG, b_x, \mu }}\simeq  {{{\CM}}^{\rm int}_{ \CG, b'_x, \mu }}$ which preserves the base points. Hence Condition c) is independent of the choice of $b_x$. }  Also, note that Condition c) implies the representability of $ \widehat{{{{\CM}}}^{\rm int}_{ \CG, b_x, \mu }}_{/x_0}$, i.e.,  Conjecture \ref{conjtubeMint} for $ \CM^{\rm int}_{ \CG, b_x, \mu }$ at $x_0$. 
 \end{conjecture}

 \begin{remark}\label{remAdm423}
 An element $b\in G(\breve \BQ_p)$ is called \emph{$\mu$-admissible}, if the homomorphism 
 $$
 \phi_b\colon {\rm Frob}^*(\CG\times{\Spec(W(k)[1/p]})\to \CG_{\Spec(W(k)[1/p])}
 $$  has pole bounded by $\mu$, in the sense of Remark \ref{rem332}.   By \eqref{ptsofvLMo},  this condition is equivalent to asking that  
 \begin{equation}
 b\in \bigcup\nolimits_{w\in {\rm Adm}(\mu^{-1})_\CG}\CG(\breve\BZ_p) w \CG(\breve\BZ_p) .
  \end{equation}
  If $[b]\in B(G, \mu^{-1})$, then  $[b]$ contains an element $b\in G(\breve\BQ_p)$ satisfying this last condition (He's theorem \cite{He}). If $b$ is $\mu$-admissible, then $\CM_{\CG, b, \mu}$ has a canonical base point $x_0$, defined as above in \eqref{basepointM} (with $b_x$ replaced by $b$). 
  
  If $b=b_x$ is attached to a point $x\in\sS_\eK(k)$ (and a $\CG$-shtuka $\sP_\eK$) as above, then $b_x$ is $\mu$-admissible. 
 \end{remark}
 
   \begin{theorem}\label{uniq}
   There is at most one system of  normal flat models  $\sS_\eK$ of ${\rm Sh}_{\eK}(\eG, X)_E$ over $O_{E}$, for $\eK=\eK_p\eK^p$, with variable sufficiently small $\eK^p$, with the  properties enumerated in Conjecture \ref{par812}. More precisely, if $\sS_\eK$ and $\sS'_\eK$ are two $O_{E}$-models which satisfy the above properties, then there are isomorphisms $\sS_\eK\simeq\sS'_\eK$ which induce the identity on the generic fibers and are compatible with changes in $\eK^p$.
   \end{theorem}
   \begin{proof}
   We imitate the proof of \cite[Thm. 6.1.5]{PCan}. It suffices to construct these isomorphisms after base change $O_{E}\to O_{\breve E}$. Let $\sS''_\eK$ be the normalization of the closure of the generic fiber in $\sS_\eK\times_{\Spec(O_E)}\sS'_\eK$. Then $\sS''_\eK$ is again a tower with finite \'etale transition morphisms, for varying $\eK^p$. The argument of the proof of \cite[Prop. 6.1.7]{PCan} shows that the morphisms
   $$
   \pi_\eK\colon \sS''_\eK\to \sS_\eK, \quad \pi'_\eK\colon \sS''_\eK\to \sS'_\eK
   $$
   are proper and isomorphisms in the generic fibers. By uniqueness, we obtain identifications of $\CG$-shtukas on $\sS''_\eK$, 
 $$
 (\pi_\eK)^*(\sP_{\eK})=(\pi'_\eK)^*(\sP'_{\eK})=\sP''_\eK .
 $$
 It suffices to show that $\pi_\eK$ and $\pi'_\eK$ induce isomorphisms on the strict completions at geometric points of the special fibers. More precisely, let $\hat x\in {\rm Sh}_{\eK_p}(\eG, X)(F)$, where $F/\breve E$ is a finite extension. 
 By  property a), $\hat x$ extends to  points $\tilde x$ of $\sS_\eK(O_F)$ and $\tilde x'$ of $\sS'_\eK(O_F)$ and a point $\tilde x''$ of $\sS''_\eK(O_F)$ mapping to $\tilde x$, resp. $\tilde x'$. By reduction, we obtain the points $ x\in\sS_\eK(k)$ and $ x'\in\sS'_\eK(k)$ and the point $ x''\in\sS''_\eK(k)$ mapping to $x$, resp. $x'$.  By uniqueness, the pullback $\CG$-shtukas ${\tilde x}^*(\sP_\eK)$ and $\tilde {x}^{\prime *}(\sP'_\eK)$ and $\tilde {x}^{\prime\prime *}(\sP''_\eK)$ on $\Spec(O_F)$ all coincide. Hence the $\sigma$-conjugacy classes $[b_x]$ and $ [b_{x'}]$ and $ [b_{x''}]$ coincide. Let $b\in G(\breve\BQ_p)$ be a representative of this class which is $\mu$-admissible, see Remark \ref{remAdm423}.

  Define,  for $r\gg0$, the $v$-sheaf $\wt{\sS}_{/x}$ over $(\wh{\sS}_{/x})^\diam$ by adding to a  point $y$ of $(\wh{\sS}_{/x})^\diam$ with values in $S=\Spa(R, R^+)\in {\rm Perfd}_k$ a trivialization of the $\CG$-shtuka
 $$
 i_r\colon G_{\CY_{[r, \infty)}(S)}\isoarrow y^*(\sP_\eK)_{| \CY_{[r, \infty)}(S)} ,
 $$
such that $\phi_{y^*(\sP_\eK)}=\phi_b=b\times{\rm Frob}$. Then the map $\wt{\sS}_{/x}\to(\wh{\sS}_{/x})^\diam$ is representable in locally spatial diamonds. It is a torsor under the diamond group 
\begin{equation}\label{diamgp}
\wt G_b= \underline{\rm Aut}(\CE^b),
\end{equation} 
{\cmag given as the automorphism $v$-sheaf of the $G$-bundle $\CE^b$ over the Fargues-Fontaine curve, see \cite[ch. III 5.1]{FS}. Then $\wt G_b$ is also the $v$-sheaf given by $S\mapsto {{\rm Aut}_{\CY_{[r, \infty)}(S)}}(G\times \CY_{[r, \infty)}(S), b\times {\rm Frob})$, automorphisms in the category of $G$-torsors over
$\CY_{[r, \infty)}(S)$ equipped with a Frobenius isomorphism. 
This diamond group is an extension of $\underline{G_b(\BQ_p)}$ by a smooth unipotent group $v$-sheaf, see \emph{loc. cit.}.} Here, to conform to the notation of \cite{FS}, we denote by $G_b$ what was denoted $J_b$ before, i.e. for a $\BQ_p$-algebra $A$
 \[
G_b(A)=\{g\in G(A\otimes_{\BQ_p}\breve\BQ_p)\ |\ b\sigma(g)=gb\}.
\]
 This is represented by a reductive group over $\BQ_p$ and $\underline{G_b(\BQ_p)}$
 is the locally profinite $v$-sheaf given by the $\BQ_p$-points $G_b(\BQ_p)$.
 We define $\wt{\sS'}_{/x'}$ and $\wt{\sS''}_{/x''}$ in an analogous way. 
 
 Recall the integral model $\CM^{\rm int}_{\CG, b, \mu}$ of the local Shimura variety corresponding to the fixed element $b$, with its tautological $\CG$-shtuka $(\CP_{\rm univ}, \phi_{\CP_{\rm univ}})$ and its trivialization $\iota_r$ for $r\gg0$. By universality, we obtain  natural morphisms fitting in a commutative diagram, 
 \begin{equation}\label{Stildediag}
 \begin{aligned}
   \xymatrix{
       &\ar[dl]_{\pi_\eK}\wt{\sS''}_{/x''}  \ar[d]^{{\rm nat}''} \ar[dr]^{\pi'_\eK}\\
      \wt{\sS}_{/x}\ar[r]^{{\rm nat}} &  \CM^{\rm int}_{\CG, b, \mu}&  \ar[l]_{{\rm nat}'}\wt{\sS'}_{/x'}.
        }
        \end{aligned}
    \end{equation}
 On the other hand, since  $\sS_\eK$ satisfies property c) of Conjecture \ref{par812}, we obtain a commutative diagram for a suitable section $g\in \wt G_b(\wh{\sS}_{/x})$,
 \begin{equation*}\label{Sdiag}
 \begin{aligned}
   \xymatrix{
       \wt{\sS}_{/x}   \ar[r]^{} \ar[d]_{{\rm nat}\circ g}\ar[dr] & \wh{\sS}_{/x}\ar[d]^{\simeq} \\
     \CM^{\rm int}_{\CG, b, \mu} &  \ar[l]\wh{\CM^{\rm int}_{\CG, b, \mu}}_{/x_0} ,
        }
        \end{aligned}
    \end{equation*}
 in which the right vertical arrow is the inverse $\Psi_x$ of the isomorphism $\Theta_x$. Furthermore, by the nature of the morphism 
  $\wt{\sS}_{/x}\to \wh{\sS}_{/x}$, this isomorphism is uniquely determined by the rest of the diagram.
Something analogous holds for $\sS'_\eK$ and a global section $g'\in \wt G_b(\wh{\sS}_{/x})$. 
\begin{proposition}\label{lemma423}
The natural projection $\wt G_b\to \underline{G_b(\BQ_p)}$ induces an isomorphism on global sections,
$$
\wt G_b(\wh{\sS}_{/x})\isoarrow G_b(\BQ_p) . 
$$
\end{proposition}
\begin{proof}
Let $\wt G_b^{>0}$ be the kernel of the natural projection. Then $\wt G_b^{>0}$ is a successive extension of positive (absolute) Banach-Colmez spaces, cf. \cite[Prop. III.5.1]{FS}. More precisely, there is a filtration $\wt G_b^{\geq \lambda}$ such that for every $\lambda>0$, there is a natural isomorphism
\[
\wt G_b^{\geq \lambda}/\wt G_b^{> \lambda}\xrightarrow{\sim} {\bf B}(({\rm ad}\CE_b)^{\geq \lambda}/({\rm ad}\CE_b)^{> \lambda}),
\]
with target the Banach-Colmez space associated to the $-\lambda$ isoclinic part of the Frobenius isocrystal 
 $({\rm Lie}(G)\otimes_{\BQ_p}\breve\BQ_p, {\rm Ad}(b)\sigma)$.

 We first prove the following lemma. A \emph{formal group} version of this lemma occurs in \cite[Prop. 4.2.11]{CaraSch}. 

\begin{lemma}\label{boundonlambda}
Let $[b]\in B(G, \mu)$, where $\mu$ is minuscule. Then $\wt G_b^{>0}$ is a successive extension of positive (absolute) Banach-Colmez spaces of slopes $\leq 1$. 
\end{lemma}

\begin{proof}
We need to prove that $\langle \nu_b, \alpha\rangle \leq 1$ for any positive root $\alpha$. Equivalently, we need to see that $\langle \nu_b, \tilde\alpha\rangle \leq 1$
for the highest root $\tilde\alpha$. But since $[b]\in B(G, \mu)$, we have $\mu-\nu_b\in C^\vee$, where $C^\vee$ denotes the obtuse Weyl chamber (spanned by the positive coroots). Since $\langle C^\vee, \tilde \alpha\rangle \geq 0$ (cf. \cite[Ch. VI, \S 1.8, Prop. 25]{Bour}), we get $\langle \nu_b, \tilde\alpha\rangle \leq \langle \mu, \tilde\alpha\rangle \leq 1$, since $\mu$ is minuscule. 
\end{proof}

We write $\bfB(\lambda)={\mathcal {BC}}(\CO(\lambda))$ and let
$\bfB(\lambda)(\wh{\sS}_{/x})$ be the group ${\rm Hom}(\Spa(\wh{O}_{\sS, x})^\diam, \bfB(\lambda))$ of maps  of $v$-sheaves over $\Spd(k)$.

 By Lemma \ref{boundonlambda}, it suffices to show

\begin{lemma}\label{BCzero}
For $0<\lambda\leq 1$, we have $\bfB(\lambda)(\wh{\sS}_{/x})=(0)$.
\end{lemma}

\begin{proof} Write $\lambda=r/s$ with coprime integers $r, s>0$. Let $\BX=\BX_\lambda$ be the simple $p$-divisible group of slope $\lambda$ over $k$ and denote by the same symbol a lift of $\BX$ over $W(k)$. Let 
$$
\wt \BX=\varprojlim_{\times p} \BX
$$
be the universal covering of $\BX$, cf. \cite{SW}. For any $p$-adically complete $W(k)$-algebra $A$,  we have
$$
\bfB(\lambda)(A)=\wt \BX(A)=\wt \BX(A/p) , 
$$ 
comp. \cite[\S 2.3]{LB}, also \cite[Prop. II.2.5]{FS}. 
Now 
$$
\wt \BX\simeq\Spf(W(k)[[T_1^{1/p^\infty},\ldots,T_r^{1/p^\infty}]]) ,
$$
We therefore obtain
$$
\bfB(\lambda)(A)=(\varprojlim_{x\mapsto x^p}A^{oo})^r .
$$
After identifying the ideal of 
 topologically nilpotent elements in $\wh{O}_{\sS, x}$ with the maximal ideal $\wh{\frak m}_{\sS, x}$, we have 
 $$
 \wt \BX(\wh{O}_{\sS, x})=(\varprojlim_{x\mapsto x^p}\wh{\frak m}_{\sS, x})^r .
 $$
Since $\cap_{n}\wh{\frak m}_{\sS, x}^{p^n}=(0)$, we see that $\bfB(\lambda)(\wh{\sS}_{/x})=(0)$. 
\end{proof} 
The   proof of Proposition \ref{lemma423} follows.
\end{proof}

 \begin{remark}  In fact,  for every positive
 Banach-Colmez space $\CE$, the group of global sections $\CE(A)$ is zero, for any noetherian $p$-adically complete $W(k)$-flat algebra $A$. Here is a sketch of the proof of this statement, which was suggested to the authors by Scholze. Observe first that it is enough to show this for $A$ a discrete valuation ring with perfect residue field and $\CE=\bfB(\lambda)$, with $\lambda=r/s>0$. Let $K$, resp. $k$, be the fraction field, resp. the residue field of $A$. Set $C=\wh{\overline K}$ and, as usual, let $O_C$ be the integral closure 
 of $A$ in $C$. By $v$-descent, a section $s\in \bfB(\lambda)(A)$ is given by a $G_K={\rm Gal}(\bar K/K)$-invariant element of $\bfB(\lambda)(O_C)=(B_{\crys}^+)^{\phi^s=p^r}\subset B_{\crys}^+$.  By a theorem of Fontaine (see \cite[Thm. 4.12]{Font} or \cite[Thm. 6.14]{FontaineO}) $(B_{\crys}^+)^{G_K}=W(k)[1/p]$, and we can see that the only such Galois invariant element
 is $0$.
 \end{remark}

Hence $g, g'\in G_b(\BQ_p)$ and, after correcting $\Psi_x$, resp. $\Psi_{x'}$ by the  automorphism of $\CM^{\rm int}_{\CG, b, \mu}$ corresponding to $g$, resp. $ g'$, we deduce from \eqref{Stildediag}  a unique left vertical isomorphism making the following diagram commutative,
\begin{equation*}\label{SMdiag}
 \begin{aligned}
   \xymatrix{
       \wt{\sS''}_{/x''}   \ar[r]^{(\pi_\eK, \pi'_\eK)} \ar[d]_{\simeq} & \wh{\sS}_{/x}\times\wh{\sS'}_{/x'}\ar[d]^{(\Psi_x, \Psi_{x'})} \\
     \wh{\CM^{\rm int}_{\CG, b, \mu}}_{/x_0}  \ar[r]^-\Delta&\wh{\CM^{\rm int}_{\CG, b, \mu}}_{/x_0}\times \wh{\CM^{\rm int}_{\CG, b, \mu}}_{/x_0} .
        }
        \end{aligned}
    \end{equation*}
  This implies that $\pi_\eK$ and $\pi'_\eK$ induce isomorphisms 
 $$
  \wh{\sS''}_{/x''}\isoarrow  \wh{\sS}_{/x}, \quad\text{ resp. } \quad \wh{\sS''}_{/x''}\isoarrow  \wh{\sS'}_{/x'}.
 $$ 
 Since the point $\hat x$ was arbitrary, the assertion follows. 
   \end{proof}

  \subsection{Some functorialities of integral models}\label{ss:funct431}
   The uniqueness statement in Theorem \ref{uniq} leads to some interesting functoriality properties.
   
   Let $(\eG, X)\hookrightarrow (\eG', X')$ be an embedding of Shimura data induced by a group embedding  $\iota: \eG\hookrightarrow \eG'$. Assume that $\eK_p$, $\eK'_p$ are parahoric subgroups of $\eG(\BQ_p)$ and $\eG'(\BQ_p)$ respectively, with corresponding Bruhat-Tits group schemes $\CG$ and $\CG'$ over $\BZ_p$, such that 
   \begin{equation}\label{intersectionCond}
   \breve \eK_p=\iota^{-1}(\breve \eK'_p)\cap G(\breve\BQ_p).
   \end{equation}
 (Note that given a parahoric $\eK_p$ of $\eG(\BQ_p)$ which is also a stabilizer, we can always find a parahoric $\eK'_p$ of $\eG'(\BQ_p)$
 which satisfies \eqref{intersectionCond}. This uses Landvogt's embedding of extended buildings $\sB^e(G,\BQ_p)\hookrightarrow \sB^e(G',\BQ_p)$, \cite{La}.) Then, $\iota: \eG\to \eG'$ extends uniquely to a group scheme homomorphism $\iota: \CG\to \CG'$ which is a dilated immersion, i.e. identifies $\CG$ with the Neron smoothening $\bar\CG^{\rm sm}$ of the Zariski closure
$\bar\CG$  of the image of $\CG$ in $G'=\eG'_{\BQ_p}$,
  \[
  \CG\cong \bar\CG^{\rm sm}\to \bar\CG\hookrightarrow \CG'.
  \]
Let $\eK^p\subset \eG(\BA^p_f)$ and $\eK'^p\subset \eG'(\BA^p_f)$ be compact open subgroups such that the natural map induced by $\iota: (\eG, X)\to (\eG', X')$ gives an immersion
\[
\iota\colon  {\rm Sh}_{\eK}(\eG, X)_E\hookrightarrow {\rm Sh}_{\eK'}(G', X')\otimes_{E'}E.
\]
  Here, as usual, $\eK=\eK_p\eK^p$, resp. $\eK'=\eK'_p\eK'^p$.   Note that, given a compact open $\eK^p\subset \eG(\BA^p_f)$, there is always a compact open
 $\eK'^p\subset \eG'(\BA^p_f)$ for which we have such an immersion, cf.  \cite[Lem. 2.1.2]{KisinJAMS}.
 If $\eK^p$ is sufficiently small, as we are always assuming, then we can also take $\eK'^p$ to be sufficiently small.

\begin{theorem}\label{functorialThm}  Let $\sS_\eK$, resp. $\sS'_{\eK'}$, be integral models of ${\rm Sh}_{\eK}(\eG, X)_E$, resp.  ${\rm Sh}_{\eK'}(\eG', X)_{E'}$ over $O_E$, resp. $O_{E'}$, with the  properties enumerated in Conjecture \ref{par812} and which are characterized by Theorem \ref{uniq}. Then, under the above assumptions which include (\ref{intersectionCond}), the immersion of $E$-schemes 
$
\iota\colon {\rm Sh}_{\eK}(\eG, X)_E\to {\rm Sh}_{\eK'}(\eG', X')_{E'}\otimes_{E'}E
$
above extends uniquely to a morphism  over $O_E$,
\[
\iota\colon\sS_{\eK}\to \sS'_{\eK'}\otimes_{O_{E'}}O_E .
\]

\end{theorem}
 Before we discuss the proof, we give a consequence:
  \begin{corollary}\label{corfunctorial}
   Let $\eK_p$ and $\eK'_p$ be parahoric subgroups of $\eG(\BQ_p)$, with corresponding parahoric group schemes $\CG$ and $\CG'$ over $\BZ_p$. Assume that  $\eK_p\subset \eK'_p$. For $\eK=\eK_p\eK^p$, resp. $\eK'=\eK'_p\eK^p$, let $\sS_\eK$, resp. $\sS_{\eK'}$, be integral models  of ${\rm Sh}_{\eK}(\eG, X)_E$, resp.  ${\rm Sh}_{\eK'}(\eG, X)_E$ over $O_{E}$ with the  properties listed in Conjecture \ref{par812} and characterized by Theorem \ref{uniq}. Then the morphism of $E$-schemes ${\rm Sh}_{\eK}(\eG, X)_E\to {\rm Sh}_{\eK'}(\eG, X)_E$ extends uniquely to $\pi_{\eK, \eK'}\colon\sS_{\eK}\to \sS_{\eK'}$ over $O_E$. 
   \end{corollary}
   \begin{proof}
  We will apply Theorem \ref{functorialThm} to $(\eG', X')=(\eG\times \eG, X\times X)$ and $\iota: \eG\to \eG'$ the diagonal embedding. Also, we take the parahoric of the target group $\eG(\BQ_p)\times\eG(\BQ_p)$ to be $\eK_p\times \eK'_p$. The intersection of $\breve\eK_p\times\breve\eK'_p$ with the diagonal is $\breve\eK_p\subset \eG(\breve\BQ_p)$, so (\ref{intersectionCond}) holds. We can see that $\sS_{\eK}\times \sS_{\eK'}$ is the (unique by Theorem \ref{uniq}) integral model 
 for ${\rm Sht}_{\eK\times \eK'}(\eG\times \eG, X\times X)$ with the  properties enumerated in Conjecture \ref{par812}. By Theorem \ref{functorialThm} the diagonal morphism extends to
  \[
  \sS_{\eK}\to \sS_{\eK}\times \sS_{\eK'}.
  \]
 This, composed with the projection, induces the  desired map $\pi_{\eK, \eK'}\colon\sS_{\eK}\to \sS_{\eK'}$.
 \end{proof}

We now give an outline of the proof of Theorem \ref{functorialThm}. This proof will be completed in \S\ref{complFunct}
after we first give the argument in the special case $(G', X')=({\rm GSp}(V), S^{\pm})$. 

Let $\sS_{\eK}^{\dagger}$ be the normalization of the closure of the image of ${\rm Sh}_{\eK}(\eG, X)_E$ in $\sS'_{\eK'}\otimes_{O_{E'}}O_{E}$. 
 This is an integral model of ${\rm Sht}_{\eK}(G, X)_E$ and comes with a morphism
 \[
\iota: \sS_{\eK}^{\dagger}\to \sS'_{\eK'}\otimes_{O_{E'}}O_E
 \]
which extends the natural morphism on the generic fibers. It will be enough to show that  $\sS^\dagger_{\eK}$ satisfies the  properties listed in Conjecture \ref{par812}. Indeed, then by Theorem \ref{uniq}, $\sS^\dagger_{\eK}=\sS_{\eK}$. Note that in the case when $(G', X')=({\rm GSp}(V), S^{\pm})$ is a Siegel Shimura datum, $(G, X)\hookrightarrow ({\rm GSp}(V), S^{\pm})$ is a Hodge embedding. Then proving these properties for the normalization  $\sS_{\eK}^\dagger$ gives the construction of an integral model as in Conjecture \ref{par812} in the Hodge type case, i.e. gives the proof of Theorem \ref{mainhodge}, as in \S\ref{ss:extofshthodge}, \S\ref{ss:complhodge}, below. In what follows, we omit subscripts and write $\sS:=\sS_{\eK}$, $\sS':=\sS_{\eK'}$, $\sP:=\sP_{\eK}$, etc.

{\sl Step  A.} We will first show that the $\CG$-shtuka $\sP_{E}$ over ${\rm Sht}_\eK(G, X)_E$ extends to a $\CG$-shtuka $\sP^\dagger$ over  $\sS^\dagger$, compatibly with the pull-back $\iota^*(\sP')$ of the $\CG'$-shtuka over $\sS$. ``Compatibly" here is meant in the sense that there is an isomorphism of $\CG'$-shtuka over $\sS^\dagger$,
\[
\CG'\times^{\CG}\sP^\dagger\simeq \iota^*(\sP') ,
\]
  i.e. also respecting the Frobenius structures. Equivalently, we can think of $\sP^\dagger$ as a shtuka obtained 
 from $\iota^*(\sP')$ by reducing the structure group from $\CG'$ to $\CG$ via $\CG\to \CG'$.
 The proof of the existence of this extension of $\sP_E$ to $\sP^\dagger$ will be explained in \S\ref{compl481}.
\smallskip
 
{\sl Step B.}  We will next  show that the restriction of the $\CG$-shtuka $\sP^\dagger$ to the formal completion of 
 any point $x$ of $\sS^\dagger(k)$ admits a framing; this provides  a morphism of $v$-sheaves,
 \[
 \widehat{{\sS^\dagger}}_{/x}\xrightarrow{\ \ } \widehat{\CM^{\rm int}_{\CG, b_x,\mu}}_{/x_0}.
 \]
 This will provide the inverse of the desired isomorphism $\Theta_x$ in Conjecture \ref{par812}. This will be explained in  \S\ref{compl482}. 
\smallskip
 
 The rest of the proof then closely follows the proof of Theorem \ref{mainhodge}, which corresponds to the special case
 $(G', X')=({\rm GSp}(V), S^{\pm})$, see \S\ref{complFunct}.

 \subsection{A conjectural prismatic refinement}\label{ss:prism}
We conjecture that the integral models $\sS_\eK$ also support an object of prismatic nature which suitably refines the universal $\CG$-shtuka $\sP_\eK$. We will now try to make this more precise. To ease the notation, we omit the subscript $\eK$. 

Recall  that according to Bhatt and Scholze (\cite{BSPrism}),  a  ``prism"  is a pair $(A, I)$ where $A$ is a $\delta$-ring and $I\subset A$ is an ideal defining a Cartier divisor in $\Spec(A)$, such that
\begin{itemize}
\item[(1)] The ring $A$ is derived $(p, I)$-adically complete.

\item[(2)] The ideal $I+\phi_A(I)A$ contains $p$, where $\phi_A(x)=x^p+p\delta_A(x)$ is the Frobenius lift $\phi_A: A\to A$
induced by the $\delta$-structure of $A$.
\end{itemize}
A map $(A, I)\to (B, J)$ of prisms is a map of $\delta$-rings $A\to B$ taking $I$ to $J$. Then, by \cite[Prop. 1.5]{BSPrism}, one has $J=IB$.

Consider the big prismatic site $(\wh\sS)_\pri$ of the $p$-adic formal scheme $\wh\sS$ given by the $p$-adic completion of $\sS$. This is the opposite of the category of pairs $((A, I), x)$ of prisms $(A, I)$ together with a map $x: \Spf(A/I)\to \wh\sS$, endowed with faithfully flat covers, as defined in \cite{BSPrism}. We have the structure sheaf  
of rings $\CO_\pri$ on $(\wh\sS)_\pri$ taking a pair $((A, I), x)$ to $A$. This admits a Frobenius   
$\phi_{\pri}: \CO_\pri\to \CO_\pri$ given by the Frobenius lifts $\phi_A: A\to A$. 
There is also the sheaf of rings $\bar\CO_{\pri}$ taking $((A, I), x)$
to $A/I$ and the sheaf of ideals $I_\pri\subset \CO_\pri$ taking $((A, I), x)$ to $I$.

 Our object, which one might call a \emph{prismatic Frobenius crystal with $\CG$-structure} over $\wh\sS$, should be a pair $(\sP_{\pri}, \phi_{\sP_{\pri}})$ of
\begin{itemize} 
\item[a)] a $\CG\otimes_{\BZ_p}\CO_{\pri}$-torsor $\sP_{\pri}$ over  $(\wh\sS)_{\pri}$, 

\item[b)]  an isomorphism
 \[
 \phi_{\sP_{\pri}}: (\phi^*_\pri\sP_{{\pri}})_{|\Spec(\CO_\pri)\setminus V(I_\pri)}\xrightarrow{\sim} {\sP_{\pri }}_{|\Spec(\CO_\pri)\setminus V(I_\pri)}.
 \]
 \end{itemize}
 
 More concretely, it should assign to each $\wt x=((A, I), x\colon \Spf(A/I)\to \wh\sS)$ a pair $ (\sP_{\wt x}, \phi_{\sP_{\wt x}})$, 
where $\sP_{\wt x}$ is a $\CG$-torsor over $\Spec(A)$ and where
\[
\phi_{\sP_x} : (\phi^*_A(\sP_x))_{|\Spec(A)\setminus V(I)}\xrightarrow{\sim} (\sP_x)_{|\Spec(A)\setminus V(I)}
\]
is a $\CG$-isomorphism, together with compatible functorial base change isomorphisms for maps of prisms $(A, I)\to (B, J)$
with commutative diagrams
\begin{equation*}
  \begin{aligned}
   \xymatrix{
         \Spf(B/J) \ar[r] \ar[rd]_{y} & \Spf(A/I) \ar[d]_{x} \\
    & \wh\sS.
        }
        \end{aligned}
    \end{equation*}
    
Let $R^{\sharp+}$ be an integral perfectoid $O_E$-algebra with a map $x^+: \Spf(R^{\sharp+})\to \wh\sS$. Denoting by  $R^{+}$  the tilt 
of $R^{\sharp+}$, then $(W(R^{+}), \ker(W(R^{+})\to R^{\sharp+}))$ is a (perfect) prism. We can then evaluate 
$(\sP_{\Prism}, \phi_{\sP_{\pri}})$ at the point $\wt x^+$ given by this prism together with $x^+$. 
Suppose now that $R^{\sharp +}$
is part of an affinoid perfectoid pair $(R^\sharp, R^{\sharp +})$. Then $x^+$ gives 
\[
x: \Spa(R^\sharp, R^{\sharp +})\to \wh\sS^{\rm ad},
\]
hence a $\Spa(R, R^+)$-point $x$ of the $v$-sheaf $\sS^\sdiam/\Spd(O_E)$. 

We ask that 
$(\sP_{\pri}, \phi_{\sP_{\pri}})$ refines the $\CG$-shtuka $(\sP, \phi_\sP)$ in the following sense: for all such choices of 
$(R^\sharp, R^{\sharp +})$ and $x$, there is
an isomorphism between the pull-back of $ (\sP_{\wt x^+}, \phi_{\sP_{\wt x^+}})$ along
\[
\CY_{[0,\infty)}(R, R^+)\to \Spec(W(R^+))
\]
and the value of the shtuka $(\sP, \phi_\sP)$ at $x$.  Furthermore, these isomorphisms are supposed to be compatible
with the base change isomorphisms. 

In the case of PEL Shimura varieties, one should be able to obtain $(\sP_{\pri}, \phi_{\sP_{\pri}})$ 
from the prismatic cohomology of the universal abelian scheme $\CA\to \sS$: the value at $\wt x=((A, I), x\colon \Spf(A/I)\to \wh\sS)$
should be the pair $ (\sP_{\wt x}, \phi_{\sP_{\wt x}})$ where $\sP_{\wt x}$ is given by appropriate frames of
\[
{\rm H}^1_\pri(\wh\CA\times_{\wh\sS, x}\Spf(A/I)/A):= {\rm H}^1_\pri((\wh\CA\times_{\wh\sS, x}\Spf(A/I)/A)_\pri, \CO_\pri),
\]
and $\phi_{\sP_{\wt x}}$ is  induced by the $\phi_A$-linearization of its Frobenius map.

   \subsection{Shimura varieties of  Hodge type}\label{ss:shimhodge}
  In this subsection, we explain the class of Shimura varieties  for which we can prove Conjecture \ref{par812}.   
 \subsubsection{Shimura data of Hodge type} Fix a $\BQ$-vector space $V$ with a perfect alternating pairing $\psi.$ 
For any $\BQ$-algebra $R,$ we write $V_R = V\otimes_{\BQ}R.$ 
Let $\GSp = \GSp(V,\psi)$ be the corresponding group of symplectic similitudes, and let $S^{\pm}$ be the Siegel 
double space, defined as the set of maps $h: \mathbb S \rightarrow \GSp_{\BR}$ such that 
\begin{enumerate}
\item The $\BC^\times$-action on $V_{\BR}$ gives rise to a Hodge structure
\begin{equation}
V_{\BC} \simeq V^{-1,0} \oplus V^{0,-1}
\end{equation}  of type $(-1,0), (0,-1)$. 
\item $(x,y) \mapsto \psi(x, h(i)y)$ is (positive or negative) definite on $V_{\BR}.$
\end{enumerate}
The Shimura datum $(\eG, X)$ is of {\sl Hodge type} if there is a symplectic faithful representation 
 $\rho: \eG\hookrightarrow {\rm {GSp}}(V, \psi)$ inducing an embedding of Shimura data
 \begin{equation}
i\colon  (\eG, X)\hookrightarrow ({\rm {GSp}}(V, \psi), S^{\pm}).
 \end{equation}

\begin{definition}\label{pglobHodgetype}
  
 Let $p$ be a prime number. The tuple $(p, \eG, X, \eK)$, with the open compact subgroup $\eK=\eK^p\eK_p$, is of \emph{global Hodge type} if the following conditions are satisfied. 
 
 \begin{itemize}
  \item[1)] $(\eG, X)$ is a Shimura datum of Hodge type. 
 \item[2)] $\eK_p=\CG(\BZ_p)$, where $\CG$ is the Bruhat-Tits stabilizer group scheme 
 $\CG_x$ of a point $x$ in the extended Bruhat-Tits building of $G(\BQ_p)$ and $\CG$ is connected, i.e., we have
 $
\CG=
\CG_x=
\CG_x^\circ$. 
 \end{itemize}
 \end{definition} 
 
 \subsubsection{Integral models in the Hodge type case}\label{ss:KPextension} We now show how to construct integral models $\sS_\eK$ as in Conjecture \ref{par812}, when $(p, \eG, X, \eK)$ is of global Hodge type. We start with a Hodge embedding 
$ i\colon \eG\hookrightarrow {\rm {GSp}}(V, \psi)$.   We can then find a parahoric group scheme $\CH$ for the symplectic similitude group 
 $H={\rm {GSp}}(V_{\BQ_p}, \psi_{\BQ_p})$ such that
 there is a  homomorphism of group schemes over $\BZ_p$ 
 \begin{equation}\label{GtoGL}
 \iota\colon  
\CG\xrightarrow{\delta} \bar\CG \hookrightarrow \CH ,
 \end{equation}
 which is a dilated immersion and extends the closed embedding in the generic fiber,
\begin{equation*}
  G_{\BQ_p}\hookrightarrow {\rm {GSp}}(V_{\BQ_p}, \psi_{\BQ_p})\subset {\rm {GL}}(\prod_i (V_{\BQ_p}\oplus V_{\BQ_p})) ,
\end{equation*} 
see \cite[4.1.5]{KP} and Lemma \ref{groupdil}, comp. \S\ref{parconstruction371}. Here, $\CH$ is the parahoric group scheme ${\rm GSp}(\Lambda_\bullet )$ given by the stabilizer of some
 periodic self-dual $\BZ_p$-lattice chain 
 \[
 \Lambda_\bullet : \cdots \subset p\Lambda_0\subset \Lambda_{r} \subset \Lambda_{r-1}\subset\cdots \subset \Lambda_0\subset \Lambda_0^\vee\subset \cdots\subset \Lambda_{r-1}^\vee\subset \Lambda_r^\vee\subset p^{-1}\Lambda_0\subset p^{-1}\Lambda^\vee_0\subset\cdots
 \] 
Then
\[
 {\rm GSp}(\Lambda_\bullet )\hookrightarrow \prod_{i=0}^r(\GL(\Lambda_i)\times \GL(\Lambda^\vee_i)) ,
 \]
 is a closed group immersion.
 
 For $\Lambda_i$ in the lattice chain $\Lambda_\bullet$, let $V_{i, \BZ_{(p)}} = \Lambda_i\cap V,$ and fix a $\BZ$-lattice $V_{i, \BZ} \subset V$ such that $V_{i, \BZ}\otimes_{\BZ}\BZ_{(p)} = V_{i, \BZ_{(p)}}$. 
 Set $V_{\rm sum}=\prod_{i=0}^r (V\oplus V)$ and
 \[
 V_{{\rm sum}, \BZ_{(p)}}=\prod_{i=0}^r V_{i, \BZ_{(p)}}\oplus V^\vee_{i, \BZ_{(p)}}\subset V_{\rm sum}.
 \]
Consider the Zariski closure $G_{\BZ_{(p)}}$ of $\eG$ in $
\GL( V_{{\rm sum}, \BZ_{(p)}})$; then $G_{\BZ_{(p)}}\otimes_{\BZ_{(p)}}\BZ_p\cong \bar\CG$. Fix a collection of  tensors $(s_{a}) \subset V_{{\rm sum}, \BZ_{(p)}}^\otimes$ 
whose stabilizer is $G_{\BZ_{(p)}}.$  This is possible by the improved\footnote{in the sense that one does not need the symmetric and alternating tensors used in \cite[Prop. 1.3.2]{KisinJAMS}.} version of \cite[Prop. 1.3.2]{KisinJAMS} given in \cite{DeligneLetter}.
Finally, set
\begin{equation}\label{defLambda}
\Lambda:=V_{{\rm sum}, \BZ}\otimes_{\BZ_{(p)}}\BZ_p=\prod_{i=0}^r\Lambda_i \oplus\Lambda_i^\vee.
\end{equation}
We have $\eK_p = 
\CG(\BZ_p),$ and we set $\eK^\flat_p = \CH(\BZ_p)=\mathrm{GSp}(\Lambda_\bullet)(\BZ_p)=\mathrm{GSp}(V_{\BQ_p})\cap \GL(\Lambda).$  By  \cite[Lem. 2.1.2]{KisinJAMS}, for any compact open subgroup $\eK^p \subset \eG(\BA^p_f)$ there exists 
$\eK^{\flat p} \subset \GSp(\BA^p_f)$ such that $i$ induces an embedding over $\sf E$
\begin{equation}\label{ShGtoSiegel}
i\colon{\rm Sh}_{\eK}(\eG,X)_\eE \hookrightarrow {\rm Sh}_{\eK^\flat}(\GSp(V, \psi), S^{\pm})_\BQ\otimes_\BQ{\sf E},
\end{equation}
where $\eK^\flat=\eK^\flat_p\eK^{\flat p}$.

The choice of lattices $V_{i, \BZ}$ gives rise to an interpretation of the Siegel Shimura variety ${\rm Sh}_{\eK^\flat}(\GSp, S^{\pm})_\BQ$ as a moduli scheme of chains of $p$-isogenies between 
polarized abelian varieties $A_i$ with $\eK^{\flat p}$-level structure;
 this extends to $\CA_{\eK^\flat}$ over $\BZ_{(p)}$ 
 (see \cite{KisinJAMS}, \cite{KP}). Denote by $\BL_\eK$ the  local system given by the Tate module of the $p$-divisible group of the product $A=\prod_{i=0}^r A_i\times A_i^\vee$ of the universal abelian schemes over ${\rm Sh}_{\eK^\flat}(\GSp, S^{\pm})_\BQ\otimes_\BQ{\sf E}$ restricted to ${\rm Sh}_{\eK}(\eG,X)_\eE$. 
 
 Recall ${\rm Sh}_{\eK}(\eG,X)_E={\rm Sh}_{\eK}(\eG,X)_\eE\otimes_\eE E$. The tensors 
 $s_a\in \Lambda^\otimes$ define, using the compatibility between Betti cohomology and  \'etale cohomology, corresponding global sections $t_{a, \et}$ of $\BL_\eK$ over ${\rm Sh}_{\eK}(\eG,X)_E$, comp. \cite[\S 2.2]{KisinJAMS} or \cite[\S 6.5]{Zhou}.
 The pro-\'etale torsor $\BP_\eK$ under $\eK_p=\CG(\BZ_p)$ is given by 
  \[
\underline{\rm Isom}_{(t_{a, \et}), (s_a)}(\BL_\eK, \Lambda_{ {\rm Sh}_{\eK}(\eG,X)}),
 \]
 i.e., $\BP_\eK$ is the torsor of  trivializations of $\BL_\eK$ that
 respect the tensors.

We denote by $\sS^-_{\eK}(\eG,X)$ the (reduced) closure of 
${\rm Sh}_{\eK}(\eG,X)_E$ in the $O_{E}$-scheme 
$\CA_{\eK^\flat}\otimes_{\BZ_{(p)}}O_{E}$. Then the integral model $\sS_{\eK}(\eG,X)$  is defined to be the normalization of $\sS_{\eK}(\eG,X)^-$, comp.   \cite{KP}.
For simplicity of notation, we set 
\begin{equation*}
\sS_{\eK}:=\sS_{\eK}(\eG,X) .
\end{equation*}  
 The morphism \eqref{ShGtoSiegel} extends to a finite morphism
\begin{equation}\label{intHodgemb}
i\colon\sS_{\eK}\to \CA_{\eK^\flat} \otimes_{\BZ_{(p)}}O_{E}.
\end{equation}
We also set 
\begin{equation*}\label{intmodlimit}
\sS_{\eK_p}:=\sS_{\eK_p}(G, X)=\varprojlim\nolimits_{\eK^p}\sS_{\eK^p\eK_p}(G, X). 
\end{equation*}
Again, we can see that the transition maps are finite \'etale and so the limit exists.
 \begin{theorem}\label{mainhodge}
Let $(p, \eG, X, \eK)$ be of global Hodge type. Then the system of integral models $\sS_{\eK}$ constructed above satisfies the properties a)--c) of Conjecture \ref{par812}. In particular (by  uniqueness, Theorem \ref{uniq}), the system $\sS_{\eK}$ is independent of the Hodge embedding used in its construction. Furthermore, for any $x\in \sS_\eK(k)$, the integral local Shimura variety $\CM^{\rm int}_{\CG, b_x, \mu}$ satisfies Conjecture \ref{conjtubeMint} at the base point $x_0\in \CM^{\rm int}_{\CG, b_x, \mu}(k)$, i.e., $\widehat{{\CM}^{\rm int }_{\CG, b_x, \mu}}_{/x_0}$ is representable.
\end{theorem}

We note that property a) (the extension property \eqref{extprop}) is very simple: for the Siegel model $\CA_{\eK^\flat}$ it holds by the N\'eron-Tate-Shafarevich criterion of good reduction, and this implies the extension property for $\sS_\eK$ by its definition as the normalization of the closure of the generic fiber in $\CA_{\eK^\flat}\otimes_{\BZ_{(p)}}O_E$. 
Property b) is proved in Subsection \ref{ss:extofshthodge}. Property c) is proved in Subsection \ref{ss:complhodge}, which also contains the representability of  $\widehat{{\CM}^{\rm int}_{\CG, b_x, \mu}}_{/x_0}$.

\subsection{Extension of shtukas}\label{ss:extofshthodge}In this subsection we give the construction of the extension of the $\CG$-shtuka $\sP_{\eK, E}$ in part b) of Conjecture \ref{par812} for the integral models $\sS_\eK$ of the last subsection. For simplicity of notation, we write $\sS=\sS_{\eK_p}$. Denote by $\wh\sS=\widehat{\sS}_{\eK_p}$ the formal scheme given by the $p$-adic completion of $\sS_{\eK_p}$.

\subsubsection{Torsors and tensors}\label{sss:torsorstensors}
We refer to \cite[App. to \S 19]{Schber} for a discussion of the various notions of a ``$\CG$-torsor".
Recall from Section \ref{ss:shimhodge} that we have
\[
\CG\xrightarrow{\delta} \bar\CG\hookrightarrow \GL(\Lambda).
\]
In this, $\delta$ is the group smoothening (a dilation) which is the identity on generic fibers and $\iota : \bar\CG\hookrightarrow \GL(\Lambda)$
is a closed immersion which
 realizes $\bar\CG$ as the {\cmag stabilizer} of a finite family of tensors $(s_a)\subset \Lambda^\otimes$, $a\in I$, i.e.
\[
\bar\CG(A)=\{g\in \GL(\Lambda\otimes_{\BZ_p}A)\ |\ g\cdot s_a=s_a, \forall a\in I\}
\]
for every $\BZ_p$-algebra $A$. Note that it follows immediately from the construction of the group smoothening 
 that $\delta$ gives a bijection
$
\CG(W(\kappa))\xrightarrow{\sim}\bar\CG(W(\kappa))
$,
for every perfect field $\kappa$, cf. \cite[3.1., Def. 1]{BLR}. {\cmag In fact, we have:
\begin{lemma}\label{BLRLemma} 
The dilation $\delta:\CG\to\bar\CG$ gives a bijection
\begin{equation}\label{barcg}
\delta(W(R)): \CG(W(R))\xrightarrow{\sim}\bar\CG(W(R)),
\end{equation}
for every perfect $\BF_p$-algebra $R$.
\end{lemma}

\begin{proof}
Both $\CG$ and $\bar\CG$ are affine group schemes of finite type over $\BZ_p$. Write $A=\Gamma(\CG, \CO_{\CG})$, $\bar A=\Gamma(\bar\CG, \CO_{\bar\CG})$ for the corresponding affine coordinate rings. The dilation $\delta$ induces an injection $\bar A\subset A$ with $A[1/p]=\bar A[1/p]$ and we can write $A=\bar A[f_1,\ldots, f_r]$, for some $f_1,\ldots, f_r\in A$. There are $n_i\geq 1$ such that $p^{n_i}\cdot f_i\in \bar A$. Since $p$ is not a zero divisor in $W(R)$ we quickly obtain that $\delta(W(R))$ is injective and it remains to show surjectivity: A $W(R)$-point of $\bar \CG$ is given by $\bar h: \bar A\to W(R)$ and we want to show that $\bar h$ extends to $h: A=\bar A[f_1,\ldots ,f_r]\to W(R)$. 
We have $\bar h: A\to W(R)[1/p]$ and $p^{n_i} \bar h(f_i)\in W(R)$. The existence of $h$ follows if, for all $i$, the element $\bar h(p^{n_i}f_i)$ lies in $p^{n_i} W(R)$; then $\bar h: A\to W(R)[1/p]$ takes values in the subring $W(R)\subset W(R)[1/p]$ and gives the desired extension $h$. 
Since $R$ is perfect, $p^nW(R)=V^nW(R)=\{(0,\ldots, 0, r_{n+1}, r_{n+2}, \ldots )\ |\ r_i\in R\}$. This implies that 
an element of $W(R)$ belongs to $p^nW(R)$ if and only if this happens after base change to all (perfect) residue fields $\alpha: R\to \kappa$. Since by the above, $\delta(W(\kappa))$ is a bijection for all perfect fields $\kappa$, $\alpha(\bar h(p^{n_i}f_i))\in p^{n_i}W(\kappa)$ for all such $\alpha$ and the result follows.
\end{proof}

}

Let $S$ be a scheme, or an adic space over $\BZ_p$. Let $\sP$ be a $\CG$-torsor over $S$. Then by the Tannakian formalism, the representation $\CG\rightarrow \GL(\Lambda)$ induces a vector bundle $\sV$ over $S$. The tensors $s_a\in \Lambda^\otimes$ induce corresponding
tensors $t_a\in \sV^{\otimes}(S)$. We can consider the sheaf of tensor-compatible trivializations of $\sV$,
\[
\bar\sT(\sV, (t_a))=\underline{\rm Isom}_{(t_a), (s_a\otimes 1)}(\sV, \Lambda\otimes_{\BZ_p}\CO_{S}) .
\]
This has a natural action of $\bar\CG$ and can be identified with the $\bar\CG$-torsor $\bar\sP$ given as the push-out
$
\bar\sP=\bar\CG\times^{\CG}\sP$
of $\sP$
by $\delta: \CG\to\bar\CG$. Note that, since $\delta[1/p]={\rm id}$, we have 
\[
\bar\sP[1/p]=\sP[1/p].
\]

Conversely, suppose that we are given a vector bundle $\sV$ over $S$
and a collection of tensors $t_a\in \sV^{\otimes}(S)$, $a\in I$, where $t_a$ has the same homogeneity as $s_a$. Then, we can consider the (fppf or \'etale) sheaf
\[
\bar\sT(\sV, (t_a)):=\underline{\rm Isom}_{(t_a), (s_a\otimes 1)}(\sV, \Lambda\otimes_{\BZ_p}\CO_{S})
\]
whose $T$-valued points for $T\to S$ are given by isomorphisms 
\[
f: \sV_T\xrightarrow{\sim}  \Lambda\otimes_{\BZ_p}\CO_{T}
\]
such that $f^{\otimes}(t_a)=s_a\otimes 1$, for all $a\in I$. There is an obvious  left action of $\bar\CG$ on $\bar\sT(\sV, (t_a))$. If $\bar\sT(\sV, (t_a))(T)$ is not empty, then
the action of $\bar\CG(T)$ on the set $\bar\sT(\sV, (t_a))(T)$
is free and transitive, i.e. $\bar\sT(\sV, (t_a))(T)\simeq \bar\CG(T)$. Under certain conditions on $T$, we will have 
$\CG(T)=\bar\CG(T)$, as for example in (\ref{barcg}) above. Then $\bar\sT(\sV, (t_a))(T)$ also acquires a free and transitive action of $\CG(T)$.

 \subsubsection{Variant for Witt vectors}\label{sss:Witt} We will apply the previous remarks not to $\CG$-torsors on $S$ but rather to $\CG$-torsors over $S\bdtimes \BZ_p$ or, more generally, over $\CY_I(S)$.   Let $S=\Spa(R, R^+)\in {\rm Perfd}_k$. If $\sP$ is a $\CG$-torsor over $\CY_I(S)$ (given for a example by a $\CG$-shtuka
 over $S$), we obtain a vector bundle $\sV$ over $\CY_I(S)$, together with tensors $(t_a)$.
 We can consider the sheaf of tensor-compatible trivializations of $\sV$,
\[
\bar\sT(\sV, (t_a))=\underline{\rm Isom}_{(t_a), (s_a\otimes 1)}(\sV, \Lambda\otimes_{\BZ_p}\CO_{\CY_I(S)}) .
\]
This has a natural action of $\bar\CG$ and can be identified with the $\bar\CG$-torsor $\bar\sP$ given as the push-out
$
\bar\sP=\bar\CG\times^{\CG}\sP$
of $\sP$
by $\delta: \CG\to\bar\CG$. 

We will also use the following:

\begin{proposition}\label{propvstack}
Let $R$ be a  perfect  $k$-algebra.
Then   $\CG$-torsors over $W(R)$  form a stack for the $v$-topology on $\Spec(R)$. 
\end{proposition}
\begin{proof}
By \cite{BS} vector bundles of fixed rank over $W(R)$ form 
 a stack for the $v$-topology on $\Spec(R)$. By the Tannakian equivalence, $\CG$-torsors over $W(R)$ are given by exact tensor functors ${\rm Rep}({\CG})\to \hbox{\rm fin. proj. } W(R)$-modules (\cite[Thm. 19.5.1]{Schber}). Using these facts we observe that it remains to show that if $R\to R'$ is $v$-surjective, a complex
 \[
 M_\bullet: 0\to M_1\to M_2\to M_3\to 0
 \]
 of finite projective $W(R)$-modules is exact if and only if the base change $M_\bullet\otimes_{W(R)}W(R')$ is exact.
 Observe that under our assumption, all maximal ideals of $\Spec(W(R))$ are in the image of $\Spec(W(R'))\to \Spec(W(R))$.
 Also $R\to R'$ is dominant, hence injective (since perfect algebras are reduced) and so $W(R)\to W(R')$ is also injective.
 The result now follows by applying Lemma \ref{exactLemma}. 
  \end{proof}

\begin{lemma}\label{lemmabarcg}
Let $R$ be a  perfect $k$-algebra  and let $\sP$, $\sP'$ be two $\CG$-torsors over $W(R)$, inducing $(\sV, (t_a))$, resp. $(\sV', (t'_a))$.
Let $\phi: \sV\xrightarrow{\sim} \sV'$ be  an isomorphism which preserves the corresponding
tensors, i.e. $\phi^{\otimes}(t_a)=t'_a$, $\forall a\in I$. Then $\phi$ is obtained from an isomorphism of $\CG$-torsors
$\tilde\phi: \sP\xrightarrow{\sim}\sP'$ which is unique.
\end{lemma}

 \begin{proof} 
 By Proposition \ref{propvstack}, it is enough to produce
 the isomorphism $v$-locally on $R$.
 Consider the natural map
 \[
  {\rm Isom}_{\CG}(\sP, \sP')\to  {\rm Isom}_{\bar\CG}(\bar\sP, \bar\sP').
 \]
 The claim is that this is bijective. The source, resp. target, are the global sections 
 of $\underline{\rm Isom}_{\CG}(\sP, \sP')$, resp. $\underline{\rm Isom}_{\bar\CG}(\bar\sP, \bar\sP')$,
 which are $\underline{\rm Aut}_{\CG}(\sP)$-, resp. $\underline{\rm Aut}_{\bar\CG}(\bar\sP)$-torsors.
 These group schemes are ``pure" inner forms of $\CG$, resp. $\bar\CG$
 (see  \cite[Prop. III 4.1]{FS}). They split $v$-locally on $R$
 because the torsors $\sP$, $\sP'$ do. Since  $\CG(W(R))=\bar\CG(W(R))$ by (\ref{barcg}),
 we deduce that $\underline{\rm Aut}_{\CG}(\sP)=\underline{\rm Aut}_{\bar\CG}(\bar\sP)$.
 We can now see $\underline{\rm Isom}_{\CG}(\sP, \sP')=\underline{\rm Isom}_{\bar\CG}(\bar\sP, \bar\sP')$
 since they are torsors for the same group and the result follows.
 \end{proof}

\begin{remark} By \cite[Prop. 19.5.3]{Schber},  for $S\in {\rm Perfd}_k$,  $\CG$-torsors on $S\bdtimes \BZ_p$ form a $v$-stack
 over $S$. Hence there is  a variant of Lemma \ref{lemmabarcg} for $\CG$-torsors $\sP$, $\sP'$ over $S\bdtimes \BZ_p$ instead of $W(R)$. 
 Indeed, it is enough to produce $\tilde\phi$ $v$-locally, and so we can see it is enough to show the result for
 $S=\Spa(C, C^+)$, where $C$ is an algebraically closed perfectoid field. Then, we can construct $\tilde\phi$ by Beauville-Laszlo glueing along $p=0$: We see that
 $\phi$ defines $\tilde\phi[1/p]:\sP[1/p]\xrightarrow{\sim} \sP'[1/p]$ while the
 completion $\hat\CO_{\CY_{[0,\infty)}(C, C^+), p=0}$ is $W(C)$ and we have $\CG(W(C))=\bar\CG(W(C))$.
\end{remark}

 \subsubsection{Extension of $\sP_E$}\label{ext463}
  Recall  the $\CG$-shtuka $\sP_E$ over the generic fiber ${\rm Sh}_E^\diam={\rm Sh}_\eK(\eG, X)_E^\diam$.
 By the Tannakian formalism as above, $\iota: \CG\rightarrow \GL(\Lambda)$ and $\sP_E$ give a vector bundle shtuka $(\sV_E, \phi_{\sV_E})$
 over the generic fiber ${\rm Sh}_E^\diam$. This is the vector space shtuka that corresponds to the de Rham local system given by   the Tate module of the pull-back of the universal abelian scheme via the Hodge embedding, cf. \S \ref{ss:dR}. Again as above, $(\sV_E, \phi_{\sV_E})$
 is endowed with a finite family of tensors $t_{a, E}\in (\sV_E, \phi_{\sV_E})^\otimes$. We can view each such tensor as a shtuka homomorphism  over ${\rm Sh}_E^\diam$,
 \begin{equation}\label{tensors}
 t_{a, E}:  (\oplus_i \sV^{\otimes m_i}_E, \phi_{\oplus_i \sV^{m_i}_E})\to (\oplus_i \sV^{\otimes n_i}_E, \phi_{\oplus_i \sV^{\otimes n_i}_E}),
 \end{equation}
 for suitable  $m_i, n_i\geq 1$. By the discussion in Section \ref{sss:Witt}, we have
 \begin{equation}\label{PQ}
\bar \sP_{E}=\bar \sT(\sV_E, (t_{a, E})).
 \end{equation}
 Recall from  (\ref{intHodgemb}) the   finite morphism
 \[
i: \sS_\eK\to   \CA_{\eK^\flat} \otimes_{\BZ_{(p)}}O_{E} ,
 \]
which extends the Hodge embedding in the generic fibers. 

We first extend $\sV_E$ to a vector shtuka $(\sV, \phi_\sV)$ over $\sS$ as follows. First observe that it is enough to extend compatibly over the $p$-adic completion $\widehat\sS$. Now consider $S=\Spa(R, R^+)\in{\rm Perfd}_k$ and a map $S\to (\widehat\sS)^\diam$ given by $\Spa(R^\sharp, R^{\sharp +})\to \widehat\sS^{\rm ad}$. Let $M_{\inf}(R^{\sharp +})$ be the BKF-module (with leg along $\phi(\xi)=0$) of  the pull-back to $\Spec(R^{\sharp +})$ of the $p$-divisible group $A[p^\infty]$ of the universal abelian scheme over $\CA_{\eK^\flat} \otimes_{\BZ_{(p)}}O_{E}$, cf. Example \ref{pdivExample}.
  This is a finite locally free module  over $A_{\inf}(R^{\sharp +})=W(R^+)$. We denote by $(\sV_S, \phi_{\sV_S})$  the corresponding minuscule shtuka of height $2g$ and dimension $g$ over $S$ with leg at $S^\sharp$, given by the restriction of
\begin{equation}\label{realM}
 M(W(R^+))=(\phi^{-1})^*M_{\inf }(R^{\sharp +})
  \end{equation}
   to $\Spa(W(R^+))\setminus \{[\varpi]=0\}$  (as in Example \ref{pdivExample}). 
   
   Using \eqref{tensors}, we see that Theorem \ref{vshtExt} implies that  the tensors $t_{a, E}$ extend
(uniquely) to tensors $t_a\in \sV^\otimes$. We  now define the $v$-sheaf over $S\bdtimes \BZ_p$ with action of $\CG$,
     \begin{equation}\label{Gtors}
  \bar\sP_S=\bar\sT(\sV, (t_a))= \underline{\rm Isom}_{(t_a), (s_a\otimes 1)}(\sV_S, \Lambda\otimes_{\BZ_p}\CO_{S\bdtimes \BZ_p}) .
  \end{equation}
We will show that $\bar\sP_S$ is induced by a $\CG$-torsor $\sP_S$ over $S\bdtimes \BZ_p$,
which is then uniquely determined. We do this in three steps.
 
\emph{Step 1.} Let $R$ be a perfect $k$-algebra and let $Z=\Spec(R)\to \sS$ 
which induces $\Spd(R)\to (\widehat\sS)^\diam$. Note that the pull-backs of the tensor powers $\sV^{\otimes m}$
to shtukas over $\Spd(R)$ are given as in  Theorem \ref{FFisocrystal} by meromorphic $F$-crystals, namely
the tensors powers $M(W(R))^{\otimes m}$ of $M(W(R))$. Note that, as in \eqref{Dnatural}, $M(W(R))$ can be identified with 
\[
\BD^\natural(W(R))=(\phi^{-1})( \BD(W(R))^*),
\]
 where $\BD(W(R))^*$ is the linear dual of the contravariant Dieudonn\'e module of the pullback of the universal $p$-divisible group $A[p^\infty]$ to $\Spec(R)$. 
By pulling back $t_a$ along
this map and using the full-faithfulness  
\[
\hbox{\rm meromorphic $F$-crystals over $R$}\to \hbox{\rm shtukas over $\Spd(R)$}
\]
given by Theorem \ref{FFisocrystal}, we obtain tensors $t_{a, \crys}\in \BD^\natural(W(R))^\otimes$. We can now consider the affine scheme with $\CG$-action over $W(R)$, 
\[
\sT_{\crys}(R):=\underline{\rm Isom}_{(t_{a,\crys}), (s_{a}\otimes 1)}( \BD^\natural(W(R)), \Lambda\otimes_{\BZ_p}W(R)) .
\]
 We also consider the corresponding  affine scheme with $G$-action over $W(R)[1/p]$,
\[
T_{\crys}(R):=\sT_{\crys}(R)[1/p]:=\underline{\rm Isom}_{ (t_{a,\crys}),(s_{a}\otimes 1)}( \BD^\natural(W(R))[1/p],\Lambda\otimes_{\BZ_p}W(R)[1/p]) ,
\]
which has a natural  Frobenius action lifting the Frobenius on $W(R)$. \
\begin{lemma}\label{GFisoLemma}
The $G$-scheme $T_{\crys}(R)$ is a $G$-torsor over $\Spec(W(R)[1/p])$. 
\end{lemma}

Using the Tannakian equivalence we see that the lemma implies that $T_{\crys}(R)$ gives a Frobenius $G$-isocrystal over 
the perfect scheme $Z=\Spec(R)$. Recall that a Frobenius $G$-isocrystal over $Z$ is a exact faithful tensor functor
\[
{\rm Rep}_{\BQ_p}(G)\to \text{$F$-}{\rm Isoc}(Z),
\]
where $\text{$F$-}{\rm Isoc}(Z)$ is the tensor category of Frobenius isocrystals. However, our argument proceeds in the opposite way. 

\begin{proof}
 We paraphrase the argument of  \cite[Cor. 1.3.12]{KMS}. Suppose that $x_0$ is a geometric point of $Z=\Spec(R)$. Let $K=\kappa(x_0)$. After replacing $Z$ by $Z\otimes_kK$, we may assume that $Z$ is defined over $K$ and that $x_0$ is a $K$-rational point.  There is a specialization functor
\begin{equation}\label{Fspec}
{\rm spec}_{x_0}\colon \text{$F$-}{\rm Isoc}(Z)\to \text{$F$-}{\rm Isoc}(x_0)
\end{equation}
between the categories of Frobenius isocrystals obtained by restricting to the point $x_0$. We may assume that $Z$ is connected. 
Then this functor is faithful and exact, and is in fact a fiber functor over $L=W(K)[1/p]$, after forgetting the Frobenius.
The $W(R)[1/p]$-scheme $T_\crys(R)$ specializes to  $T_\crys(x_0)$ which is an affine $L$-scheme with $G$-action. Lift $x_0$ to a point $\tilde x_0$ of $\sS$ with values in a complete discrete valuation ring $O_K$. Using the \'etale-crystalline comparison isomorphism applied to the crystalline representation of ${\rm Gal}(\bar K/K)$  given by the tensor powers of the Tate module $T_p(A[p^\infty])\otimes_{\BZ_p}\BQ_p$, we see that $T_\crys(x_0)$ is a $G$-torsor and so it gives a Frobenius $G$-isocrystal over $x_0$, i.e. the exact tensor functor 
$$
\omega_{T_0}\colon {\rm Rep}_{\BQ_p}(G)\to \text{$F$-}{\rm Isoc}(x_0), \quad W\mapsto W\times^G T_\crys(x_0) .
$$
 We claim that  $\omega_{T_0}$ factors as a composition of tensor functors
$$
\omega_{T_0}={\rm spec}_{x_0} \circ\omega,
$$
 where the tensor functor $\omega\colon  {\rm Rep}_{\BQ_p}(G)\to \text{$F$-}{\rm Isoc}(Z)$ takes the representation $\Lambda_\BQ$ of $G$ to $\BD^\natural(W(R))[1/p]$. This last condition forces on us the definition of $\omega(V_{m, n})$, where $V_{m, n}=\Lambda_\BQ^{\otimes n}\otimes \Lambda_\BQ^{* \otimes m} $. But any object of ${\rm Rep}_{\BQ_p}(G)$ is the kernel of a $G$-invariant map $e\colon W\to W$, where $W$ is a direct sum of objects of the form $V_{m, n}$; we can even assume that $e$ is an idempotent. By considering $e$ as a tensor to which we apply the previous construction, we define
a homomorphism $\omega(e)\colon\omega(W)\to \omega(W)$ of $F$-isocrystals over $Z$ and set $\omega(\Ker e)=\Ker(\omega(e))$. It is easy to see that this definition is independent of the presentation and defines an exact faithful tensor functor into the category $\text{$F$-}{\rm Isoc}(Z)$. We claim that the $G$-space $T_\crys(Z)$ is isomorphic to the  $G$-torsor which corresponds to $\omega$. It suffices to prove that the corresponding pushout functors coincide,
\begin{equation*}
\begin{aligned}
&{\rm Rep}_{\BQ_p}(G)\to \text{$W(R)[1/p]$-modules}\\
& W\mapsto W\times^G T_\crys(Z),\quad \quad W\mapsto \omega(W) .
\end{aligned}
\end{equation*}
However, by construction, both functors send $\Lambda_\BQ$ to $\BD(W(R))[1/p]$ and the morphisms 
$ s_{a}:  \oplus_i\, \Lambda_\BQ^{\otimes m_i}\to \oplus_i\, \Lambda_\BQ^{\otimes n_i}$ to $ t_{a, \crys}:  \oplus_i\, \BD^\natural(W(R))[1/p]^{\otimes m_i}\to \oplus_i \,\BD^\natural(W(R))[1/p]^{\otimes n_i}$. Since the tensor algebra of $\Lambda_\BQ$ and the tensors $s_a$ allow us to recover the group algebra of $G$, it follows that both functors coincide.  \end{proof}

\emph{Step 2.} Now consider $\Spa(C, C^+)$ with $C$ a complete non-archimedean algebraically closed field of characteristic $0$
 and with tilt $\Spa(C^\flat, C^{\flat +})\in  {\rm Perfd}_k$. Let
 \[
 x: \Spa(C, C^+)\to S^\sharp=\Spa(R^\sharp, R^{\sharp +})\to \widehat\sS^{\rm ad}
 \]
 be a morphism giving the point $\Spa(C^\flat, C^{\flat +})\to S\to (\widehat\sS)^\diam/\Spd(O_E)$.
 Since $C^+\subset O_C$, this also gives $\Spa(C, O_C)\to \Spa(C, C^+)\to \widehat{\sS}^{\rm ad}$,
 which we still denote by $x$. We set
 \[
 \bar C^+=C^+/\mathfrak m_C\subset k(C)=k(C^\flat)
 \]
 which is a valuation ring of $k(C)=O_C/\mathfrak m_C$. 
   
 As in Example \ref{pdivExample}, we have the Breuil-Kisin-Fargues module $M$  over $A_{\rm inf}(C^+)$ of the pull-back  $x^*(A[p^\infty])$ of
 the universal $p$-divisible group. By the existence of the $\CG$-shtuka over the generic fiber $\sS_E$, we have  a $\CG$-torsor 
 \begin{equation}\label{sTinf}
 \sT_{[0,\infty)}=\sT_{[0,\infty)}(C, C^{+})
 \end{equation}
  over $\CY_{[0, \infty)}(C^\flat, C^{\flat +})$. This induces a vector bundle $\sV_{[0,\infty)}$ given by the pullback of $M$ to $\CY_{[0, \infty)}(C^\flat, C^{\flat +})$ and Frobenius invariant
 tensors $x^*(t_a)\in \sV_{[0,\infty)}^\otimes$. By our construction of $t_a$ (see the proof of Proposition \ref{Extperfd}) the tensors $x^*(t_a)$ extend to tensors over the pullback of $M^\otimes$ to 
 \[
 \CY_{[0,\infty]}(C^\flat, C^{\flat +})=\Spa(A_{\rm inf}(C^+))\setminus \{[\varpi]=0, p=0\}.
 \]
 Furthermore, by full-faithfulness, these extend to tensors $t_{a, x}\in M^\otimes$ over $A_{\rm inf}(C^+)$.

 \begin{lemma} 
The $\CG$-torsor $\sT_{[0,\infty)}(C, C^{+})$ from \eqref{sTinf} extends to 
a scheme theoretic $\CG$-torsor $\sT(C^+)$ over $A_{\rm inf}(C^+)$  with the property that 
the construction of \S \ref{sss:Witt} applied to $\sT(C^+)$ gives $M$ with the tensors $t_{a, x}\in M^{\otimes}$. 
This extension is canonical, i.e unique, up to a unique isomorphism. The $\CG$-torsor $\sT(C^+)$ is trivial, since 
 $W(C^+)$ is strictly henselian. 
 \end{lemma}
 \begin{proof}

i) By Lemma \ref{GFisoLemma}, applied to $R=\bar C^+$, the tensors $t_{a, \crys}$ give  a $G$-torsor
 \[
 T_\infty=T_{\rm crys}(\bar C^+)
 \]
  over $W(\bar C^+)[1/p]$ which underlies a Frobenius $G$-isocrystal over $\bar C^+$.
  
ii) By Proposition \ref{exttoW} and the main result of \cite{An}, the $\CG$-torsor $\sT_{[0,\infty)}$
extends to a (trivial) $\CG$-torsor $\sT(O_C)$ over $A_{\rm inf}(O_C)=W(O_{C^\flat})$. Using the faithfulness of the pullback along 
\[
\CY_{[0, \infty)}(C^\flat, O_C^{\flat})\to \CY_{[0, \infty]}(C^\flat, O_C^{\flat})\to  \Spec(A_{\rm inf}(O_C)),
\]
on vector bundles and their homomorphisms, we see that this $\CG$-torsor gives, by the
construction of \S \ref{sss:Witt}, the BKF-module $M\otimes_{A_{\rm inf}(C^+)}A_{\rm inf}(O_C)$ of $x^*(A[p^\infty])$ with the tensors $t_{a, x}\in M^{\otimes}\otimes_{A_{\rm inf}(C^+)}A_{\rm inf}(O_C)$.

 The idea now is that $T_\infty$, $\sT_{[0, \infty)}$, and $\sT(O_C)$ combine to give 
 the $\CG$-torsor $\sT(C^+)$, by considering the fibered product
 \[
 C^+=\bar C^+\times_{k(C)} O_C ,
 \]
 which gives
\[
A_{\rm inf}(C^+)[1/p]=W(\bar C^+)[1/p]\times_{W(k(C))[1/p]} A_{\rm inf}(O_C)[1/p].
\]
Set 
\[
A=A_{\rm inf}(C^+)[1/p],
\]
\[
A_1=W(\bar C^+)[1/p], \quad 
A_2=A_{\rm inf}(O_C)[1/p], \quad A_0=W(k(C))[1/p]
\]
so that we have
\[
A=A_1\times_{A_0} A_2\ .
\]
Set
\[
\sT[1/p]=\sT(C^+)[1/p]={\rm Isom}_{(t_{a, x}), (s_a\otimes 1)}( M[1/p],\Lambda\otimes_{\BZ_p}A), 
\]
\[
 \sT_i[1/p]=\sT[1/p]\otimes_A A_i={\rm Isom}_{ (t_{a, x, i}),(s_a\otimes 1)}( M_i[1/p], \Lambda\otimes_{\BZ_p}A_i). 
\]
For each $i=0, 1, 2$, we know  by the above that $\sT_i[1/p]$ is a $G$-torsor. We now use \cite[Lem. 11.3]{An}, according to which the functor that associates to a perfect valuation ring $R$ over $k$ the groupoid of $G$-torsors over $W(R)[1/p]$ is a stack for the arc topology. This
implies that  $\sT_i[1/p]$, $i=0, 1, 2$, ``glue" to give a $G$-torsor $\sT'$ over $A$. The $G$-torsor $\sT'$ can also be described by 
a finite projective $A$-module $M'$ with tensors $t'_{a, x}\in M'^\otimes$. By its construction and the fact that $M$ and $t_a\in M^\otimes$ also
satisfy the required compatibilities with $M_i$ and $t_{a, x, i}\in M^\otimes_i$, we see that there is an isomorphism $M\simeq M'$
taking $t_{a, x}$ to $t_{a, x}'$. Therefore, $\sT[1/p]\simeq \sT'$ and $\sT[1/p]$ is also a $G$-torsor. Now the $G$-torsor $\sT[1/p]$ glues with the $\CG$-torsor $\sT_{[0,\infty)} $ over $\CY_{[0,\infty)}(C^\flat, C^{\flat +})$ to produce
a $\CG$-torsor over $\CY_{[0, \infty]}(C^\flat, C^{\flat +})=\Spa(A_{\rm inf}(C^+))\setminus \{[\varpi]=0, p=0\}$. This, using GAGA and  \cite[Cor. 11.6]{An}, 
extends to a $\CG$-torsor $\sT(C^+)$ over $\Spec(A_{\rm inf}(C^+))$ which is trivial.
By the full-faithfulness of the restriction to $\CY_{[0, \infty]}(C^\flat, C^{\flat +})$
and the construction, we see that the $A_{\rm inf}(C^+)$-module, resp. the tensors, obtained from the $\CG$-torsor $\sT(C^+)$ by the Tannakian formalism is $M$, resp. $t_{a, x}\in M^{\otimes}$. The rest of the properties of $\sT(C^+)$ 
in the statement also follow from the above full-faithfulness.
 \end{proof}
 
\emph{Step 3.} We now show that we obtain a corresponding $\CG$-torsor in characteristic $p$ also. Take $\Spa(D, D^{+})\in {\rm Perfd}_k$ 
given by an algebraically closed affinoid field $D$, equal to its untilt, 
and a point
\[
\Spa(D, D^+)\to \widehat{\sS}.
\]
We can lift this to a point of $\widehat\sS$ with values in $\Spa(C, C^+)$ with $C$ as above (of characteristic $0$), and with $\bar C^+=D^{+}$. Then, by Step 2, we obtain a $\CG$-torsor over $A_{\rm inf}(C^+)=W(C^{+\flat})$ which, via $C^{+\flat}/p\to D^+$, reduces to a $\CG$-torsor over $W(D^{+})$. By its construction, this (trivial) $\CG$-torsor $\sT(D^+)$ underlies the 
Dieudonn\'e crystal $\BD^\natural$ of the pull-back of the universal $p$-divisible group and the tensors $(t_{a, \crys})$. 
The induced $\bar\CG$-torsor $\bar\sT(D^+)$ is uniquely determined (does not depend on the lifting) 
as we can see by using the tensors $(t_{a,\crys})$, and so is the $G$-torsor $\sT(D^+)[1/p]=\bar\sT(D^+)[1/p]$
which in fact is trivial.
Since $\CG(W(D^+))=\bar\CG(W(D^+))$, it follows that the $\CG$-torsor $\sT(D^+)$ is also uniquely determined and does not depend on the choice of lifting above, cf. Lemma \ref{lemmabarcg}.
\smallskip

 Consider now 
 \[
 y: T=\Spa(B, B^+)\to S\to  (\widehat\sS)^\diam/\Spd(O_E) ,
 \]
  a corresponding product of points as in Steps 2 and 3 above with
 $B^+=\prod_j C^+_j$ such that $T\to S$ is a $v$-cover. We first give a $\CG$-torsor $\sP_T$ over $T\bdtimes \BZ_p$.
 By the work above, we have (trivial) $\CG$-torsors $\sP_{j}=\sT(C^+_j)$ over $A_{\rm inf}(C_j^+)=W(C^{+\flat}_j)$, for all $j$.
 By using the Tannakian formalism and the fact that families $(M_j)$ of finite free $W(C^{+\flat}_j)$-modules of constant rank 
 correspond to finite free modules over $\prod_j W(C^{+\flat}_j)=W(B^+)$, we obtain a $\CG$-torsor over $W(B^+)$
 whose restriction along the $j$-th component gives $\sP_j$.
 By restriction from $W(B^+)$ to $T\bdtimes \BZ_p$, this gives the $\CG$-torsor $\sP_T$ over $T\bdtimes \BZ_p$.

 By \cite[Prop. 19.5.3]{Schber}, $\CG$-torsors on $S\bdtimes \BZ_p$ form a $v$-stack
 over $S$ and so there is an equivalence of categories between $\CG$-torsors on $S\bdtimes \BZ_p$
 and $\CG$-torsors on $T\bdtimes \BZ_p$ with suitable descent data. Here, we can obtain a descent datum on $\sP_T$
 by using Lemma \ref{lemmabarcg} and that both the underlying vector bundle $\sV_T$ and the tensors $t_a\in \sV^\otimes_T$ have such descent data, since they are obtained by base-change from $S\bdtimes\BZ_p$. This then gives a $\CG$-torsor $\sP_S$ over $S\bdtimes\BZ_p$ which,
 in turn, gives $\sV_S$ and $(t_a)\in \sV^\otimes_S$. The $\CG$-torsor $\sP_S$ supports a Frobenius structure $\phi_{\sP_S}$
 obtained from the Frobenius structure of $\sV_S$, and $(\sP_S, \phi_{\sP_S})$ is a $\CG$-shtuka over $S$. Also, 
 $(\sP_S, \phi_{\sP_S})$ has leg bounded by 
 $\mu$. Indeed,  this is true in the generic fiber by construction. In general, it follows by reducing to $(C, O_C)$-valued points 
 and using \S \ref{ideaSpecialization}. (Note  that $\sV_S$ has dimension $g$ and height $2g={\rm rank}_{\BZ_p}(\Lambda)$, and that $(\BM^{\rm loc}_{\CG, \mu})^\diam(C, O_C)\subset {\rm Gr}(g, \Lambda)^\diam_{O_E}(C, O_C)$.) The association $S\mapsto (\sP_S, \phi_{\sP_S})$ gives the desired $\CG$-shtuka over $(\widehat\sS)^\diam$. 
 This concludes the proof of the existence of the extension $\sP$ of the $\CG$-shtuka $\sP_{E}$ of $\sS$.
   \hfill $\square$

\subsection{Local completions}\label{ss:complhodge}
Recall the $\CG$-shtuka $\sP_{\eK, E}$ over $ {\rm Sh}_{\eK}(\eG,X)_E$.
By our work above, this extends to a $\CG$-shtuka $\CP_{\eK}$ over ${\sS}_{\eK}$.
For simplicity of notation, we omit the subscript $\eK$ of $\sS_\eK$ in what follows.

\subsubsection{Completed local rings}\label{sss:comploc} Take $x\in \sS(k)$ and consider  the strict completion of the local ring $ \hat R_x=\hat\CO_{\sS, x}$ at the corresponding point of $\sS$.
Then $\hat R_x$ is a Noetherian normal integral local domain. The induced morphism
\[
\Spec(\hat R_x)\to \Spec(\hat\CO_{\CA_{\eK^\flat} \otimes_{\BZ_{(p)}}O_{E}, i(x)})
\]
is finite. Set $\hat A_x=\hat\CO_{\CA_{\eK^\flat} \otimes_{\BZ_{(p)}}O_{E}, i(x)}$ for simplicity. Also set $b=b_x$.

\begin{proposition}\label{prop451}
There is a $\Spd(\hat R_x)$-valued point  of $\CM^{\rm int}_{\CG, b, \mu}$
such that the corresponding $\CG$-shtuka over $\Spd(\hat R_x)$ is equal to the pull-back of 
the $\CG$-shtuka $\CP_\eK$ via the morphism $\Spd(\hat R_x)\to \widehat \sS^\diam/\Spd(O_E)$,
and which lifts the base point $x_0\colon\Spd(k)\to \CM^{\rm int}_{\CG, b, \mu}$.
\end{proposition}
Here we recall that $\Spd(\hat R_x)$ denotes the $v$-sheaf $\Spa(\hat R_x, \hat R_x)^\diam$. Note that $\Spd(\hat R_x)$ is quasi-compact and can be covered in the $v$-topology
by a finite union of representable affinoid perfectoids $S$. Then the statement above essentially amounts
to showing that there is $r\gg0$ such that the pull-back of the $\CG$-shtuka $\CP_\eK$  
under  morphisms $S\to \Spd(\hat R_x)\to \widehat \sS^\diam/\Spd(O_E)$ admits a compatible
(equivalence class of) trivialization(s) $i_r$ over the sectors  $\CY_{[r, \infty)}(S)$ which lift the trivializations induced by the
base point.

 Assuming this for the moment, we can show part (c) of Conjecture \ref{par812} for 
 the models $\sS_\eK$ as defined above. 
 Indeed, by Proposition \ref{prop451} we obtain using the definition of ${\mathcal M}^{\rm int}_{\CG, b, \mu}$ as a moduli functor, a morphism of $v$-sheaves
\[
\Psi_G: \Spd(\hat R_x)\to  {\mathcal M}_{\CG, b, \mu}^{\rm int} ,
\]
where $b=b_x$. This fits into a commutative diagram

 \begin{equation}\label{SCompletions}
 \begin{aligned}
   \xymatrix{
       (\wh{\sS}_{/x})^\diam=\Spd(\hat R_x)    \ar[r]^{\ \ \  \ \Psi_{G, x}} \ar[d]_{i_x} &  \wh{{\mathcal M}^{\rm int}_{\CG, b, \mu}}_{/x_0}\ar[d]^{\wh i_{/x_0}} \\
     \Spd(\hat A_x) \ar[r]^{\Psi_{H, x} \ \ \ \ \ \ \ \ \ \ \ }&\wh{{\mathcal M}^{\rm int}_{\CH, i(b), i(\mu)}}_{/x_0}\times_{\Spd(\breve\BZ_p)}\Spd(O_{\breve E}).
        }
        \end{aligned}
    \end{equation}

In this diagram of $v$-sheaves over $\Spd(\CO_{\breve E})$, the map $\Psi_{H, x}$ is an isomorphism,
and the two vertical maps $i_x$ and $\wh i_{/x_0}$ are closed immersions. 
The generic fibers of all four $v$-sheaves in the diagram are representable by smooth rigid analytic spaces over $\breve E$.
Both generic fibers $\Spd(\hat R_x)_\eta$ and $(\wh{{\mathcal M}^{\rm int}_{\CG, b, \mu}}_{/x_0})_\eta$ are smooth of the same dimension  and by the above, are closed in $\Spd(\hat A_x)_\eta$. By Proposition \ref{topflat} (3), $(\wh{{\mathcal M}^{\rm int}_{\CG, b, \mu}}_{/x_0})_\eta$
is connected, hence $\Spd(\hat R_x)_\eta=(\wh{{\mathcal M}^{\rm int}_{\CG, b, \mu}}_{/x_0})_\eta$.
It then follows by  Proposition \ref{topflat} (2) (topological flatness) that $|\Spd(\hat R_x) |$ and $|\wh{{\mathcal M}^{\rm int}_{\CG, b, \mu}}_{/x_0}|$
agree as closed subsets of $|\Spd(\hat A_x)|$; then $\Psi_{G, x}$ is an isomorphism by \cite[Lemma 17.4.1]{Schber}
(or \cite[Prop. 12.15 (iii)]{Sch-Diam}).

This shows part (c) of Conjecture \ref{par812} for the
 integral model $\sS_\eK$, by taking $\Theta_{ x}=\Psi_{G, x}^{-1}$. In fact,  this also shows that
 $\Spd(\hat R_x)$ is isomorphic to  the formal completion $\wh{{\mathcal M}^{\rm int}_{\CG, b_x, \mu}}_{/x_0}$, which is therefore representable. 
\medskip

\subsubsection{Proof of Proposition \ref{prop451}} We need some preliminaries in which we use set-up and notations from \cite{FS}.

 Let $S=\Spa(R, R^+)\in {\rm Perfd}_k$.
In the following, we also need to use the ``schematic" (or ``algebraic") Fargues-Fontaine curve $X^{\rm alg}_S$
defined, for example, in \cite[II.2.3]{FS} (see also \cite{KL}).  Recall that by the GAGA type theorem \cite[Thm. 6.3.9]{KL}
(also \cite[Prop. II. 2.7]{FS}), there is an exact tensor equivalence of categories between vector bundles on $X^{\rm alg}_S$
and on $X_S$. Hence, by the Tannakian formalism, we have a similar equivalence between categories of $G$-torsors. In addition, vector bundles that correspond under the equivalence have isomorphic cohomology groups.
We now 
consider the automorphism group schemes
\[
\sG_b:=\underline\Aut_{G}(\CE_b),\quad \sH_b:=\underline\Aut_{H}(\CE_b)
\]
over $X^{\rm alg}_{S}$. 
These are forms of $G$, resp. $H$, over $X^{\rm alg}_S$. We can also consider
\[
\sG^{\geq 0}_b=\underline\Aut_{G, {\rm filt}}(\CE_b),\quad \sH^{\geq 0}_b=\underline\Aut_{H, {\rm filt}}(\CE_b)
\]
which are parabolic subgroup schemes of $\sG_b$, resp. $\sH_b$. These group schemes support ``HN filtrations"
 $\sG^{\geq \lambda}_b$,  $\sH^{\geq \lambda}_b$, for $\lambda\geq 0$, defined as in \cite[\S III.5]{FS} . Their global sections 
 (see the proof of Prop. III.5.1) are, for $S\in {\rm Perfd}_k$,
 \[
 \sG^{\geq \lambda}_b(X^{\rm alg}_S)=\wt G_b^{\geq \lambda}(S) ,
 \]
with $\wt G_b^{\geq \lambda}$ as in \cite[Prop. III.5.1]{FS}. The group $v$-sheaves $\wt G_b^{\geq \lambda}$ satisfy
\[
\wt G_b=\wt G_b^{\geq 0}=\wt G_b^{>0}\rtimes \underline {G_b(\BQ_p)},
\]
and  for every $\lambda>0$, there is a natural isomorphism
\[
\wt G_b^{\geq \lambda}/\wt G_b^{> \lambda}\xrightarrow{\sim} {\bf B}(({\rm ad}\CE_b)^{\geq \lambda}/({\rm ad}\CE_b)^{> \lambda}),
\]
with target the Banach-Colmez space associated to the $-\lambda$ isoclinic part of the Frobenius isocrystal 
 $({\rm Lie}(G)\otimes_{\BQ_p}\breve\BQ_p, {\rm Ad}(b)\sigma)$.
 We can also consider the fpqc quotient 
 \[
 \sQ_b^{\geq 0}:=\sH^{\geq 0}_b/\sG^{\geq 0}_b
 \]
  over $X^{\rm alg}_S$. To describe $\sQ_b^{\geq 0}$, we need two lemmas. 
 The proof of the first is left to the reader.
  
  \begin{lemma}\label{twofacts}
 
 i) The quotient $T_{H/G}={\rm Lie}(H)/{\rm Lie}(G)$ is naturally identified with the tangent space of the quotient 
affine scheme $H/G$ at the identity coset. Then $(T_{H/G}\otimes_{\BQ_p}\breve\BQ_p, {\rm Ad}(b)\cdot\sigma)$ 
is a Frobenius isocrystal. 

ii)   The quotient $H_b/G_b$ is represented (as a quotient of a reductive group by a closed
reductive subgroup) by an affine $\BQ_p$-scheme. The tangent space $T_{H_b/G_b}$ at the identity coset can be
identified with the slope zero part of $(T_{H/G}\otimes_{\BQ_p}\breve\BQ_p, {\rm Ad}(b)\cdot\sigma)$.  \hfill $\square$
\end{lemma}

The points $(H_b/G_b)(\BQ_p)$ give a locally profinite set with a continuous action of the group $H_b(\BQ_p)$. 
We will consider the corresponding $v$-sheaf $\underline{(H_b/G_b)(\BQ_p)}$.

\begin{lemma}\label{lemma452} Let $S\in {\rm Perfd}_k$. 
The quotient $\sQ_b^{\geq 0}$ is represented
by a relatively affine scheme over $X_S^{\rm alg}$. 
There is a decreasing exhausting filtration 
\[
\sQ_b^{\geq \lambda}\subset \sQ_b^{\geq 0}
\]
by closed subschemes $\sQ_b^{\geq \lambda}$ over $X^{\rm alg}_S$, for $\lambda\geq 0$, such that

1) We have $\sQ_b^{\geq 0}\simeq ({(H_b/G_b)}\times_{\BQ_p} X^{\rm alg}_S)\times_{X^{\rm alg}_S} \sQ^{>0}_b $.
  
  2) For each $\lambda>0$, there is a surjective morphism
  \[
 f_\lambda:  \sQ_b^{\geq \lambda}\to 
\CV_{H/G, b, \lambda},
\]
with
\[
f_\lambda^{-1}(0)=\sQ_b^{> \lambda}.
\]
Here, $\CV_{H/G, b, \lambda}$ is the vector bundle  over $X^{\rm alg}_S$ which is associated to the $-\lambda$ isoclinic part of the Frobenius isocrystal 
 $(T_{H/G}\otimes_{\BQ_p}\breve\BQ_p, {\rm Ad}(b)\sigma)$.
  \end{lemma}
  
  \begin{proof}
  We set
  \[
  \sQ^{\geq \lambda}_b:=\sH_b^{\geq \lambda}/\sG^{\geq \lambda}_b.
  \]
 We first show that the natural map
  \[
  \sH_b^{\geq \lambda}/\sG^{\geq \lambda}_b\to \sQ^{\geq 0}_b=\sH_b^{\geq 0}/\sG^{\geq 0}_b
  \] 
  is injective, i.e.
  that
  \[
  \sH^{\geq \lambda}_b\cap \sG^{\geq 0}_b=\sG^{\geq \lambda}_b.
  \]
  It suffices to check this on Lie algebras. Then it follows  from
   the fact that $i: G\hookrightarrow H$ is faithful.
   The rest is
  similar to the proof of \cite[Prop. III 5.1]{FS}: We first show the analogue of \cite[Prop. III 5.2]{FS} 
  for the quotient $H/G$. For that we can reduce to the  
  case that the $G$-bundle is trivial. The result then follows as in \emph{loc. cit.}
  by eventually reducing to considering the Lie algebras and using Lemma \ref{twofacts}. 
  \end{proof}

\noindent{\sl We can now prove  Proposition \ref{prop451}.}  Consider the formal scheme $\Spf(R)$, where $R=\hat R_x$. Since $\Spf(R)$ maps to the Rapoport-Zink formal scheme for the symplectic group, we have for every $R$-algebra  $B\in {\rm Nilp}_W$ a universal quasi-isogeny
\[
q: \BX_0\otimes_k B/pB\dashrightarrow \BX\otimes_B B/pB
\]
with $\BX=A[p^\infty]$. We can take the base point here to be $\BX_0=\BX\otimes_R k$ and $q$ the identity modulo $\mathfrak m_R$. Also, the Dieudonn\'e crystal ${\mathbb D}(\BX_0)(W(k))$ supports the 
Frobenius invariant tensors $t_{a, \crys}\in {\mathbb D}(\BX_0)(W(k))^\otimes$. 

By \cite[Thm. 25.1.2, Cor. 25.1.3]{Schber}, the $v$-sheaf of the RZ formal scheme for the symplectic group coincides with  
$ {\CM}^{\rm int}_{\CH, i(b), i(\mu)}$. Therefore, we can obtain an equivalence class of a framing of the corresponding $\CH$-shtuka
over $\Spd(R)$. By following the proof of \cite[Thm. 25.1.2]{Schber}, we see that this is given as follows:

Let $\Spa(B, B^{+})$ be affinoid perfectoid over $k$ and let a point $\Spa(B, B^{+})\to \Spd(R)$ be given by a map from 
$\Spa(B^\sharp, B^{\sharp +})$  to $\Spa(R, R)$, so we have a  continuous map $R\to B^{\sharp +}$. For a pseudo-uniformizer
$\pi_{B^\sharp}$ of $B^\sharp$, we have $R/\frak m^N\to B^{\sharp +}/(\pi_{B^\sharp})=B^+/(\pi_B)$ for some $N$. By evaluating the quasi-isogeny on the Dieudonn\'e crystal associated to the $p$-divisible group $\BX=A[p^\infty]$, we obtain a trivialization over $B^+_{\rm crys}(B^{+}/\pi_{B})$ and so a trivialization of the $\CH$-shtuka over $\CY_{[r, \infty)}(B, B^+)$,
for $r\gg0$.
This trivialization is compatible among all points $\Spa(B, B^+)\to \Spd(R)$, so it gives a framing of the corresponding $\CH$-shtuka
over $\Spd(R)=\Spa(R, R)^\diam$. Note that $\Spd(R)$ accepts a surjective $v$-cover from a
finite union of affinoid perfectoids, so we can pick a single $r\gg0$ such that the trivialization is defined 
over $\CY_{[r,\infty)}$.

We now want to show that this descends to a framing  of the $\CG$-shtuka $\CP$ over $\Spd(R)$, i.e. that the trivialization giving the framing respects the $\CG$-structure. Denote by $\CE=\CP_{FF}$ the
$G$-torsor over\footnote{Note that we do not really define a relative FF curve $X_{\Spd(R)}$, but just consider a category whose objects we think of as the ``$G$-torsors over $X_{\Spd(R)}$".} the ``relative FF curve" $X_R:=X_{\Spd(R)}$ obtained in the usual way 
by first restricting the shtuka $\CP$ to a sector $(r, \infty)$ and then descending to the quotient by Frobenius.
It is enough to show that the quasi-isogeny gives a trivialization of $\CE$, i.e. a $G$-torsor isomorphism 
$\CE_b\simeq \CE$ over $X_R$.

 First we show that the Harder-Narasimhan (HN) filtration on $\CE$ is constant over points over $\Spd(R)$, i.e., for any representation $\rho$ of $G$, the HN polygon of 
the vector bundle $\rho_*(\CE)$ is constant, i.e. the same on all geometric points $\Spa(C, O_C)\to \Spd(R, R)^\diam$.
Such points are given by $R\to O_{C^\sharp}$, for which $k=R/\frak m\hookrightarrow O_C/\pi^{1/N}_C$ and
the constancy follows from a theorem of Fargues-Fontaine, according to which Frobenius isocrystals over $O_C/\pi_C$ are
isotrivial, i.e. obtained from $k$ via base change by $k\to O_C/\pi_C$ (see \cite[\S 11.1]{FarguesFontaine} , especially Cor. 11.1.14). It now follows
from \cite[Thm. II. 2.18]{FS} (as in the proof of \cite[Prop. III. 5.3]{FS})  that the $G$-torsor $\CE$ admits
a HN filtration over $\Spd(R)$.  
This allows us to consider trivializations of $\CE$ over $X_{\Spd(R)}$ which respect this filtered structure.

 For any $a: S\to \Spd(R)$ with $S\in {\rm Perfd}_k$, consider now
 \[
 \sT_{G, S}=\underline{\rm Isom}_{G, {\rm filt}}(\CE_b, a^*(\CE)),\quad  \sT_{H, S}=\underline{\rm Isom}_{H, {\rm filt}}(\CE_b, a^*(\CE))
 \]
 which is a $\sG^{\geq 0}_b$-torsor, resp. a $\sH^{\geq 0}_b$-torsor, over $X^{\rm alg}_S$. By a 
  $\sG^{\geq 0}_b$-torsor  $\sT_G$  over $X_R:=X_{\Spd(R)}$, resp. a $\sH^{\geq 0}_b$-torsor $\sT_H$ over $X_R$, we mean a  collection of  $\sT_{G, S}$, resp. $\sT_{H, S}$, over varying $a: S\to \Spd(R)$, together with appropriate glueing 
  data. In our situation, the $\sH^{\geq 0}_b$-torsor $\sT_H$ over $X_R$ is trivial with a section provided by the universal quasi-isogeny as explained above. By its definition, $\sT_G$ is a reduction of $\sT_H$ in the sense that it comes
  with an isomorphism
  \[
  \sT_H\simeq \sH^{\geq 0}_b\times^{\sG^{\geq 0}_b}\sT_G.
  \]
  The set of such reductions of $\sT_H\simeq \sH^{\geq 0}_b$ to a $\sG^{\geq 0}_b$-torsor over $X_R$ are in bijection
  with the set  ${\rm H}^0(X_R, \sQ^{\geq 0}_b)$ of global sections of the quotient $\sQ^{\geq 0}_b=\sH^{\geq 0}_b/\sG^{\geq 0}_b$.  
Set
\[
\wt Q^{\geq \lambda}_b(R)={\rm H}^0(X_R, \sQ^{\geq \lambda}_b).
\]
We have
\[
{\rm H}^0(X_R, \sQ^{\geq 0}_b)=\wt Q^{\geq 0}_b(R)=\wt Q^{>0}_b(R)\times \underline{(H_b/G_b)(\BQ_p)}(\Spd(R)).
\]
Recall that the global sections of the vector bundle $\CV_{H/G, b, \lambda}$ associated to the Frobenius isocrystal $(T_{H/G}\otimes_{\BQ_p}\breve\BQ_p, {\rm Ad}(b)\sigma)_{-\lambda}$ are given by an absolute Banach-Colmez space
\[
{\mathbf B}(\lambda)^{\oplus m}={\mathcal {BC}}(\CO({\lambda}))^{\oplus m}.
\]
By Lemma \ref{boundonlambda}, all the slopes $\lambda$ that appear here satisfy $0< \lambda\leq 1$.
By Lemma \ref{lemma452}, induction, and Lemma \ref{BCzero}, we obtain
\[
\wt Q^{>0}_b(\Spd(R))=0.
\]
So, we have
\[
{\rm H}^0(X_{\Spd(R)}, \sQ^{\geq 0}_b)=\underline{(H_b/G_b)(\BQ_p)}(\Spd(R)).
\]
Similarly $\wt Q^{>0}_b(\Spd(k))=0$ and
\[
\underline{(H_b/G_b)(\BQ_p)}(\Spd(R))=\underline{(H_b/G_b)(\BQ_p)}(\Spd(k))
\]
by specialization. Hence, we obtain that specialization identifies
\[
{\rm H}^0(X_{\Spd(R)}, \sQ^{\geq 0}_b)={\rm H}^0(X_{\Spd(k)}, \sQ^{\geq 0}_b).
 \]
 Similarly, we have 
 \[
 {\rm H}^0(X_{\Spd(R)}, \sH^{\geq 0}_b)={\rm H}^0(X_{\Spd(k)}, \sH^{\geq 0}_b),\quad  {\rm H}^0(X_{\Spd(R)}, \sG^{\geq 0}_b)={\rm H}^0(X_{\Spd(k)}, \sG^{\geq 0}_b).
 \]
Hence, the reductions of the trivial torsor $\sT_H$ to a $\sG^{\geq 0}_b$-torsor are uniquely determined
by their specialization over $X_{\Spd(k)}$. The torsor $\sT_G$ is such a reduction and its specialization over $X_{\Spd(k)}$
is trivial with a section given by the quasi-isogeny (which respects the $G$-structure over the point $x_0$). Therefore, $\sT_G$
is trivial over $X_R$: in fact, the section of $\sT_H$ given by the quasi-isogeny gives a section of $\sT_G$ because this is true after specialization to $\Spd(k)$.  This section of $\sT_G$ provides
the desired framing of the $\CG$-shtuka over $\Spd(R)$. This concludes the proof of Proposition \ref{prop451}. \hfill$\square$

\subsection{Completion of the  proof of Theorem \ref{functorialThm}}\label{complFunct}

In order to harmonize  the notation with that used in the Hodge type case above, we denote the group $\eG'$ by $\eH$ and the parahoric group schemes by $\CG$ and $\CH$ so that, by assumption, we have a dilated immersion
\[
\CG=\bar\CG^{\rm sm}\to \bar\CG\hookrightarrow \CH.
\]
 By \cite[Prop. 1.3.2]{KisinJAMS}, \cite{DeligneLetter}, there is a closed group immersion $\CH\hookrightarrow \GL(\Lambda)$, for some finite free $\BZ_p$-module $\Lambda$ and finite collections of tensors 
$s_a$, $s_b\in \Lambda^{\otimes}$, with $a\in I$, $b\in J$, such that
\[
\CH=\{g\in \GL(\Lambda)\ |\ g\cdot s_b=s_b, \forall b\in J\},
\]
\[
\bar\CG=\{g\in \GL(\Lambda)\ |\ g\cdot s_a=s_a, g\cdot s_b=s_b, \forall a\in I, b\in J\}.
\]
Recall that $\sS^{\dagger}$ denotes the normalization of the closure of the image of ${\rm Sh}_{\eK}(\eG, X)_E$ in $\sS'\otimes_{O_{E'}}O_{E}$.  Hence, by construction, $\sS^\dagger$ is normal and flat over $O_E$. We denote the natural morphism by 
\[
\iota: \sS^\dagger\to \sS'\otimes_{O_{E'}}O_{E}.
\]
 As explained in \S \ref{ss:funct431}, it is enough to show that the integral model $\sS^\dagger$ satisfies the conditions in Conjecture \ref{par812}. 
Condition a) in Conjecture \ref{par812} follows
 from the corresponding condition for $\sS'$ and the construction of $\sS^\dagger$ as a normalization.
 It remains to show that  conditions b) and c) in Conjecture \ref{par812} are satisfied for $\sS^\dagger$.

  \subsubsection{Condition b)}\label{compl481}
  We  show that we can extend the $\CG$-shtuka $\sP_E$ over ${\rm Sh}_{\eK}(\eG, X)_E$
  to a $\CG$-shtuka $\sP^\dagger$ over the integral model $\sS^\dagger$ such that 
  $\CH\times^{\CG} \sP^\dagger$
  is isomorphic to the pull-back via $\iota$ of the ``universal" $\CH$-shtuka $\sP'$ over $\sS'\otimes_{O_{E'}}O_{E}$.
  This is done by following the arguments of \S\ref{ss:extofshthodge}, see especially \S\ref{ext463}; below we point out the additions and adjustments needed for the argument. 
  
  Applying $\CH\hookrightarrow \GL(\Lambda)$ to $\sP'$ gives, by the Tannakian formalism, a vector space shtuka $\sV$ over $\sS'$. We restrict this via $\iota$ above to obtain a vector space shtuka $\sV^\dagger$ over $\sS^\dagger$. The main idea now is to give tensors $t_a\in (\sV^\dagger)^\otimes$, $a\in I$, corresponding to $s_a\in \Lambda^\otimes$ that extend the tensors $t_{a, E}$ given by the $\CG$-torsor $\sP_E$ on the generic fiber. (Note that we already have tensors $t_b\in (\sV^\dagger)^\otimes$, $b\in J$, that correspond to $s_b\in \Lambda^\otimes$, which are given using $\sP'$). Then we use the collection of $t_a$, together with $t_b$, to give 
  the $\CG$-torsor underlying the desired shtuka $\sP^\dagger$ as in \S\ref{ext463}. We need to be careful in the following two points
  where the argument somewhat deviates from \S\ref{ext463}: First, we do not have a universal $p$-divisible group, or a 
  corresponding ``global" BKF module that underlies the vector space shtuka $\sV$, hence we need to find an alternative argument  to obtain the meromorphic Frobenius crystal over the special fiber of $\sS^\dagger$ used in Step 1 of \S\ref{ext463}.  Instead of this explicit construction, we use at this point Theorem \ref{FFisocrystal}. Second, to obtain the $A_{\rm inf}(C^+)$-module $M$ that appears in Step 2, we  use Proposition \ref{propBKFshtuka} in the easier
  case of an algebraically closed perfectoid field. Given these ingredients, the argument proceeds as in \S\ref{ext463}.

  \subsubsection{Condition c)}\label{compl482}
  We now show that $\sS^\dagger$, together with the $\CG$-shtuka $\sP^\dagger$ given as in \S\ref{compl481} above satisfies condition c) of Conjecture \ref{par812}.  The proof  follows the arguments in \S\ref{ss:complhodge} with $\sS_\eK$ and $\sP_\eK$ now replaced by $\sS^\dagger$ and $\sP^\dagger$. Indeed, Proposition \ref{prop451} extends to this situation with the same proof and the rest of the result quickly follows. Note that in this, the fact that $\Psi_{H, x}$ of (\ref{SCompletions}) is an isomorphism, is provided  by our assumption that the model $\sS'$ for $\eH$ satisfies condition c) of Conjecture \ref{par812}.

\subsection{The local model diagram}\label{ss:LMD}
Let $\sS_\eK$ be the integral model of ${\rm Sh}_\eK(\eG, X)_E$ over $O_E$, satisfying Conjecture \ref{par812}. In the classical theory, the local model diagram gives a global way to relate the singularities of the integral model $\sS_\eK$ to the singularities of the local model. We now interpret this construction in our set-up.

 \subsubsection{The global $v$-sheaf local model diagram}\label{sss:vLMD}
{\cmag  Recall that the $v$-sheaf group $\CG^\diam$ has $S$-valued points $\CG^\diam(S)$, for $S=\Spa(R,R^+)\in {\rm Perfd}$,  given by pairs of an untilt $S^\sharp=\Spa(R^\sharp, R^{\sharp+})$ of $S$ and an element of $\CG(R^\sharp)$. We  define as follows a $v$-sheaf $\CG^\diam$-torsor over  $\sS_\eK^{\diam/}\to \Spd(O_E)$. For simplicity, we will often omit the subscript $\eK$.

For $S\in {\rm Perfd}_{\kappa_E}$ and $x: S=\Spa(R, R^+)\to \sS^{\diam/}$ 
with corresponding untilt  
$ S^\sharp=\Spa(R^\sharp, R^{\sharp +})$, we consider the pull-back 
 (restriction) $\sP_{ \phi ({S^\sharp})}=\phi^*(\sP )_{|S^\sharp}$ of the $\CG$-torsor $\phi^*(\sP )$ 
along  
\[
S^\sharp\hookrightarrow \CY_{[0,\infty)}(S)=S\bdtimes\BZ_p .
\]
This pull back gives a $\CG$-torsor over $S^\sharp$ which, by \cite[Thm. 19.5.3]{Schber}, corresponds to a unique $\CG$-torsor over $\Spec(R^\sharp)$. For $S=\Spa(R, R^+)\in {\rm Perfd}_{\kappa_E}$, we now define $\wt\sS_\eK^v(S)$ to be the set of pairs $(x, \alpha)$,
where $x: S=\Spa(R, R^+)\to \sS^{\diam/}$ is a point of $\sS^{\diam/}$, and 
\[
\alpha: \CG\times {S^\sharp}  \xrightarrow{\sim }  \sP_{  \phi(S^\sharp)} 
\]
a $\CG$-isomorphism (i.e. a section), for $S^\sharp=\Spa(R^\sharp, R^{\sharp +})$ the untilt given by $x$.
This defines a $v$-sheaf with a natural morphism of $v$-sheaves
\begin{equation}
\pi^v: \wt\sS^v\to \sS^{\diam/}
\end{equation}
which is a $v$-sheaf $\CG^\diam$-torsor.}

Given a section $\alpha$ of the $\CG$-torsor $\sP_{ \phi(S^\sharp)}$ over $S^\sharp$ (which amounts to a section of the scheme theoretic $\CG$-torsor over $\Spec(R^\sharp)$) we can extend it, using the smoothness of $\CG$, to a section 
\begin{equation}\label{sectalph}
\wh\alpha: \CG\times \wh { S^\sharp }  \xrightarrow{\sim }(\phi^*(\sP))_{\wh { S^\sharp}}
\end{equation}
of the pull-back of $\phi^*(\sP)$ over the formal completion $\wh { S^\sharp }:=\Spec(\wh\CO_{S\bdtimes\BZ_p,  S^\sharp})$ of $S\bdtimes\BZ_p$ along $ S^\sharp $.  Then by using Beauville-Laszlo glueing, we see that
the pair
\[
(\sP, \    \phi_\sP \circ  \wh\alpha: \CG\times ({\wh S^\sharp\setminus S^\sharp}) \xrightarrow{\sim}  \sP_{\wh S^\sharp\setminus S^\sharp} )
\]
gives an $S$-valued point in the affine $\BB_{\rm dR}$-Grassmannian ${\rm Gr}_{\CG, \Spd(O_E)}$ over $\Spd(O_E)$
(see \cite[Prop. 19.1.2, Prop. 20.3.2]{Schber}).
Since the $\CG$-shtuka $\sP $ has leg bounded by $\mu$, this point is, by definition, a point of the $v$-sheaf local model $\BM^v_{\CG, \mu}\subset {\rm Gr}_{\CG, \Spd(O_E)}$ (see the text after Definition \ref{defGsht}). This defines a morphism of $v$-sheaves over $\Spd(O_E)$ which is $\CG^\diam$-equivariant, 
\begin{equation}\label{morphqv}
q^v\colon\wt\sS^v\to \BM^v_{\CG, \mu} .
\end{equation}
Indeed, let $\CL^+\CG$ be the \emph{positive $v$-sheaf loop group} over $\Spd (O_E)$ with values in $S=\Spa(R, R^+)$ given by the untilt $S^\sharp$ and 
$$
 \CL^+\CG(S)=\CG(\BB^+_{\rm dR}(R^\sharp)) .
$$
Then $\BB^+_{\rm dR}(R^\sharp)\to R^\sharp$ induces a homomorphism $\CL^+\CG\to \CG^\diam\times\Spd(O_E)$. Letting $\CL^1\CG$ denote the kernel of this homomorphism, any two choices of extensions $\wh\alpha$ as in \eqref{sectalph} differ by a section of $\CL^1\CG$. Since $\mu$ is minuscule, the action of $\CL^+\CG$ on $\BM^v_{\CG, \mu}$ factors through $\CG^\diam\times\Spd(O_E)$ (for details, see \cite{AnRicLou}), and so, in the definition of the morphism \eqref{morphqv}, the image is indeed independent of the choice of $\wh\alpha$. 

 Altogether we obtain a diagram of morphisms of $v$-sheaves
\begin{equation}\label{LMDvsheaf}
\begin{gathered}
   \xymatrix{
	     &\widetilde{\sS}^v_{\eK}  \ar[dl]_-{\text{$\pi^v_\eK$}} \ar[dr]^-{\text{$q^v_\eK$}}\\
	   {\sS}^{\diam/}_{\eK}   & & \ \BM^v_{\CG, \mu} \ ,
	}
\end{gathered}
\end{equation}
which we call the \emph{$v$-sheaf theoretic local model diagram} for $\sS_\eK$. 
{\cmag Here $\pi^v$ is a $v$-sheaf $\CG^\diam$-torsor and $q^v$ is $\CG^\diam$-equivariant.}
Note that the existence of this diagram is a formal consequence
of the existence of the $\CG$-shtuka over $\sS_\eK$, which is bounded 
by $\mu$.

\subsubsection{The scheme-theoretic local model diagram} 
We consider a diagram of $O_E$-schemes
\begin{equation}\label{LMD}
\begin{gathered}
   \xymatrix{
	     &\widetilde{\sS}_{\eK}  \ar[dl]_-{\text{$\pi$}} \ar[dr]^-{\text{$q$}}\\
	   {\sS}_{\eK}   & & \ \BM_{\CG, \mu} \ ,
	}
\end{gathered}
\end{equation}
where $\pi$ is a $\CG$-torsor and $q$ a $\CG$-equivariant morphism which is smooth. Here $ \BM_{\CG, \mu}={\mathbb M}^{\rm loc}_{\CG, \mu}$ denotes the scheme local model
as in Theorem \ref{LMconj}.

The generic fiber of such a  diagram can be given by the classical Borel embedding construction: It is the canonical model over $E$ 
of the natural diagram obtained from the Borel embedding of the domain $X$, cf. \cite[III, 4]{MilneAA}. In the Hodge type case, it can also be constructed using the de Rham cohomology of the universal 
abelian scheme (see \cite[2.3]{CaraSch}). Since the $\CG$-shtuka $\sP_\eK$ extends $\sP_{\eK, E}$, 
we can see that the corresponding $v$-sheaf diagram
agrees with the base change of  (\ref{LMDvsheaf}) to $\Spd(E)$. 

{\cmag  Applying the $\diam/$-functor, cf. Definition \ref{def:diam/},
from schemes  over $O_E$ to $v$-sheaves over $\Spd(O_E)$ to (\ref{LMD}) gives
\begin{equation}\label{LMDdiam}
\begin{gathered}
   \xymatrix{
	     &\widetilde{\sS}_{\eK}^{\diam/}  \ar[dl]_-{\text{$\pi^{\diam/}$}} \ar[dr]^-{\text{$q^{\diam/}$}}\\
	   {\sS}_{\eK}^{\diam/}   & & \ \BM_{\CG, \mu}^{\diam} \  ,
	}
\end{gathered}
\end{equation}
where $\pi^{\diam/}$ is a $\CG^{\diam/}$-torsor, and $q^{\diam/}$ is $\CG^{\diam/}$-equivariant. Here, note that since $\BM_{\CG, \mu}$ is proper over $O_E$, we have $\BM_{\CG, \mu}^{\diam}=\BM_{\CG, \mu}^{\diam/}$ and the group $v$-sheaf $\CG^\diam$ acts on $\BM_{\CG, \mu}^{\diam}$. 

 The following definition attempts to give a $v$-sheaf interpretation of the scheme local model diagram,  analogous to Scholze's $v$-sheaf definition of scheme local models, cf. Theorem \ref{LMconj}.
 
  \begin{definition}
A \emph{scheme-theoretic local model  diagram} for $\sS_\eK$ is a diagram \eqref{LMD} (where $\pi$ is a $\CG$-torsor and $q$ is a $\CG$-equivariant smooth map) such that the generic fiber is given by the Borel embedding construction and which gives the $v$-sheaf local model diagram (\ref{LMDvsheaf}) after pushing out  the $\CG^{\diam/}$-torsor $\pi^{\diam/}$ of (\ref{LMDdiam}) by 
$\CG^{\diam/}\to \CG^\diam$. 
\end{definition}

\begin{conjecture}\label{conjLMD}
A scheme-theoretic local model diagram for $\sS_\eK$ exists. 
\end{conjecture}

Let us remark that  
it is not clear that a 
scheme-theoretic local model diagram for $\sS_\eK$  is uniquely determined if it exists.
Indeed, note that we cannot apply Corollary \ref{corFF} since the morphism $\pi^v$, which is obtained from the universal shtuka as above, is a $\CG^\diam$-torsor but does not obviously come from a $\CG^{\diam/}$-torsor. A canonical $\CG^{\diam/}$-torsor inducing $\pi^v$ by push out by $\CG^{\diam/}\to \CG^\diam$ might be obtained by assuming the existence of a ``stronger'' structure underlying the universal shtuka, for example by assuming that the conjectural  prismatic refinement  described in \S\ref{ss:prism} exists.}

\quash{ We note that,  it follows from 
Corollary \ref{corFF} that a scheme-theoretic local model diagram for $\sS_\eK$  is uniquely determined if it exists: it is  the unique diagram of morphisms which represents
 (\ref{LMDvsheaf}) and has the given generic fiber.}
 
  This conjecture follows when $(p, \eG, X, \CG)$ is of global Hodge type from the main result 
of \cite{KP} under some additional tameness hypotheses: $p\neq 2$, $G$ splits over a tamely ramified extension of $\BQ_p$,
and $p\nmid |\pi_1(G_{\rm der})|$. Here we are using implicitly the uniqueness of the integral model (Theorem \ref{uniq}), which allows us to replace the original Hodge embedding by the Hodge embedding that is used in \cite{KP}. The tameness conditions are relaxed in \cite{KZhou},where the result is extended to most
reductive groups whose corresponding adjoint groups are {\cmag essentially tamely ramified in the sense of Remark \ref{Gacc}, i.e., Weil restrictions of scalars of tame groups from wildly ramified extensions. }

 {\cmag
\subsubsection{The local $v$-sheaf local model diagram}\label{sss:localLMD} 
The exact same argument
 as in \S \ref{sss:vLMD} also gives the {$v$-sheaf theoretic local model diagram} for $\CM^{\rm int}_{\CG, b, \mu}$, i.e.  a diagram of morphisms of $v$-sheaves,
\begin{equation}\label{LMDvsheafloc}
\begin{gathered}
   \xymatrix{
	     &\widetilde\CM^{\rm int}_{\CG, b, \mu}  \ar[dl]_-{\text{$\pi^v$}} \ar[dr]^-{\text{$q^v$}}\\
	   \CM^{\rm int}_{\CG, b, \mu}   & & \ \BM^v_{\CG, \mu} \ ,
	}
\end{gathered}
\end{equation}
where $\pi^v$ is a $\CG^\diam$-torsor, and $q^v$ is $\CG^\diam$-equivariant. Again, there is a representability conjecture.
\begin{conjecture}
There is a diagram of morphisms of normal formal schemes flat and locally  formally  of finite type
\begin{equation}\label{formalLMD}
\begin{gathered}
   \xymatrix{
	     &\widetilde\sM_{\CG, b, \mu}  \ar[dl]_-{\text{$\pi$}} \ar[dr]^-{\text{$q$}}\\
	   \sM_{\CG, b, \mu}   & & \ \wh{\BM}_{\CG, \mu} \ ,
	}
\end{gathered}
\end{equation}
where $\pi$ is a $\wh\CG$-torsor and $q$ is formally smooth and $\wh\CG$-equivariant. In addition, 
the diagram \eqref{LMDvsheafloc}
 is  obtained from (\ref{formalLMD}) by applying the $\diam$-functor on formal schemes, followed by push out by $\CG^\sdiam=(\wh\CG)^\diam\to \CG^\diam$. Here $\wh{\BM}_{\CG, \mu}$, resp. $\wh\CG$, denotes the completion  of $\BM_{\CG, \mu}$, resp. $\CG$, along its special fiber. 
\end{conjecture}

  }

\subsection{Uniformization by LSV}\label{ss:uni}

 We continue with the same assumptions and notations as in \S \ref{ss:KPextension}.
 In particular, $(p, \eG, X, \eK)$ is of global Hodge type and we choose an appropriate Hodge embedding
 which produces
 \[
 i: \sS_\eK\to \mathcal A_{\eK^\flat}\otimes_{\BZ_{(p)}}O_E,
 \]
as in \eqref{intHodgemb} and which factors through the closed embedding $\sS_\eK^-\hookrightarrow \mathcal A_{\eK^\flat}\otimes_{\BZ_{(p)}}O_E$. We also choose a point $x_0\in \sS_\eK(k)$.

\subsubsection{Construction of RZ spaces}
Following \cite{HamaKim}, \cite{HP}, we  construct a ``Rapoport-Zink (RZ) formal scheme" $\RRZ_{\CG, \mu,  x_0 }$ over $\Spf(O_{\breve  E})$.   

Let $\RRZ_{\CH, i(x_0) }$  be the RZ formal scheme over $\Spf(\breve \BZ_p)$ for $(\CH, \iota(\mu))$ with framing object given by the $p$-divisible group of the product of the abelian varieties in the chain of isogenies that corresponds by the moduli interpretation of $\CA_{\eK^\flat}$ to the point $i( x_0 )$.   We will denote this ``framing" $p$-divisible group by $\BX_0$. 
 
 The uniformization map for the Siegel moduli scheme is a morphism of formal schemes
  \begin{equation}\label{unifxSieg}
  \Theta^{\rm RZ}_{\CH,i( x_0 )}: \RRZ_{\CH,  i(x_0) }\to \widehat\CA _{\eK^\flat}.
  \end{equation}

  Now consider the $p$-adic completion $\widehat\sS_{\eK }$ of $\sS_\eK$  
  and the fiber product induced by the morphism \eqref{intHodgemb}, 
  \begin{equation}
  \RRZ'_{\CG, \mu,  x_0 }:=\widehat\sS_{\eK }\times_{\widehat
\CA_{\eK^\flat}\otimes_{\BZ_{p}}O_E} (\RRZ_{\CH, i( x_0 )}\otimes_{\breve \BZ_p} O_{\breve E}).
  \end{equation}
  The fiber product $ \RRZ'_{\CG, \mu,  x_0 }$ is a formal scheme which is locally of formally finite type over $\Spf(O_{\breve E})$. Recall that $\widehat\sS_{\eK }^\diam$ supports the  $\CG$-shtuka $\sP_\eK$.
  By its construction, it corresponds to a vector shtuka $\sV$ over $\widehat\sS_{\eK }^\diam$
  (obtained from the BKF-module of the universal $p$-divisible group) and Frobenius invariant tensors
  $t_a\in \sV^{\otimes}$, $a\in I$. The tensors $t_a$ give corresponding Frobenius invariant tensors $t_{a, \crys}$
  on the Dieudonn\'e isocrystal  of the universal $p$-divisible group over the perfection 
  $(\sS_\eK\otimes_{O_E}k)^{\rm perf}$. On the other hand, over 
  $(\RRZ_{\CH,  i(x_0) }\otimes_{\breve\BZ_p}O_{\breve E})_\red^{\rm perf}$, there is a universal 
  quasi-isogeny $q$ between the universal $p$-divisible group and the $p$-divisible group $\BX_0$.
  Hence, over $(\RRZ'_{\CG, \mu,  x_0 })_\red^{\rm perf}$, we have Frobenius invariant tensors $t_{a, \crys}$ and $q^*(t_{a, \crys, 0})$  both on the 
 Dieudonn\'e isocrystal of the universal $p$-divisible group. By definition, the specialization
 of $t_{a, \crys}$ at the base point $x_0$ is equal to $q^*(t_{a, \crys, 0})$. Now, by \cite[Lem. 4.3.3]{HamaKim}, 
 the locus where 
 \[
 t_{a, \crys}=q^*(t_{a, \crys, 0}),\quad \forall a\in I,
 \]
  is an open and closed subscheme of $(\RRZ'_{\CG, \mu,  x_0 })_\red^{\rm perf}$, hence a union of connected components of $(\RRZ'_{\CG, \mu,  x_0 })_\red^{\rm perf}$.
  We let $\RRZ_{\CG,\mu, x_0}$ be the unique open and closed formal subscheme of $\RRZ_{\CG,\mu, x_0}'$ whose perfect reduced locus is this union. 
 By construction,  we have a morphism of formal schemes
  \begin{equation}\label{ThetaRZ}
  \Theta^{\rm RZ}_{\CG, x_0 }:  \RRZ_{\CG, \mu,  x_0 }\to \widehat\sS_{\eK } ,
  \end{equation}
  which fits in a commutative diagram of formal schemes
  \begin{equation}\label{RZM}
  \begin{aligned}
   \xymatrix{
         \RRZ_{\CG, \mu,  x_0 } \ar[r] \ar[d]_{\Theta^{\rm RZ}_{\CG, x_0 }}&\RRZ_{\CH,  i(\mu),  i(x_0) }\otimes_{\breve\BZ_p}O_{\breve E}  \ar[d]^{  \Theta^{\rm RZ}_{\CH, i(x_0) }\otimes_{\BZ_p}O_{E} }\\
      \widehat\sS_{\eK } \ar[r] & \widehat\CA _{\eK^\flat}\otimes_{\BZ_p}O_{\breve E}\, .
        }
        \end{aligned}
    \end{equation}
      \begin{lemma}\label{RZnorm}
  The formal schemes $\RRZ'_{\CG, \mu,  x_0 }$ and $\RRZ_{\CG, \mu,  x_0 }$ are normal and flat  over $\Spf(O_{\breve E})$. Furthermore, for each $x\in  \RRZ_{\CG, \mu,  x_0 }(k)$, the morphism \eqref{ThetaRZ} induces an isomorphism of formal completions
  $$
  \wh{\RRZ_{\CG, \mu,  x_0}} _{/x}\xrightarrow{\ \sim\ } \widehat{\sS_{\eK }}_{/x} .
  $$
  Here on the RHS we have written $x$ for $\Theta^{\rm RZ}_{\CG, x_0 }(x)$.
  \end{lemma}
  \begin{proof}
  The standard deformation theory of $p$-divisible groups (as for example in \cite{R-Z}) implies that the morphism $\Theta^{\rm RZ}_{\CH, i(x_0)}$ 
  induces an isomorphism of formal completions
  $$
  \wh{\RRZ_{\CH, \mu,  i(x_0)}} _{/i(x)}\xrightarrow{\ \sim\ } \widehat{\CA _{\eK^\flat}}_{/i(x)} .
  $$
  It follows from the construction of $\RRZ'_{\CG, \mu,  x_0 }$ as fiber product and then of 
  $\RRZ_{\CG, \mu,  x_0 }$ as an open and closed formal subscheme, that \eqref{ThetaRZ} 
  gives an isomorphism of formal completions
   $$
   \wh{\RRZ'_{\CG, \mu,  x_0}} _{/x}\simeq \wh{\RRZ_{\CG, \mu,  x_0}} _{/x}\xrightarrow{\ \sim\ } \widehat{\sS_{\eK }}_{/x} .
  $$
  The result now follows since $\sS_\eK$ is normal and excellent.
  \end{proof}

 By \cite[Cor. 25.1.3]{Schber}, there is an isomorphism of $v$-sheaves
 \[
 \RRZ_{\CH,  i(\mu),  i(x_0) }^\diam\xrightarrow{\sim}\CM^{\rm int}_{\CH,  b, i(\mu)}.
 \]
 The top horizontal arrow in the diagram \eqref{RZM} becomes under this identification 
 \begin{equation}\label{iRZ}
 i_{\mathrm {RZ}}: \RRZ_{\CG, \mu,  x_0 }^\diam\to \CM_{\CH,  b, i(\mu)}^{\rm int}\times_{\Spd(\breve\BZ_p)}\Spd(O_{\breve E}) .
 \end{equation}
 \begin{lemma}\label{RZembedd}
 a) The morphism \eqref{iRZ} factors 
 as
 \[
 i_{\mathrm {RZ}}\colon \RRZ_{\CG, \mu,  x_0 }^\diam\xrightarrow{c} \CM^{\rm int}_{\CG, b, \mu}\xrightarrow{\ i\ }\CM^{\rm int}_{\CH,  b, i(\mu)}\times_{\Spd(\breve\BZ_p)}\Spd(O_{\breve E}) ,
 \]
 where $i$ is the natural morphism. 
 
 b) For each $x\in  \RRZ_{\CG, \mu,  x_0 }(k)$, the morphism $c$ induces an isomorphism on formal completions
 \[
   \wh{\RRZ_{\CG, \mu,  x_0}} _{/x}\xrightarrow{\ \sim\ } \wh{\CM^{\rm int}_{\CG, b, \mu}}_{/c(x)}.
 \]
 In particular, $\wh{\CM^{\rm int}_{\CG, b, \mu}}_{/c(x)}$ is representable by the formal spectrum of a complete local ring.
 \end{lemma}
 
 \begin{proof}
 Consider the pullback under \eqref{ThetaRZ} of the $\CG$-shtuka $\sP_\eK$ on  $ \widehat{\sS_{\eK }}$. The corresponding $\CH$-shtuka $ \CH\times^\CG\sP_\eK$ is equipped with a framing from the morphism \eqref{iRZ}. By its construction, this framing respects the tensors $t_{a, {\rm crys}}$, hence comes from a framing of the $\CG$-shtuka. This defines the factorization in a).  

 By Lemma \ref{RZnorm}, 
 $ \wh{\RRZ_{\CG, \mu,  x_0}} _{/x}\simeq \wh{\sS_{\eK}}_{/x}$. On the other hand, by the proof of Proposition \ref{prop451} 
 the morphism   
 \[
\wh{\sS_{\eK}}_{/x}\to \wh{\CM^{\rm int}_{\CG,  b, \mu}}_{/x},
 \]
given by $\Psi_{G, x}^*$ there, is in fact an isomorphism.  This shows b). 
\end{proof}

 The following statement is to be compared to \cite[Cor. 6.3]{Zhou}.
 
 \begin{proposition}  
 The morphism $c:  \RRZ_{\CG, \mu,  x_0 }^\diam\xrightarrow{ \ } \CM^{\rm int}_{\CG, b, \mu}$ is a closed immersion.
 \end{proposition}
 
 \begin{proof}  
 Recall that the map $\CM^{\rm int}_{\CG, b, \mu}\to \CM^{\rm int}_{\CH, i(b), i(\mu)}\times_{\Spd(\breve\BZ_p)}\Spd(O_{\breve E})$ is a closed immersion, cf. Proposition \ref{prop332}. 
{\cmag Consider the  morphism of formal schemes in the top of diagram (\ref{RZM}),
\[
d: \RRZ_{\CG, \mu,  x_0 }\to \sM_{\CH,  b, i(\mu)}\otimes_{\breve\BZ_p}O_{\breve E}=\RRZ_{\CH,  i(\mu),  i(x_0) }\otimes_{\breve\BZ_p}O_{\breve E}  .
\]
 {\cmag By its construction, the source $\RRZ_{\CG, \mu,  x_0 }$ of this morphism is a union of connected components of $\RRZ'_{\CG, \mu,  x_0 }$ which then maps to the target by a finite morphism. Hence, the morphism $d$ is quasi-compact and quasi-separated and the induced map on topological spaces has Zariski closed image. The formal scheme-theoretic image of $d$ is defined and it is a formal closed subscheme $\sM'_{\CG,b,\mu}$ of the formal scheme $\sM_{\CH,  b, i(\mu)}\otimes_{\breve\BZ_p}O_{\breve E}$.} The corresponding $v$-sheaf $(\sM^{\prime}_{\CG, b, \mu})^\diam$ is, by Lemma \ref{RZembedd} (a), a closed $v$-subsheaf of $\CM^{\rm int}_{\CG, b, \mu}$. It suffices to prove that  the resulting map $c: \RRZ_{\CG, \mu,  x_0 }\to \sM^{\prime}_{\CG, b, \mu}$ is an isomorphism of formal schemes.
 By Lemma \ref{RZnorm} and  Lemma \ref{RZembedd}, the map $c$ induces isomorphisms on formal completions at each $k$-point.  Hence it suffices to  prove that the induced map on $k$-points
 \[
 c:  \RRZ^\diam_{\CG, \mu,  x_0 }(k)\xrightarrow{ \ } \CM^{\rm int}_{\CG, b, \mu}(k)
 \] 
 is injective.}

  Recall that we have an inclusion of $(b, \mu)$-admissible sets,
 \[
 \CM^{\rm int}_{\CG, b, \mu}(k)=X_\CG(b, \mu^{-1})(k)\subset X_{\CH}(i(b), i(\mu)^{-1})(k)=\CM^{\rm int}_{\CH, i(b), i(\mu)}(k).
 \]
 Suppose we have $x\neq x'\in \RRZ^\diam_{\CG, \mu,  x_0 }(k)$ with $c(x)=c(x')$. Then $x$, $x'$ map to the same point $y=i(x)=i(x')$ in $\RRZ^\diam_{\CH, i(\mu), i(x_0)}(k)$. The points $x$, $x'$ also give $k$-points of $\sS_\eK$ which we still denote by $x$, $x'$ and we have $y=i(x)=i(x')$
  in $\CA_{\eK^\flat}(k)$. Consider the $\CG$-shtukas $\sP_x$, $\sP_{x'}$ over $\Spd(k)$ obtained by specializing the $\CG$-shtuka $\sP_\eK$ over $\sS^\diam_\eK$ to the two points $x$ and $x'$. Here, these are given together with quasi-isogeny trivializations that respect the $\CG$-structure and provide elements $[(\sP_{x}, i_{r, x})]$ and $[(\sP_{x'}, i_{r, x'})]$ of $ \CM^{\rm int}_{\CG, b, \mu}(k)=X_\CG(b, \mu^{-1})(k)$. Under our assumption $c(x)=c(x')$ we have $s_0:=[(\sP_{x}, i_{r, x})]=[(\sP_{x'}, i_{r, x'})]$ in $ \CM^{\rm int}_{\CG, b, \mu}(k)$ with image $t_0\in \CM^{\rm int}_{\CH, i(b), i(\mu)}(k)$.
Consider the map
 \[
 \wh{\sS_\eK}_{/x}\sqcup  \wh{\sS_\eK}_{/x'}\to \wh{\sS_{\eK}^-}_{/y}\subset \wh{(\CA_{\eK^\flat}\otimes_{\breve\BZ_p}O_{\breve E})}_{/y}.
 \]
Set $\wh U_{x}=\wh{\sS_\eK}_{/x}$ and  $\wh U_{x'}=\wh{\sS_\eK}_{/x'}$ and $\wh U^-_{y}=\wh{\sS^-_\eK}_{/y}$ and $\wh V_{y}=\wh{(\CA_{\eK^\flat}\otimes_{\breve\BZ_p}O_{\breve E})}_{/y}$. By Proposition \ref{prop451} and (\ref{SCompletions}),  
we obtain commutative diagrams
\begin{equation*}\label{SC3}
 \begin{aligned}
   \xymatrix{
     \wh U_{x} \ar[r]^{\wh{\Psi_{G}}_{,x}} \ar[d]_{i_x} &  \wh{{\CM}^{\rm int}_{\CG, b, \mu}}_{/s_0}\ar[d]^{i_{s_0}} \\
   \wh V_{y}    \ar[r]^{\wh{\Psi_{H}}_{,y}\ \ \ \ \ \ \ \ \ \ \ \ \ }&\wh{({\CM}^{\rm int}_{\CH, i(b), i(\mu)}\otimes_{\BZ_p}O_{\breve E})}_{/t_0},
        }
        \end{aligned}\quad 
        \begin{aligned}
   \xymatrix{
       \wh U_{x'} \ar[r]^{\wh{\Psi_{G}}_{,x'}} \ar[d]_{i_{x'}} &  \wh{{\CM}^{\rm int}_{\CG, b, \mu}}_{/s_0}\ar[d]^{i_{s_0}} \\
   \wh V_{y}    \ar[r]^{\wh{\Psi_{H}}_{,y}\ \ \ \ \ \ \ \ \ \ \ \ \ }&\wh{({\CM}^{\rm int}_{\CH, i(b), i(\mu)}\otimes_{\BZ_p}O_{\breve E})}_{/t_0}.
        }
        \end{aligned}
    \end{equation*}
Note that these two diagrams share three vertices and the arrows between them. All horizontal arrows are isomorphisms (by the proof of \ref{par812} (c) for 
$\wh{\Psi_{H}}_{,y}$, see the text under (\ref{SCompletions})).
It follows that the scheme-theoretic images of $\wh U_{x}\to \wh V_y$ and $\wh U_{x'}\to \wh V_y$ are equal. Since $\wh U_x$ and $\wh U_{x'}$
are both irreducible components of the normalization of the image $\wh U^-_{y}=\wh{\sS^-_\eK}_{/y}$ in $\wh V_y$, it follows 
that $\wh U_{x}=\wh U_{x'}$ and so $x=x'$  in $\sS_\eK(k)$, and $x=x'$ in $\RRZ_{\CG, \mu,  x_0 }(k)$. This contradiction shows the asserted injectivity. 
    \end{proof}

  \subsubsection{Condition $\rm U$}\label{Uconj} 
  We conjecture that $c:  \RRZ_{\CG, \mu,  x_0 }^\diam\xrightarrow{ \ } \CM^{\rm int}_{\CG, b, \mu}$ is an isomorphism.
  By the above, this is equivalent to asking that $c: \RRZ^\diam_{\CG, \mu,  x_0 }(k)\to \CM^{\rm int}_{\CG, b, \mu}(k)=X_\CG(b, \mu^{-1})(k)$
  is bijective, or also  that it is surjective. In turn, this is equivalent to the following condition.
  
  \begin{itemize}
  \item[] 
  {\sl $({\rm U}_{x_0})$: Let $x_0\in \sS_\eK(k)$, $b=b_{x_0}$.  There is a morphism
  \[
  \Theta_{\CG, x_0}: X_\CG(b, \mu^{-1})\to \sS_{\eK}
  \]
  which sends the base point to $x_0$ and is such that the diagram
  \begin{displaymath}
  \xymatrix{
     X_\CG(b, \mu^{-1})(k)\ar[r] \ar[d]_{\Theta_{\CG, x_0 }(k)}& X_{\CH}(i(b), i(\mu^{-1}))(k)
     \ar[d]^{  \Theta_{\CH, i(x_0) }(k)}\\
      \sS_{\eK}(k) \ar[r] & \CA_{\eK^\flat}(k)\, }
    \end{displaymath}
   commutes.}
\end{itemize}

Indeed, assuming   $({\rm U}_{x_0})$ we obtain a morphism $X_\CG(b, \mu^{-1})\to \RRZ'_{\CG,\mu, x_0}$ whose image lands
in $\RRZ_{\CG,\mu, x_0}$.  This morphism gives on $k$-points a map $X_\CG(b, \mu^{-1})(k)\to \RRZ_{\CG, \mu,  x_0 }(k)$ which is easily seen to be an inverse of $c:  \RRZ_{\CG, \mu,  x_0 }(k)\to X_\CG(b, \mu^{-1})(k)$. In particular, there is at most one such map $\Theta_{\CG, x_0}$.

  Assuming $({\rm U}_{x_0})$ we can identify $\RRZ^\diam_{\CG, \mu, x_0}\simeq \CM^{\rm int}_{\CG, b, \mu}$. In particular,  $\CM^{\rm int}_{\CG, b, \mu}$ is representable by a formal scheme $\sM_{\CG, b, \mu}$ and we obtain a ``uniformization" morphism
 \begin{equation}\label{ThetM}
 \Theta^{\rm {RZ}}_{\CG, x_0}: \sM_{\CG, b, \mu}\to \wh\sS_\eK.
 \end{equation}
 which maps the base point to $x_0\in \sS_\eK(k)$.
 \begin{remark}\label{prevres}
Suppose that $p$ is odd, $G=\eG\otimes_\BQ \BQ_p$ splits over a tamely ramified extension of $\BQ_p$ and $p\nmid |\pi_1(G_{\rm der})|$.  In this case condition $({\rm U}_{x})$ is identical with Axiom A in \cite[\S 4.3]{HamaKim}; Hamacher and Kim also conjecture that,
under these hypotheses, it is always satisfied. By work of Zhou \cite{Zhou, vanH}, condition $({\rm U}_x)$ is known to hold under these hypotheses if in addition $x$ is basic, or $G$ is residually split or $\CG$ is absolutely special.  By Nie \cite{Nie}, it is known if $G$ is unramified. {\cmag As mentioned in the Introduction, condition $({\rm U}_x)$ is now known to hold in general, due to work of Gleason-Lim-Xu \cite[Cor. 1.10]{GLX22}.}
\end{remark}
 \subsubsection{Uniformization}
 Let $x_0\in\sS_\eK(k)$. We assume that Condition $({\rm U}_{x_0})$ from \S \ref{Uconj} is satisfied. Using the action of $\eG(\BA_f^p)$ on $\sS_{\eK_p}$ as a group of Hecke correspondences, the morphism $\Theta^{\rm {RZ}}_{\CG, x_0}$ from \eqref{ThetM} induces a  morphism
 \begin{equation}\label{unifmorKP}
 \sM_{\CG, b_{x_0}, \mu}\times \eG(\BA_f^p)\to \sS_{\eK_p}\otimes_{O_E}{O_{\breve E}} . 
 \end{equation}
 The image $\CI(x_0)$ on $k$-points is called the isogeny class of $x_0$. A point $x'\in\sS_{\eK_p}(k)$ lies in $\CI(x_0)$ if and only if the corresponding polarized abelian varieties $A_{x_0}$ and $A_{x'}$ are related by a quasi-isogeny respecting weakly the polarizations and such that under the induced maps on rational Dieudonn\'e modules, resp. on rational Tate modules for $\ell\neq p$, the tensors $s_{a, {\rm crys}, x_0}$ and $s_{a, {\rm crys}, x'}$, resp. $s_{a, {\ell}, x_0}$ and $s_{a, {\ell}, x'}$, are mapped to one another, comp. \cite[Prop. 1.4.15]{KisinLR}.  
 
  \begin{remark}\label{LRreal}
 If $x_0$ is basic, {\cmag i.e., the class $[b_{x_0}]\in B(G, \mu^{-1})$ is basic, } then   $\CI(x_0)$ is the set of $k$-points of a closed subscheme of $\sS_{\eK_p}\otimes_{O_E}k$ {\cmag which is a proper $k$-scheme.  In fact, one expects the following concrete description of $\CI(x_0)$ in this case.} Consider the basic locus $ \sS_{\eK_p, {\rm basic}}$,  a closed subscheme of  $\sS_{\eK_p}\otimes_{O_E}k$, cf. \cite[Thm. 1.3.13]{KMS}, 
 \begin{equation}
 \sS_{\eK_p, {\rm basic}}(k)=\{ x'\in \sS_{\eK_p}(k)\mid [b_{x'}]\in B(G, \mu^{-1})\text{ the unique basic element}\} .
 \end{equation} 
  Then $\sS_{\eK_p, {\rm basic}}(k)$ should always coincide with $\CI(x_0)$, for any basic $x_0$. This is known to be true at least if $p$ is odd and $G$ is unramified and $\eK_p$ hyperspecial, cf. \cite[Cor. 7.2.16]{XZ}. 
 It also holds if $p$ is odd, $G$ is quasisplit and splits over a tame extension, and $p\nmid |\pi_1(G_{\rm der})|$ and $\eK_p$ is special, cf. \cite{HZZ}. 
 The statement is also compatible with the Langlands-Rapoport conjecture which enumerates isogeny classes by equivalence classes of  admissible homomorphisms $\phi: \frak Q\to {\frak G}_G$ from the quasi-motivic gerb $\frak Q$ to the neutral gerb ${\frak G}_G$. Indeed,  up to equivalence,  there is a unique  such $\phi$ whose  $p$-component $\phi_p$ is basic, cf. \cite{HZZ}.   \end{remark} 
 
 The morphism \eqref{unifmorKP} induces a morphism
 \begin{equation}\label{unifK}
  \sM_{\CG, b_{x_0}, \mu}\times \eG(\BA_f^p)/\eK^p\to \sS_{\eK}\otimes_{O_E}{O_{\breve E}}. 
 \end{equation}
 Let $I_{x_0}(\BQ)$ be the group of self-isogenies of $A_{x_0}$ which respect the tensors $s_{a, {\rm crys}, x_0}$, resp. $s_{a, {\ell}, x_0}$. Then we have an action of $I_{x_0}(\BQ)$ on $\RRZ_{\CG, \mu, x_0}\simeq \sM_{\CG, b, \mu}$ (through its action on the Dieudonn\'e module respecting $s_{a, {\rm crys}, x_0}$), and an action of $I_{x_0}(\BQ)$ on $\eG(\BA_f^p)/\eK^p$,
  after fixing
  a level $\eK^p$-structure on the rational Tate modules of $A_{x_0}$ respecting $s_{a, {\ell}, x_0}$. 
  These two actions combine to an action of $I_{x_0}(\BQ)$ on the LHS of \eqref{unifK}. 
  
\begin{theorem}\label{HK}
 Let $(p,\eG, X, \eK)$ be of global Hodge type. Let $x_0\in \sS_\eK(k)$ and assume that Condition $({\rm U}_{x_0})$ in \S \ref{Uconj} holds. Then $\CM^{\rm int}_{\CG, b_{x_0}, \mu}$ is representable by a formal scheme $\sM_{\CG, b_{x_0}, \mu}$ and 
 there is non-archimedean uniformization along the isogeny class $\CI(x_0)$ in $\sS_\eK\otimes_{O_E} k$, i.e., an isomorphism of formal schemes over $O_{\breve E}$,
$$
 I_{x_0}(\BQ)\backslash ({\sM_{\CG,b_{x_0},\mu}}\times G(\BA_f^p)/\eK^p)\isoarrow \wh{({\sS_\eK\otimes_{O_E}O_{\breve E})}}_{/\CI(x_0)} .
$$
\emph{This isomorphism is to be interpreted (especially when $x_0$ is non-basic) as for its PEL counterpart in \cite[Thm. 6.23]{R-Z}.} 
\end{theorem} 
{\cmag
 \begin{proof} As noted in \eqref{ThetM}, the $v$-sheaf $\CM^{\rm int}_{\CG, b_{x_0}, \mu}$ is representable by $\RRZ_{\CG, \mu,  x_0 }$. We give the argument for the uniformization statement only in the basic case, when $\CI(x_0)$ is a closed subset  of $\sS_\eK\otimes_{O_E} k$. In the non-basic case, even the interpretation of the statement of the theorem is more involved since  $\CI(x_0)$ is then a countable union of locally  closed subsets of $\sS_{\eK_p}\otimes_{O_E}k$, cf. the explanation in the PEL case in \cite[Thm. 6.23]{R-Z}. In the basic case, one checks: 
 \begin{altenumerate}
 \item The formal scheme on the LHS is locally formally of finite type. Indeed, in the basic case, $I_{x_0}$ is an inner form of $G$ which is anisotropic at the archimedean place, cf. \cite[Cor. 2.1.7]{KisinLR}. Hence the LHS is a finite disjoint sum of formal schemes of the form $\Gamma\backslash \sM_{\CG,b_{x_0},\mu}$, where $\Gamma\subset J(\BQ_p)$ is a discrete cocompact subgroup, comp. \cite[proof of Thm. 6.23]{R-Z}. Furthermore, by replacing $K^p$ by a subgroup of finite index, we may assume that the action of $\Gamma$ is free (see loc.~cit.). Since $\sM_{\CG,b_{x_0},\mu}$ is  locally formally of finite type, the assertion follows. 
 
 \item The morphism induced on the underlying reduced schemes is proper. Indeed, in the diagram \eqref{RZM} the upper arrow induces a finite morphism on the underlying reduced schemes, with target a union of projective schemes, cf. \cite[Prop. 2.32]{R-Z}.  From (i) it follows that the underlying reduced scheme of the LHS is proper over $k$. 
 \item The map induced on $k$-points
 $$
 I_x(\BQ)\backslash ({X_\CG(b_{x_0}, \mu^{-1})(k)}\times G(\BA_f^p)/\eK^p)\to \CI(x)
 $$
 is a bijection. This follows from \cite[Prop. 9.1]{Zhou} (which is based on \cite[Prop. 2.1.3]{KisinLR}). Indeed, the assumption \cite[Assumpt. 6.1.7]{Zhou} is satisfied, thanks to hypothesis $({\rm U}_{x_0})$.  
 
 \item For each point of $ I_{x_0}(\BQ)\backslash {X_\CG(b_{x_0}, \mu^{-1})(k)}\times G(\BA_f^p)/\eK^p$, represented by a pair $(\tilde x', g)\in {X_\CG(b_{x_0}, \mu^{-1})(k)}\times G(\BA_f^p)$, with corresponding point $ {x'}\in\CI(x_0)$, the induced map on completed local rings 
 $$
 \hat\CO_{\sS, x'}\to  \hat\CO_{\sM, \tilde x'}
 $$
 is an isomorphism. Indeed, this follows from Lemma \ref{RZnorm} after identifying  $\sM_{\CG, b_{x_0}, \mu}$   with $\RRZ_{\CG, \mu,  x_0 }$.
 \end{altenumerate}
 The assertion follows, since any morphism of finite type between locally noetherian formal schemes that is formally \'etale, proper and radicial is an isomorphism, comp. \cite[proof of Thm. 6.23]{R-Z}. 
 \end{proof}
 
 }
 
\section{Appendix: Parahoric extension}\label{s:parext}
This appendix is a remnant of an earlier version of the paper. We summarize here some facts on extending torsors under parahoric group schemes over the punctured spectrum of $A_{\rm inf}$. Most of these results are due to Ansch\"utz \cite{An}, which is our main reference. At the time when we wrote the earlier version, the main conjecture was proved by Ansch\"utz in many cases, but not all. In the meantime, he has proved the conjecture in its entirety. We give here a proof of the conjecture in the \emph{tame case or when $p\geq 5$.}  (In fact, our result is somewhat more general, see Theorem \ref{extTHM}.) The reason for leaving in this appendix is that our proof under these restrictive conditions is simpler and more direct.  We also give a simple proof of the extension lemma of Kisin and the first author \cite[Prop. 1.4.3]{KP}. In particular, we do not use  results on ``Serre II" \cite{Gi} which are used by Ansch\"utz and in \cite{KP}, which ultimately proceed on a case-by-case basis. 

\subsection{Statement of the conjecture}

Let $E$ be a $p$-adic field with perfect residue field $k$ of characteristic $p$ and uniformizer $\pi$ (Ansch\"utz considers also local fields of characteristic $p$). Let $C$ be a   perfect non-archimedean field of characteristic $p$ which is a $k$-algebra, and let $C^+$ be an open and bounded valuation subring. We consider the local ring  $A_E=W_E(C^+)=W(C^+)\hat\otimes_{W(k)} O_E$, and
\begin{equation}
X=\Spec (A_E), \quad U=X\setminus S, \quad V=\Spec (A_E\otimes_{O_E} E) . 
\end{equation}
Here $S=V(\pi, [\varpi])$, where $\varpi$ is a pseudouniformizer of $C^+$. If $C^+=O_C$, then $S=\{s\}$ with $s$ the unique closed point of $X$. When we want to stress the dependence on $(C, C^+)$ or also on $E$, we also write $X_{(C, C^+)}$  and $U_{(C, C^+)}$  and $V_{(C, C^+)}$, or $X_{(E; C, C^+)}$  and $U_{(E; C, C^+)}$  and $V_{(E; C, C^+)}$. We consider the following functors on the categories of coherent sheaves of modules,
\begin{equation}
{\rm Coh}(X)\to{\rm Coh}(U), \ \CM\mapsto \CM_{|U}; \quad {\rm Coh}(U)\to{\rm Coh}(X), \ \CM\mapsto {{\rm H}^0(U,\CM)}^{\prime}.
\end{equation}
Here we use that ${\rm H}^0(U, \CM)$ is a finitely generated $A_E$-module, and denote by ${{\rm H}^0(U,\CM)}^{\prime}$ the associated coherent sheaf of modules on $X$. 
\begin{theorem}
The above functors induce mutually inverse  tensor equivalences between ${\rm Bun}(X)$ and ${\rm Bun}(U)$. 
In particular, all vector bundles on $U$ are trivial, i.e., free. 
\end{theorem}
\begin{proof}
 This follows from a theorem of Kedlaya \cite{Kedlaya}, cf. \cite[Prop. 14.2.6]{Schber}.
\end{proof}

Now let $G$ be a reductive group over $E$. 
\begin{conjecture}[Extension conjecture]\label{paraextconj}
Let $\CG$ be a parahoric model of $G$ over $O_E$. Then any $\CG$-torsor over $U$ extends to a $\CG$-torsor over $X$. 
\end{conjecture}
Note that if such an extension exists, then it is unique. As pointed out above, Ansch\"utz has proved this conjecture
(see \cite[Thm 1, Cor.  11.6]{An}). 
\begin{proposition}\label{cohoV}
Let $C$ and the residue field $k$ of $E$ be algebraically closed. Let $\CG$ be a parahoric model of $G$. Then a $\CG$-torsor $\CP$ on $U$ extends to $X$ if and only if the $G$-torsor $\CP_{|V}$ is trivial. Moreover,  this holds for all $\CG$-torsors $\CP$ on $U$  if and only if ${\rm H}^1(V, G)=\{1\}$.

In particular, if the Extension Conjecture \ref{paraextconj} is true for one parahoric model $\CG$ of $G$, then it is true for any  parahoric model $\CG$ of $G$.
\end{proposition}
\begin{proof}
We show that the proof of \cite[Cor. 9.3]{An} (in which $C^+=O_C$) carries over. Since $C$ is algebraically closed, 
$C^+$ is strictly henselian. Since $O_E$ is strictly henselian, so is $W_E(C^+)=W(C^+)\hat\otimes_{W(k)} O_E$, and therefore $\CG$-torsors
over $W_E(C^+)$ are trivial.  Hence a $\CG$-torsor on $U$ extends to $X$ if and only if it is trivial. Let $\varpi\in C^+$ be a pseudo-uniformizer. The $\pi$-adic completion of $W_E(C^+)[1/[\varpi]]$ is $W_E(C)=W(C)\hat\otimes_{W(k)} O_E$, which is a complete dvr with algebraically closed residue field $C$, so is also strictly henselian. Set $O_\CE:=W_E(C)$ and let $\CE$ be its fraction field. Then any $\CG$-torsor on $O_\CE$ is trivial and, by Steinberg's theorem, any $G$-torsor on $\CE$ is trivial.  By Beauville-Laszlo glueing, there is a bijection between the set of isomorphism classes of $\CG$-torsors on $U$ which are trivial on $V=\Spec (W_E(C^+)[1/\pi])$ and on $\Spec (O_\CE)$ and the double coset space 
\[
\CG(O_\CE)\backslash G(\CE)/G(W_E(C^+)[1/\pi]). 
\]
Hence for the first assertion of the proposition, since the triviality on $O_\CE$ is automatic, it suffices to show that this double coset space consists of only one element.  This follows as in the proof of \cite[Prop. 9.2]{An} from the ind-properness of the Witt vector Grassmannian ${\rm Gr}_\CG$, which implies identifications 
\begin{equation*}
\begin{aligned}
G(W_E(C)[1/\pi])/\CG(W_E(C))&={\rm Gr}_\CG(W_E(C)))\\&={\rm Gr}_\CG(W_E(C^+))= G(W_E(C^+)[1/\pi])/\CG(W_E(C^+)) .
\end{aligned}
\end{equation*}
In the second assertion of the proposition, if ${\rm H}^1(V, G)=\{1\}$, then any $G$-torsor on $V$ is trivial. Hence, again by the triviality of the double coset space, any $\CG$-torsor on $U$ is trivial and therefore extends to $X$. Conversely, let $P$ be a $G$-torsor on $V$. By Steinberg's theorem, the pullback $P_\CE$ of $P$ to $\CE$ is trivial. Hence, by  Beauville-Laszlo glueing, we can glue $P$ and the trivial $\CG$-bundle on $O_\CE$ to obtain a $\CG$-bundle $\CP$ on $U$. But by assumption, $\CP$ extends to $X$ and hence is trivial. But then $P$ is trivial, as claimed. 
\end{proof}
We can reduce Conjecture \ref{paraextconj} to the case considered in the previous proposition by the following descent lemma.
\begin{lemma}\label{redalgcl}
a) Let $C'^+$ be a faithful flat extension of $C^+$. If Conjecture \ref{paraextconj} holds for $\CG$-torsors over $U_{(C', C'^+)}$, then it does for $\CG$-torsors over $U_{(C, C^+)}$. 

b) Let $E'/E$ be an unramified extension. If Conjecture \ref{paraextconj} holds for parahoric $O_{E'}$-models of $G'=G\otimes_E E'$, then it holds for parahoric models of $G$. 
\end{lemma}
\begin{proof}
For a), the argument of \cite[Lem. 9.1]{An} applies: any  tensor functor $\omega\colon {\rm Rep}(\CG)\to {\rm Bun}(X_{(C, C^+)})$ which is exact after composition with the restriction functor ${\rm Bun}(X_{(C, C^+)})\to {\rm Bun}(U_{(C, C^+)})$ and with the base extension functor ${\rm Bun}(X_{(C, C^+)})\to {\rm Bun}(X_{(C', C'^+)})$ is itself exact.   For b), one uses that the base change under $O_E\to O_{E'}$ of a parahoric model is again a parahoric model, cf.  \cite[p. 25]{An}. 
\end{proof} 
\begin{proposition}\label{unrext}
Assume that $G$ is unramified. Then $G$ satisfies Conjecture \ref{paraextconj}.
\end{proposition}
\begin{proof}
Let $\CG$ be the reductive model of $G$ over $O_E$. Then the proof of \cite[Prop. 8.5]{An} applies since there exists an embedding $\CG\hookrightarrow \GL_n$ such that the quotient $\GL_n/\CG$ is affine. 
\end{proof}
\begin{proposition}\label{extconjfunct}
\noindent (i)  $G$  satisfies the Extension Conjecture if and only if $G_\ad$ does. 

\smallskip

\noindent (ii) The class of groups  satisfying the Extension Conjecture is closed under  direct products. 

\smallskip

\noindent (iii) Let $E'/E$ be a finite extension. If $G'$ over $E'$  satisfies the Extension Conjecture, then so does $\Res_{E'/E}(G')$.  Let $G$ be over $E$ and let $G'=G\otimes_E E'$. If $G'$ satisfies the Extension Conjecture, then so does $G$, provided that $E'/E$ is an unramified extension. 
\end{proposition}
\begin{proof}
Statement (i) is \cite[Prop. 9.6 and Lem. 9.8]{An}. Statement (ii) is trivial. The first statement of (iii) follows, using Lemma \ref{redalgcl}, from Proposition \ref{cohoV} by Shapiro's lemma. For the last statement in (iii), we note that the base change $\CG\otimes_{O_E}O_{E'}$ is again a parahoric model. Since $O_{E'}=W(k')\otimes_{W(k)}O_E$, we have  $ W(C^+)\otimes_{W(k')}O_{E'}=W(C^+)\otimes_{W(k')}({W(k')}\otimes_{W(k)}O_{E})=W(C^+)\otimes_{W(k)}O_{E}$. Hence  the argument in the proof of \cite[Lemma 9.1]{An} applies.  
\end{proof}

\subsection{Descent under a tamely ramified extension}
Let $E'/E$ be a tamely ramified finite extension, which is Galois with Galois group $\Gamma$. For simplicity of notation, set $O'=O_{E'}$ and $O=O_E$ and denote by $\breve O'$ and $\breve O$ the completions of the corresponding maximal unramified extensions. Let $G$ be a reductive group over $E$ and let $\CG$ be a parahoric model. By definition, $\CG=\CG_x^\circ$ where $x\in \CB(G, E)$ is a point in the (extended) Bruhat-Tits
building, and $\CG_x$ is the Bruhat-Tits smooth group scheme over $O$ with $\breve O$-points given the stabilizer
\[
\CG_x(\breve O)=\{g\in G(\breve E)\ |\ g\cdot x=x\} ,
\]
and $\CG_x^\circ$ is the neutral component. Set $G'=G\otimes_E E'$. Recall that by work of Prasad-Yu \cite{PYu}, since $E'/E$ is tame, we have an equivariant identification
\[
\CB(G, E)=\CB(G', E')^\Gamma.
\]
The point $x$ can be considered also as a point of $\CB(G', E')$ which is fixed by $\Gamma$. 

By functoriality, the corresponding
Bruhat-Tits group scheme $\CG'_x$ over $O'$ supports a semilinear action of $\Gamma$
and so does its neutral component $\CG'=(\CG'_x)^\circ$.
Hence, the Weil restrictions of scalars 
\[
{\rm Res}_{O'/O}(\CG'_x), \quad {\rm Res}_{O'/O}(\CG'),
\]
are smooth affine group schemes over $O$ with an action of $\Gamma$. 
Note that the second group scheme is also connected. By \cite[Prop. 1.3.9]{KP}, we have
\begin{equation}\label{tamedescBT}
\CG_x\cong {\rm Res}_{O'/O}(\CG'_x)^\Gamma.
\end{equation}
Let us recall here the argument for the proof of this isomorphism:
The fixed point loci
\[
{\rm Res}_{O'/O}(\CG'_x)^\Gamma, \quad {\rm Res}_{O'/O}(\CG')^\Gamma,
\]
are smooth (using tameness and \cite[Prop. 3.4]{Edix}). 
We have 
\[
\CG_x(\breve O)=\{g\in G(\breve E)\ |\ g\cdot x=x\}=\{g'\in G(\breve E')^\Gamma\ |\ g'\cdot x=x\}=\CG'_x(\breve O')^\Gamma.
\]
Now \eqref{tamedescBT} follows from the   characterization of the Bruhat-Tits group schemes as the unique smooth group schemes with
 group of $\breve O$-points given by the stabilizer.

Note that this gives a closed group scheme immersion $\CG_x\hookrightarrow {\rm Res}_{O'/O}(\CG'_x)$ which, by adjunction, induces a group scheme homomorphism $\CG_x\otimes_O O'\to \CG'_x$. 

Set $X'=\Spec(A_{E'})$, $X=\Spec(A_E)$ and denote by $U'$, resp. $U$, the complement
of the closed point. We have $X'=X\otimes_O O'$ and 
$U'=U\otimes_O O'$. The following proposition extends the second statement of Proposition \ref{extconjfunct}, (iii), which concerned unramified extensions.

\begin{proposition}\label{tamedescent} Let $E'/E$ be a finite tamely ramified Galois extension, with Galois group $\Gamma$. Let $G$ be a reductive group over $E$ and let $G'=G\otimes_E E'$. If $G'$ satisfies the Extension Conjecture \ref{paraextconj}, then so does $G$.
\end{proposition}

\begin{proof}  By Proposition \ref{cohoV} and the argument in the proof of \cite[Theorem 11.4]{An}, we see that 
it is enough to assume $C^+=O_C$ throughout. By Proposition \ref{extconjfunct} (i) and (ii), we may assume first that $G_\ad$ is simple and, passing through $G_\ad$, that $G$ is also simply connected.  It is enough to consider parahoric group schemes $\CG$ and $\CG'$ as above and, assuming that every $\CG'$-torsor over $U'$ extends to $X'$, to  show that every $\CG$-torsor over $U$ extends to $X$. Since $G$ is simply connected,  $\CG'=\CG'_x$ and $\CG=\CG_x= {\rm Res}_{O'/O}(\CG')^\Gamma$. 

Using Lemma \ref{redalgcl} a), we can reduce to the case that $C$ is algebraically closed, and by Lemma \ref{redalgcl} b) that  $k$ is also algebraically closed and that $X$ and  $X'$ are strictly local. Hence, every $\CG$-, resp. $\CG'$-torsor, over $X$, resp. $X'$, is trivial.

Note that $\CG'$ over $O'$ is a smooth affine 
$\Gamma$-group scheme, i.e. it affords a ($O'$-semilinear) $\Gamma$-action which is compatible with the Hopf $O'$-algebra
structure on $O(\CG')$.  We can make sense of the notion of a $(\Gamma, \CG')$-torsor over $X'$ or $U'$ as in Balaji-Seshadri \cite[\S4]{BalaS}.

 Let $\CP$ be a $\CG$-torsor over $U$. The base change of $\CP$ by $U'\to U$ followed by the push-out by $\CG\otimes_O O'\to \CG'$ gives naturally a $(\Gamma,\CG')$-torsor
$\CP'$ over $U'$. Recall that we assume that all $\CG'$-torsors over $U'$ extend to $X'$. Consider the $\CG'$-torsor $\CP'$ over $U'$ obtained by forgetting  the $\Gamma$-structure. This  extends to a $\CG'$-torsor 
$\widetilde\CP'$ over $X'$; by faithfully flat descent, this is affine and given by a flat $A_{E'}$-algebra $O(\widetilde \CP')$. Now observe that 
the restriction from $X'$ to $U'$ gives a fully faithful functor from affine flat schemes over $X'$ to affine flat schemes
over $U'$ (see \cite[Prop. 8.2]{An}; this uses Lazard's theorem to write a general flat algebra such as $O(\widetilde\CP')$ as a direct limit of finite free modules.
Alternatively, using the Tannakian equivalence to reduce to vector bundles, we see that the restriction of $\CG'$-torsors over $X'$ to $\CG'$-torsors over $U'$ is fully faithful, cf. \cite{An} Lemma 8.4.) By applying this full faithfulness to the isomorphisms given by 
the elements of the Galois group, we obtain that the $\Gamma$-action on $\CP'$ extends to a $\Gamma$-action on $\widetilde \CP'$. Hence, we obtain a $(\Gamma, \CG')$-torsor over $X'$ which enhances $\widetilde\CP'$.
\smallskip

\noindent {\it \ Claim:  ${\rm Res}_{X'/X}(\wt\CP')^\Gamma$ is a $\CG$-torsor over $X$ which extends the $\CG$-torsor $\CP$ over $U$.}

\smallskip

The proof that follows is due to Scholze. What has to be shown is that the fiber of ${\rm Res}_{X'/X}(\wt\CP')^\Gamma$ over the special point $s\in X$ is non-empty. Indeed, if this fiber is non-empty, a section can be lifted to a section of ${\rm Res}_{X'/X}(\wt\CP')^\Gamma$ over $X$ (use the smoothness of ${\rm Res}_{X'/X}(\wt\CP')^\Gamma$), hence ${\rm Res}_{X'/X}(\wt\CP')^\Gamma$ is a (trivial) $\CG$-torsor over $X$. 

Assume, by way of contradiction, that the fiber of ${\rm Res}_{X'/X}(\wt\CP')^\Gamma$ over $s$ is empty. Then $\CP={\rm Res}_{U'/U}(\CP')^\Gamma={\rm Res}_{X'/X}(\wt\CP')^\Gamma$ is an affine scheme. Hence, $\CP\times_{U}U'=\CP\otimes_O O'$ is also affine. Consider the push-out morphism
$$
\pi\colon \CP\times_{U}U'\to \CP'. 
$$
It induces a map on cohomology,
\begin{equation}\label{pusho}
{\rm H}^1(\CP', \CO)\to {\rm H}^1(\CP\times_U U', \CO) .
\end{equation}
The map \eqref{pusho} is an isomorphism up to bounded $p$-torsion. Indeed, since $\pi$ is an affine morphism, the map  \eqref{pusho}  is induced by the map of sheaves on $\CP'$ given by 
$$
\CO_{\CP'}\to \pi_*(\CO_{\CP\times_U U'}) ,
$$
and this map is injective, with cokernel a skyscraper sheaf on $U'\otimes_{O'} k$. Now ${\rm H}^1(\CP\times_U U', \CO)=0$ since $\CP\times_U U'$ is affine. 
Since $\CP'$ is a trivial $\CG'$-torsor over $U'$, the source of this map can be identified with ${\rm H}^1(U', \CO)\otimes {\rm H}^0(\CG', \CO)$. Since ${\rm H}^1(U', \CO)={\rm H}^2_{\mathfrak m}(W_{E'}(O_C))$ is not of bounded $p$-torsion (it contains the images of $\frac{1}{\pi^a[\varpi]^b}\in W_{E'}(O_C)_{\pi[\varpi]}$ for any $a>0, b>0$), this is the desired contradiction. 
\end{proof}

\begin{corollary}\label{cortame}
 If there exists a tamely ramified extension $E'$ of $E$ such that $G'=G\otimes_E E'$ is of the form $G'\simeq \Res_{\tilde E'/E}(\tilde G)$, where $\tilde E'$  is a finite extension  of $E'$ and  $\tilde G$ is an unramified group over $\tilde E'$, then $G$ satisfies the Extension Conjecture \ref{paraextconj}. 
\end{corollary}
\begin{proof}
 This follows from Proposition \ref{tamedescent}, Proposition \ref{extconjfunct}, (iii) and Corollary \ref{unrext}.
\end{proof}
\subsection{Summary}
 Now combining everything above, we  obtain the following result on the Extension Conjecture. We introduce the following terminology. We call a reductive group $G$ over $E$ \emph{essentially tamely ramified} if
  $
 G_{\rm ad}\simeq \prod_i {\rm Res}_{E_i/E} (H_i),
$
  where, for all $i$, $H_i$ splits over a tamely ramified extension of $E_i$.
 \begin{remark}\label{Gacc}
 A reductive group $G$ is essentially tamely ramified under either of the following hypotheses:
 \begin{itemize}
 \item[a)]  If $p\geq 5$.
 \item[b)] If $p=3$ and $G_{\rm ad}\otimes_E\breve E$ has no simple factors of type $D_4^{(3)}$ or $D_4^{(6)}$
 (ramified triality).
 \end{itemize}
 Indeed, it is enough to show the following: If $G$ is an adjoint simple group $G$ over $E$ and, either $p\geq 5$, or $p=3$ and the condition in (b) above is satisfied, then $G\simeq {\rm Res}_{E'/E} (G')$, where
$G'$ splits over a tamely ramified extension of $E'$. We can always write $G\simeq {\rm Res}_{E'/E} (G')$ and by Steinberg's theorem, $G'\otimes_{E'} \breve E'$ is quasi-split. Hence, $G'\otimes_{E'} \breve E'$ is isomorphic to
the quasi-split outer form of its split form $H'$ and can be written 
\[
G'\otimes_{E'} \breve E'={\rm Res}_{E''/\breve E'}(H'\otimes_{\BQ_p} \breve E')^\Gamma ,
\]
where $E''/\breve E'$ is Galois with (inertial) group $\Gamma$ which acts 
via diagram automorphisms. By examining the possible local Dynkin diagrams (comp. \cite[\S 7a]{PRTwisted}), we see
that $e=[E'':\breve E']$ can only take the values $1$, $2$ and $3$. Hence,
if $p\geq 5$, $G'$ splits over a tamely ramified extension.
 We have $e=3$ only in one case, when the local Dynkin diagram is of type $G^I_2$ which 
 corresponds to the ramified triality $D_4^{(3)}$ or  $D_4^{(6)}$; this shows the result in case (b).
 \end{remark}

 \begin{theorem}\label{extTHM}
 Let $G$  be a reductive group over $E$ which is essentially tamely ramified. Then $G$ satisfies the Extension Conjecture \ref{paraextconj}. 
\end{theorem}
 \begin{proof}
 By Proposition \ref{extconjfunct} (i), (ii), we can assume that $G$ is adjoint and simple. Then  $G={\rm Res}_{E'/E}(G')$,
 and $G'$ splits over a tamely ramified extension.  By Proposition \ref{extconjfunct} (iii), we reduce to the case
 that $G$ splits over a tamely ramified extension. Finally,  by Proposition \ref{tamedescent}, we reduce to the case that 
 $G$ is, in fact, split. Then the result follows from Proposition \ref{unrext}.
 \end{proof}
 
 \begin{corollary}
 a) If $p\geq 5$, then $G$ satisfies the Extension Conjecture \ref{paraextconj}. 
 
 b) If $p=3$ and $G_{\rm ad}\otimes_E\breve E$ has no simple factors of type $D_4^{(3)}$ or $D_4^{(6)}$
 (ramified triality), then $G$ satisfies the Extension Conjecture \ref{paraextconj}. 
 \end{corollary}
 
\begin{proof} 
This follows from Theorem \ref{extTHM} and Remark \ref{Gacc}. \end{proof}

\begin{remark}
In the proof in Remark \ref{Gacc} above, we have $e=\#\Gamma=2$, when the local Dynkin diagram is of types $B$-$C_n$
($n\geq 3$), $C$-$B_n$ ($n\geq 2$), $C$-$BC_n$ ($n\geq 1$), or $F^I_4$.  The types $B$-$C_n$, $C$-$BC_n$, correspond to ramified unitary groups and $C$-$B_n$
to (ramified) even orthogonal groups. The cases of wildly ramified unitary groups (for $p=2$) can be handled
by using unpublished  work of Kirch, as shown in the first version of \cite[\S 9]{An}. Ansch\"utz's result covers this case, as well as the wildly ramified triality group (for $p=3$), wildly ramified even orthogonal groups (for $p=2$), 
and the wildly ramified outer form of $E_6$ (for $p=2$). 
\end{remark}

\subsection{On the extension theorem of \cite{KP}}
As before, let $G$ be a reductive group over $E$ and $\CG$ a parahoric model over $O_E$. In \cite{KP}, it is assumed that $O_E=W(k)$, where $k$ is a finite field or algebraically closed. The following generalizes \cite[Prop. 1.4.3]{KP}.
\begin{proposition}\label{KPnew}
Assume that $G$ is essentially tamely ramified.    Then any $\CG$-torsor over $\Spec (O_E\lps u\rps)\setminus \{\frak m\}$ extends to $\Spec (O_E\lps u\rps)$. (Here $\frak m$ denotes the maximal ideal of the local ring $O_E\lps u\rps$.)
\end{proposition}

\begin{proof}
 Let $C$ be a perfect non-archimedean field which is a $k$-algebra, and let 
 $$
 f\colon O_E\lps u\rps\to W_E(O_C)
 $$  be the homomorphism defined by sending $u$ to $[\varpi]$, for a fixed pseudouniformizer $\varpi\in \frak m_C$. Then $f$ is faithfully flat, cf. \cite[Lem. 10.2]{An}. Hence Proposition \ref{KPnew} follows from Theorem \ref{extTHM} by using descent  \cite[Prop. 8.2, Lemma 8.3]{An}, comp. the argument in the proof of Proposition \ref{tamedescent}. 
 \end{proof}

\begin{remark}
Again, Ansch\"utz proves this in complete generality, i.e.,  without any tameness hypothesis, as a consequence of his proof of the Extension Conjecture \ref{paraextconj}, cf. \cite[Prop. 10.3]{An}.
\end{remark}

\vfill\eject

\end{document}